%% file: book.tex
\documentclass[graybox,envcountchap,sectrefs]{svmono}

\usepackage{makeidx}         % allows index generation
\usepackage[utf8]{inputenc}
\usepackage{bm}
\usepackage{amsmath}
\usepackage{amssymb}
\usepackage{booktabs}
\usepackage{amsfonts}
\usepackage{amsthm}
\usepackage{array}
\usepackage{url}
\usepackage{multirow}
\usepackage{multicol}  
\usepackage{hyperref}
\usepackage{wrapfig}
\usepackage{longtable}
\usepackage{enumitem}
\usepackage{dsfont}
\usepackage[round]{natbib}

\usepackage{longtable}

\makeindex

\usepackage{hyperref}
\hypersetup{
    bookmarks=true,         % show bookmarks bar?
    unicode=false,          % non-Latin characters in Acrobat’s bookmarks
    pdftoolbar=true,        % show Acrobat’s toolbar?
    pdfmenubar=true,        % show Acrobat’s menu?
    pdffitwindow=false,     % window fit to page when opened
    pdfstartview={FitH},    % fits the width of the page to the window
    pdftitle={My title},    % title
    pdfauthor={Author},     % author
    pdfsubject={Subject},   % subject of the document
    pdfcreator={Creator},   % creator of the document
    pdfproducer={Producer}, % producer of the document
    pdfkeywords={key1, key2}, % list of keywords
    pdfnewwindow=true,      % links in new PDF window
    colorlinks=true,        % false: boxed links; true: colored links
    linkcolor=blue,         % color of internal links (change box color with linkbordercolor)
    citecolor=blue,         % color of links to bibliography
    filecolor=blue,         % color of file links
    urlcolor=cyan           % color of external links
}

\allowdisplaybreaks

\def\P{{\mathrm P}}
\def\Q{{\mathrm Q}}
\def\H{{\mathrm H}}
\def\W{{\mathrm W}}
\def\PP{{\mathbb{P}}}
\def\QQ{{\mathbb{Q}}}
\def\HH{{\mathbb{H}}}
\def\E{{\mathrm E}}
\def\cov{{\rm Cov}}
\def\cor{{\rm Cor}}
\def\var{{\rm Var}}
\def\ind{{\mathds 1}}
\def\d{{\mathrm d}}
\def\sign{{\rm sign}}
\def\ba{{\bm a}}
\def\bd{{\bm d}}
\def\bv{{\bm v}}
\def\bh{{\bm h}}

\def\bk{{\bm k}}
\def\bx{{\bm x}}
\def\by{{\bm y}}
\def\bz{{\bm z}}
\def\beps{{\bm \epsilon}}
\def\btheta{{\bm \theta}}

\def\Ab{\mathbf{A}}
\def\Bb{\mathbf{B}}

\def\Db{\mathbf{D}}

\def\Pb{\mathbf{P}}
\def\Qb{\mathbf{Q}}
\def\Ub{\mathbf{U}}
\def\Vb{\mathbf{V}}
\def\Wb{\mathbf{W}}
\def\Xb{\mathbf{X}}
\def\Ib{\mathbf{I}}
\def\bSigma{\mathbf{\Sigma}}

\def\bbeta{{\bm \beta}}
\def\bb{{\bm b}}
\def\bc{{\bm c}}
\def\bp{{\bm p}}

\def\bt{{\bm t}}
\def\bu{{\bm u}}
\def\bv{{\bm v}}
\def\bW{{\bm W}}
\def\bX{{\bm X}}
\def\bY{{\bm Y}}
\def\bZ{{\bm Z}}
\def\bpi{{\bm \pi}}
\def\bsigma{{\bm \sigma}}

\def\bI{{\bm I}}
\def\bmu{{\bm \mu}}
\def\cA{\mathcal{A}}
\def\cB{\mathcal{B}}
\def\cO{\mathcal{O}}
\def\cH{\mathcal{H}}
\def\cI{\mathcal{I}}
\def\cJ{\mathcal{J}}

\def\cN{\mathcal{N}}
\def\cC{\mathcal{C}}
\def\cD{\mathcal{D}}
\def\cE{\mathcal{E}}
\def\cF{\mathcal{F}}
\def\cG{\mathcal{G}}
\def\cL{\mathcal{L}}
\def\cM{\mathcal{M}}
\def\cS{\mathcal{S}}
\def\cU{\mathcal{U}}
\def\cW{\mathcal{W}}
\def\cX{\mathcal{X}}
\def\cZ{\mathcal{Z}}
\def\reals{\mathbb{R}}
\def\GG{\mathbb{G}}
\def\BB{\mathbb{B}}

\def\Pr{\mathrm{Pr}}

\def\diam{{\rm diam}}
\def\op{{\rm op}}
\def\tr{{\rm tr}}
\DeclareMathOperator*{\argmax}{arg\,max} 
\DeclareMathOperator*{\argmin}{arg\,min}
\newcommand{\norm}[1]{\|#1\|}
\newcommand{\bnorm}[1]{\Big\|#1\Big\|}

\let\hat\widehat
\let\tilde\widetilde
\let\bar\overline

\title{An Introduction to Permutation Processes (version 0.5)}

\begin{document}

\author{Fang Han}
\maketitle

\frontmatter
\input{chapters/preface}

\input{chapters/notation}
\tableofcontents

\mainmatter
\include{chapters/intro}
\include{chapters/random-permutation}
\include{chapters/basic-prob}

\include{chapters/combin-prob}
\include{chapters/combin-process}
\include{chapters/ep}
\include{chapters/cp}
\include{chapters/rmt}

\include{chapters/stat-app}

\include{chapters/m-est}
\include{chapters/per-test}

\printindex

\bibliographystyle{apalike}
\bibliography{AMS}

\end{document}

%% file: chapters/preface.tex
\preface
\addcontentsline{toc}{chapter}{Preface}

This book is an introduction to the theory of stochastic processes whose randomness involves only a random permutation. Such randomnesses can occur either by design (e.g., through a simple sampling without replacement or a randomized controlled trial) or as a consequence of certain data analysis operations (e.g., through ranking).

The book pursues three objectives. First, it aims to provide an exposition on the theory of permutation statistics. The classical foundation of permutation statistics was summarized in the book by \cite{sidak1999theory}. Permutation inequalities and central limit theorems were well-established as early as the 1970s. More recently, a deeper understanding of the connections between permutation statistics and Stein's methods has emerged, accompanied by the development of new tools and results, as partly outlined in works such as \cite{chen2010normal} and \cite{chatterjee2005concentration}. Part I of this book is dedicated to giving an exposition to this foundational theory.

The second objective is to construct a theory of permutation processes within the framework of the already established empirical process theory. In empirical process theory, randomness stems from independent sampling from a specific distribution, corresponding to a super-population perspective on the sampling paradigm. Conversely, in permutation processes, randomness is solely derived from a permutation, representing a finite-population view of sampling. The classical theory of finite-population sampling was summarized in \cite{hajek1981sampling} and \cite{fuller2011sampling}. Part II of this book presents alternative results that more closely align with the theory of stochastic processes, akin to the approaches adopted in, e.g., the works of \cite{vaart1996empirical} and \cite{gine2021mathematical}.

The third objective of this book is to apply the developed permutation process theory to various statistical applications, for now merely focusing on the theory of M- and Z-estimators and permutation tests. The inherent triangular array nature of permutation randomness, as observed by numerous researchers, calls for Lindeberg-Feller-type analyses of stochastic processes. They occasionally demand nuanced treatments compared to classical arguments. On the other hand, the exploration of finite-population statistical inference, as evidenced by the research endeavors of successive generations of statisticians, holds promise for offering alternative perspectives on a range of data analysis tasks, particularly those rooted in {\it design-based inference}, and can deliver more accurate and robust uncertainty quantification. Part III of this book represents the author's initial efforts, inspired by his colleagues, towards a more systematic treatment of permutation statistical inference. Subsequent enrichment of this section is anticipated in the coming years.

\section*{Underlying philosophy}

The author's intent is to craft a {\it simple and short} book, featuring minimal notation and concise proofs. The inherent elegance of uniform permutation greatly facilitates this aim. As elucidated in Part II, by confining the support to a finite set, the typically intricate issue of measurability, which holds significant sway in empirical process analyses, can be elegantly resolved in a rigorous manner. Furthermore, the formidable results established by Bobkov \citep{bobkov2004concentration} and Tolstikhin \citep{tolstikhin2017concentration} endow Talagrand-type inequalities with remarkable utility within the permutation framework.

Simplicity also entails the author's deliberate selectivity regarding the inclusion of results. It is important to remark that this book exclusively focuses on permutation objects of {\it statistical relevance} and is confined to a narrow range of structures. Consequently, practitioners in fields like survey sampling or causal inference may find themselves grappling with the absence of the precise results they require, particularly those involving more intricate sampling methodologies. Conversely, researchers with a penchant for theoretical exploration may observe that this book barely touches upon modern combinatorial theories, such as those concerning graphs and other complex structures. However, despite these limitations, it is hoped that this book can still offer utility to researchers whose interests align with the book's scope.

\subsection*{Reading guide}

This book is structured into three parts, each comprising several chapters. Part I comprises three chapters and is aimed at laying the groundwork for the subsequent discussions. Chapter \ref{chapter:flavors} provides a brief overview of the permutation objects under examination in this work, alongside establishing essential notation. Chapter \ref{chapter:basic-prob} equips the reader with foundational knowledge in mathematical analysis, probability theory, and statistics, essential for understanding the subsequent discussions. Additionally, it emphasizes the triangular array framework, accommodating scenarios where the data-generating distribution varies with the sample size. Chapter \ref{chapter:combin-prob} introduces Stein-type analyses of permutation statistics, giving an exploration of weak convergence theorems and moment inequalities.

Part II concerns stochastic processes stemmed from a permutation measure and indexed by a class of functions. It is structured to align with their empirical process counterparts. Chapter \ref{chapter:ep} introduces elementary empirical process theory, covering Dudley's metric entropy bounds (Chapter \ref{chapter:dudley}), Glivenko-Cantelli theorems (Chapter \ref{chapter:gc}), and Donsker theorems (Chapter \ref{chapter:donsker}). Chapter \ref{chap:cpt} delves into the theory of permutation processes, corresponding to simple samplings without replacement from a finite population. Within this chapter, Chapter \ref{chap:rosen} outlines the stochastic structure of such samplings, while Chapter \ref{chap:cmi} presents maximal inequalities essential for stochastic process analyses. Chapters \ref{chap:cti}, \ref{chap:cgc}, and \ref{chap:cd} are dedicated to the permutational Talagrand inequality, Glivenko-Cantelli bounds, and Donsker bounds, respectively, constituting the foundational components of this book. Lastly, Chapter \ref{chapter:crmt} focuses on a special stochastic process, the spectrum of a combinatorial random matrix, providing a specialized exploration within this domain.

Part III is currently a work in progress and comprises two chapters. Chapter \ref{chap:m-est} concerns M- and Z-estimation theory in scenarios where data are sampled without replacement from a finite population. The triangular array setting introduces subtle differences from classical arguments, warranting careful consideration.  In Chapter \ref{chapter:rbs}, preliminary results on permutation tests within a finite-population setting are presented. %However, the current focus is primarily on the unconditional performance of these tests. The author anticipates adding the conditional version in subsequent updates.

\subsection*{Why writing this book}

The author's motivation to write this book is two-fold. 

First, to his knowledge, in the literature {\it there has not been a book that relates probabilistic permutation theory, finite-population statistical inference, and empirical process theory}. It is believed that topics like Stein's method, combinatorial moment and Talagrand inequalities, permutation stochastic process theory, and the connections between them to mathematical statistics—particularly in the realm of hypothesis testing and confidence intervals for complex statistical inference problems—are not adequately covered in any book; their significance will only grow over time. 

Second, the author wishes to give a more accessible introduction to the theory of empirical processes, which holds a foundational position in statistics, without sacrificing any mathematical rigor. The simplicity of permutation objects, for the first time, makes it possible.

\subsection*{Intended audience}

This book has been used for a special topics course the author gave in the Department of Statistics at the University of Washington, Seattle, in Spring 2024. Accordingly, it is initially designed as a textbook for one-quarter or one-semester graduate-level study of finite-population statistical inference, or, if more ambitious, empirical process theory from a finite-population perspective. In this regard, it aims at {\it graduate students in regular statistics, biostatistics, or econometrics programs}.

This book could also be used for self-study and is suitable for {\it any researcher} with half a year of graduate-level measure/probability theory and a year of graduate-level mathematical statistics, who is {\it interested in learning statistical theory for inferring permutation objects}. In this regard, we hope that this book can clarify important concepts, provide useful theoretical justifications, and be beneficial to researchers in statistics, biostatistics, econometrics, and any other field whose research involves quantifying randomness stemming from permutations.

%% file: chapters/notation.tex
\chapter*{Notation}
\label{appendix:notation}
\addcontentsline{toc}{chapter}{Notation}

This section puts down the abbreviation and notation this book tries to keep coherently throughout. 

\subsection*{Abbreviations}

\begin{longtable}{p{0.3\linewidth}p{0.7\linewidth}}
CDF & cumulative distribution function;\\
i.i.d. & independent and identically distributed; \\
LLN & the law of large numbers, either weak or strong;\\
CLT & central limit theorem;\\
\end{longtable}

\subsection*{Symbols}

\subsubsection*{Reserved symbols}
\begin{longtable}{p{0.3\linewidth}p{0.7\linewidth}}
$\pi$ & the uniform random permutation;\\
$n$ & the size of the sample;\\
$N$ & the size of the (finite) population;\\
$\Ib_d$ & the $d$-dimensional identity matrix.\\
\end{longtable}

\subsubsection*{Sets}
\begin{longtable}{p{0.3\linewidth}p{0.7\linewidth}}
$\emptyset$ & the empty set;\\
$\reals$ & the set of real numbers;\\
$\reals^{\geq 0}$ & the set of nonnegative real numbers;\\
$\reals^{>0}$ & the set of positive real numbers;\\
$\reals^p$ & the set of $p$-dimensional real vectors;\\
$\mathbb{Z}$ & the set of integers;\\
$\mathbb{N}$ & the set of natural numbers;\\
$\mathbb{Q}$ & the set of rational numbers;\\
$[N]$ & the set of integers from 1 to $N$: $\{1,2,\ldots, N\}$.\\
\end{longtable}

\subsubsection*{Functions}
\begin{longtable}{p{0.3\linewidth}p{0.7\linewidth}}
$K$, $C$, $C'$, ... & the generic positive finite constants;\\
$O(\cdot)$, $o(\cdot)$ & the big-$O$ and small-$o$ notation;\\
$O_{\Pr}(\cdot)$ and $o_{\Pr}(\cdot)$ & the stochastic big-$O$ and small-$o$ notation in a probability space $(\Omega,\cA,\Pr)$;\\
$\norm{\cdot}_2$ & the Euclidean norm;\\
$|\cdot|$ & the cardinality of a set;\\
$\ind(\cdot)$ & the indicator function;\\
$\diam(\cdot; d)$ & the diameter of a set in a metric space $(T,d)$;\\
$\vee$ and $\wedge$ & the maximum and the minimum between the two;\\
$\psi_p(\cdot)$ & the function $e^{x^p}-1$ with domain $[0,\infty)$;\\
$\gtrsim$ and $\lesssim$ & asymptotically (up to a constant) greater and less than;\\
$\Phi(\cdot)$ & the CDF of the standard Gaussian.\\
\end{longtable}

\subsubsection*{Fonts}
\begin{longtable}{p{0.3\linewidth}p{0.7\linewidth}}
$\mathcal{X}$, $\mathcal{S}$, ... & Calligraphic font typically refers to sets.\\
$\bm{x}$, $\bm{z}$, ... & Bold italic font typically refers to vectors.\\
$\Xb$, $\Ab$, ... & Bold upright font typically refers to matrices.\\
$x$, $\bm{x}$, ... & Lowercase letters typically refers to non-random objects/vectors.\\
$X$, $\bm{X}$, ... & Uppercase letters typically refers to random variables/vectors.\\
\end{longtable}

\subsubsection*{Probabilities and Distributions}
\begin{longtable}{p{0.3\linewidth}p{0.7\linewidth}}
$(\Omega, \cA, \Pr)$ & the probability space; \\
$\E[\cdot]$ & the expected value of a random variable; \\
$\var(\cdot)$ & the variance of a random variable; \\
$\cov(\cdot)$ & the covariance of two random variables;\\
$\cor()$ & the correlation of two random variables;\\
$\P$, $\Q$ & the laws of random variables;\\
$\PP_n$ & the empirical measure;\\
$\P_N$ & the finite-population measure;\\
$\PP_{\pi,n}$ & the permutation measure.
\end{longtable}

%% file: chapters/intro.tex
\begin{partbacktext}
\part{Introduction}
\end{partbacktext}

%% file: chapters/random-permutation.tex
\chapter{Uniform Random Permutation}
\label{chapter:flavors}

\section{Uniform random permutation}

The primary focus of this book is on uniform random permutation and permutation statistics/processes as its functionals. 

A uniform random permutation is a {\it random mapping},
\[
\pi (=\pi_N): [N] \to [N],
\]
such that
\[
\Pr\Big(\pi(1)=i_1,\pi(2)=i_2,\ldots,\pi(N)=i_N\Big)=\frac{1}{N!}
\]
for any rearrangement $(i_1,\ldots,i_N)$ of $[N]:=\{1,2,\ldots,N\}$. Define $\cS_N$ to be the permutation group over $[N]$. The random mapping $\pi$ is a uniform distribution over $\cS_N$. %, and can simply called a {\it a uniform random permutation}.

Some basic properties for $\pi$ are listed below.

\begin{proposition}\label{prop:basic} For any $k\in[N]$, it holds true that
\begin{enumerate}[label=(\roman*)]
\item\label{prop:basic1} for any $i_1,\ldots,i_k\in[N]$,
\[
\Pr\Big(\pi(1)=i_1,\ldots,\pi(k)=i_k\Big)=\frac{(N-k)!}{N!}\ind(i_1\ne i_2\ne\cdots\ne i_k);
\]
\item for any $q\in[N]$,
\[
\Pr\Big(\pi(1)\leq q, \pi(2)\leq q, \ldots,\pi(k)\leq q\Big)=\frac{q(q-1)\cdots(q-k+1)}{N(N-1)\cdots(N-k+1)}\ind(q\geq k);
\]
\item we have
\[
\E\Big[\pi(1)\pi(2)\cdots\pi(k)\Big]=\frac{(N-k)!}{N!}\sum_{i_1\ne\cdots\ne i_k}i_1i_2\cdots i_k;
\]
in particular,
\begin{align*}
\E\Big[\pi(1)\Big]=\frac{N+1}{2}, ~~~\var\Big(\pi(1)\Big)=\frac{N^2-1}{12}, \\
~~{\text and }~~\cor\Big(\pi(1),\pi(2)\Big)&=-\frac{1}{N-1}~{\rm as}~ N>1;
\end{align*}
\item for any real sequence $a_1,a_2,\ldots,a_N$, we have
\begin{align*}
\E\Big[a_{\pi(1)}a_{\pi(2)}\cdots a_{\pi(k)}\Big]=\frac{(N-k)!}{N!}\sum_{i_1\ne\cdots\ne i_k}a_{i_1}a_{i_2}\cdots a_{i_k};
\end{align*}
in particular,
\begin{align*}
\E\Big[a_{\pi(1)}\Big]&=\frac1N\sum_{i=1}^Na_i,~~\var\Big(a_{\pi(1)}\Big)=\frac1N\sum_{i=1}^Na_i^2-\Big(\frac{1}{N}\sum_{i=1}^Na_i \Big)^2,\\
~~{\text and }~~\cor\Big(a_{\pi(1)},a_{\pi(2)}\Big)&=-\frac{1}{N-1}~~{\rm if}~\var\Big(a_{\pi(1)}\Big)>0;
\end{align*}
\item defining 
\[
\bar a_{n,\pi}=n^{-1}\sum_{\pi(i)\leq n}a_i,
\]
it then holds true that
\[
\E\Big[\bar a_{n,\pi}\Big]=\frac{1}{N}\sum_{i=1}^Na_i=:\bar a_N,~~
\var\Big(\bar a_{n,\pi}\Big)=\frac{N-n}{Nn}\cdot \frac{1}{N-1}\sum_{i=1}^N(a_i-\bar a_N)^2,
\]
and
\[
\E\Big[\frac{1}{n-1}\sum_{\pi(i)\leq n}(a_i-\bar a_{n,\pi})^2\Big]=\frac{1}{N-1}\sum_{i=1}^N(a_i-\bar a_N)^2;
\]
\item for any two real sequences $\{a_i;i\in[N]\}$ and $\{b_i;i\in[N]\}$ with $\bar b_{n,\pi}$ and $\bar b_N$ similarly defined as above, we have
\[
\cov(\bar a_{n,\pi}, \bar b_{n,\pi})=\frac{N-n}{Nn}\cdot\frac{1}{N-1}\sum_{i=1}^N(a_i-\bar a_N)(b_i-\bar b_N)
\]
and
\[
\E\Big[\frac{1}{n-1}\sum_{\pi(i)\leq n}(a_i-\bar a_{n,\pi})(b_i-\bar b_{n,\pi}) \Big]=\frac{1}{N-1}\sum_{i=1}^N(a_i-\bar a_N)(b_i-\bar b_N).
\]
\end{enumerate}
\end{proposition}

\begin{exercise}
Prove Proposition \ref{prop:basic}.
\end{exercise}

\begin{proposition}\label{prop:1entropy}
The uniform random permutation $\pi$ maximizes the entropy
\[
H(\sigma):=-\sum_{s\in\cS_N} \Pr\Big(\sigma=s\Big)\log\Big\{\Pr\Big(\sigma=s\Big)\Big\}
\]
among all random permutations over $\cS_N$ if no constrain on the permutation pattern is enforced.
\end{proposition}
\begin{proof}
Use the fact that, for the maximization problem
\begin{align*}
\max _{p} -\sum_{i} p_i\log (p_i)\\
\text{subject to } p_i\geq 0~~{\rm and}~~\sum_i p_i=1,
\end{align*}
the maximum is attained when all $p_i$'s are equal.
\end{proof}

\begin{remark}
H\'ajek \citep[Chapter 3]{hajek1981sampling} advocated using the entropy $H(\sigma)$ as a measure of spread of a permutation distribution. 
\end{remark}

\section{Statistical permutation}

Survey statisticians perceive random permutations as the tool to sample data {\it without replacement}. In the survey sampling literature, this is called {\it simple random sampling without replacement}, although in this book we shall simply call it the {\it permutation sampling}. The idea is simple: given a {\it finite population} 
\[
\Big\{z_1,z_2,\ldots,z_N\Big\},
\]
we sample without replacement a subset of it based on whether, for each $i\in[N]$, whether $\pi(i)\leq n$. Here $n$ is a pre-determined size of the sample. Statistical inference for data collected from the permutation sampling paradigm belongs to the realm of {\it design-based inference}.

Permutation sampling is subtly different from {\it independence sampling}, with which statisticians are much more familiar. In the independence sampling paradigm, it is hypothesized that a researcher samples data points {\it independently} from a {\it superpopulation}, which follows a certain {\it unknown} distribution. Inference based on the independence sampling paradigm belongs to the realm of the {\it model-based inference}.

The two frameworks (design v.s. model) are intrinsically related. First, independence sampling can arise from a random permutation; this is characterized by the following proposition.

\begin{proposition}[permutation v.s. independence sampling]\label{prop:iid}
Suppose that the finite population, $\{X_1,\ldots,X_N\}$, are formed by points that are independently sampled from a superpopulation that is characterized by the {\it law} $\P$. We then have, for any $n\in[N]$,
$X_{\pi^{-1}(1)},\ldots,X_{\pi^{-1}(n)}$ are independently and identically distributed with the same law $\P$.
\end{proposition}
\begin{proof}
Without loss of generality, let us assume that $\P$ has support over $\mathbb{R}$. We then have, by independence between $X_i$'s and $\pi$,
\begin{align*}
&\Pr(X_{\pi^{-1}(1)}\leq t_1,\ldots, X_{\pi^{-1}(n)}\leq t_n)\\
=&\frac{(N-n)!}{N!}\sum_{i_1\ne i_2 \ne \cdots \ne i_n}\Pr(X_{i_1}\leq t_1,\ldots,X_{i_n}\leq t_n)\\
=&\frac{(N-n)!}{N!}\sum_{i_1\ne i_2 \ne \cdots \ne i_n}\Pr(X_{1}\leq t_1)\cdots\Pr(X_{n}\leq t_n)\\
=&\prod_{i=1}^n\Pr(X_i\leq t_i).
\end{align*}
This yields the claim.
\end{proof}

Conversely, a uniform random permutation can arise from an independence sampling through the concept of {\it rank-based statistics}. For any real sequence of random variables $X_1,\ldots,X_n$, define the rank of each entry $X_i$, $i\in[n]$, as
\[
R_i=R_i(\{X_i;i\in[n]\})=\sum_{j=1}^n\ind(X_j\leq X_i).
\] 
It is obvious then that $\{R_i;i\in[n]\}$ is a subset of $[n]$. In particular, when $X_i$'s are distinct, the mapping, $i\to R_i$, is a permutation in $\cS_n$.

The next proposition links ranking to permutation.

\begin{proposition}[Interplay between permutation and ranking] \label{prop:1rank-perm}
Suppose $X_1,\ldots,X_n\in\mathbb{R}$ are independently sampled from a continuous distribution $\P$ so that, with probability 1, there is no tie. Let $\{R_i, i \in[N]\}$ be the rank of $\{X_i, i\in [N]\}$ such that
\[
X_{R_1}<X_{R_2}<\cdots < X_{R_N}.
\]
The following then holds. 
\begin{enumerate}[label=(\roman*)]
\item The mapping $\pi: i \to R_i$ is distributed uniformly over $\cS_n$;
\item $R_i=\sum_{j=1}^N\ind(X_j\leq X_i)=N\hat F(X_i)$, where $\hat F(t):=\frac1N\sum_{i=1}^N\ind(X_i\leq t)$ is the empirical cumulative distribution function (empirical CDF);
\item $R_i/N-F(X_i)$ converges to 0 almost surely\footnote{We have not defined the meaning of ``convergence almost surely; this will be done in the next chapter.''}, where $F(\cdot)$ is the CDF of the probability measure $\P$;
\item letting $\tilde U_i=R_i/N$ and $U_i=F(X_i)$, we then have
\[
\cov(\tilde U_1, U_1)=\frac{n-1}{12n}~~{\rm and}~~\cov(\tilde U_1, U_2)=-\frac{1}{12n}.
\]
\end{enumerate}
\end{proposition}

\begin{exercise}
Please give a proof of Proposition \ref{prop:1rank-perm}.
\end{exercise}

%To wrap up, this book is mainly interested in uniform random permutation aroused in the above two statistical scenario: design-based inference and rank-based statistics. %Any impact of book on areas beyond these two fields, on the other hand, would be a happy surprise to the author.

\section{Permutation statistics}

A permutation statistic is a functional of the random permutation. Early focus is on rank tests, e.g., measuring the disarray of two permutations, $\pi$ and $\sigma$, and the closeness of $\pi$ to a uniform permutation.
\begin{example}[Measure of disarray]\label{eg:disarray}
\begin{enumerate}[label=(\roman*)]
\item Spearman's footrule:\\ $D(\pi,\sigma)=\sum_{i=1}^N|\pi(i)-\sigma(i)|$;
\item Spearman's rho: $\rho(\pi,\sigma)=\sum_{i=1}^N(\pi(i)-\sigma(i))^2$;
\item Kendall's tau: $\tau(\pi,\sigma)=\sum_{i,j}\sign(\pi(i)-\pi(j))\sign(\sigma(i)-\sigma(j))$;
\item Chatterjee's rank correlation: $\xi(\pi,\sigma)=\sum_{i=1}^{N-1}|\pi([i+1])-\pi([i])|$, where the indices $[i]$'s satisfy $\sigma([1])<\sigma([2])<\cdots<\sigma([N])$.
\end{enumerate}
\end{example}

\begin{example}[Measure of uniformness]\label{eg:two-sample}
\begin{enumerate}[label=(\roman*)]
\item Wilcoxon rank sum: $W(\pi)=\sum_{i=1}^m\pi(i)$, where $m\in [N]$ is a preset positive integer;
\item Mann-Whitney: $U(\pi)=\sum_{i=1}^{m}\sum_{j=1}^{n}\ind(\pi(i)<\pi(m+j))$, where $m,n$ are two positive integers.
\end{enumerate}
\end{example}

In statistics, resampling methods constitute to one of the most exciting and deep directions.

\begin{example}[Two-sample permutation testing]\label{example:1perm}
Consider $X_1,\ldots,X_m$  and $Y_1,\ldots,Y_n$ to be two samples, and 
\[
\hat\theta_m(\{X_i;i\in[m]\}) ~~~{\rm and}~~~ \hat\theta_n(\{Y_i;i\in[n]\}) 
\]
to be two statistics calculated using the two samples, respectively. Letting $Z_1=X_1,\ldots,Z_m=X_m, Z_{m+1}=Y_1,\ldots,Z_{m+n}=Y_n$, permutation testing concerns using 
\[
\hat\theta_m(\{Z_i;\pi(i)\leq m\})-\hat\theta_n(\{Z_i;\pi(i)> m\})
\]
to infer
\[
\hat\theta_m(\{X_i;i\in[m]\})-\hat\theta_n(\{Y_i;i\in[n]\}).
\]
\end{example}

What plays the most important role in the scope of this book is the framework of finite-population/design-based statistical inference, where a size-$n$ sample of points is drawn uniformly without replacement from a finite population. 

\begin{example}[Finite-population inference]\label{example:1finite-pop}
Consider
\[
\Big\{z_i=z_{N,i}\in\cZ, i\in[N] \Big\}
\]
to be a finite population, whose elements may not be distinct and can change as $N$ varies. The observed data can then be represented as
\[
\Big\{z_{i}, i\in [N], \pi(i)\leq n\Big\}
\]
Any statistical method about inferring a functional of the finite population using the above observed sample is a permutation functional. 
\end{example}

Recent surge in design-based causal inference brings up new application scenarios. 

\begin{example}[Design-based causal inference]\label{example:1causal} In a classical design, there is an unobserved finite population of  size $N$,
\[
\Big\{(x_i,y_i(0),y_i(1)), i\in [N]\Big\},
\]
where we sample without replacement a size-$n$ subset: 
\[
\Big\{(x_i,y_i(0),y_i(1)), \pi(i)\leq n\Big\}.
\]
Here the $(y_i(0), y_i(1))$'s are potential outcomes in the causal inference terminology, and can be interpreted as the outcomes of the $i$-th individual if being treated (i.e., $y_i(1)$) or not (i.e., $y_i(0)$). 

We next randomly assign $m$ of the sample points to the case and the rest to the control group, yielding the following observation:
\begin{align}\label{eq:1causal-design}
\Big\{(x_i, y_i, D_i), \pi(i)\leq n \Big\}, \text{ with }D_i=\ind(\pi(i)>m) \text{ and }y_i=y_i(D_i),
\end{align}
where $D_i$ indicates the treatment status ($D_i=1$ or $0$ signifies being treated or not, respectively). Inferences on the difference between $y_i(1)$'s and $y_i(0)$'s using only the observed data, \eqref{eq:1causal-design}, constitute to permutation inference problems.
\end{example}

\section{Notes}

The results in Proposition \ref{prop:basic} can be found in standard survey sampling books; see, e.g., \citet[Chapter 4]{deming1966some}, \citet[Chapter 2]{cochran1977sampling}, \citet[Chapter 3]{levy2013sampling}, and \citet[Chapter 2]{chaudhuri2014modern}. \cite{scheaffer1990elementary}, \cite{chaudhuri2005survey}, \cite{fuller2011sampling} cover more designs beyond simple random sampling. Proposition \ref{prop:1entropy} is a direct consequence of the entropy argument; cf. \citet[Chapter 3]{hajek1981sampling}. \\

Proposition \ref{prop:iid} is well-known; e.g.,in commenting on a paper of Godambe and Thompson, Baranrd \citep{baranrd1971bayes} mentioned that ``a simple random sample of a simple random sample is itself a simple random sample''. Proposition \ref{prop:1rank-perm} is well-known as well.\\

Example \ref{eg:disarray} mentions four notable rank correlations. A classical reference to rank correlations is Sir. Kendall's book \citep{kendall1948rank}. We adopt the present version from \cite{diaconis1977spearman}. Chatterjee's rank correlation is a recent breakthrough in rank correlation methods, proposed by Sourav Chatterjee \citep{chatterjee2021new}; see, also, \cite{shi2022power}, \cite{lin2023boosting}, and \cite{lin2022limit}. \\

Example \ref{eg:two-sample} mentions two popular two-sample rank tests. A classical reference is \cite{sidak1999theory}; see, also, \cite{lehmann2006nonparametrics} and \cite{nikitin1995asymptotic}. The present version is adopted from \cite{zhao1997error}.\\

Example \ref{example:1perm} is about permutation tests, for which good referencing books include \cite{good2005permutation}, \cite{bonnini2014nonparametric}, \cite{berry2018permutation}; see, also, \citet[Chapter 17]{romano2005testing} and \citet[Chapter 13]{van2000asymptotic}, and \cite{chung2013exact}. \\

Example \ref{example:1finite-pop} concerns survey sampling, which we have given a brief review in the first paragraph. \\

Example \ref{example:1causal} concerns causal inference in a finite-population sampling without replacement framework. This track of study was initiated by Neyman in \cite{splawa1990application}. For more discussions, we refer readers to a modern survey made by Li and Ding \citep{li2017general} and Bai, Shaikh, and Tabord-Meehan \citep{bai2024primer}.

%% file: chapters/basic-prob.tex
\chapter{Technical Preparation}
\label{chapter:basic-prob}

\section{Basic analysis}\label{sec:basic-analysis}

A {\it topological space} $(\Omega, \cO)$ contains a collection of subsets, $\cO=\{O\subset \Omega\}$, such  that 
\begin{enumerate}[label=(\roman*)]
\item both the empty set $\emptyset$ and the whole set $\Omega$ belong to $\cO$;
\item $\cO$ is closed to finite intersection; 
\item $\cO$ is closed to arbitrary union. 
\end{enumerate}
Elements in $\cO$ are called {\it open sets}. A set of $B\subset\Omega$  is called {\it closed} if its complement, denoted by $B^c$, is open. The {\it closure} of a set $D\subset\Omega$, denoted by $\bar{D}$, is the intersection of all closed sets that cover $D$. The {\it interior} of $D$, denoted by $D^\circ$, is the union of all open subsets of $D$. A subset $A$ of $\Omega$ is said to be {\it dense} if $\bar A=\Omega$; in that case, $\Omega$ is called {\it separable}.

Any open set that contains $\omega\in\Omega$ is called a {\it neighborhood} of $\omega$. A sequence of points $\{\omega_n\}$ is said to converge to $\omega$ in $(\Omega,\cO)$, denoted as $\omega_n\to\omega$, if every neighborhood of $\omega$ contains all but finitely many $\omega_n$'s. If distinct points in $\Omega$ contain distinct neighborhoods, then $(\Omega,\cO)$ is said to be {\it Hausdorff}. A set $K\subset \Omega$ is said to be {\it compact} if for arbitrary open union that covers $K$, it contains a finite open union that still covers $K$. A compact set in a Hausdorff topological space is closed. 

A mapping $f$ between two topological spaces is said to be {\it continuous} if the inverse of any open set is open. If $\omega_n\to\omega$ and $f$ is continuous, then $f(\omega_n)\to f(\omega)$.

A metric space $( T,d)$ contains a metric $d: T\times T\to [0,\infty)$ such that, for any $s,t,u\in T$,
\begin{enumerate}[label=(\roman*)]
\item $d(s,t)=d(t,s)$;
\item $d(s,u)\leq d(s,t)+d(t,u)$;
\item $d(s,t)=0$ if and only $s=t$. 
\end{enumerate}
A pseudo-metric space $( T,d)$ contains a pseudo-metric $d$ that satisfies (i) and (ii), but not (iii); it thus indues an equivalent class. A (pseudo-)metric space induces a topological space, where the open set is defined as arbitrary unions of the open $r$-balls, $B(t,r):=\{s\in T; d(s,t)<r\}$ for some $r\geq 0$. We can accordingly define topological notions in a (pseudo-)metric space.

A Cauchy sequence $\{t_n\}$ in $( T,d)$ is such that $d(t_n,t_m)\to 0$ as $n,m\to\infty$. The space $( T,d)$ is said to be complete if every Cauchy sequence converges to a limit in $ T$. A separable complete metric space is called a {\it Polish space}.

A set $K\subset  T$ is said to be {\it totally bounded} if, for any $r>0$, $K$ can be covered by finitely many open $r$-balls. If $( T,d)$ is complete, then $K$ is compact if and only if $K$ is totally bounded and closed.

A {\it normed space} $( T,\norm{\cdot})$ contains a vector space $ T$ equipped with a norm $\norm{\cdot}: T\to[0,\infty)$ such that, for any $s,t\in T$ and $\alpha\in\reals$,
\begin{enumerate}[label=(\roman*)]
\item $\norm{s+t}\leq \norm{s}+\norm{t}$;
\item $\norm{\alpha t}= |\alpha|\cdot \norm{t}$;
\item $\norm{t}=0$  if and only if $t=0$.
\end{enumerate}
The space $( T,\norm{\cdot})$ is called a {\it pseudo-normed} space if it satisfies everything except for (iii) above. A normed space induces a metric space with $d(s,t)=\norm{s-t}$ for any $s,t\in T$. A complete normed spaced is called a {\it Banach space}.

Any real space $\reals^d$ equipped with the Euclidean norm $\norm{x}_2=(\sum_{j=1}^dx_j^2)^{1/2}$ is a Banach space. Another Banach space that plays a special role in this book is the set of all bounded real functions $f: T\to\reals$ equipped with the uniform norm $\norm{f}_{ T}=\sup_{t\in T}|f(t)|$ (also written as $\norm{f}_{\infty}$), denoted as $(\ell^{\infty}( T),\norm{\cdot}_{ T})$. This space is not separable unless $ T$ is countable, and $(\ell^{\infty}( T),\norm{\cdot}_{ T})$ is generally not Polish.

A subspace of $\ell^{\infty}( T)$, $UC( T,d)$, contains all bounded functions that are {\it uniformly $d$-continuous}, i.e., 
\[
\lim_{\delta\to0}\sup_{d(s,t)\leq \delta}|f(s)-f(t)|=0.
\]
The space $(UC( T,d), \norm{\cdot}_{ T})$ is Polish.

\section{Stochastic convergence}\label{sec:stochastic-convergence}

A collection of subsets of $\Omega$, denoted by $\cA$, is a $\sigma$-algebra if
\begin{enumerate}[label=(\roman*)]
\item $\emptyset \in \cA$;
\item $\cA$ is closed under complement;
\item $\cA$ is closed under countable union. 
\end{enumerate}
A space $(\Omega,\cA)$ is said to be {\it measurable} if $\cA$ is a $\sigma$-algebra. If $(\Omega,\cA)$ is measurable and $(\Omega',\cO')$ is topological, then $f:\Omega\to\Omega'$ is said to be a {\it measurable function} if the inverse of any element in $\cO'$ belongs to $\cA$. 

For arbitrary collection of subsets of $\Omega$, denoted by $\cB$, we define $\sigma(\cB)$ to be the smallest $\sigma$-algebra that contains $\cB$; this is called the $\sigma$-algebra generated by $\cB$. For a topological space $(\Omega,\cO)$, $\sigma(\cO)$ is called its {\it Borel $\sigma$-algebra}. A map $X$ between two topological spaces $(\Omega,\cO)$ and $(\Omega',\cO')$ is said to be {\it Borel measurable} if $X^{-1}(\cO')\subset \sigma(\cO)$.

For a measurable space $(\Omega,\cA)$, a map $\mu:\cA\to[0,\infty]$ is said to be a {\it measure} if
\begin{enumerate}[label=(\roman*)]
\item $\mu(\emptyset)=0$;
\item $\mu$ is countably additive, i.e., $\mu(\cup_{i=1}^{\infty}A_i)=\sum_{i=1}^{\infty}\mu(A_i)$ for any countably disjoint sets $A_i\in\cA$.
\end{enumerate}
We call such $(\Omega,\cA,\mu)$ a {\it measure space}. In particular, if $\mu(\Omega)=1$, the corresponding space is said to be a {\it probability space}, written as $(\Omega,\cA,\Pr)$. 

An $\cX$-valued random variable $X:\Omega\to \cX$ is a Borel measurable function mapping from a probability space $(\Omega,\cA,\mu)$ to a Polish space $(\cX,d)$, the latter of which can be infinite-dimensional. For any random variable $X\in\cX$, its {\it law}, written as $\P_X$, is defined to be the {\it induced measure} such that $\P_X(A)=\Pr(X^{-1}(A))$ for any Borel measurable $A$ in $(\cX,d)$. 

Consider a sequence of $\cX$-valued random variables $\{X_n;n=1,2,\cdots\}$ defined over the same probability space $(\Omega,\cA,\Pr)$. 

\begin{definition}
The stochastic sequence $X_n$ is said to {\it converge in probability} to a random variable $X$, written as $X_n\stackrel{\Pr}{\to}X$, if for any $\epsilon>0$, 
\[
\lim\limits_{n\to \infty}\Pr\Big\{d(X_n,X)>\epsilon\Big\}=0.
\] 
It is further said to be {\it converging almost surely} to $X$, written as $X_n\stackrel{\rm a.s.}{\to}X$, if
\[
\Pr\Big\{\lim_{n\to \infty}d(X_n,X)>0\Big\}=0.
\]
In particular, two random variables $X=Y$ a.s. if and only if $\Pr(X=Y)=1$.
\end{definition}

%Let's further assume the space $(\cX,d)$ to be Polish, i.e., complete and separable. One can then introduce the notion of {\it weak convergence}.

\begin{definition} Let $(\cX,d)$ be a Polish space. A sequence of $\cX$-valued Borel measurable random variables $X_n$ is then said to {\it weakly converge} to another $\cX$-valued Borel measurable random variable $X$, written as 
\[
X_n\Rightarrow X, 
\]
if for all bounded continuous function $f:\cX\to \reals$, it holds true that
\[
\lim_{n\to\infty}\E f(X_n) \to \E f(X).
\]
\end{definition}

\begin{exercise}
Show that the notion of weak convergence generalizes that of ``convergence in distribution'' --- i.e., the CDF of $X_n$ converges to the CDF of $X$ at every continuity point of $F_X$ --- when we take $(\cX,d)$ to be $(\reals^p,\norm{\cdot})$, the multivariate real space equipped with the Euclidean distance.
\end{exercise}

The following proposition characterizes weak convergence; this is the famous Portmanteau lemma \citep[Lemma 2.2]{van2000asymptotic}.

\begin{proposition}[Portmanteau Lemma]\label{prop:portmanteau} Let $(\cX,d)$ be a Polish space and $\{X_n\}$ be a sequence of $\cX$-valued random variables. The following are then equivalent.
\begin{enumerate}[label=(\roman*)]
\item $X_n\Rightarrow X$;
\item $\E f(X_n)\to \E f(X)$ for any bounded and uniformly continuous $f: \cX\to\reals$;
\item $\E f(X_n)\to \E f(X)$ for any bounded and Lipschitz continuous\footnote{A function $f:\cX\to\reals$ is Lipschitz continuous if there exists a universal constant $L>0$ such that, for any $x,y\in\cX$, $|f(x)-f(y)|\leq Ld(x,y)$.} $f: \cX\to\reals$;
\item for any closed set $U \subset \cX$,
\[
\limsup_{n\to\infty} \Pr(X_n\in U) \leq \Pr(X\in U);
\]
\item for any open set $V\subset \cX$,
\[
\liminf_{n\to\infty} \Pr(X_n \in V)\geq \Pr(X \in V);
\]
\item for any Borel set $A\subset\cX$ such that the topological boundary of $A$, denoted by $\partial A$, satisfies $\Pr(X\in \partial A)=0$, we have
\[
\lim_{n\to\infty}\Pr(X_n\in A)=\Pr(X\in A);
\]
\item $\E f(X_n)\to \E f(X)$ for any bounded, measurable, and continuous almost everywhere with regard to $\P_X$, $f:\cX\to\reals$.
\end{enumerate}
\end{proposition}
\begin{proof}
(i)$\Longrightarrow$ (ii) $\Longrightarrow$ (iii): obvious.

(iii)$\Longrightarrow$ (iv): let's introduce a function 
\[
f(x):=d(x,U):=\inf_{y\in U}d(x,y),~~~\text{ for any }x\in \cX.
\]
It is then clear that, for any $x,y\in \cX$, 
\[
f(x)\leq d(x,y)+f(y),
\]
so that, by symmetry, $|f(x)-f(y)|\leq d(x,y)$ and thus $f\in [0,\infty]$ is Lipschitz continuous. As a consequence, introducing $g_k(x):=(1-kf(x))^+$ being the positive part of $1-kf(x)$, we have: (a) $g_k(x)$ is Lipschitz continuous; (b) $g_k(x)\in [0,1]$; (c) $g_k(x)\geq \ind(x\in U)$ for any $x\in\cX$; (d) for any $x\in\cX$, $\lim_{k\to\infty}g_k(x)=\ind(x\in U)$. Here we use the fact that, since $U$ is closed, $f(x)=0$ if and only if $x\in U$.

Summarizing what we have obtained, by (iii), 
\begin{align*}
\limsup_{n\to\infty}\Pr(X_n\in U)\leq \lim_{k\to\infty}\limsup_{n\to\infty}\int g_k(x)\d \P_{X_n}(x) = \lim_{k\to\infty}\int g_k(x)\d \P_X(x)\\
=\int \ind(x\in U)\d \P_X(x)=\Pr(X\in U).
\end{align*}

(iv)$\Longrightarrow$ (v): by symmetry.

(iv)+(v) $\Longrightarrow$ (vi): introduce $U$ and $V$ be the closure and interior of $A$. It is then true that (a) $U$ is closed; (b) $V$ is open; (c) $\Pr(X\in U)=\Pr(X\in V)=\Pr(X\in A)$ since the difference between $U$ and $V$ is $\partial A$, whose $\P_X$-measure is 0. Accordingly
\begin{align*}
\Pr(X\in A)=\Pr(X\in V)\leq \liminf \Pr(X_n\in V)\leq \liminf \Pr(X_n\in A)\leq\\
 \limsup \Pr(X_n\in A)\leq \limsup \Pr(X_n \in  U) \leq \Pr(X \in U)=\Pr (X\in A).
\end{align*}

(vi) $\Longrightarrow$ (vii): without loss of generality, assume that $f\in [0,1]$ and let $D_f$ be the set of all discontinuous points of $f$. We then have
\begin{align*}
\lim_{n\to\infty}\E f(X_n)&=\lim_{n\to\infty}\int_0^1 \Pr(f(X_n)\geq t)\d t\\
&=\int_0^1 \Pr(f(X)\geq t)\d t\\
&=\E f(X),
\end{align*}
where in the second equality we used the dominated convergence theorem, claim (vi), and the fact that $\partial \Big\{x: f(x)\geq t\Big\}$ has $\P_X$-measure 0 for Lebesgue almost all $t$. The first equality, on the other hand, is due to Exercise \ref{exe:integration-tail} ahead. 

(vii) $\Longrightarrow$ (i): obvious.
\end{proof}

A probability measure, $\Q$, on a general metric space $(\cX,d_{\cX})$ is said to be {\it tight} if for any $\epsilon>0$, there exists a compact set $K_{\epsilon}\subset \cX$ such that $\Q(K_{\epsilon})\geq 1-\epsilon$. A random variable $X: \Omega\to \cX$ is said to be tight if its law is tight. It is immediate that every Borel measurable $X$ in the Polish space is tight.

\begin{theorem}[Prokhorov] If a sequence of probability measures on a Polish space weakly converges, then this sequence is (uniformly) tight.
\end{theorem}
\begin{proof}
This is the second half of Prokhorov's Theorem. For a proof, check Theorem 5.2 in \cite{billingsley1999convergence}.
\end{proof}

Specializing to the case when $(\cX,d)$ is $(\reals^d,\norm{\cdot}_2)$, a random variable $X$ is a measurable map between $(\Omega,\cA,\Pr)$ and the topological space induced by the Euclidean norm $\norm{\cdot}_2$. For any $X\in\reals$, any probability measure $\Q$ over $(\reals, \cB(\reals))$, and any $p\geq 1$, we define the $L^p(\Q)$ norm of $X$ as
\[
\norm{X}_{L^p(\Q)}=\Big(\int |x|^p\d \Q(x)\Big)^{1/p}.
\]
The subscript $\Q$ can be further suppressed when the law of $X$ is explicit from the context. In this case,
\[
\norm{X}_{L^p}=(\E|X|^p)^{1/p}=\Big(\int |x|^p\d \P_X(x)\Big)^{1/p}.
\]
We further define
\[
\norm{X}_{L^{\infty}} = \inf\Big\{K\in[0,\infty]; \Pr(|X|\leq K)=1  \Big\},
\]
which is the {\it essential supremum} of $X$. 

\begin{exercise}\label{exe:integration-tail}
For any random variable $X$ such that $\E|X|<\infty$, we have
\[
\E|X|=\int_0^{\infty}\Pr(|X|> t)\d t.
\] 
\end{exercise}

\begin{theorem}[Minkowski-Riesz-Fischer] The $L^p$ space, containing all $X$'s such that $\norm{X}_{L^p}<\infty$, is a Banach space.
\end{theorem}
\begin{proof}
Theorem 19.1 in \cite{billingsley2008probability}.
\end{proof}

\begin{definition}
The sequence $X_n$ is said to {\it converge in $L^p$ norm} to another random variable $X$, written as $X_n\stackrel{L^p}{\to}X$, if
\[
\lim_{n\to\infty}\norm{X_n-X}_{L^p}=0.
\] 
\end{definition}

\begin{proposition}[Continuous mapping theorem]\label{prop:cmt} Let $f:\reals^d\to\reals^m$ be continuous at a $\P_{\bX}$-measure 1 set. The following are then true.
\begin{enumerate}[label=(\roman*)]
\item If $\bX_n\Rightarrow \bX$, then $f(\bX_n)\Rightarrow f(\bX)$;
\item if $\bX_n \stackrel{\Pr}{\to}\bX$, then $f(\bX_n) \stackrel{\Pr}{\to}f(\bX)$;
\item if $\bX_n \stackrel{\rm a.s.}{\to}\bX$, then $f(\bX_n) \stackrel{\rm a.s.}{\to}f(\bX)$.
\end{enumerate}
\end{proposition}

\begin{exercise}
Prove Proposition \ref{prop:cmt}.
\end{exercise}

\begin{exercise}
Prove the Slutsky's lemma, that is, for any random vectors $\bX_n,\bY_n,\bX\in\reals^d$, if $\bX_n\Rightarrow \bX$ and $\bY_n\Rightarrow \bc$ for some constant $\bc$, then
\begin{enumerate}[label=(\roman*)]
\item $\bX_n+\bY_n\Rightarrow \bX+c$;
\item $\bY_n^\top\bX_n\Rightarrow \bc^\top\bX$.
\end{enumerate}
\end{exercise}

The distribution of an $\reals^d$-valued random vector $\bX$ is uniquely determined by its characteristic function.

\begin{definition}[Moment generating function and characteristic function] For any $\reals^d$-valued random vector $\bX$, its moment generating function (MGF) is defined to be
\[
m_{\bX}(\bt)=\E\exp(\bt^\top\bX);
\]
its characteristic function (cf) is defined to be
\[
\phi_{\bX}(\bt)=\E\exp({\rm i}\bt^\top\bX),
\]
where ${\rm i}$ is the imaginary number.
\end{definition}

\begin{theorem}[L\'{e}vy continuity theorem]\label{thm:levy} A sequence of $\reals^d$-valued random vectors $\bX_n$'s weakly converges to another random vector $\bX$ if and only $\phi_{\bX_n}$ converges pointwisely to $\phi_{\bX}$.
\end{theorem}
\begin{proof}
Theorem 6.6.3 in \cite{chung2001course}.
\end{proof}

\begin{corollary}[Cram\'{e}r-Wold device]\label{cor:cw-device} Let $\bX_n$ be a sequence of $\reals^d$-valued random vectors and $\bX$ be another random vector in $\reals^d$. Then $\bX_n\Rightarrow \bX$ if and only if $\bt^\top\bX_n\Rightarrow \bt^\top\bX$ for any $\bt\in\reals^d$.
\end{corollary}
\begin{proof}
If $\bX_n\Rightarrow \bX$, then for any bounded continuous function $f$, we have
\[
\E f(\bX_n)\to \E f(\bX) \text{ as }n\to\infty.
\]
In particular, for any $\bt\in\reals^d$,
\[
\E f(\bt^\top\bX_n)\to \E f(\bt^\top\bX) \text{ as }n\to\infty.
\]
Thusly, $\bt^\top\bX_n\Rightarrow \bt^\top\bX$.

On the other hand, if $\bt^\top\bX_n\Rightarrow \bt^\top\bX$ for all $\bt\in\reals^d$, then
\[
\phi_{\bX_n}(\bt)=\E\exp({\rm i}\bt^\top\bX_n)\to \E\exp({\rm i}\bt^\top\bX)=\phi_{\bX}(\bt).
\]
Theorem \ref{thm:levy} then yields $\bX_n\Rightarrow \bX$.
\end{proof}

The last theorem connects, regarding weak convergence, the pointwise convergence to the uniform convergence.

\begin{theorem}[Poly\'a's Theorem]\label{thm:polya} Suppose $\bX_n\Rightarrow \bX\in\reals^d$ such that $\bX$ has a continuous CDF $F_{\bX}$. Then, denoting $F_{\bX_n}$ to be the CDF of $\bX_n$, we have 
\[
\sup_{\bt\in\reals^d}\Big|F_{\bX_n}(\bt)-F_{\bX}(\bt)\Big|=0.
\]
\end{theorem}

\begin{exercise}
Prove Theorem \ref{thm:polya}.
\end{exercise}

\section{Weak convergence of stochastic processes}\label{sec:weak-convergence-sp}

Let $(\Omega,\cA,\Pr)$ be a probability space and $( T,d)$ be a pseudo-metric space. A stochastic process $\{X(t);t\in T\}$ defined on $(\Omega,\cA,\Pr)$ is a function
\[
X:  T\times\Omega \to \reals
\]
such that $X(t,\omega)$ is an $\reals$-valued random variable for any $t\in T$ and $\omega\in\Omega$. For any finite set $F\subset  T$, $\omega\to\{X(t,\omega);t\in F\}$ is then measurable, and we call the distribution of it a {\it finite-dimensional distribution} of $X$. {\it Kolmogorov existence theorem} states then that a {\it consistent} family of finite distributions of $X$ defines a unique probability measure $\mu$ over the cylindrical $\sigma$-algebra $\cC$ of $\reals^ T$. Thusly, $X$ is a process mapping from $(\Omega,\cA)$ to $(\reals^ T,\cC)$ and admits the law $\mu$. 

In general, $\mu$ cannot be further extended to a tight Borel probability measure with regard to the topological space induced by $\norm{\cdot}_T$. Accordingly, the classical theory of weak convergence of stochastic processes in Polish spaces has to be refined. 

An important notion about $X$ that we will repeatedly use is {\it separable}.

\begin{definition} A process $\{X(t);t\in T\}$, indexed by a pseudo-metric space $( T,d)$, is said to be separable if there exists a countable subset $S\subset T$ and a null set $N$ such that for any $\omega\not\in N$ and $t\in T$, there exists a sequence $s_n\in S$ satisfying $d(s_n,t)\to 0$ so that $|X(s_n,\omega)-X(t,\omega)|\to 0$ as $n\to\infty$.
\end{definition}

By definition, if $X$ is separable, then for any $t\in T$, there exists a sequence $s_n\in S$ so that $s_n$ converges to $t$; this implies that $(T,d)$ has to be also separable. In addition, if $X$ is separable, then 
\[
\sup_{t\in T}|X(t)|=\sup_{t\in S}|X(t)|, a.s.. 
\]
Accordingly, the supremum of a separable stochastic process is always measurable. We will appeal to this property in the following chapters.

\begin{definition}\label{def:2version}
A process $Y$ is said to be a {\it version} of another process $X$ if they have the same finite distribution for any finite $t_1,\ldots,t_n\in T$ and any $n=1,2,\ldots$.
\end{definition}

\begin{definition}
A process $X$ is said to be {\it sample bounded} if it admits a version $\tilde X$ satisfying $\sup_{t\in T}|\tilde X_t|<\infty$ a.s.. It is said to be {\it sample continuous} if it admits a version $\tilde X$ satisfying that, for almost all $\omega$, $\{X(t,\omega),t\in T\}$ is bounded and uniformly $d$-continuous.
\end{definition}

A sample bounded $X$ can be embedded into $\ell^{\infty}(T)$ with the corresponding $\sigma$-algebra $\Sigma$ as the intersection between the cylindrical $\cC$ and the Borel $\sigma$-algebra generated by the $\norm{\cdot}_T$ norm. This metric space, $(\ell^{\infty}(T),\norm{\cdot}_T)$, will play a pivotal role in the definition of weak convergence of stochastic processes. 

Recall that $X$ is usually not Borel measurable in $(\ell^{\infty}(T),\norm{\cdot}_T)$. However, when $X$ is indeed Borel measurable, the following proposition outlines the relation between tightness and sample continuity of $X$.

\begin{proposition}
Let $X\in\ell^{\infty}(T)$ be a Borel measurable stochastic process. The following two are then equivalent. 
\begin{enumerate}[label=(\roman*)]
\item $X$ is tight;
\item There exists a pseudo-metric $d$ on $T$ and a version $\tilde X$ of $X$ such that $(T,d)$ is totally bounded and $\tilde X$ is uniformly $d$-continuous.
\end{enumerate}
\end{proposition}
\begin{proof}
Proposition 2.1.7 in \cite{gine2021mathematical} and Lemma 7.2 in \cite{kosorok2008introduction}.
\end{proof}

%\begin{definition}[Bounded stochastic process] A bounded stochastic process $\{X(t);t\in T\}$ defined on a measure space $(\Omega,\mathcal{A},\P)$ is a measurable map from $(\Omega,\mathcal{A})$ to $(\ell^{\infty}( T),\Sigma)$, with $\ell^{\infty}( T)$ standing for the space of all bounded real functions on $ T$, and $\Sigma$ representing the cylindrical $\sigma$-algebra of $\ell^{\infty}( T)$.
%\end{definition}

\begin{definition}[Weak convergence of stochastic processes]\label{def:weak-convergence} A sequence of (sample) bounded stochastic processes $\{X_n(t);t\in T\}$ is said to be weakly converging to a tight Borel measurable  stochastic process $\{X(t);t\in T\}$ on $\ell^\infty( T)$, written as  $X_n \Rightarrow X$ in $\ell^{\infty}( T)$, if 
\begin{enumerate}[label=(\roman*)]
\item any finite-dimensional distribution of $X_n(t)$, $(X_n(t_1),\ldots,X_n(t_m))^\top$, weakly converge to that of $(X(t_1),\ldots,X(t_m))^\top$;
\item for any bounded continuous function $H:\ell^{\infty}( T)\to\reals$, we have
\[
\E^*H(X_n) \to \E H(\tilde X),
\]
as $n\to\infty$; here $\E^*$ represents the outer expectation and $\tilde X$ is a separable version of $X$.
\end{enumerate}
\end{definition}

In reality, verifying the second condition of Definition \ref{def:weak-convergence} is inessential, and the real deal is the following theorem. It connects weak convergence of bounded stochastic processes to another maximal inequality; this is the famous {\it stochastic equicontinuity.}

\begin{theorem}\label{thm:weak-convergence-key} Consider a sequence of bounded stochastic processes $\{X_n(t);t\in T\}$ in $\ell^{\infty}( T)$. The following two are then equivalent.
\begin{enumerate}[label=(\roman*)]
\item Any finite-dimensional distribution of $\{X_n(t);t\in T\}$ weakly converges to some distribution, and there exists a pseudo-metric space $( T,d)$ such that it is totally bounded and, for any $\epsilon>0$,
\begin{align}\label{eq:sec}
 \lim_{\delta\to 0}\limsup_{n\to\infty}\Pr^*\Big\{\sup_{d(s,t)\leq\delta}|X_n(t)-X_n(s)|>\epsilon  \Big\}=0,
\end{align}
where $\Pr^*$ represents the outer probability. 
\item There exists a tight Borel measurable stochastic process $X$ such that $X_n\Rightarrow X$ in $\ell^{\infty}( T)$.
\end{enumerate}
\end{theorem}
\begin{proof}
Theorem 3.7.23 in \cite{gine2021mathematical}.
\end{proof}

Theorem \ref{thm:weak-convergence-key} reduces proving weak convergence of stochastic processes to handling (a) weak convergence of any finite-dimensional realization, which is usually a consequence of multivariate central limit theorems, and (b) a new set of maximal inequalities in the form of \eqref{eq:sec}.

In statistics, weak convergence of stochastic processes is often used combined with the following generalized version of the continuous mapping theorem that extends Proposition \ref{prop:cmt}.

\begin{theorem}[Continuous mapping, general]\label{thm:cmt} Let $(\cX,d_{\cX})$ and $(\cW,d_{\cW})$ be two pseudo-metric spaces and let $f:\cX\to\cW$ be continuous. Then, if $X_n \Rightarrow X$ in $(\cX,d_{\cX})$, we have $g(X_n)\Rightarrow g(X)$ in $(\cW,d_{\cW})$.
\end{theorem}
\begin{proof}
Theorem 7.7 in \cite{kosorok2008introduction}.
\end{proof}

\section{Elementary probability inequalities}\label{sec:chap2-inequ}

\subsection{Markov's inequality}

\begin{theorem}[Markov's inequality] For any $\reals$-valued random variable $X\geq 0$ and any $t>0$,
\[
\Pr(X\geq t)\leq \frac{\E X}{t}.
\]
\end{theorem}

\begin{example}[Longest increasing sequence of $\pi$] Consider $\pi$ to be a uniform random permutation in $\cS_N$ and let 
\[
L_N \text{ be the length of the longest increasing subsequence of } \pi 
\]
such that both $i_1<\cdots<i_k$ and $\pi(i_1)<\cdots<\pi(i_k)$ holds. For example, letting $\pi([5])=[2,3,1,4,5]$, we have $L_N=4$, corresponding to the sequence $1,2,4,5$. 

The following is (a simplified version of) the famous Erd\'{o}s–Szekeres theorem
\begin{theorem}[Erd\'{o}s–Szekeres]\label{thm:erdos} It holds true that
\[
1\leq \liminf_{N\to\infty}\frac{\E [L_N]}{\sqrt{N}}\leq \limsup_{N\to\infty}\frac{\E [L_N]}{\sqrt{N}}\leq e. \footnote{\cite{baik1999distribution} showed that $\E[L_N]=2\sqrt{N}+cN^{1/6}+o(N^{1/6})$, where $c \approx -1.77$. Furthermore, $(L_N-2\sqrt{N})/N^{1/6}$ weakly converges to the 2nd type Tracy-Widom distribution \citep{tracy1994level}.}
\]
\end{theorem}
\begin{proof}
We first prove the upper bound using only Markov's inequality. Note that, for any $\ell>0$,
\begin{align}\label{eq:ES1}
\E [L_N] \leq \ell \Pr(L_N\leq \ell)+N\cdot\Pr(L_N\geq \ell) \leq \ell + N\cdot \Pr(L_N\geq \ell).
\end{align}
Introduce $X_N$ to represent the number of increasing subsequences of length $\ell$ in $\pi$. There are apparently ${N \choose \ell}$ many such subsequences, and, by symmetry (essentially the uniformness of $\pi$), each has probability $1/\ell!$ to be increasing. Accordingly
\[
\Pr(L_N\geq \ell)=\Pr(X_N>0)\leq \E[X_N] =\frac{1}{\ell!}{N \choose \ell}\leq \Big(\frac{e\sqrt{N}}{\ell}\Big)^{2\ell}.
\]
Taking $\ell=(1+\delta)e\sqrt{N}$, plugging the above inequality with the chosen $\ell$ to \eqref{eq:ES1}, first letting $N$ go to infinity and then push $\delta\to0$ yields the upper bound.

{\bf [Optional]} We then move to the lower bound using a combinatorial argument due to Erd\'{o}s and Szekeres. Introduce $D_n$ to be the length of the longest decreasing sequence in $\pi$. By symmetry, 
\[
\E[L_N]=\E\Big[\frac{L_N+D_N}{2} \Big] \geq \E[(L_ND_N)^{1/2}].
\]
It remains to lower bound the product of $L_N$ and $D_N$. To this end, introduce $L_N^{(k)}$ and $D_N^{(k)}$ to be the length of the longest increasing and decreasing sequences ending at position $k$. The following two facts then hold:
\begin{enumerate}[label=(\roman*)]
\item for any $k\in[n]$, $L_N^{(k)}\leq L_N$ and $D_N^{(k)}\leq D_N$;
\item the set $\{(L_N^{(k)},D_N^{(k)});k\in[N]\}$ contains distinct pairs (by noticing that, for any $j>k$, either $L_N^{(j)}>L_N^{(k)}$ or $D_N^{(j)}>D_N^{(k)}$ holds via separately discussing the case that $\pi(j)>\pi(k)$ or the converse).
\end{enumerate}
Combining the above two and noticing that everything under consideration is a natural number, we obtain
\[
L_ND_N \geq \Big|\Big\{(L_N^{(k)},D_N^{(k)});k\in[N]\Big\}\Big|=N
\]
and thus $\E[L_N]\geq \sqrt{N}$.
\end{proof}
\end{example}

\subsection{Jensen's inequality}

Letting $I$ be an interval in $\reals$, a function $g:I\to \reals$ is said to be convex if for any $x,y\in I$ and any $t\in[0,1]$,
\[
g(tx+(1-t)y)\leq tg(x)+(1-t)g(y).
\]

\begin{theorem}[Jensen's inequality]
Suppose $X\in I$ be an $\reals$-valued random variable such that $\E[X]$ and $\E[f(X)]$ both exist. It then holds true that
\[
g(\E[X]) \leq \E g(X).
\]
\end{theorem}
\begin{proof}
Standard mathematical analysis gives that, for any $x,y$,
\[
g(x) \geq g(y)+(x-y)g_r'(y),
\]
where $g_r'$ is $g$'s right derivative. This yields
\[
g(X)\geq g(\E X)+(X-\E X)g_r'(\E X).
\]
so that $\E[g(X)] \geq g(\E X)$.
\end{proof}

A quite useful consequence of Jensen's inequality is the monotonicity of the $L^p$ norms. 

\begin{corollary}\label{cor:monotone-Lp} For any $1\leq p\leq q \leq \infty$, we have
\[ 
\norm{X}_{L^p}\leq \norm{X}_{L^q}.
\]
\end{corollary}
\begin{exercise}
Prove Corollary \ref{cor:monotone-Lp}.
\end{exercise}

\begin{exercise}[Young's entropy inequality]\label{lem:young-entropy} For any random variable $U\geq 0$ such that $\E U=1$ and any random variable $V\geq 0$, please show that 
\[
\E[U \log V] \leq \log\E [V]  + \E[U\log U].
\]
\end{exercise}

\subsection{Maximal inequalities}

A large fraction of this book concerns processes and their performance {\it at the worst case}. To this end, we usually have to bound the supremum of a set of random variables, i.e., to establish maximal inequalities. This point will be made much more clearly when we move on to the second part of this book.
 
 Let us start with the introduction to the Orlicz norm. 
 
 \begin{definition}[Young's function]\label{def:young} A function $\psi:[0,\infty)\to\reals^{\geq 0}$ is said to be a Young's function if it is {\it convex strictly increasing}, and satisfies
 \[
 \lim_{t\to\infty}\psi(t)=\infty~~~{\rm and}~~~\psi(0)=0.
 \]
 \end{definition}
 
 \begin{definition}[Orlicz norm]
 For any Young's function $\psi$ and any $\reals$-valued random variable $X$, define 
 \[
 \norm{X}_{\psi} = \inf\Big\{C>0; \E\psi\Big(\frac{|X|}{C}\Big)\leq 1  \Big\}.
 \]
 \end{definition}
 
 \begin{theorem}\label{thm:orlicz-banach}
 The space of all $\reals$-valued random variables $Z$ such that $\norm{Z}_{\psi}<\infty$, denoted by $L_{\psi}$, is a Banach space equipped with the norm $\norm{\cdot}_{\psi}$.
 \end{theorem}
 \begin{proof}
 Chapter II.9 in \cite{rutickii1961convex}.
 \end{proof}
 
\begin{proposition}\label{prop:maximal1}
\begin{enumerate}[label=(\roman*)]
\item As choosing $\psi: x\to x^p$ for some $p\geq 1$, the Orlicz norm reduces to the $L^p$ norm, $\norm{\cdot}_{L^p}$;
\item as choosing 
\[
\psi=\psi_p: x\to \exp(x^p)-1 
\]
for some $p\geq 1$, the corresponding Orlicz norm $\norm{X}_{\psi_p}\geq \norm{X}_{L^p}\geq \norm{X}_{L^q}$ for any $q\in[1,p]$;
\item for any $1\leq q\leq p$, we have $\norm{X}_{\psi_q}\leq \norm{X}_{\psi_p}(\log 2)^{1/p-1/q}$;
\item for any $p\geq 1$, we have $\norm{X}_{L^p}\leq p! \norm{X}_{\psi_1}$;
\item $\norm{X^2}_{\psi_1}= \norm{X}_{\psi_2}^2$;
\item for any $t>0$ and $a\geq 0$, we have
\[
\Pr(|X|>t)\leq \frac{1+a}{\psi(t/\norm{X}_{\psi})+a}.
\]
\end{enumerate}
\end{proposition}

\begin{exercise}
Prove Proposition \ref{prop:maximal1}.
\end{exercise}

Some quite powerful maximal inequalities can be derived based on the notion of Orlicz norm. 

%\begin{proposition}[Lemma 2.2.2, \cite{vaart2023empirical}]\label{prop:vw-maximal} For arbitrarily dependent $\reals$-valued random variables $X_1,\ldots,X_k$, it holds true that
%\[
%\bnorm{\max_{i\in[m]}X_i}_{\psi} \leq C\psi^{-1}(m)\max_{i\in[m]}\norm{X_i}_{\psi}
%\]
%as long as $\psi$ is nondegenerate and satisfies $\limsup_{x,y\to\infty}\psi(x)\psi(y)/\psi(cxy)<\infty$ for some positive constant $c$.
%\item 
%\end{proposition}

\begin{lemma}\label{lem:orlicz}
For arbitrarily dependent $\reals$-valued random variables $X_1,\ldots,X_m$, it holds true that
\[
\E\max_{i\in[m]}|X_i|\leq \max_{i\in[m]}\norm{X_i}_{\psi}\cdot \psi^{-1}(m).
\]
\end{lemma}
\begin{proof}
First of all, we have, for any $i\in[m]$ and any measurable set $A\subset\Omega$,
\begin{align*}
\int_A|X_i|\d\Pr&=\norm{X_i}_{\psi}\int_A \psi^{-1}\circ \psi\Big(\frac{|X_i|}{\norm{X_i}_{\psi}}\Big)\d\Pr\\
&\leq \norm{X_i}_{\psi}\Pr(A)\psi^{-1}\Big(\frac{1}{\Pr(A)} \int\psi\Big(\frac{|X_i|}{\norm{X_i}_{\psi}}\Big)\d\Pr \Big)\\
&=\norm{X_i}_{\psi}\Pr(A)\cdot \psi^{-1}\Big(\frac{1}{\Pr(A)}\Big),
\end{align*}
where in the inequality we used the Jensen's inequality. 

Next, taking $\{\Omega_i;i\in[m]\}$ to be a partition of $\Omega$ such that $X_i=\max_{\ell\in[m]}X_{\ell}$ over $\Omega_i$. It then holds true that
\begin{align*}
\E\max_{i\in[m]}|X_i| =\sum_{i=1}^m\int_{\Omega_i} |X_i|\d \Pr \leq \max_{i\in[m]}\norm{X_i}_{\psi}\cdot \sum_{i=1}^m\Pr(\Omega_i)\psi^{-1}\Big(\frac{1}{\Pr(\Omega_i)} \Big)\\
\leq \max_{i\in[m]}\norm{X_i}_{\psi}\cdot \psi^{-1}(m),
\end{align*}
where in the last inequality we used Jensen's inequality again. This completes the proof.
\end{proof}

Most of the random variables we are going to handle in this book are those with either $\psi_2$- or $\psi_1$-Orlicz norm bounded. It is hence useful to discuss a little bit more about these two special norms.

\begin{definition}[subgaussian distribution]
An $\reals$-valued random variable $X$ is said to be a subgaussian distribution if $\norm{X}_{\psi_2}<\infty$. 
\end{definition}

\begin{definition}[subexponential distribution]
An $\reals$-valued random variable $X$ is said to be a subexponential distribution if $\norm{X}_{\psi_1}<\infty$. 
\end{definition}

\begin{lemma}\label{lem:subgaussian} The following hold true. 
\begin{enumerate}[label=(\roman*)]
\item $X$ is subgaussian if and only if for all $t\geq 0$, $\Pr(|X|\geq t)\leq 2\exp(-t^2/K_1^2)$ (for some constant $K_1>0$), which, when $\E X=0$, is further equivalent to assuming that, for all $\lambda\in\reals$, $\E\exp\{\lambda X\}\leq \exp(K_2^2\lambda^2)$ (for some constant $K_2>0$).
\item $X$ is subexponential if and only if for all $t\geq 0$, $\Pr(|X|\geq t)\leq 2\exp(-t/K_1)$ (for some constant $K_1>0$), which, when $\E X=0$, is further equivalent to assuming that, for all $|\lambda|\leq 1/K_2$, $\E\exp\{\lambda X\}\leq \exp(K_2^2\lambda^2)$ (for some constant $K_2>0$).
\end{enumerate}
\end{lemma}
\begin{proof}
(i) The last assertion (no need to assume $\E X=0$) implies that
\[
\Pr(X \geq t)\leq \frac{\E e^{\lambda X}}{e^{\lambda t}}\leq e^{-\lambda t + K_2^2\lambda^2}.
\]
Picking $\lambda=t/(2K_2^2)$, we obtain
\[
\Pr(X\geq t)\leq \exp\Big(-\frac{t^2}{4K_2^2} \Big),
\]
yielding the second assertion.

If the second assertion is true, then for any $C>0$ (no need to assume $\E X=0$),
\begin{align*}
\E\exp(X^2/C^2)&=\int_1^{\infty}\Pr(e^{X^2/C^2}\geq t)\d t\\
&=1+\int_1^{\infty}\Pr(X^2\geq C^2\log t)\d t\\
&\leq 1+2\int_1^{\infty}\exp\Big(-\frac{C^2\log t}{K_1^2}  \Big)\d t\\
&=1+2\int_1^{\infty} t^{-C^2/K_1^2}\d t.
\end{align*}
Picking $C^2=3K_1^2$, we have $\E\exp(X^2/(3K_1^2))\leq 2$, and thus $\norm{X}_{\psi_2}\leq \sqrt{3}K_1$.

Thirdly, if the first assertion is true and $\E X=0$, then using the numeric inequality that $e^x\leq x+e^{x^2}$, we have
\begin{align*}
\E e^{\lambda X} \leq \E[\lambda X+e^{\lambda^2X^2}]=\E[e^{\lambda^2X^2}]=\E\Big[ \Big(e^{X^2/\norm{X}_{\psi_2}}\Big)^{\lambda^2\norm{X}_{\psi_2}^2}\Big]\leq 2^{\lambda^2\norm{X}_{\psi_2}^2},
\end{align*}
whenever $\lambda\leq 1/\norm{X}_{\psi_2}$. If $\lambda > 1/\norm{X}_{\psi_2}$, then using the numeric inequality that
\[
2\lambda x\leq 2\norm{X}_{\psi_2}^2\lambda^2+\frac{x^2}{2\norm{X}_{\psi_2}^2},
\]
we derive
\begin{align*}
\E e^{\lambda X}\leq e^{\norm{X}_{\psi_2}^2\lambda^2}\cdot \E e^{X^2/4\norm{X}_{\psi_2}^2}\leq e^{\norm{X}_{\psi_2}^2\lambda^2} \cdot 2^{1/4} \leq e^{2\norm{X}_{\psi_2}^2\lambda^2},
\end{align*}
where in the second inequality we used Jensen and in the last inequality we used the fact that $e^{\norm{X}_{\psi_2}^2\lambda^2}>e>2^{1/4}$ since $\lambda > 1/\norm{X}_{\psi_2}$.

Lastly, if the first assertion is true and $\E X$ is not necessarily 0, Proposition \ref{prop:maximal1}(vi) directly implies the second assertion by choosing $a=1$.

(ii) Left to the readers.
\end{proof}

\begin{exercise}
Prove Lemma \ref{lem:subgaussian}, Part (ii).
\end{exercise}

Specializing to the Orlicz-$\psi_p$ norms, the following lemma gives an alterantive bound to Lemma \ref{lem:orlicz}.

\begin{lemma}\label{lem:orlicz2} Suppose $p\in [1,\infty)$. Then, for arbitrarily dependence $\reals$-valued random variables $X_1,\ldots,X_m$, it holds true that
\[
\bnorm{\max_{i\in[m]}|X_i|}_{\psi_p} \leq \max_{i\in[m]}\bnorm{X}_{\psi_p}\cdot\psi_p^{-1}(m).
\]
\end{lemma}
\begin{proof}
Notice first that, for any $x,y\geq 0$, 
\begin{align*}
\psi_p(x)\psi_p(y)=\Big( e^{x^p}-1\Big)\cdot \Big( e^{y^p}-1\Big) =e^{(xy)^p}-e^{x^p}-e^{y^p}+1\\
\leq e^{(xy)^p}-1 =\psi_p(xy).
\end{align*}
Accordingly, for any $k,y>0$, we obtain
\begin{align*}
\max_{i\in[m]}\psi_p\Big(\frac{|X_i|}{ky} \Big) \cdot \psi_p(y) \leq \max_{i\in[m]}\psi_p\Big(\frac{|X_i|}{k}\Big)\leq \sum_{i=1}^m\psi_p\Big(\frac{|X_i|}{k}\Big).
\end{align*}
Taking expectations on both sides, we obtain
\[
\psi_p(y)\cdot\E\psi_p\Big(\frac{\max_{i\in[m]}|X_i|}{ky} \Big)\leq \sum_{i=1}^m\E\psi_p\Big(\frac{|X_i|}{k}\Big).
\]
Picking $k=\max_i \norm{X_i}_{\psi_p}$ yields then
\[
\psi_p(y) \cdot\E\psi_p\Big(\frac{\max_{i\in[m]}|X_i|}{ky} \Big) \leq m.
\]
Lastly, picking $y=\psi_p^{-1}(m)$, we obtain
\[
\E\psi_p\Big(\frac{\max_{i\in[m]}|X_i|}{ky} \Big) \leq 1.
\]
In other words, we have proven that
\[
\bnorm{\max_{i\in[m]}|X_i|}_{\psi_p} \leq \max_{i\in[m]}\bnorm{X}_{\psi_p}\cdot\psi_p^{-1}(m).
\]
This completes the proof.
\end{proof}

Invoking Proposition \ref{prop:maximal1}, Lemma \ref{lem:orlicz2} gives an (up to constants) equivalent bound to Lemma \ref{lem:orlicz2}:
\[
\E\max_{i\in[m]}|X_i| \leq \frac{1}{\sqrt{\log 2}}\cdot \max_{i\in[m]}\bnorm{X}_{\psi_p}\cdot\psi_p^{-1}(m).
\]

Lastly, the following result connects random variables of a Bernstein tail to the Orlicz-$\psi_1$ norm.

\begin{lemma}[Bernstein tails]\label{lem:2bernstein} Suppose that there exist two fixed constants $a,b>0$ such that the random variable $X$ satisfies
\[
\Pr(|X|>t)\leq 2\exp\Big(-\frac{t^2}{b+at} \Big),~~~\text{ for all }t\geq 0.
\]
We then have 
\[
\norm{X}_{\psi_1}\leq 6a+\sqrt{\frac{6b}{\log 2}}.
\]
\end{lemma}
\begin{proof}
We have, when $t\leq b/a$, $\Pr(|X|>t)\leq 2\exp(-t^2/(2b))$; when $t>b/a$, $\Pr(|X|>t)\leq 2\exp(-t/(2a))$. Accordingly,
\begin{align*}
\Pr\Big\{|X|\ind(|X|\leq b/a)>t\Big\}\leq 2\exp(-t^2/(2b))\\
~~~{\rm and}~~~\Pr\Big\{|X|\ind(|X|> b/a)>t\Big\} \leq 2\exp(-t/(2a)).
\end{align*}
Leveraging Lemma \ref{lem:subgaussian} (noticing that $\E X=0$ is not needed here), we then obtain
\[
\bnorm{|X|\ind(|X|\leq b/a)}_{\psi_2} \leq \sqrt{6b}~~~{\rm and}~~~\bnorm{|X|\ind(|X|> b/a)}_{\psi_1} \leq 6a.
\]
Lastly, employing Proposition \ref{prop:maximal1}, we have
\[
\bnorm{|X|\ind(|X|\leq b/a)}_{\psi_1}\leq (\log 2)^{-1/2}\bnorm{|X|\ind(|X|\leq b/a)}_{\psi_2},
\]
and thus, using Theorem \ref{thm:orlicz-banach} completes the proof.
\end{proof}

\section{Independence sampling}\label{sec:iid}

Statisticians are inevitably familiar with the independence sampling paradigm, which has produced fruitful results in both probability and mathematical statistics. Below we briefly review some of the most fundamental ideas when independence between observed points can be assumed. 

Notably speaking, unless otherwise emphasized, in this book we always take a triangular array perspective towards sampling. In other words, for each positive integer $n$, what we observe is
\begin{align}\label{eq:indep}
\bX_{n1},\ldots, \bX_{nn} \text{ that are mutually independent}.
\end{align}
Different $n$ may generate totally different $\bX_{ni}$'s. It is also worthwhile pointing out that, for each fixed $n$, we do {\it not} require $\{\bX_{ni}\}$'s to be identically distributed.

Needless to say, none of the results presented in this chapter can be directly applied to analyzing statistics whose randomness comes solely from a random permutation. 

%In the sequel, for presentation clearness, we often drop the subscript $n$ in $\bX_{ni}$'s and 

\subsection{Law of large numbers}

We start with the case that all $X_{i}=X_{ni}$'s are random scalars.  Introduce 
\[
\mu_{ni}=\E[X_{ni}], ~\sigma^2_{ni}=\var(X_{ni}), \text{ and } s_n^2=\sum_{i=1}^n\var(X_{ni}). 
\]
Let $\bar X_n=n^{-1}\sum_{i=1}^nX_{ni}$ be the sample mean. 

\begin{theorem}[Law of large numbers] Assume \eqref{eq:indep}. It then holds true that, for each $t>0$,
\begin{align*}
\Pr\Big(\Big|\bar X_n-\frac{1}{n}\sum_{i=1}^n\mu_{ni}\Big|>t\Big)\leq \frac{s_n^2}{n^2t^2}.
\end{align*}
In particular, 
\begin{enumerate}[label=(\roman*)]
\item if $s_n/n\to 0$ as $n\to\infty$, we have $\bar X_n-\E \bar X_n\stackrel{\Pr}{\to} 0$;
\item furthermore, if it is assumed that 
\[
\sup_{n,i}\E|X_{ni}-\mu_{ni}|^4<\infty,
\]
then $\bar X_n-\E \bar X_n\stackrel{\rm a.s.}{\to} 0$.
\end{enumerate}
\end{theorem}
\begin{proof}
Use Markov's inequality and the first lemma of Borel-Cantelli.
\end{proof}

\begin{example}[Number of circles in a permutation]\label{example:circle} Consider a permutation in $\cS_5$ such that $\pi([5])=[2,3,1,5,4]$. Then one circle sends 1 to 2, 2 to 3, 3 to 1, and another circle sends 4 to 5 and 5 to 4. There are accordingly two circles in this particular permutation. Letting $S_N$ be the number of circles in a uniform random permutation $\pi$, the next theorem establishes the convergence of $S_N$ to its mean.
\begin{theorem} We have
\[
\frac{S_N-\log N}{\log N} \stackrel{\Pr}{\to} 0.
\]
\end{theorem}
\begin{proof}
Consider the sequence $1, \pi(1), \pi(\pi(1)),\ldots$, which will eventually get back to 1. This then gives the first circle, written as $(1,\pi(1),\cdots, \pi^k(1))$ with $\pi^{k+1}(1)=1$ for the first time. Then, picking the smallest integer that is not in the first circle, say $i$, repeating the same process yields the second circle $(i,\pi(i),\ldots,\pi^{\ell}(i))$ with $\pi^{\ell+1}(i)=i$ for the first time. This yields the circle decomposition; e.g., when $\pi([5])=[2,3,1,5,4]$, its circle decomposition is $(1,2,3)(4,5)$. 

Let $X_{N,k}=\ind(\text{the }k\text{-th item in circle decomposition completes a circle})$; e.g., when $\pi([5])=(1,2,3)(4,5)$, $X_{5,1}=X_{5,2}=X_{5,4}=0$ and $X_{5,3}=X_{5,5}=1$. We then have that $S_N=\sum_{i=1}^N X_{N,i}$. Furthermore, the following two facts hold:
\begin{enumerate}[label=(\roman*)]
\item for any $N$, $X_{N,1},\ldots,X_{N,N}$ are independent (think about why);
\item for each $N$ and any $i\in[N]$, $X_{N,i}$ is Bernoulli distributed with $\Pr(X_{N,i}=1)=(N-i+1)^{-1}$.
\end{enumerate}
It then holds true that
\[
\E S_N=\sum_{i=1}^{N}\frac{1}{i}=\log N(1+o(1))~~~{\rm and}~~\var(S_N)=\sum_{i=1}^N\frac{1}{i}\Big(1-\frac{1}{i}\Big)=\log N(1+o(1)).
\]
Accordingly, by Markov's inequality, for any fixed $t>0$,
\[
\Pr\Big(|S_N-\E S_N|>t\log N \Big) \leq \frac{\var(S_N)}{t^2(\log N)^2}\to 0,
\]
and thus
\[
\frac{S_N-\log N}{\log N} \stackrel{\Pr}{\to} 0.
\]
This completes the proof.
\end{proof}
\end{example}

\subsection{Central limit theorems}

Central limit theorems (CLTs), exemplified by the weak convergence of a sequence of statistics (random variables) to a continuous limit distribution, typically the Gaussian, not only play a pivotal role but also arguably serve as the cornerstone in the domain of statistical inference. They are crucial in uncertainty quantification, useful for hypothesis testing and building confidence intervals.

This section endeavors to offer a concise review of this pivotal topic, setting the stage for resonance in the subsequent chapters.

\subsubsection{Lindeberg-Feller-Lyapunov CLT}

Lindeberg-Feller CLT concerns the triangular array setting with independent but possible non-identically distributed random variables. In this regard, it is the most powerful result available to statisticians.

\begin{theorem}[Lindeberg-Feller-Lyapunov CLT]\label{thm:LFL-CLT} Assume \eqref{eq:indep}.
\begin{enumerate}[label=(\roman*)]
\item Supposing that, for any $\epsilon>0$,
\[
\text{(Lindeberg condition)}~~\lim_{n\to\infty}\frac{1}{s_n^2}\sum_{i=1}^n\E\Big[(X_{ni}-\mu_{ni})^2\cdot \ind(|X_{nk}-\mu_{nk}|>\epsilon s_n) \Big]=0,
\]
then
\begin{align}\label{eq:LF-CLT}
\lim_{n\to\infty}\sup_{t\in\mathbb{R}}\Big|\Pr\Big(\frac{\sum_{i=1}^n(X_{ni}-\mu_{ni})}{s_n}\leq t\Big) - \Phi(t)\Big|= 0,
\end{align}
where $\Phi(\cdot)$ is the CDF of the standard Gaussian.
\item In particular, if there exists some $\delta>0$ such that
\[
\text{(Lyapunov condition)}~~\lim_{n\to\infty}\frac{1}{s_n^{2+\delta}}\sum_{i=1}^n\E|X_{ni}-\mu_{ni}|^{2+\delta}=0,
\]
the above uniform convergence \eqref{eq:LF-CLT} holds.
\end{enumerate}
\end{theorem}
\begin{proof}
The following telescoping lemma will be used in the proof.
\begin{lemma}[Telescoping lemma]For any real sequences $\{a_i,b_i; i\in[n]\}$ such that $\sup_{i\in[n]}\max\{|a_i|,|b_i|\leq 1$, it holds true that
\[
\Big|\prod_{i=1}^na_i-\prod_{i=1}^nb_i  \Big| \leq \sum_{i=1}^n|a_i-b_i|.
\]
\end{lemma}

\noindent {\bf [Proof of Part (i)].} By translation invariance, without loss of generality, assume $\mu_{ni}=0$ and $s_n^2=1$ so that Linderberg condition translates to 
\[
\lim_{n\to\infty}\sum_{i=1}^n\E[X_{ni}^2\ind(|X_{ni}|>\epsilon)]=0
\]
and we aim to prove $\sum_{i=1}^nX_{ni}\Rightarrow N(0,1)$. 

Invoking Theorem \ref{thm:levy}, it suffices to consider 
\[
\phi_{S_n}(t)=\prod_{i=1}^n \phi_{X_{ni}}(t).
\]
Noting that, by Taylor expansion, for each $n$ and $i\in[n]$, we have
\begin{align*}
\Big|\phi_{X_{ni}}(t)-1+\frac{t^2\sigma^2_{ni}}{2} &=\Big|\E\Big(e^{{\rm i}tX_{ni}}-1-{\rm i}tX_{ni}+\frac{t^2X_{ni}^2}{2}  \Big)\Big|\\
&\leq \E\min\Big\{t^2X_{ni}^2, \frac{|t|^3|X_{ni}|^3}{6}  \Big\},
\end{align*}
so that, by telescoping,
\begin{align*}
\Big|\phi_{S_n}(t)-\prod_{i=1}^n\Big(1-\frac{t^2\sigma_{ni}^2}{2}\Big)\Big| &\leq \sum_{i=1}^n\Big|\phi_{X_{ni}}(t)-1+\frac{t^2\sigma_{ni}^2}{2} \Big|\\
&\leq \sum_{i=1}^n\E\min\Big\{t^2X_{ni}^2,\frac{|t|^3|X_{ni}|^3}{6} \Big\}\\
&\leq t^2\sum_{i=1}^n\E[X_{ni}^2\ind(|X_{ni}|>\epsilon)]+\frac{|t|^3\epsilon}{6},
\end{align*}
where in the last inequality, we used the fact that
\begin{align*}
 \E\min\Big\{t^2X_{ni}^2, \frac{|t|^3|X_{ni}|^3}{6}  \Big\} &\leq \E[t^2X_{ni}^2\ind(|X_{ni}|\geq \epsilon)]+\frac{|t|^3\epsilon \E|X_{ni}|^2}{6}\\
 &=t^2\E[X_{ni}^2\ind(|X_{ni}|\geq \epsilon)]+\frac{|t|^3\epsilon \sigma_{ni}^2}{6}.
\end{align*}
By first employing Lindeberg condition and then pushing $\epsilon$ goes to 0, the above derivation implies
\[
\lim_{n\to\infty}\Big|\phi_{S_n}(t)-\prod_{i=1}^n\Big(1-\frac{t^2\sigma_{ni}^2}{2}\Big)\Big| =0.
\]

On the other hand, by telescoping, 
\begin{align*}
\Big|e^{-t^2/2}-\prod_{i=1}^n\Big(1-\frac{t^2\sigma_{ni}^2}{2}\Big)  \Big| &\leq \sum_{i=1}^n\Big\{e^{-t^2\sigma_{ni}^2}-\Big(1-\frac{t^2\sigma_{ni}^2}{2}\Big) \Big\}\\
&\leq \frac12\sum_{i=1}^nt^4\sigma_{ni}^4\\
&\leq \frac{t^4\max_{i\in[n]}\sigma_{ni}^2}{2},
\end{align*}
where in the second inequality we used the fact that, for all $x\geq 0$, $|e^x-(1-x)|\leq x^2/2$ and the last inequality is true since we have $s_n^2=1$.

It remains to prove that 
\[
\limsup_{n\to\infty}\max_{i\in[n]}\sigma_{ni}^2=0,
\]
which is indeed true since 
\begin{align*}
\limsup_{n\to\infty}\max_{i\in[n]}\sigma_{ni}^2&=\limsup_{n\to\infty}\max_{i\in[n]}\Big\{ \E[X_{ni}^2\ind(|X_{ni}|\leq \epsilon)]+ \E[X_{ni}^2\ind(|X_{ni}|>\epsilon)] \Big\}\\
&\leq \epsilon^2+\limsup_{n\to\infty}\max_{i\in[n]}\E[X_{ni}^2\ind(|X_{ni}|>\epsilon)]\\
&=\epsilon^2
\end{align*}
holds for arbitrarily small $\epsilon>0$.

Accordingly, piecing together, we obtain
\[
\lim_{n\to\infty}\Big|\phi_{S_n}(t)-e^{-t^2/2}\Big|=0~~~\text{for any }t\in\reals,
\]
and thus using Theorem \ref{thm:levy} and then Theorem \ref{thm:polya} completes the proof of the first part. 

\noindent {\bf [Proof of Part (ii)].} Notice that, for any $\delta>0$ and $\epsilon>0$,
\[
(X_{ni}-\mu_{ni})^2\ind(|X_{ni}-\mu_{ni}|>\epsilon s_n )\leq \frac{|X_{ni}-\mu_{ni}|^{2+\delta}}{(\epsilon s_n)^\delta}.
\]
Therefore,
\begin{align*}
\frac{1}{s_n^2}\sum_{i=1}^n\E\Big[(X_{ni}-\mu_{ni})^2\cdot \ind(|X_{nk}-\mu_{nk}|>\epsilon s_n) \Big]\leq  \frac{1}{s_n^2}\sum_{i=1}^n\E\frac{|X_{ni}-\mu_{ni}|^{2+\delta}}{(\epsilon s_n)^\delta},
\end{align*}
the right of which will go to zero under Lyapunov condition.
\end{proof}

\begin{example}[Number of circles in a permutation, cont.]\label{example:circle2} Let's continue Example \ref{example:circle}. Recall that the number of circles in $\pi_N$ is $S_N=\sum_{i=1}^NX_{Ni}$ with
\[
\Pr(X_{Ni}=1)=1-\Pr(X_{Ni}=0)=\frac{1}{N-i+1}.
\]
Accordingly, verifying Lyapunov's condition and recalling 
\[
s_N^2 = \log N(1+o(1)) \text{ and }\sum_{i=1}^N\E|X_{Ni}-\mu_{Ni}|^3=O(\log N),
\]
we obtain that 
\[
\text{(Goncharov) }~~~~~~ \frac{S_N-\log N}{\sqrt{\log N}}\Rightarrow N(0,1).
\]
\end{example}

\begin{example}[Linear regression with a fixed design] Suppose we observe $\{(Y_{ni},\bx_{ni})\in\reals^{d_n+1};i\in[n]\}$  that satisfies a linear regression model
\[
Y_i=\beta_{n0}+\bbeta_{n1}^\top\bx_{ni}+\epsilon_{ni},
\]
where the design vectors $\bx_{ni}$'s are assumed to be fixed (a.k.a. a fixed design) and the only randomness comes from noises $\epsilon_{ni}$'s that we assume to be i.i.d., of mean 0, variance $\sigma_n^2$, and finite third moment. The following theorem is due to Peter Huber.
\begin{theorem}\label{thm:huber} The ordinary least squares estimator, defined as 
\[
\hat\bbeta_n:=(\Xb_n^\top \Xb_n)^{-1} \Xb_n^\top \bY_n
\]
with
\begin{align*}
  \Xb_n&= \left(\begin{matrix}
  1 & 1 & \ldots & 1 \\
  \bx_{n1} & \bx_{n2} & \ldots & \bx_{nn}
  \end{matrix}\right)^\top ~~~{\rm and }~~~\bY_n=(Y_{n1},\ldots,Y_{nn})^\top,
\end{align*}
is an asymptotically normal estimator of $\bbeta_n=(\beta_{n0},\bbeta_{n1}^\top)^\top$\footnote{In a high-dimensional setting when $d_n$ diverges, this means any normalized linear projection is asymptotically standard normal.} if, letting $\ba_{ni}\in\reals^{d_n+1}$ denote the $i$-th column in $(\Xb_n^\top \Xb_n)^{-1/2} \Xb_n^\top$, we have
\[
\max_{i\in[n]}\norm{a_{ni}}_2\to 0~~~{\rm as}~n\to\infty.
\]
\end{theorem}
\begin{proof}
Write $\beps_n=(\epsilon_{n1},\ldots,\epsilon_{nn})^\top$. Simple algebra yields
\[
\hat\bbeta_n= \bbeta_n + (\Xb_n^\top \Xb_n)^{-1} \Xb_n^\top \beps_n.
\]
so that
\[
  (\Xb_n^\top \Xb_n)^{1/2}\left(\hat{\bbeta}_n-\bbeta_n\right)= (\Xb_n^\top \Xb_n)^{-1/2} \Xb_n^\top \beps_n.
\]
It remains to prove that, for any $\bt_n\in\mathbb{R}^{d_n+1}\backslash \{\bm{0}_{d_n+1}\}$,
\[
 T_n:= \frac{\bt_n^\top (\Xb_n^\top \Xb_n)^{-1/2} \Xb_n^\top \beps_n}{\sqrt{\var(\bt_n^\top (\Xb_n^\top \Xb_n)^{-1/2} \Xb_n^\top \beps_n)}}=  \frac{\sum_{i=1}^n(\bt_n^\top \ba_{ni}) \epsilon_{ni}}{\sqrt{\var(\sum_{i=1}^n\bt_n^\top \ba_{ni})}}
  \]
  is asymptotically normal. Notice that
    \begin{align*}
  \sigma_{ni}^2:= {\rm var}\left([\bt_n^\top \ba_{ni}] \epsilon_{ni}\right) = [\bt_n^\top \ba_{ni}]^2\var\left(\epsilon_{ni}\right) = [\bt_n^\top \ba_{ni}]^2\sigma_n^2
  \end{align*}
  and
  \begin{align*}
  s_n^2&:= \sum_{i=1}^n \sigma_{ni}^2 = \sigma_n^2\sum_{i=1}^n [\bt_n^\top \ba_{ni}]^2 \\
  &= \sigma_n^2 \bt_n^\top (\Xb_n^\top \Xb_n)^{-1/2} \Xb_n^\top \Xb_n (\Xb_n^\top \Xb_n)^{-1/2} \bt_n = \sigma_n^2 \norm{\bt_n}_2^2.
  \end{align*}
  We obtain that
  \[
  T_n=\sum_{i=1}^n Z_{ni},~~~{\rm with}~Z_{ni}:=\frac{(\bt_n^\top\ba_{ni})\epsilon_{ni}}{s_n}~~{\rm and }~~\var(T_n)=1.
  \]
  Now verifying Lyapunov's CLT, we have 
  \begin{align*}
\sum_{i=1}^n\E|Z_{ni}|^3= \sum_{i=1}^n\frac{[\bt_n^\top\ba_{ni}]^3\sigma_n^3}{s_n^3}\leq \sigma_n\norm{\bt_n}_2\max_{i\in[n]}\norm{\ba_{ni}}_2\cdot\frac{\sum_{i=1}^n[\bt_n^\top\ba_{ni}]^2\sigma_n^2}{s_n^3}\\
=\max_{i\in[n]}\norm{\ba_{ni}}_2\to 0,
  \end{align*}
  which completes the proof.
\end{proof}
\end{example}

Next, consider multivariate $\bX_{ni}\in\mathbb{R}^d$ with a fixed $d$. We can similarly define 
\[
\bmu_{ni}=\E[\bX_{ni}]~~~{\rm and}~~\bSigma_n=\sum_{i=1}^n\cov(\bX_{ni}).
\]
\begin{theorem}[Multivariate Lindeberg-Feller-Lyapunov CLT]\label{thm:mlfclt} Assume \eqref{eq:indep}. Denote $\bW_{ni}=\bSigma_n^{-1/2}(\bX_{ni}-\bmu_{ni})$.
\begin{enumerate}[label=(\roman*)]
\item Suppose for every $\epsilon>0$,
\[
\text{(Lindeberg condition)}~~\sum_{i=1}^n\E[\|\bW_{ni}\|_2^2\ind(\|\bW_{ni}\|_2\geq \epsilon)]\to 0.
\]
We then have $\sum_{i=1}^n \bW_{ni}\Rightarrow N(0,\bI_p)$.
\item  In particular, if there exists some $\delta>0$ such that
\[
\text{(Lyapunov condition)}~~\lim_{n\to\infty}\sum_{i=1}^n\E\|\bW_{ni}\|_2^{2+\delta}=0,
\]
 we have $\sum_{i=1}^n \bW_{ni}\Rightarrow N(0,\Ib_d)$.
\end{enumerate}
\end{theorem}

\begin{exercise}
Prove Theorem \ref{thm:mlfclt} using Corollary \ref{cor:cw-device}.
\end{exercise}

\subsubsection{Bounds on CLTs}

Lindeberg-Feller CLTs and its variants establish convergence of one sequence of CDFs to its limit. However, in many applications it may also be of interest to learn {\it how fast} the convergence can be. This question is resolved in the next two theorems.

\begin{theorem}[Berry–Esseen Theorem]\label{thm:berry-esseen} Assume \eqref{eq:indep}. There exists a universal constant $K>0$ such that for all $n$,
\[
\sup_{t\in\mathbb{R}}\Big|\Pr\Big(\frac{\sum_{i=1}^n(X_{ni}-\mu_{ni})}{s_n}\leq t\Big) - \Phi(t)\Big| \leq \frac{K\sum_{i=1}^n\E|X_{ni}-\mu_{ni}|^3}{s_n^{3}}.
\]
\end{theorem}

Let's prove a weaker version of Theorem \ref{thm:berry-esseen}, which can be wrapped up in one page while also involving an important trick we will use later. 

In this regard, we consider a simpler setting that 
\[
X_{n1},\ldots,X_{nn} \text{ are i.i.d.}. 
\]
We further drop the subscript $n$ to save the notation. Without loss of generality then, let's assume they are mean-zero, of unit variance and finite third moment. Consider any smooth function $\phi$ such that the first three derivatives are bounded. One could then establish that, for 
\[
Z_n:=\frac{1}{\sqrt{n}}\sum_{i=1}^nX_n, 
\]
there exists a universal constant $K>0$ such that
\[
\Big|\E\phi(Z_n)-\E\phi(Y)\Big|\leq \norm{\phi'''}_{\infty}\cdot \frac{K\E|X|^3}{\sqrt{n}},
\]
where $Y\sim N(0,1)$.

Indeed, let $\{Y_i;i\in[n]\}$ be i.i.d. standard Gaussian random variable so that $\E\phi(Y)=\E\phi(\sum_{i=1}^nY_i/\sqrt{n})$. Let's then implement an individual swapping trick and introduce
\[
Z_{n,i}:=\frac{X_1+\ldots+X_i+Y_{i+1}+\ldots+Y_n}{\sqrt{n}}~~~\text{for }i=0,\ldots,n.
\]
Accordingly, we have
\begin{align*}
&\Big|\E\phi(Z_n)-\E\phi(Y)\Big|=\Big|\sum_{i=0}^{n-1}\E\Big\{\phi(Z_{n,i+1})-\phi(Z_{n,i})\Big\}\Big|\\
=&\Big|\sum_{i=0}^{n-1}\E\Big\{\phi'(\tilde Z_{n,i})\frac{X_{i+1}}{\sqrt{n}}+\frac{1}{2}\phi''(\tilde Z_{n,i})\frac{X_{i+1}^2}{n}+R_{n,i}-\phi'(\tilde Z_{n,i})\frac{Y_{i+1}}{\sqrt{n}}-\frac{1}{2}\phi''(\tilde Z_{n,i})\frac{Y_{i+1}^2}{n}-\tilde R_{n,i}\Big\}\Big|\\
\leq& \sum_{i=1}^{n-1}|\E R_{n,i}- \E \tilde R_{n,i}|\\
\leq&\norm{\phi'''}_{\infty}\cdot \frac{K\E|X|^3}{\sqrt{n}},
\end{align*}
where in the second equality, for each $i$, we Taylor expand both $\phi(Z_{n,i+1})$ and $\phi(Z_{n,i})$ at $\tilde Z_{n,i}:=(X_1+\cdots+X_i+Y_{i+2}+\cdots+Y_n)/\sqrt{n}$, and the remainder terms satisify
\[
R_{n,i}=O(|X_{i+1}|^{3}/n^{3/2})\norm{\phi'''}_{\infty}~~~{\rm and}~~~\tilde R_{n,i}=O(|Y_{i+1}|^{3}/n^{3/2})\norm{\phi'''}_{\infty}.
\]

As a direct consequence, by smoothing the indicator function $\ind(x\leq t)$, we obtain a smooth function $\phi$ that takes values 1 and 0 over $(-\infty, t)$ and $(t+\epsilon,\infty)$, respectively. This function has the third derivative upper bounded by $O(\epsilon^{-3})$. Optimizing over $\epsilon$, it yields
\begin{align}\label{eq:exe1}
\Big| \Pr(Z_n\leq t)-\Phi(t)\Big| \leq Kn^{-1/8},
\end{align}
for some $K>0$.

This clever swapping idea is due to Lindeberg \citep{lindeberg1922neue}, for which we call it {\it Lindeberg swapping method}.

\begin{exercise}
Please complete the proof of \eqref{eq:exe1} with all missing parts filled.
\end{exercise}

Berry-Esseen bound can usually control the tail probabilities up to an order of logarithmic-$n$, which falls in the regime of {\it small deviations}. The next theorem, on the other hand, establishes tail probability bounds at an order of polynomial-$n$. This type of theorems can thus handle  {\it moderate deviations}, the research of which was initiated by Khinchin \citep{khintchine1929neuen} and Cram\'er \citep{cramer1994nouveau}. The following gives such as a result  that applies to triangular arrays.

\begin{theorem}[Cramer's moderate deviation] Assume \eqref{eq:indep}. Further assume there exist universal constants $c_1,c_2,t_0>0$ such that
\[
\inf_{n}s_n^2/n\geq c_1^2~~~\text{and}~~~\sup_{i,n}\E\Big[\exp\Big(t_0\sqrt{|X_{ni}-\mu_{ni}|}\Big)\Big]\leq c.
\] 
We then have
\[
\Pr\Big(\frac{\sum_{i=1}^n(X_{ni}-\mu_{ni})}{s_n}>t\Big)\Big/ (1-\Phi(t))=1+M_n(1+t^3)/\sqrt{n}
\]
holds for all $t\in(0, (c_1t_0^2)^{1/3}n^{1/6})$, where $\sup_n|M_n|$ is bounded by a constant only depending on $c_1t_0^2$ and $c_2$.
\end{theorem}
\begin{proof}
Proposition 4.6 in \cite{chen2013stein}.
\end{proof}

\subsection{Moment inequalities}

Inequalities that bound the moments of a random variable can usually be translated to those that bound the tail probabilities; cf. Lemma \ref{lem:subgaussian}. In this regard, although many useful ones exist, this book is only interested in the following two, namely, Hoeffding's \citep{hoeffding1963probability} and Bernstein's \citep{bernstein1924modification}.

\begin{theorem}[Hoeffding's inequality]\label{thm:2hoeffding} Assume \eqref{eq:indep}. and suppose $X_{ni}\in [a_{ni},b_{ni}]$. We then have, 
\[
\log\E\exp\Big\{\lambda\sum_{i=1}^n(X_{ni}-\mu_{ni}) \Big\} \leq \frac{\lambda^2}{8}\sum_{i=1}^n(b_{ni}-a_{ni})^2~~~\text{for any }\lambda\in\reals,
\]
and thus, for any $t\geq 0$,
\[
\Pr\Big(\sum_{i=1}^n(X_{ni}-\mu_{ni})\geq t\Big) \leq \exp\Big(-\frac{2t^2}{\sum_{i=1}^n(b_{ni}-a_{ni})^2} \Big).
\]
\end{theorem}

Here we give a proof of the Hoeffding's inequality with a slightly weaker bound that yields a worse constant. However, the derivation is easier to understand. Since the bound is location-invariant, we can always shift the random variable $X_{ni}-\E X_{ni}$ to be within $[-\frac{b_{ni}-a_{ni}}{2}, \frac{b_{ni}-a_{ni}}{2}]$ so that
\[
\bnorm{X_{ni}-\E X_{ni}}_{\psi_2}\leq \frac{b_{ni}-a_{ni}}{2}.
\]
Without loss of generality then, assume $\E X_{ni}=0$ for all $i\in[n]$. Using Lemma \ref{lem:subgaussian}, we have 
\begin{align*}
\E\exp\Big(\lambda\sum_{i=1}^nX_{ni}\Big)=\prod_{i=1}^n\E\exp\Big(\lambda X_{ni}\Big)\leq \prod_{i=1}^n\exp\Big(2\lambda^2\norm{X_{ni}}_{\psi_2}^2\Big)\\
=\exp\Big(\frac{\lambda^2\sum_{i=1}^n(b_{ni}-a_{ni})^2}{2}\Big).
\end{align*}
The tail probability can then be obtained via Markov's inequality.

\begin{example}[Symmetrized t-statistic] Assume \eqref{eq:indep}. Further suppose $X_{ni}$ to be symmetric around its mean and consider the following t-statistic
\[
T_n=\frac{\sum_{i=1}^n(X_{ni}-\mu_{ni})}{\sqrt{\sum_{i=1}^n(X_{ni}-\mu_{ni})^2}}.
\]
It may be a little bit surprising that, for any $t>0$, 
\[
\Pr(T_n>t)\leq \exp(-t^2/2).
\] 
The proof is based on Hoeffding's inequality (Theorem \ref{thm:2hoeffding}) and the symmetrization trick, which we will heavily use in subsequent chapters. Without loss of generality, let's assume $\mu_{ni}=0$ and thus
\begin{align*}
\Pr(T_n>t)=\Pr\Big(\frac{\sum_{i=1}^n X_{ni}}{\sqrt{\sum_{i=1}^nX_{ni}^2}}>t \Big)=\Pr\Big(\frac{\sum_{i=1}^n \epsilon_iX_{ni}}{\sqrt{\sum_{i=1}^nX_{ni}^2}}>t \Big),
\end{align*}
where $\{\epsilon_i;i\in[n]\}$ are independent mean-zero symmetric Bernoulli taking values in $\{-1,1\}$. Then, by Hoeffding's inequality, we obtain
\[
\Pr\Big(\frac{\sum_{i=1}^n \epsilon_iX_{ni}}{\sqrt{\sum_{i=1}^nX_{ni}^2}}>t \mid X_{n1},\ldots,X_{nn}\Big)\leq \exp(-t^2/2),
\]
so that
\[
\Pr(T_n>t)=\E\Big\{\P\Big(\frac{\sum_{i=1}^n \epsilon_iX_{ni}}{\sqrt{\sum_{i=1}^nX_{ni}^2}}>t \mid X_{n1},\ldots,X_{nn}\Big)\Big\}\leq \exp(-t^2/2).
\]
This quite clever bound is due to Bahadur and Eaton \citep{efron1969student}.
\end{example}

\begin{theorem}[Bernstein's inequality]\label{thm:bernstein} Assume \eqref{eq:indep} and the existence of a non-random constant $M_n$ such that 
\[
\sup_{i\in[n]}\Big|X_{ni}-\mu_{ni}\Big|\leq M_n. 
\]
We then have
\[
\Pr\Big(\sum_{i=1}^n(X_{ni}-\mu_{ni})\geq t\Big) \leq \exp\Big(-\frac{t^2/2}{s_n^2+M_nt/3} \Big).
\]
\end{theorem}
\begin{proof}
Without loss of generality, let's again assume $\mu_{ni}=0$ for all $i\in[n]$. Taylor expanding $e^{\lambda x}$, for any $\lambda\in (0, 1/M)$, then gives
\[
\E e^{\lambda X_{ni}} \leq \frac{\lambda^2\sigma_{ni}^2}{2} + \sum_{i=3}^{\infty}\frac{\lambda^q\sigma_{ni}^2M_n^{i-2}}{q!}\leq \frac{\sigma_{ni}^2}{2}\sum_{i=2}^{\infty}\lambda^iM^{i-2}=\frac{\sigma_{ni}^2\lambda^2}{2(1-\lambda M)},
\]
implying that
\[
\E e^{\lambda \sum_{i=1}^nX_{ni}}\leq \frac{\lambda^2s_n^2}{2(1-\lambda M)}.
\]
Optimizing $\lambda$ over $(0,1/M)$ then yields the desired bound.
\end{proof}

\begin{exercise}\label{exe:bernstein}
Please complete the proof of Theorem \ref{thm:bernstein}.
\end{exercise}

\begin{exercise}\label{exe:berstein2}
An alternative version to the above Bernstein's inequality is the following. Assume \eqref{eq:indep}. Then there exists a universal constant $C>0$ such that for any $t\geq 0$,
\[
\Pr\Big(\sum_{i=1}^n(X_{ni}-\mu_{ni})\geq t\Big) \leq \exp\Big(-\frac{Ct^2}{\sigma_n^2+a_nt} \Big),
\]
where
\[
\sigma_n^2:=\sum_{i=1}^n\bnorm{X_{ni}-\mu_{ni}}_{\psi_1}^2~~~{\rm and}~~~a_n:=\max_{i\in[n]}\bnorm{X_{ni}-\mu_{ni}}_{\psi_1}.
\]
Please prove it and specify a value of the constant $C$.
\end{exercise}

\begin{example}[Hanson-Wright inequality] Let $\bX_n=(X_{n1},\ldots,X_{nn})^\top$ contain mean-zero independent entries satisfying $K_n:=\max_{i\in[n]}\norm{X_i}_{\psi_2}<\infty$ and  let $\Ab_n\in\reals^{n\times n}$ be an arbitrary non-zero deterministic matrix. There then exists a universal constant $C>0$ such that, for any $t>0$,
\[
\Pr\Big\{\Big|\bX_n^\top\Ab_n\bX_n-\E[\bX_n^\top\Ab_n\bX_n]\Big|\geq t\Big\} \leq 2\exp\Big(-\frac{Ct^2}{K_n^4\norm{\Ab_n}_{\rm F}^2+K_n^2\norm{\Ab_n}_{\op}t} \Big).
\]
Indeed, dropping the subscript $n$ and letting $\Ab=[a_{i,j}]$, we have 
\[
\bX^\top\Ab\bX-\E[\bX^\top\Ab\bX]=\sum_{i=1}^na_{i,i}(X_i^2-\E[X_i^2])+\sum_{i\ne j}a_{i,j}X_iX_j.
\]
For the diagonal term, noticing that 
\[
\norm{a_{i,i}(X_i^2-\E X_i^2)}_{\psi_1}\lesssim |a_{i,i}|\cdot\norm{X_i}_{\psi_2}^2\leq \norm{\Ab}_{\op}K^2,
\]
Bernstein inequality in the form of Exercise \ref{exe:berstein2} then shows
\begin{align}\label{eq:hw1}
\Pr\Big(\Big|\sum_{i=1}^na_{i,i}(X_i^2-\E[X_i^2])\Big|>t\Big)\leq 2\exp\Big(-\frac{Ct^2}{K^2\norm{\Ab}_{\op}t+K^4\norm{\Ab}_{\rm F}^2}\Big).
\end{align}
For the off-diagonal term, we need a quite deep decoupling technique, which is due to \citet[Theorem 3.1.1]{de2012decoupling}.
\begin{theorem}For any $p\geq 1$, we have
\[
\bnorm{\sum_{i\ne j}a_{i,j}X_iX_j}_{L^p}\leq \bnorm{\sum_{i\ne j}a_{i,j}X_iX_j'}_{L^p},
\]
where $\bX'=(X_1',\ldots,X_n')$ is an independent copy of $\bX$. 
\end{theorem}
Introducing $\Ab^o$ to be $\Ab$ with all diagonals replaced by 0 so that
\[
\sum_{i\ne j}a_{i,j}X_iX_j=\bX^\top\Ab^o\bX.
\]
Using the above bound, we can now upper bound the MGF of $\bX^\top\Ab^o\bX$ by that of $\bX^\top\Ab^o\bX'$, the latter of which is subgaussian conditional on either $\bX$ or $\bX'$ such that
\begin{align*}
\E[\E[e^{\lambda\bX^\top\Ab^o\bX'}\mid \bX]] &\leq \E[\exp(C\lambda^2K^2\norm{\Ab^o \bX'}^2)]\\
&=\E[\E\exp(\sqrt{2C}K\lambda\cdot \bZ^\top\Ab^o\bX' \mid \bX')]\\
&=\E[\E\exp(\sqrt{2C}K\lambda\cdot \bZ^\top\Ab^o\bX' \mid \bZ)]\\
&=\E[\exp(2C^2K^4\lambda^2\cdot \norm{\Ab^o\bZ}^2)]\\
&\leq \E[\E\exp(2C\lambda K^2\cdot \bZ^\top\Ab^o\bZ' \mid \bZ')]\\
&=\E\exp(2C\lambda K^2\cdot \bZ^\top\Ab^o\bZ')~~~\text{ for all }\lambda\in\reals,
\end{align*}
where $\bZ,\bZ'$ are independent standard multivariate Gaussian independent of $\bX',\bX$.

Lastly, to control the MGF of $\bZ^\top\Ab^o\bZ'$, one uses the property of multivariate Gaussian, whose distribution is rotation invariant, and singular value decomposing
\[
\Ab^o=\sum_{i=1}^n\lambda_i\bu_i\bv_i^\top
\]
 to deduce that, 
 \begin{align}\label{eq:hw2}
&\E\exp(\lambda\cdot\bZ^\top\Ab^o\bZ') = \E\exp\Big(\lambda\cdot\sum_{i=1}^n\lambda_iZ_iZ_i'\Big)\leq \prod_{i=1}^n\E\exp\Big(\frac{\lambda^2}{2}\cdot\lambda_i^2Z_i^2 \Big)\notag\\
\leq& \exp(\lambda^2\norm{\Ab^o}_{\rm F}^2)\leq \exp(\lambda^2\norm{\Ab}_{\rm F}^2),~~~\text{ for all }\lambda^2\leq 1/(2\norm{\Ab}_{\op}^2),
 \end{align}
 where $Z_1,\ldots,Z_n,Z_1',\ldots,Z_n'$ are i.i.d. standard Gaussian and in the last inequality we used the property of chi-square distributions that
 \[
 \log\E\exp(\lambda Z_1^2)=-\frac{1}{2}\log(1-2\lambda)\leq 2\lambda~~\text{ for all }\lambda<1/4.
 \]
 Combining \eqref{eq:hw1} and \eqref{eq:hw2} then completes the proof.
\end{example}

\begin{exercise}
Please specify a value of the constant $C$ (no need to be sharp) in the Hanson-Wright inequality.
\end{exercise}

\section{Notes}

{\bf Chapter \ref{sec:basic-analysis}.} The standard reference to basic analysis is \cite{rubin1987real} and \cite{rubin1991real}. We took most of materials in this chapter from Chapter 6.1 in \cite{kosorok2008introduction}, which provides a concise summary.\\

{\bf Chapter \ref{sec:stochastic-convergence}.} The standard reference to the topic of measure-theoretical probability theory is \cite{billingsley2008probability}, along with the books of Chung \citep{chung2001course} and Durret \citep{durret2019probability}. \citet[Chapter 2]{van2000asymptotic} gives a summary of stochastic convergence results, and the author learns the concept of weak convergence on metric spaces from Gallen Shorack \citep[Chapter 14]{shorack2017probability}.\\ 

{\bf Chapter \ref{sec:weak-convergence-sp}.} \cite{billingsley1999convergence} outlined the framework of weak convergence of stochastic processes. The present version of weak convergence for sample bounded stochastic processes was adapted from \citet[Section 3.7]{gine2021mathematical}. Part I of \cite{vaart1996empirical} and Chapter 7 of \cite{kosorok2008introduction} contain further results.\\

{\bf Chapter \ref{sec:chap2-inequ}.} Results in this chapter were scattered in many useful textbooks. The author learnt Theorem \ref{thm:erdos} from \cite{roch2024modern}; see, also, the book by Baik, Deift, and Suidan \citep{baik2016combinatorics}. More in-depth discussions on Orlicz norms and related properties can be found in, e.g., \citet[Chapter 11]{ledoux2013probability}, \citet[Chapter 2.2]{vaart1996empirical}, and \citet[Chapter 2]{vershynin2018high}.\\

{\bf Chapter \ref{sec:iid}.} The standard reference to the probability of sums of independent random variables is \cite{petrov1972independent}; \citet[Chapter 2.1]{pena2009self} gives a concise summary. Examples \ref{example:circle} and \ref{example:circle2}, on the number of circles in a permutation, are adapted from \cite{durret2019probability}. Theorem \ref{thm:huber} is due to \cite{huber1973robust}. The author learnt the Lindeberg swapping trick, used for proving Theorem \ref{thm:berry-esseen}, from Terrence Tao \citep[Chapter 2.2]{tao2023topics}. Lastly, a standard reference to moment and concentration inequalities under independence is \cite{boucheron2013concentration}. Hanson-Wright inequality is due to \cite{hanson1971bound} and the present form is developed in \cite{rudelson2013hanson}; see, also, \citet[Chapter 6]{vershynin2018high}.

%\begin{conjecture} Let $\pi$ be a uniform permutation, $\ba=(a_1,\ldots,a_N)^\top$ be a deterministic vector, and $\Cb\in\reals^{N\times N}$ be a deterministic matrix. Denote 
%\[
%\Qb=\sum_{i,j\in\cT}c_{ij}a_ia_j~~~{\rm with}~~~\cT:=\{i\in[N]: \pi(i)\leq n\}.
%\]
%We conjecture the existence of a universal constant $K>0$ such that, for any $t>0$,
%\[
%\Pr\Big\{\Big|\Qb-\E[\Qb]\Big|\geq t\Big\} \leq 2\exp\Big(-\frac{Kt^2}{\tau^2\norm{\Cb}_{\rm F}^2+\rho\norm{\Cb}_{\op}t} \Big),
%\]
%where
%\[
%\tau^2:=\frac{(N-n)}{nN(N-1)}\norm{\ba}_2^4~~~{\rm and}~~~\rho:=\norm{\ba}_{\infty}^2.
%\]
%\end{conjecture}

%% file: chapters/combin-prob.tex
\chapter{Combinatorial Probability}\label{chapter:combin-prob}

\section{Overview}\label{sec:3overview}

In permutation statistics, a classical object of interest is a deterministic $N$ by $N$ real matrix $\Ab=\Ab_N\in\reals^{N\times N}$, whose entries are written  as $\{a_{i,j}=a^N_{i,j}; i,j\in[N]\}$. The matrix $\Ab=[a_{i,j}]$ changes with $N$. For reasons to be explained later, $\Ab$ is often coupled with its normalized version,
\[
\Db=[d_{i,j}], \text{ with } d_{i,j}=a_{i,j}-\frac1N\sum_{k=1}^Na_{k,j}-\frac1N\sum_{\ell=1}^Na_{i,\ell}+\frac{1}{N^2}\sum_{k,\ell=1}^Na_{k,\ell}.
\]
It is straightforward to verify that
\[
\sum_{i=1}^N d_{i,j}=\sum_{j=1}^Nd_{i,j}=0~~\text{for any } i,j\in[N],
\]
which may help explain why we call $\Db$ a normalized version of $\Ab$. Furthermore, let's define three features of $\Ab$ that we will repeatedly use in the subsequent sections:
\[
\mu_A:=\frac{1}{N}\sum_{i,j=1}^Na_{i,j}, \quad \sigma_A^2:=\frac{1}{N-1}\sum_{i,j=1}^Nd_{i,j}^2, ~~~{\rm and}~~B_A:=\max\limits_{i,j\in[N]}\big|d_{i,j}\big|.
\]

In permutation statistics, an analogy to the sample sum, $\sum_{i=1}^nX_i$, in the independence sampling paradigm (i.e., Section \ref{sec:iid}) is the following combinatorial sum,
\[
Y:=\sum_{i=1}^N a_{i,\pi(i)},
\]
where $\pi=\pi_N$ is uniformly distributed over $\cS_N$.

\begin{example}[Spearman's rho] Consider the Spearman's rho statistic of the form $\sum_{i=1}^Ni\pi(i)$. It is then a combinatorial sum with the corresponding matrix $a_{i,j}=ij$. 
\end{example}

\begin{example}[Spearman's footrule]\label{eg:footrule} Consider the Spearman's footrule $\sum_{i=1}^N|\pi(i)-i|$. It is then a combinatorial sum with the corresponding matrix $a_{i,j}=|i-j|$.
\end{example}

\begin{example}[Survey sample mean] Consider the survey sample mean that is the average of a sample drawn uniformly without replacement from a finite population $\{z_i;i\in[N]\}$. It is then a combinatorial sum with the corresponding matrix $a_{i,j}=n^{-1}z_i\ind(j\leq n)$.
\end{example}

\section{Combinatorial law of large numbers}\label{sec:clln}

The mean and variance of a combinatorial sum, whose randomness only comes from the random permutation, was calculated in \citet[Theorem 2]{hoeffding1951combinatorial} and we summarize the results below.

\begin{proposition}[Mean and variance of $Y$] \label{prop:2basic1} We have
\[
\E [Y]=\mu_A~~{\rm and}~~\var(Y)=\sigma_A^2.
\]
\end{proposition}
\begin{proof} Using Proposition \ref{prop:basic}\ref{prop:basic1}, we have
\[
\E[Y]=\E\Big[\sum_{i=1}^N a_{i,\pi(i)}\Big]=\sum_{i=1}^N\E[a_{i,\pi(i)}]=\sum_{i,j=1}^Na_{i,j}\Pr[\pi(i)=j]=\frac{1}{N}\sum_{i,j=1}^Na_{i,j}.
\]

The variance calculation, on the other hand, is a little bit delicate. One important feature of the combinatorial sum that we will use here and also repeatedly in the future is the following lemma.

\begin{lemma}\label{lem:comb-sum-key}
It holds true that  $Y-\E Y=\sum_{i=1}^Nd_{i,\pi(i)}$.
\end{lemma}

Using Lemma \ref{lem:comb-sum-key} and Proposition \ref{prop:basic}\ref{prop:basic1} again, one can obtain
\begin{align*}
&\E[d_{i,\pi(i)}]=\frac{1}{N}\sum_{i,j=1}^Nd_{i,j}=0,
\var(d_{i,\pi(i)})=\E[d_{i,\pi(i)}^2]=\frac{1}{N}\sum_{j=1}^Nd_{i,j}^2,~~~{\rm and}\\
&\text{(whenever }i\ne j)~~\E[d_{i,\pi(i)}d_{j,\pi(j)}]=\frac{1}{N(N-1)}\sum_{k\ne \ell}d_{i,k}d_{j,\ell}=-\frac{1}{N(N-1)}\sum_{k=1}^Nd_{i,k}d_{j,k}.
\end{align*}
Thusly, we obtain
\begin{align*}
&\var(Y)=\E\Big[\sum_{i=1}^Nd_{i,\pi(i)}^2\Big]=\sum_{i=1}^N\E d_{i,\pi(i)}^2+\sum_{i\ne j}\E[d_{i,\pi(i)}d_{j,\pi(j)}]\\
&=\frac{1}{N}\sum_{i,j=1}^Nd_{i,j}^2-\frac{1}{N(N-1)}\sum_{k=1}^N\sum_{i\ne j}d_{i,k}d_{j,k}\\
&=\frac{1}{N}\sum_{i,j=1}^Nd_{i,j}^2+\frac{1}{N(N-1)}\sum_{i,k=1}^Nd_{i,k}^2=\frac{1}{N-1}\sum_{i,j=1}^Nd_{i,j}^2.
\end{align*}
The proof is thus complete.
\end{proof}

A direct consequence of Proposition \ref{prop:2basic1} is the following (weak) law of large numbers.

\begin{corollary}[Combinatorial LLN] It holds true that
\[
Y-\mu_A \stackrel{\Pr}{\to} 0
\]
if $\sigma_A\to 0$ as $n\to\infty$.
\end{corollary}

\section{Combinatorial CLT}\label{sec:cclt}

The following is the celebrated combinatorial central limit theorem that also gives a Berry-Esseen-type bound.

\begin{theorem}[Combinatorial CLT]\label{thm:CCLT-key} There exists a universal constant $K>0$ such that
\begin{align*}
\sup_{t\in\mathbb{R}}\Big|\Pr\Big(\frac{Y-\E[Y]}{\sqrt{\var(Y)}}\leq  t\Big) - \Phi(t)\Big| \leq 
K\cdot \frac{\sum\limits_{i,j\in[N]}\big|d_{i,j}\big|^3}{N\sigma_A^3}
\end{align*}
holds for all $N\geq 3$. Here we remind that $\Phi(\cdot)$ represents the CDF of the standard Gaussian.
\end{theorem}
\begin{proof}
Theorem 6.2 in \cite{chen2010normal}; see also Section \ref{sec:stein-cclt} ahead for a proof of a weaker version.
\end{proof}

A direct consequence of Theorem \ref{thm:CCLT-key} is the following corollary, which gives an easier-to-check condition for asymptotic normality of $Y$.

\begin{corollary}\label{cor:2-clt} There exists a universal constant $K>0$ such that
\begin{align*}
\sup_{t\in\mathbb{R}}\Big|\Pr\Big(\frac{Y-\E[Y]}{\sqrt{\var(Y)}}\leq  t\Big) - \Phi(t)\Big| \leq  K\cdot \frac{B_A}{\sigma_A}
\end{align*}
holds for all $N\geq 3$.
\end{corollary}

\begin{proof}
We have
\begin{align*}
\sum\limits_{i,j\in[N]}\big|d_{i,j}\big|^3 \Big/(N\sigma_A^3) \leq \frac{B_A\sum_{i,j\in[N]}d_{i,j}^2}{N\sigma_A^3}\leq \frac{B_A}{\sigma_A},
\end{align*}
where the last inequality is due to $\sum_{i,j\in[N]}d_{i,j}^2=(N-1)\sigma_A^2$.
\end{proof}

\begin{example}[Spearman's rho] Consider $Y_N=\sum_{i=1}^Ni\pi(i)$, with a corresponding matrix $a_{i,j}=ij$. One could then use  Proposition \ref{prop:2basic1} to deduce
\begin{align*}
\E[Y_N]=\frac{1}{N}\sum_{i,j=1}^Nij=\frac{N(N+1)^2}{2}~~\text{and }~~\var(Y_N)=\frac{N^5}{144}+O(N^4).
\end{align*}
Noting that $B_A=O(N^2)$,  we reach
\[
B_A\Big/ \sigma_A=O(N^2)/\sqrt{N^5/144} \to 0.
\]
Thusly, invoking Corollary \ref{cor:2-clt}, $Y_N$ is asymptotically normal.
\end{example}

\begin{example}[Spearman's footrule]\label{eg:footrule} Consider 
$D_N=\sum_{i=1}^N|\pi(i)-i|$. It is straightforward to verify
\[
D_N=\sum_{i=1}^Na_{i,\pi(i)}  \text{ with } a_{i,j}=|i-j|.
\]
Accordingly, by Proposition \ref{prop:2basic1}, 
\[
\E[D_N]=\frac{1}{N}\sum_{i,j=1}^N|i-j|=\frac{N^2-1}{3}.
\]
A similar but much lengthy derivation finds
\[
\var(D_N)=\frac{2}{45}N^3+O(N^2).
\]
Noting that $B_A=O(N)$, we reach
\[
B_A/ \sigma_A=O(N)/\sqrt{2N^3/45} \to 0.
\]
\end{example}

\subsection{Stein's method of exchangeable pairs}\label{sec:stein-cclt}

This section gives the proof of a weaker version of the combinatorial CLT. Compared to the proof of \citet[Theorem 6.2]{chen2010normal}, the following one is easier to parse and thus fits more to the philosophy of this book. For {\it all} subsequence results this book is going to cover, this weaker result also suffices.

Let's first introduce Stein's identity that is due to  \cite{stein1972bound}.

\begin{lemma}[Stein's identity, Lemma 2.1 in \cite{chen2010normal}]\label{lem:stein-identity} If a random variable $Z$ is standard Gaussian, then for all absolutely continuous functions $f:\reals\to\reals$ with finite $\E|f'(Z)|$, we have
\[
\E f'(Z)=\E[Zf(Z)].
\]
Conversely,  if the above identity holds for all bounded, continuous and piecewise continuously differentiable functions $f$ with finite $\E|f'(Z)|$, then $Z$ is standard Gaussian distributed.
\end{lemma}

The idea of characterizing the Gaussian by checking the averaged difference between $f'(Z)$ and $Zf(Z)$ is ingenious, and also proves to be extremely useful in bounding the distance between probability measures. It turns out to be directly related to a information-theoretical metric, the Wasserstein distance.

In detail, for any two $\reals$-valued random variables $W$ and $Z$, the Wasserstein-1 distance between $W$ and $Z$ is defined as
\[
d_W(W,Z):=\sup_{g\in\cL(1)}\Big|\E[g(W)]-\E[g(Z)]\Big|,
\]
where the supremum is over $\cL(1)$, the set of all 1-Lipschitz functions $g$, that is, functions satisfying $|g(w)-g(z)|\leq |w-z|$ for all $w,z\in\reals$. 

\begin{exercise}\label{exe:wasserstein}
For any two $\reals$-valued random variables $W,Z$ such that $Z$ has a bounded Lebesgue density, please show that
\[
\sup_{t\in\reals}\Big|\Pr(W\leq t)-\Pr(Z\leq t)\Big| \leq K\big\{d_W(W,Z)\big\}^{1/2},
\]
where the constant $K>0$ only depends on the density of $Z$. Please also specify an explicit $K$.
\end{exercise}

By Stein's identity, the following lemma then holds; its proof is a little bit technical and not quite related to the mainstream of this book, so we relegate it to the end of this chapter. 

\begin{lemma}\label{lem:stein-1} For $Z\sim N(0,1)$ and any random variable $W$, we have
\[
d_W(W,Z)\leq \sup_{f\in\cL'(1)}\Big|\E\Big[f'(W)-Wf(W)\Big]\Big|,
\]
where 
\[
\cL'(1):=\Big\{f:\reals\to\reals; \norm{f}_\infty\leq 1, \norm{f'}_{\infty}\leq \sqrt{\frac{2}{\pi}}, \norm{f''}_{\infty}\leq 2\Big\}.
\]
\end{lemma}

With Lemma \ref{lem:stein-1}, the problem of quantifying the convergence rate of any sequence of $\reals$-valued random variables $X_n$ to its limit $Z$, if Gaussian, reduces to bounding $|\E[f'(W)-Wf(W)]|$ over $\cL'(1)$.

\begin{definition}[Exchangeable pair]\label{def:exchangeable}
For any random variable $W$, it is said that $(W,W')$ forms an exchangeable pair if $(W,W')$ has the same distribution as $(W',W)$.
\end{definition}

The following theorem connects the Wasserstein distance between any random variable $W$ and $Z\sim N(0,1)$ to, quite surprisingly, the structure of $(W,W')$. This is the celebrated {\it Stein's method of exchangeable pairs}. 

\begin{lemma}\label{lem:stein-2}
Suppose that $(W,W')$ forms an exchangeable pair such that
\begin{align}\label{eq:stein-key2}
\E(W'-W \mid W)=-\lambda W,~~~\text{ for some }\lambda \in (0,1),
\end{align}
which implies that $\E[W]=0$. Further assume $\E [W^2]=1$. It is then true that
\[
d_W(W,Z)\leq \Big( \frac{2}{\pi}\var\Big[\E\Big\{\frac{1}{2\lambda}(W'-W)^2 \mid W \Big\}\Big]\Big)^{1/2}+\frac{1}{3\lambda}\E\Big|W'-W\Big|^3.
\]
\end{lemma}
\begin{proof}
Using Lemma \ref{lem:stein-1}, it suffices to uniformly bound
\[
\Big|\E\Big[f'(W)-Wf(W)\Big]\Big|
\]
over those functions $f$ such that $\norm{f}_\infty\leq 1, \norm{f'}_{\infty}\leq \sqrt{2/\pi}, \norm{f''}_{\infty}\leq 2$. Since 
\[
\E(W'-W \mid W)=-\lambda W,
\]
we have 
\[
\E\Big[(W'-W)f(W) \Big]=\E\Big[\E[W'-W\mid W]f(W) \Big]=-\lambda\E[Wf(W)].
\]

Next, introduce $F(x)$ whose derivative is $f(x)$. By Taylor expanding $F(\cdot)$ at $W$, we have
\[
\E\Big[F(W')-F(W)\Big]=\E\Big[(W'-W)f(W)\Big]+\frac{1}{2}\E\Big[(W'-W)^2f'(W)\Big]+R,
\]
where, using the property that $\norm{f''}_{\infty}<2$, we have $|R| \leq \frac13\E|W'-W|^3$. Accordingly,
\begin{align*}
\E\Big[(W'-W)f(W)\Big]&=\E\Big[F(W')-F(W)\Big]-\frac{1}{2}\E\Big[(W'-W)^2f'(W)\Big]-R\\
&=-\frac{1}{2}\E\Big[(W'-W)^2f'(W)\Big]-R,
\end{align*}
where in the second equality we used that $(W,W')$ is an exchangeable pair so that 
\[
\E\Big[F(W')-F(W)\Big]=0.
\]

Then, we have 
\[
\E[Wf(W)]=\frac{1}{2\lambda}\E\Big[(W'-W)^2f'(W)\Big]+\frac{R}{\lambda},
\]
and thus
\begin{align*}
&\Big|\E\Big[f'(W)-Wf(W)\Big]\Big|\\
=&\Big|\E\Big[\Big\{\frac{1}{2\lambda}(W'-W)^2-1\Big\}f'(W)\Big]+\frac{R}{\lambda}  \Big|\\
\leq& \Big|\E\Big[\Big\{\frac{1}{2\lambda}(W'-W)^2-1\Big\}f'(W)\Big]\Big| + \frac{1}{3\lambda}\E|W'-W|^3\\
= &\Big|\E\Big[(U-1)f'(W)\Big]\Big| + \frac{1}{3\lambda}\E|W'-W|^3\\
\leq& \sqrt{\frac{2}{\pi}}\E|U-1|+ \frac{1}{3\lambda}\E|W'-W|^3,
\end{align*}
where we define 
\[
U=\E\Big\{\frac{1}{2\lambda}(W'-W)^2\mid W\Big\}.
\]
It remains to control 
\begin{align*}
\E|U-1| = \norm{U-1}_{L^1}\leq \norm{U-1}_{L^2}=\sqrt{\var(U)},
\end{align*}
where in the last equality we used the fact that $\E W^2=1$ so that
\[
\E [U] = \frac{1}{2\lambda}\E(W'-W)^2=\frac{2-2\E[WW']}{2\lambda}=\frac{-2\E[(W'-W)W]}{2\lambda}=1.
\]

This shows that, for any $f\in\cL'(1)$, 
\[
\Big|\E\Big[f'(W)-Wf(W)\Big]\Big| \leq  \Big(\frac{2}{\pi} \var\Big[\E\Big\{\frac{1}{2\lambda}(W'-W)^2 \mid W \Big\}\Big]\Big)^{1/2}+\frac{1}{3\lambda}\E\Big|W'-W\Big|^3.
\]
Combining the above inequality with Lemma \ref{lem:stein-1} then completes the proof.
\end{proof}

\subsection{Proof of a weaker version of  the combinatorial CLT}

Using Lemma \ref{lem:stein-2} yields the following weaker version of Theorem \ref{thm:CCLT-key}.

\begin{theorem}[Combinatorial CLT, weaker version]\label{thm:weak-cclt} There exists a universal constant $K>0$ such that, for any $N\geq 4$,
\begin{align*}
d_W\Big(\frac{Y-\E[Y]}{\sqrt{\var(Y)}}, Z\Big) \leq 
K\cdot \left\{\frac{\sum\limits_{i,j=1}^N\big|d_{i,j}\big|^3}{N\sigma_A^3}+\sqrt{\frac{\sum\limits_{i,j=1}^Nd_{i,j}^4}{N\sigma_A^4}}\right\},
\end{align*}
where $Z\sim N(0,1)$. 
\end{theorem}

\begin{corollary} There exists a universal constant $K>0$ such that, for any $N\geq 4$, 
\begin{align*}
\sup_{t\in\mathbb{R}}\Big|\Pr\Big(\frac{Y-\E[Y]}{\sqrt{\var(Y)}}\leq  t\Big) - \Phi(t)\Big| \leq  K\cdot \sqrt{\frac{B_A}{\sigma_A}}.
\end{align*}
\end{corollary}
\begin{proof}
Use Exercise \ref{exe:wasserstein}.
\end{proof}

\begin{proof}[Proof of Theorem \ref{thm:weak-cclt}]The proof of Theorem \ref{thm:weak-cclt} involves constructing an specific exchangeable pair for the combinatorial sum. To this end, let's introduce a couple of $\pi$ to be
\[
\pi'  := \begin{cases}
\pi(i), \quad \text{if }i\ne I,J,\\
\pi(J), \quad \text{if } i=I,\\
\pi(I), \quad \text{if } i=J,
\end{cases}
\]
with $I,J$ uniformly and independently sampled from $[N]$.

\begin{exercise}\label{exe:exchangeable}
Please show that $(\pi,\pi')$ forms an exchangeable pair.
\end{exercise}

Using Exercise \ref{exe:exchangeable} and Lemma \ref{lem:comb-sum-key}, it is natural to construct
\[
W=Y-\E[Y]=\sum_{i=1}^Nd_{i,\pi(i)}~~~{\rm and}~~~W'=\sum_{i=1}^Nd_{i,\pi'(i)}
\]
and it is immediate that $(W,W')$ also forms an exchangeable pair. In addition, by proper standardization, without loss of generality, we can assume 
\[
\sigma_A^2=1.
\]
{\bf Step 1.} Let's first verify Condition \eqref{eq:stein-key2} in Lemma \ref{lem:stein-2}:
\begin{align*}
\E[W'-W\mid \pi]&= \E\Big[d_{I,\pi(J)}+d_{J,\pi(I)}-d_{I,\pi(I)}-d_{J,\pi(J)}\mid \pi\Big]\\
&=\frac{2}{N^2}\sum_{i,j=1}^Nd_{i,\pi(j)}-\frac{2}{N}\sum_{i=1}^Nd_{i,\pi(i)}\\
&=-\frac{2}{N}W.
\end{align*}
Accordingly, $\E[W'-W\mid W]=-\lambda W$ with 
\[
\lambda=\frac{2}{N}\in (0,1) \text{ whenever } N\geq 3.
\]
{\bf Step 2.} Next, let's bound
\begin{align*}
\frac{1}{3\lambda}\E|W'-W|^3 &=\frac{N}{6}\E\Big|d_{I,\pi(J)}+d_{J,\pi(I)}-d_{I,\pi(I)}-d_{J,\pi(J)} \Big|^3\\
&\leq \frac{32N}{6}\Big(\E|d_{I,\pi(J)}|^3+\E|d_{I,\pi(I)}|^3\Big)\\
&= \frac{32N}{6}\Big( \frac{1}{N^2}\sum_{i,j=1}^N|d_{i,j}|^3 + \frac{1}{N(N-1)}\sum_{i,j=1}^N|d_{i,j}|^3\Big)\\
&\leq \frac{16}{N}\sum_{i,j=1}^N|d_{i,j}|^3.
\end{align*}
{\bf Step 3.} Lastly, we calculate
\begin{align*}
\E[(W'-W)^2\mid \pi]&= \E\Big[\Big(d_{I,\pi(J)}+d_{J,\pi(I)}-d_{I,\pi(I)}-d_{J,\pi(J)}\Big)^2\mid \pi\Big]\\
&=\frac{1}{N^2}\sum_{i,j=1}^N\Big(d_{i,\pi(j)}+d_{j,\pi(i)}-d_{i,\pi(i)}-d_{j,\pi(j)}  \Big)^2\\
&=: \frac{1}{N^2}\sum_{i,j=1}^N X_{ij}^2.
\end{align*}
It is then immediate that 
\[
\var\Big(\frac{1}{N^2}\sum_{i,j=1}^NX_{ij}^2\Big)=\frac{1}{N^4}\sum_{i,j,k,\ell}\Big[\E[X_{ij}^2X_{k\ell}^2]-\E[X_{ij}^2]\E[X_{k\ell}^2]\Big].
\]
It remains to bound 
\[
\sum_{i,j,k,\ell}\Big[\E[X_{ij}^2X_{k\ell}^2]-\E[X_{ij}^2]\E[X_{k\ell}^2]\Big].
\]
Let's first study the case when $i \ne j \ne k \ne \ell$. For them, we have 
\[
\E[X_{ij}^2X_{k\ell}^2]=\frac{1}{N(N-1)(N-2)(N-3)}\sum_{i_1\ne i_2\ne i_3 \ne i_4}A_{iji_1i_2}A_{k\ell i_3i_4}
\]
and
\[
\E[X_{ij}^2]\E[X_{k\ell}^2]=\frac{1}{N^2(N-1)^2}\sum_{i_1\ne i_2}A_{iji_1i_2}\sum_{i_3\ne i_4}A_{jki_3i_4},
\]
where
\[
A_{ijk\ell}:=\Big(d_{i,k}+d_{j,\ell}-d_{i,\ell}-d_{j,k}  \Big)^2.
\]
Accordingly,
\begin{align*}
&\E[X_{ij}^2X_{k\ell}^2]-\E[X_{ij}^2]\E[X_{k\ell}^2]\leq \frac{C}{N^5}\sum_{i_1\ne i_2 \ne i_3 \ne i_4}A_{iji_1i_2}A_{k\ell i_3i_4}\\
\leq& \frac{4C}{N^5}\sum_{i_1\ne i_2 \ne i_3 \ne i_4}\Big[d_{i,i_1}^4+d_{j,i_2}^4+d_{i,i_2}^4+d_{j,i_1}^4+d_{k,i_3}^4+d_{\ell,i_4}^4+d_{k,i_4}^4+d_{\ell,i_3}^4 \Big],
\end{align*}
so that
\begin{align}\label{eq:cclt-stein-1}
\sum_{i\ne j \ne k \ne \ell}\Big\{\E[X_{ij}^2X_{k\ell}^2]-\E[X_{ij}^2]\E[X_{k\ell}^2]\Big\} \leq C'N\sum_{i,j=1}^Nd_{i,j}^4.
\end{align}

It remains to consider the case when $i=k\ne j \ne \ell$. For this, we have
\begin{align}\label{eq:cclt-stein-2}
&\sum_{i\ne j\ne \ell}\Big\{\E[X_{ij}^2X_{i\ell}^2]-\E[X_{ij}^2]\E[X_{i\ell}^2]\Big\} \leq \sum_{i\ne j\ne \ell}\E[X_{ij}^2X_{i\ell}^2]\leq \frac{1}{2}\sum_{i\ne j\ne \ell}[\E X_{ij}^4+\E X_{i\ell}^4]\notag\\
= & N\sum_{i\ne j}\E X_{ij}^4=\frac{1}{N-1}\sum_{i\ne j}\sum_{i_1\ne i_2}A_{iji_1i_2}^2\leq \frac{8}{N-1}\sum_{i\ne j}\sum_{i_1\ne i_2}(d_{i,i_1}^4+d_{j,i_2}^4+d_{i,i_2}^4+d_{j,i_1}^4)\notag\\
\leq& 16N \sum_{i,j=1}^N d_{i,j}^4.
\end{align}
Combining \eqref{eq:cclt-stein-1} and \eqref{eq:cclt-stein-2}, we obtain
\[
\var\Big(\frac{1}{N^2}\sum_{i,j=1}^NX_{ij}^2\Big)\leq \frac{C''}{N^3}\sum_{i,j=1}^Nd_{i,j}^4
\]
so that 
\[
\frac{1}{4\lambda^2}\var\Big(\frac{1}{N^2}\sum_{i,j=1}^NX_{ij}^2\Big) \leq \frac{16C''}{N}\sum_{i,j=1}^Nd_{i,j}^4.
\]
In other words, we showed
\begin{align}\label{eq:cclt-stein-3}
\var\Big[\E\Big\{\frac{1}{2\lambda}(W'-W)^2 \mid \pi \Big\}\Big]\leq \frac{16C''}{N}\sum_{i,j=1}^Nd_{i,j}^4.
\end{align}

\noindent {\bf Step 4.} Lastly, we use the following lemma.
\begin{lemma}\label{lem:basic-ce}
For any random variables $X,Y,Z$ such that $X$ is a measurable function of $Z$, it holds true that
\[
\var(\E[Y\mid X])\leq \var(\E[Y \mid Z]).
\]
\end{lemma}
Combining the above lemma (by picking $X$ to be $W$ and $Z$  to be $\pi$) with \eqref{eq:cclt-stein-3}, and plugging everything to Lemma \ref{lem:stein-2} then complete the proof.
\end{proof}

\begin{exercise}
Prove Lemma \ref{lem:basic-ce}.
\end{exercise}

\section{A variant of the combinatorial CLT}\label{sec:cclt-variant}

To analyze Chatterjee's rank correlation in Example \ref{eg:disarray}, the classical combinatorial CLT is not helpful and we need a variant. To this end, let's define the oscillation sum as
\[
W=\sum_{i=1}^{N}a_{\pi(i),\pi(i+1)}~~\text{with the convention that }\pi(N+1)=\pi(1).
\]

Some calculations give the mean and variance of $W$. 
\begin{proposition}\label{prop:oscillation}
We have
\begin{align*}
\E[W]=&\frac{1}{N-1}\sum_{i\ne j\in[N]}a_{ij}\\
~~{\rm and}~~\var(W)=&\frac{1}{N-2}\sum_{i,j\in[N]}d_{i,j}^2-\frac{1}{(N-1)(N-2)}\sum_{i,j\in[N]}d_{i,j}d_{j,i}+\\
&\frac{1}{(N-1)^2(N-2)}\Big(\sum_{i\in[N]}d_{i,i}\Big)^2-\frac{N}{(N-1)(N-2)}\sum_{i\in[N]}d_{i,i}^2 .
\end{align*}
\end{proposition}

\begin{exercise}
Please prove Proposition \ref{prop:oscillation}.
\end{exercise}

\begin{theorem}[Oscillation combinatorial CLT] \label{thm:OCCLT} Assume
\[
\sum_{i=1}^nd_{ii}^2=O(\sigma_A^2).
\]
Then there exists a universal constant $K>0$ such that
\[
\sup_{t\in\mathbb{R}}\Big|\Pr\Big(\frac{W-\E[W]}{\sqrt{\var(W)}}\leq t\Big)-\Phi(t)\Big| \leq \frac{K}{\sqrt{N}}\Big(\frac{\sqrt{\sum\limits_{i,j\in[N]}d_{i,j}^4}}{\sigma_A^2}+\frac{\sqrt{\sum\limits_{i,j\in[N]}|d_{i,j}|^3}}{\sigma_A^{3/2}}\Big).
\]
\end{theorem}
\begin{proof}
Proof of Theorem 1 in \cite{chao1996estimating}.
\end{proof}

\begin{corollary}
The random variable $W$ is asymptotically normal if the ratio $B_A/\sigma_A\to 0$.
\end{corollary}
\begin{proof}
We have
\begin{align*}
\frac{\sqrt{\frac1N\sum\limits_{i,j\in[N]}d_{i,j}^4}}{\sigma_A^2}+\frac{\sqrt{\frac1N\sum\limits_{i,j\in[N]}|d_{i,j}|^3}}{\sigma_A^{3/2}}\leq \frac{B_A\sqrt{\sigma_A^2}}{\sigma_A^2}+\frac{B_A^{1/2}\sqrt{\sigma_A^2}}{\sigma_A^{3/2}}=\frac{B_A}{\sigma_A}+\sqrt{\frac{B_A}{\sigma_A}},
\end{align*}
which goes to 0 if $B_A/\sigma_A\to0$.
\end{proof}

\begin{example}[Chatterjee's rank correlation] Let $\xi_N=\sum_{i=1}^{N-1}|\pi(i+1)-\pi(i)|$, with a corresponding matrix $a_{i,j}=|i-j|$. Using a tedious revision of Proposition \ref{prop:oscillation}, we could then derive
\[
\E[\xi_N]=\frac{N^2-1}{3}
\]
and
\[
\var(\xi_N)=\frac{2}{45}N^3+O(N^2).
\]
Invoking exactly the same argument of Example \ref{eg:footrule}, we also have $\xi_n$ is asymptotically normal. 
\end{example}

\section{Combinatorial moderate deviations}\label{sec:cmd}

\begin{theorem}[Combinatorial moderate deviations]\label{thm:2CMD} There exists a universal constant $M>0$ such that
\[
\Pr\Big(\frac{Y-\E[Y]}{\sqrt{\var(Y)}}\geq  t\Big)\Big/ (1-\Phi(t))=1+M(1+t^3)B_A/\sigma_A
\]
holds for all $t\in [0, (\sigma_A/B_A)^{1/3}]$.
\end{theorem}

The proof of this result is based on a zero-biased coupling technique.

\begin{lemma}[Chen-Fang-Shao]\label{lem:ChenFangShao} Assume a zero-biased couple $(X,X^*)$, which satisfies 
\begin{align*}
\E[X]=&0,~\var(X)=1, \text{and } \E[Xf(X)]=\E [f'(X^*)], \text{ for any bounded and}\\ &\text{absolutely continuous function }f \text{ with a bounded derivative }f'.
\end{align*}
Further assume the existence of a constant $\delta$ such that
\[
\Pr(|X-X^*|\leq\delta)=1.
\]
We then have
\[
\frac{\Pr(X\geq t)}{1-\Phi(t)}=1+O(1)(1+t^3)\delta
\]
holds for all $t\in[0,\delta^{-1/3}]$.
\end{lemma}
\begin{proof}
Theorem 3.1 and Corollary 3.1 of \cite{chen2013stein}.
\end{proof}

Get back to the proof of Theorem \ref{thm:2CMD}. Let the random variable $X$ in Lemma \ref{lem:ChenFangShao} take the value
\[
\tilde Y=\sum_{i=1}^Nd_{i,\pi(i)}/\sigma_A.
\]

\begin{lemma}[Goldstein] There exists a zero-biased couple of $\tilde Y$, denoted by $\tilde Y^*$, such that
\[
\Pr(|\tilde Y-\tilde Y^*|\leq 8B_A/\sigma_A)=1.
\]
\end{lemma}
\begin{proof}
Theorem 2.1 in \cite{goldstein2005berry}.
\end{proof}

\section{Combinatorial moment inequalities}\label{sec:3cmi}

\subsection{Chatterjee's method}\label{sec:chatterjee-method}

This section aims to establish concentration inequalities for the combinatorial sums. It turns out that such inequalities can be derived through a novel use of Stein's exchangeable method. The idea first appeared in the Ph.D. thesis of Sourav Chatterjee, bearing the title ``Concentration inequalities with exchangeable pairs'' \citep{chatterjee2005concentration} and later published in the Annals of Probability \citep{chatterjee2006stein}.

To introduce Chatterjee's ingenious approach, let's consider a general setting with $X\in\mathcal{X}$ to be a generic random variable, and $f:\mathcal{X}\to\mathbb{R}$ to be the object of interest such that, without loss of generality, 
\[
\E[f(X)]=0. 
\]
Introduce $X'$ so that $(X,X')$ form an exchangeable pair; see Definition \ref{def:exchangeable}. Now, seek a couple of $f$, denoted by
\[
F:\mathcal{X}^2\to \mathbb{R},
\]
such that
\[
F(X,X')=-F(X',X)~~~\text{ and }~~~\E[F(X,X') \mid X]=f(X).
\]
Define 
\[
v(x):=\frac12\E\Big[\Big|(f(X)-f(X'))F(X,X')\Big| \mid X=x\Big].
\]
Lastly, assume
\begin{align}\label{eq:assump-cbi}
\E[f^2(X)]<\infty, ~\E[F^2(X,X')]<\infty, ~{\rm and}~\E e^{tf(X)}<\infty \text{ for all }t\in\reals.
\end{align}

\begin{lemma}[Master lemma]\label{lem:chatterjee-master} Assume \eqref{eq:assump-cbi} and define $M_\lambda:=\E[\exp(\lambda f(X))]$. We then have
\[
\frac{\d}{\d \lambda}M_\lambda \leq |\lambda|\E\Big[e^{\lambda f(X)}v(X)\Big],~~\text{ for all }\lambda\in\mathbb{R}.
\]
\end{lemma}
\begin{proof}
 We have, due to \eqref{eq:assump-cbi}, 
\[
\frac{\d}{\d \lambda}M_\lambda=\E\Big[\frac{\d}{\d \lambda}e^{\lambda f(X)}\Big]=\E\Big[f(X)e^{\lambda f(X)}\Big].
\]
Notice that, for any square-integrable $h:\mathcal{X}\to\mathbb{R}$,
\begin{align}\label{eq:2chatterjee-equ-1}
\E[h(X)f(X)]&=\E[h(X)\cdot\E[F(X,X')\mid X]]=\E[h(X)F(X,X')]\notag\\
&=\frac12\E[(h(X)-h(X'))F(X,X')].
\end{align}
We can continue to write
\begin{align}
\frac{\d}{\d \lambda}M_\lambda&=\frac12 \E\Big[\Big(e^{\lambda f(X)}-e^{\lambda f(X')} \Big)F(X,X')  \Big] \notag \\
&\leq \frac{|\lambda|}{4}\E\Big[\Big(e^{\lambda f(X)}+e^{\lambda f(X')} \Big)\Big|(f(X)-f(X'))F(X,X') \Big|  \Big] \notag \\
&\leq |\lambda|\E\Big[e^{\lambda f(X)}v(X)\Big], \notag
\end{align}
where in the first inequality we use the fact that
\[
\Big|\frac{e^x-e^y}{x-y} \Big| \leq \frac12(e^x+e^y)~~~\text{for any }x,y\in\mathbb{R}.
\]
This completes the proof.
\end{proof}

The following is Chatterjee's first lemma, which gives a Hoeffding type inequality for those $f(X)$ with $v(x)$ bounded almost surely.

\begin{lemma}[Chatterjee's first lemma]\label{lem:Cha-1stLemma} Assume \eqref{eq:assump-cbi} and suppose that there exists a finite positive constant $M$ such  that $\P(|v(X)|\leq M)=1$. It then holds true that
\[
\E\Big[\exp(\lambda f(X))\Big] \leq \exp(M\lambda^2/2),~~~\text{for all }\lambda\in\mathbb{R}
\]
and thus
\[
\Pr(f(X)\geq t)\leq \exp(-t^2/2M),~~~\text{for all }t\geq 0.
\]
\end{lemma}
\begin{proof}
Using Lemma \ref{lem:chatterjee-master}, we can continue to write
\begin{align*}
\frac{\d}{\d \lambda}M_\lambda&\leq |\lambda|\E\Big[e^{\lambda f(X)}v(X)\Big] \leq M|\lambda|M_\lambda.
\end{align*}
The above equation, combined with the initial condition that $M_0=1$, yields the conclusion. 
\end{proof}

\begin{lemma}[Chatterjee's second lemma]\label{lem:2chattejee-2} Assume \eqref{eq:assump-cbi}. Suppose further that there exist finite positive constants $A,B$ such that
\[
\Pr\Big\{v(X)\leq Af(X)+B\Big\}=1.
\]
It then holds true that
\[
\E\Big[\exp(\lambda f(X))\Big] \leq \exp\Big[\frac{B\lambda^2}{2(1-A\lambda)} \Big]~~\text{for all }\lambda\in [0, 1/A),
\]
and for all $t\geq 0$,
\[
\Pr(f(X)\geq t)\leq \exp\Big(-\frac{t^2}{2B+2At}\Big).
\]
\end{lemma}
\begin{proof}
Using Lemma \ref{lem:chatterjee-master}, we obtain
\begin{align*}
\frac{\d}{\d \lambda}M_\lambda \leq |\lambda|\E\Big[e^{\lambda f(X)}v(X)\Big]\leq |\lambda|\E\Big[e^{\lambda f(X)}(Af(X)+B)  \Big] = A|\lambda|\frac{\d}{\d \lambda}M_\lambda  + B|\lambda|M_\lambda.
\end{align*}
which yields
\[
\frac{\d}{\d \lambda}\log M_\lambda \leq \frac{B\lambda}{1-A\lambda}~~\text{for all }\lambda\in [0, 1/A).
\]
Combining the above with the initial condition that $\log M_0=0$ yields
\[
\log M_\lambda \leq \int_0^\lambda \frac{Bs}{1-As}\d s \leq \int_0^\lambda \frac{Bs}{1-A\lambda}\d s = \frac{B\lambda^2}{2(1-A\lambda)}. 
\]
Lastly, invoking Exercise \ref{exe:bernstein} gives the tail probability bound.
\end{proof}

\begin{lemma}[Chatterjee's third lemma]\label{lem:2chattejee-3} Assume \eqref{eq:assump-cbi} and introduce the function
\[
r(\psi)=\frac{1}{\psi}\log\E e^{\psi v(X)},~~\text{ for any }\psi>0.
\]
It then holds true that
\[
\log\E e^{\lambda f(X)}\leq \frac{\lambda^2r(\psi)}{2(1-\lambda^2/\psi)}, \text{ for all }\psi>0 \text{ and }0\leq \lambda<\sqrt{\psi},
\]
and thus, for any $t\geq 0$ and $\psi>0$,
\[
\Pr\Big\{f(X)\geq t\Big\}\leq \exp\Big\{-\frac{t^2}{2r(\psi)+2t/\sqrt{\psi}}  \Big\}.
\]
\end{lemma}
\begin{proof}
Using Lemma \ref{lem:chatterjee-master}, we obtain
\begin{align*}
\frac{\d}{\d \lambda}M_\lambda &\leq |\lambda|\E\Big[e^{\lambda f(X)}v(X)\Big]\leq \frac{\lambda M_\lambda}{\psi}\E[\psi v(X)\cdot W_\lambda],
\end{align*}
where we introduce 
\[
W_\lambda:=\frac{1}{M_\lambda}\cdot e^{\lambda f(X)},
\]
whose expectation is one. Jensen's inequality then yields, for any $\lambda\geq 0$,
\begin{align*}
\frac{\d}{\d \lambda}M_\lambda &\leq \frac{\lambda M_\lambda}{\psi}\E\Big[W_\lambda\cdot \log e^{\psi v(X)} \Big]\\
&\leq \frac{\lambda M_\lambda}{\psi}\log \E[e^{\psi v(X)}]+\frac{\lambda M_\lambda}{\psi}\E[W_\lambda\log W_\lambda], 
\end{align*}
where in the second inequality we invoke Exercise \ref{lem:young-entropy}. Since
\[
\frac{\d}{\d \lambda}M_\lambda\Big|_{\lambda=0}=\E[f(X)]=0,~~~M_0=1,~~~{\rm and}~~M_{\lambda} \text{ is convex},
\]
we derive $M_\lambda\geq 1$ for all $\lambda\in\reals$. Consequently, $\log W_\lambda\leq \lambda f(X)$ and thus
\[
\frac{\d}{\d \lambda}M_\lambda \leq \lambda M_\lambda r(\psi) + \frac{\lambda^2}{\psi}\frac{\d}{\d \lambda}M_\lambda,
\]
yielding
\[
\frac{\d}{\d \lambda}M_\lambda \leq \frac{r(\psi)\lambda}{1-\lambda^2/\psi}M_\lambda,~~~\text{ for all }0\leq \lambda<\sqrt{\psi}.
\]
This implies
\[
\frac{\d}{\d \lambda}\log M_\lambda \leq \frac{r(\psi)\lambda}{1-\lambda^2/\psi}
\]
so that, combining with the fact that $\log M_0=0$,
\[
\log M_\lambda \leq \int_0^\lambda \frac{r(\psi)s}{1-s^2/\psi}\d s \leq \int_0^\lambda \frac{r(\psi)s}{1-\lambda^2/\psi}\d s = \frac{r(\psi)\lambda^2}{2(1-\lambda^2/\psi)},
\]
which completes the proof.
\end{proof}

\begin{lemma}[Chatterjee's fourth lemma]\label{lem:2chattejee-4} For any positive integer $k$,
\[
\E[(f(X))^{2k}]\leq (2k-1)^k\E[v(X)^k].
\]
\end{lemma}
\begin{proof}
Invoking \eqref{eq:2chatterjee-equ-1} and choosing $h(x)=x^{2k-1}$, we obtain
\begin{align*}
\E[f(X)^{2k}]&=\frac12\E[(f(X)^{2k-1}-f(X')^{2k-1})F(X,X')]\\
&\leq (2k-1)\E[f(X)^{2k-2}v(X)]\\
&\leq (2k-1)\{\E[f(X)^{2k}]\}^{(k-1)/k}\{\E[v(X)^k]\}^{1/k},
\end{align*}
where the first inequality is due to
\[
\Big|x^{2k-1}-y^{2k-1}\Big|\leq \frac{2k-1}{2}(x^{2k-2}+y^{2k-2})|x-y| ~~\text{ for any }x,y\in\mathbb{R}.
\]
Rearranging the bound yields the conclusion.
\end{proof}

\subsection{Combinatorial moment inequalities}\label{sec:cmi-chatterjee}

Next, let's apply Chatterjee's method to studying the combinatorial sums. To this end, we construct an exchangeable pair of $Y$. Like Theorem \ref{thm:weak-cclt}, we introduce a couple of $\pi$ to be
\[
\pi' :=\pi \circ (I,J) = \begin{cases}
\pi(i), \quad \text{if }i\ne I,J,\\
\pi(J), \quad \text{if } i=I,\\
\pi(I), \quad \text{if } i=J,
\end{cases}
\]
with $I,J$ uniformly and independently sampled from $[N]$. It is easy to check that $(\pi,\pi')$ is an exchangeable pair. Recall Lemma \ref{lem:comb-sum-key} that
\[
Y-\E[Y]=\sum_{i=1}^Nd_{i,\pi(i)}.
\]
Now take
\[
f(\pi)=\sum_{i=1}^N d_{i,\pi(i)} ~~~{\rm and}~~~F(\pi_1,\pi_2)=\frac{N}{2} \Big(\sum_{i=1}^N d_{i,\pi_1(i)}-\sum_{i=1}^N d_{i, \pi_2(i)}\Big).
\]
The following lemma shows that the couple $(f,F)$ satisfies the conditions of Chatterjee's method.

\begin{lemma}\label{lem:2chatterjee} We have, 
\[
F(\pi,\pi')=-F(\pi',\pi)~~~{\rm and}~~~\E[F(\pi,\pi')\mid \pi] = f(\pi).
\]
In addition,
\begin{align}\label{eq:2chatterjee-1}
v(\pi) =& \frac12\E\Big[\Big|(f(\pi)-f(\pi'))F(\pi,\pi')\Big| \mid \pi\Big]\notag\\
=& \frac{1}{4N}\sum_{i,j\in[N]}\Big(d_{i,\pi(i)}+d_{j,\pi(j)}-d_{i,\pi(j)}-d_{j,\pi(i)}  \Big)^2\notag \\
\leq& 2\sum_{i\in[N]}d_{i,\pi(i)}^2+2\sigma_A^2.
\end{align}
\end{lemma}
\begin{proof}
We have 
\begin{align*}
2\E[F(\pi,\pi')\mid \pi]&=N\E[d_{I,\pi(I)}+d_{J,\pi(J)}-d_{I,\pi(J)}-d_{J,\pi(I)}\mid \pi]\\
&=2\sum_{i=1}^N d_{i,\pi(i)}-\frac{2}{N}\sum_{i,j\in[N]}d_{i,\pi(j)}\\
&=2f(\pi).
\end{align*}
Therefore, the couple $(f,F)$ satisfies the required conditions. It remains to calculate and bound the corresponding $v(\cdot)$ function; for this, we have
\begin{align*}
v(\pi)&=\frac{N}{4}\E\Big[\Big(\sum_{i=1}^N d_{i,\pi(i)}-\sum_{i=1}^N d_{i, \pi'(i)}\Big)^2  \mid \pi\Big]\\
&=\frac{1}{4N}\sum_{i,j\in[N]}\Big(d_{i,\pi(i)}+d_{j,\pi(j)}-d_{i,\pi(j)}-d_{j,\pi(i)}  \Big)^2\\
&\leq 2\sum_{i\in[N]}d_{i,\pi(i)}^2+2\sigma_A^2.
\end{align*}
This completes the proof.
\end{proof}

\begin{theorem}[Combinatorial Hoeffding's inequality, version I]\label{thm:hb-ineq} We have, for all $t\geq 0$,
\[
\Pr(Y-\E[Y]\geq t)\leq \exp\Big(-\frac{t^2}{4NB_A^2+4\sigma_A^2} \Big).
\]
\end{theorem}
\begin{proof}
Continue from \eqref{eq:2chatterjee-1} and use a crude bound
\begin{align*}
2\sum_{i\in[N]}d_{i,\pi(i)}^2+2\sigma_A^2\leq  2NB_A^2+2\sigma_A^2.
\end{align*}
Invoking Lemma \ref{lem:Cha-1stLemma} then completes the proof.
\end{proof}

\begin{theorem}[Combinatorial Hoeffding's inequality, version II]\label{cor:c-hoeffding-2}
Suppose that $A=[a_{i,j}]$ can be decomposed as $a_{i,j}=a_ib_j$. We then have, for all $t\geq 0$,
\[
\Pr(Y-\E[Y]\geq t)\leq \exp\Big(-\frac{t^2}{4\bar\sigma_A^2+4\sigma_A^2} \Big),
\]
where
\begin{align*}
\bar\sigma_A^2&=\sup_{s\in\cS_N}\sum_{i=1}^N(a_i-\bar a_N)^2(b_{s(i)}-\bar b_N)^2\leq \sqrt{\sum_{i=1}^N(a_i-\bar a_N)^4}\cdot \sqrt{\sum_{i=1}^N(b_i-\bar b_N)^4}\\
~~{\rm and}~~\sigma_A^2&=\frac{1}{N-1}\sum_{i=1}^N(a_i-\bar a_N)^2\sum_{j=1}^N(b_{j}-\bar b_N)^2
\end{align*}
with $\bar a_N:=N^{-1}\sum_{i=1}^Na_i$ and $\bar b_N:=N^{-1}\sum_{i=1}^Nb_i$.
\end{theorem}
\begin{proof}
Continuing from \eqref{eq:2chatterjee-1} and using the fact that, as $a_{i,j}=a_ib_j$, 
\[
d_{i,j}=(a_i-\bar a_N)(b_j-\bar b_N), 
\]
we obtain
\begin{align*}
2\sum_{i\in[N]}d_{i,\pi(i)}^2+2\sigma_A^2 \leq 2\sum_{i=1}^N(a_i-\bar a_N)^2(b_{\pi(i)}-\bar b_N)^2 + 2\sigma_A^2
\leq  2\bar\sigma_A^2 + 2\sigma_A^2,
\end{align*}
which concludes the proof using Lemma \ref{lem:Cha-1stLemma}.
\end{proof}

\begin{theorem}[Combinatorial Bernstein's inequality]\label{thm:cb-ineq} We have, for all $t\geq 0$,
\[
\Pr(Y-\E[Y]\geq t)\leq \exp\Big(-\frac{t^2}{12\sigma_A^2+4\sqrt{2}B_At} \Big).
\]
\end{theorem}
\begin{proof} 
Continuing \eqref{eq:2chatterjee-1}, we obtain
\begin{align}
v(\pi)&\leq 2\sum_{i\in[N]}d_{i,\pi(i)}^2+\frac{2}{N}\sum_{i,j\in[N]}d_{i,\pi(j)}^2 \leq W+4\sigma_A^2, \label{eq:2chatterjee-1-2}
\end{align}
where we introduce
\[
W=2\sum_{i\in[N]}d_{i,\pi(i)}^2-\frac{2(N-1)}{N}\sigma_A^2,
\]
whose mean is easily verified to be 0.

Plugging it into Lemma \ref{lem:2chattejee-3}, we obtain
\begin{align*}
r(\psi)=\frac{1}{\psi}\log \E e^{\psi v(X)} \leq \frac{1}{\psi}\log \E e^{\psi (W+4\sigma_A^2)}\leq 4\sigma_A^2+\frac{1}{\psi}\log\E e^{\psi W}.
\end{align*}
It remains to bound $\log\E e^{\psi W}$. To this end, let's employ the Stein exchangeable pair again, which gives
\begin{align*}
v_W(\pi)&=\frac{1}{N}\sum_{i,j\in[N]}[d^2_{i,\pi(i)}+d^2_{j,\pi(j)}-d^2_{i,\pi(j)}-d^2_{j,\pi(i)}]^2\\
&\leq \frac{4B_A^2}{N}\sum_{i,j\in[N]}[d^2_{i,\pi(i)}+d^2_{j,\pi(j)}+d^2_{i,\pi(j)}+d^2_{j,\pi(i)}]\\
&\leq 4B_A^2(W+4\sigma_A^2).
\end{align*}
Plugging the above bound into Lemma \ref{lem:2chattejee-2}, we obtain
\[
\log \E e^{\psi W}\leq \frac{8B_A^2\sigma_A^2\psi^2}{1-4B_A^2\psi},
\]
implying
\[
r(\psi) \leq 4\sigma_A^2+\frac{8B_A^2\sigma_A^2\psi}{1-4B_A^2\psi}.
\]
Sending the above bound into Lemma \ref{lem:2chattejee-3} and setting $\psi=(8B_A^2)^{-1}$ completes the proof.
\end{proof}

%\begin{theorem}[Combinatorial Burkholder-Davis-Gundy inequality] Denote the random variable
%\[
%\Delta = \frac{1}{4N}\sum_{i,j\in[N]}\Big(d_{i,\pi(i)}+d_{j,\pi(j)}-d_{i,\pi(j)}-d_{j,\pi(i)}  \Big)^2.
%\]
%We then have, for any positive integer $k$,
%\[
%\E(Y-\E[Y])^{2k}\leq (2k-1)^k\E[\Delta^k].
%\]
%\end{theorem}

%This theorem can be directly verified by combining Equation \eqref{eq:2chatterjee-1} with Lemma \ref{lem:2chattejee-4}.

%\begin{theorem}[Combinatorial Bernstein-type inequality]\label{thm:2chatterjee-2} Suppose $a_{i,j} \in (0, 1)$ for all $i,j\in[N]$. We then have, for all $t>0$,
%\[
%\P(Y-\E[Y]\geq t)\leq \exp\Big(-\frac{t^2}{4\E[Y]+2t} \Big).
%\]
%\end{theorem}

%The last theorem is the sharpest one.

\section{Combinatorial multivariate CLT}\label{sec:cmclt}

Lastly, we extend the above results to multivariate $\ba_{i,j}\in\mathbb{R}^p$. We could then similarly define
\[
\ba_{i,\bullet}=\frac{1}{N}\sum_{j=1}^N\ba_{ij}, ~~~\ba_{\bullet,j}=\frac1N\sum_{i=1}^N\ba_{ij},~~~ \ba_{\bullet,\bullet}=\frac{1}{N^2}\sum_{i,j=1}^N\ba_{i,j},
\]
and $\bd_{i,j}=\ba_{i,j}-\ba_{i,\bullet}-\ba_{\bullet,j}+\ba_{\bullet,\bullet}$.
It is easy to verify that $\sum_{i=1}^N\bd_{i,j}=\sum_{j=1}^N\bd_{i,j}=\bm{0}~~\text{for any }i,j\in[N]$.
Define
\[
\bY:=\sum_{i=1}^N\ba_{i,\pi(i)},~~~\bmu_A=\frac{1}{N}\sum_{i,j=1}^N\ba_{i,j},~~~{\rm and}~~~\bSigma_A=\frac{1}{N-1}\sum_{i,j\in[N]}\bd_{i,j}\bd_{i,j}^\top.
\]

\begin{proposition}[Mean and covariance matrix of $\bY$]\label{prop:mult-comb-mean-and-variance} We have
\[
\E[\bY]=\bmu_A~~~{\rm and}~~\cov(\bY)=\bSigma_A.
\]
\end{proposition}

\begin{exercise}
Please prove Proposition \ref{prop:mult-comb-mean-and-variance}.
\end{exercise}

\begin{proposition}[Multivariate combinatorial CLT]\label{prop:mc-CLT} Suppose $p$ is a fixed positive integer and it further holds true that
\begin{align}\label{eq:mc-CLT-1}
\frac{1}{N}\sum_{i=1}^N\sum_{j=1}^N\bnorm{\bSigma_A^{-1/2}\bd_{i,j}}_2^3\to 0.
\end{align}
We then have $\bSigma_A^{-1/2}(\bY-\bmu_A) \Rightarrow N(0, \Ib_p)$.
\end{proposition}
\begin{proof}
Invoking the Cramer-Wold devise (Corollary \ref{cor:cw-device}), it suffices to show that, for any $\bv\in\mathbb{R}^p$ such that $\|\bv\|_2=1$, we have
\[
\bv^\top\bSigma_A^{-1/2}(\bY-\bmu_A) \Rightarrow N(0,1).
\]
Notice that 
\begin{align*}
\bv^\top\bSigma_A^{-1/2}(\bY-\bmu_A)=  \sum_{i=1}^N\bv^\top\bSigma_A^{-1/2}\bd_{i,\pi(i)}=\sum_{i=1}^N w_{i,\pi(i)}
\end{align*}
with $w_{i,j}=\bv^\top\bSigma_A^{-1/2}\bd_{i,j}$ and 
\begin{align*}
\var\Big(\sum_{i=1}^nw_{i,\pi(i)} \Big)&=\frac{1}{N-1}\sum_{i,j=1}^N\bv^\top\bSigma_A^{-1/2}\bd_{i,j}\bd_{i,j}^\top\bSigma_A^{-1/2}\bv\\
&=\bv^\top\bSigma_A^{-1/2}\Big(\frac{1}{N-1}\sum_{i,j=1}^N\bd_{i,j}\bd_{i,j}^\top\Big)\bSigma_A^{-1/2}\bv\\
&=\bv^\top\bSigma_A^{-1/2}\bSigma_A\bSigma_A^{-1/2}\bv=1.
\end{align*}
Invoking Theorem \ref{thm:CCLT-key}, we then obtain $\sum_{i=1}^N w_{i,\pi(i)} \Rightarrow N(0,1)$ if 
\[
\frac{1}{N}\sum_{i,j\in[N]}|w_{i,j}|^3\leq \frac1N\sum_{i,j\in[N]}\bnorm{\bSigma_A^{-1/2}\bd_{i,j}}_2^3\to 0.
\]
This completes the proof.
\end{proof}

\section{Hoeffding's convex ordering inequality}\label{sec:hoeffding-convex-order}

We end this chapter with a useful result that compares linear statistics of entries sampled without replacement to those with replacement. This inequality is due to Wassily Hoeffding's pathbreaking 1963 paper \citep{hoeffding1963probability} that also proposed another, maybe more famous, inequality named after him (Theorem \ref{thm:2hoeffding}).

The observation that sampling without replacement would lead to more accurate estimates than sampling with replacement is intuitive. In the case of a linear statistics, the following argument of Debabrata Basu \citep{basu1958sampling} is well-known.

%Let's consider $\{z_1,\ldots,z_N\} \in \cZ$ to be a set of non-random elements that are {\it not} necessarily distinct. For arbitrary $n\in[N]$, let $Y_1,\ldots,Y_n$ be sampled uniformly {\it without} replacement from $\{z_1,\ldots,z_n\}$ and $X_1,\ldots,X_n$ be sampled uniformly with replacement (i.e., {\it independently}) from $\{z_1,\ldots,z_n\}$.

\begin{lemma}
Fix an arbitrary scalar set $\{z_1,\ldots,z_N\}$ that contains not necessarily distinct elements. Then, for any $n\in[N]$, we have
\[
\var\Big(\sum_{i=1}^nY_i\Big)\leq \var\Big(\sum_{i=1}^nX_i\Big),
\]
where $Y_1,\ldots,Y_n$ are sampled uniformly {\it without} replacement from $\{z_1,\ldots,z_N\}$, and $X_1,\ldots,X_n$ are sampled uniformly with replacement (i.e., {\it independently}) from $\{z_1,\ldots,z_N\}$.
\end{lemma}
\begin{remark}
We note that the permutation sum $\sum_{i=1}^nY_i$ can be easily written in the form of a combinatorial sum:
\[
\sum_{i=1}^nY_i = \sum_{i=1}^N z_i\ind(\pi(i)\leq n).
\]
Accordingly, any concentration inequalities for independent sums and derived using Chernoff-type bounds can yield a couple for permutation sums. 
\end{remark}
\begin{proof}
It is immediate that, for each $i\in[n]$, $Y_i$ and $X_i$ are identically distributed, so that they have identical moments. Accordingly, we have 
\begin{align*}
&\var\Big(\sum_{i=1}^nX_i\Big)-\var\Big(\sum_{i=1}^nY_i\Big)=\E\Big[\sum_{i=1}^nX_i\Big]^2-\E\Big[\sum_{i=1}^nY_i\Big]^2\\
&=\sum_{i\ne j}\E[X_iX_j]-\sum_{i\ne j}\E[Y_iY_j]=n(n-1)\Big(\E[X_1X_2]-\E[Y_1Y_2]\Big)\\
&=n(n-1)\Big(\frac{1}{n^2}\sum_{i,j=1}^nz_iz_j-\frac{1}{n(n-1)}\sum_{i\ne j}z_iz_j\Big)=\frac{1}{n}\Big\{(n-1)\sum_{i=1}^nz_i^2-\sum_{i\ne j}z_iz_j\Big\}\\
&=\frac1n\Big\{\sum_{i\ne j}\frac{z_i^2+z_j^2}{2}-\sum_{i\ne j}z_iz_j\Big\}\geq 0.
\end{align*}
The proof is thus complete.
\end{proof}

Wassily Hoeffding extended the above argument to a more general framework that can potentially compare the stochastic ordering of vector-valued linear functionals arising from the two types of samplings. The following is his convex ordering inequality.

\begin{theorem}[Hoeffding's convex ordering inequality]\label{thm:hoeffding-amazing} Consider any {\it vector space} $\cZ$ and any $z_1,\ldots,z_N$ (not necessarily distinct) in $\cZ$. For any convex function $f:\cZ\to\reals$  and any $n\in[N]$, we have
\[
\E f\Big(\sum_{i=1}^nY_i\Big) \leq \E f\Big(\sum_{i=1}^n X_i \Big),
\]
where $Y_1,\ldots,Y_n$ are sampled uniformly {\it without} replacement from $\{z_1,\ldots,z_N\}$, and $X_1,\ldots,X_n$ are sampled uniformly with replacement (i.e., {\it independently}) from $\{z_1,\ldots,z_N\}$.
\end{theorem}
\begin{proof}
Hoeffding's original proof is hard to digest. In the following we used a coupling argument that the author learnt from \cite{ben2018weighted}. 

Let $\{J_1,J_2,\ldots\}$ be random integers sampled uniformly and independently from $[N]$. For each $k\in [N]$, let 
\[
I_k=J_{T_k},
\]
where $T_k$ indexes the $k$-th distinct item appearing in $\{J_1, J_2, \ldots\}$. It is then immediate that 
\[
\sum_{i=1}^nX_i \stackrel{d}{=}\sum_{i=1}^n z_{J_i}~~~{\rm and}~~~\sum_{i=1}^nY_i \stackrel{d}{=}\sum_{i=1}^nz_{I_i},
\]
so it suffices to compare the two sums on the righthand sides. To this end, we have that, for any distinct points $i_1,\ldots,i_n\in[N]$ and $k\in[n]$,
\begin{align*}
&\E\Big[z_{J_k}\mid \{I_\ell\}_{\ell=1}^n=\{i_\ell\}_{\ell=1}^n\Big] = \sum_{j=1}^n \Pr\Big(J_k=i_j \mid \{I_\ell\}_{\ell=1}^n=\{i_\ell\}_{\ell=1}^n\Big)z_{i_j}\\
=&\sum_{j=1}^n\frac{1}{n}z_{i_j}= \frac{1}{n}\E\Big[\sum_{i=1}^nY_i \mid \{I_\ell\}_{\ell=1}^n=\{i_\ell\}_{\ell=1}^n\Big],
\end{align*}
so that
\[
\E\Big[\sum_{i=1}^nX_i \mid \{I_\ell\}_{\ell=1}^n \Big] = \sum_{k=1}^n \E\Big[z_{J_k}\mid \{I_\ell\}_{\ell=1}^n\Big]=\sum_{i=1}^nY_i.
\]
Accordingly, 
\[
\E\Big[ \sum_{i=1}^nX_i \mid \sum_{i=1}^nY_i\Big]=\E\Big[ \E\Big[\sum_{i=1}^nX_i\mid \{I_\ell\}_{\ell=1}^n\Big] \mid \sum_{i=1}^nY_i\Big]=\sum_{i=1}^nY_i.
\]
Thusly, $(\sum X_i,\sum Y_i)$ forms a {\it martingale coupling}. In particular, we have, by (finite form) Jensen's inequality (Lemma \ref{lem:finite-form-jensen}),
\begin{align*}
\E f\Big(\sum_{i=1}^nY_i\Big) &= \E \Big[f\Big(\E\Big[\sum_{i=1}^nX_i\mid \sum_{i=1}^nY_i\Big]\Big)\Big]\\ 
&\leq \E\Big[\E\Big[f\Big(\sum_{i=1}^nX_i\Big)\mid \sum_{i=1}^nY_i \Big]\Big]\\
&=\E\Big[f\Big(\sum_{i=1}^nX_i\Big)\Big].
\end{align*}
This completes the proof.
\end{proof}

\begin{lemma}[Finite-form Jensen's inequality]\label{lem:finite-form-jensen} Assume 
a function $f:\cZ\to\reals$ to be convex, i.e., for any nonnegative real numbers $\lambda_1,\lambda_2$ such that $\lambda_1+\lambda_2=1$, we have 
\[
f(\lambda_1z_1+\lambda_2 z_2)\leq \lambda_1f(z_1)+\lambda_2f(z_2),~~~\text{ for any }z_1,z_2\in\cZ.
\]
We then have, for any finite positive integer $n$ and any nonnegative real numbers $\lambda_1,\ldots,\lambda_n$ such that $\sum_{i=1}^n\lambda_i=1$, 
\[
f\Big(\sum_{i=1}^n\lambda_iz_i\Big) \leq \sum_{i=1}^n \lambda_i f(z_i).
\]
\end{lemma} 

\begin{exercise}
Please prove lemma \ref{lem:finite-form-jensen}.
\end{exercise}

\section{Notes}

{\bf Chapters \ref{sec:3overview}-\ref{sec:clln}.} Study of combinatorial sums of the form $\sum_{i=1}^N{a_{i,\pi(i)}}$ originated from Wald and Wolfowitz \citep{wald1944statistical}, who were focused on a more special case of $a_{i,j}=a_ib_j$; see, also, \cite{noether1949theorem}, \cite{erdos1959central}, and \cite{hajek1961some}. The current form of the combinatorial sum was pinned down by Hoeffding \citep{hoeffding1951combinatorial}; see, also, \cite{motoo1956hoeffding}.\\

{\bf Chapters \ref{sec:cclt}-\ref{sec:cmd}.} Wald and Wolfowitz \citep{wald1944statistical}, Noether \citep{noether1949theorem}, and Dwass \citep{dwass1955asymptotic} proved CLTs for linear statistics of the form $\sum_{i=1}^Na_ib_{\pi(i)}$, for which H\'ajek \citep{hajek1961some} gave a final say: a necessary and sufficient condition. Hoeffding introduced and proved the first CLT for the general permutation statistic $\sum_{i=1}^N{a_{i,\pi(i)}}$ in \cite{hoeffding1951combinatorial}. A Lindeberg condition was later established in \cite{motoo1956hoeffding}. 

Von Bahr \citep{von1976remainder} gave the first Berry-Esseen-type bound for quantifying the convergence of combinatorial CLTs. Ho and Chen \citep{ho1978l_p} are the first to introduce Stein's method into analyzing the combinatorial sum. Using a more refined analysis based on Stein's method, Bolthausen \citep{bolthausen1984estimate} gave the present theorem (Theorem \ref{thm:CCLT-key}).

Charles Stein introduced Stein's method in his famous 1972 paper \citep{stein1972bound}. A standard reference to his method is the monograph written by Stein himself \citep{stein1986approximate}. Another popular and also systematic introduction to his method is \cite{chen2010normal}. See, also, \cite{ross2011fundamentals} and \cite{chatterjee2014short} for some concise surveys of the literature. In particular, the author learnt the arguments made in Chapter \ref{sec:cclt} from \cite{ross2011fundamentals}.

Chapter \ref{sec:cclt-variant} concerns an oscillation setting that was related to counting Eulerian number and the number of runs \citep{barton1965some,knuth1997art}. A related setting concerns double- and multiply-indexed permutation statistics \citep{zhao1997error,shi2022distribution}, which are  U-statistics counterpart to the combinatorial sum studied in this chapter. 

{\bf Chapters \ref{sec:3cmi}-\ref{sec:hoeffding-convex-order}.} The idea of using Stein's method to derive sharp concentration bounds is due to Sourav Chatterjee in \cite{chatterjee2005concentration}, which we followed closely in this chapter.  Specializing to survey sample means, \cite{serfling1974probability} introduced a set of concentration inequalities based on the martingale method; see, also, \cite{bardenet2015concentration} and \cite{greene2016finite} for surveys of related inequalities. We will give a relatively more detailed discussion on this literature in Chapter \ref{chap:cpt}.

\citet[Appendix Section A]{shi2022berry} gave a set of Berry-Esseen bounds for multivariate combinatorial CLT. See, also, \cite{fraser1956vector} for an early attempt and \cite{bolthausen1993rate}, \cite{chatterjee2007multivariate}, and \cite{fang2015rates} for some more recent progress.

Hoeffding's original convex ordering inequality, presented in \citet[Theorem 4]{hoeffding1963probability}, only concerns real-valued random variables. The present theorem, instead, can handle general vector-valued random objects including, in particular, stochastic processes. This form was introduced in \citet[Proposition A.1.9]{vaart1996empirical}. The present proof is due to \cite{ben2018weighted}, which only concerned real-valued random variables but whose proof idea can be easily generalized to analyzing vector-valued ones. This type of the generalization is of particular usefulness in stochastic process analysis.

\section*{Appendix}

\begin{lemma}\label{lem:stein-soln} Suppose $g\in \cL(1)$ and $Z\sim N(0,1)$. Then the Stein's equation for $g$,
\begin{align}\label{eq:stein-equation}
f'(w)-wf(w)=g(w)-\E g(Z),
\end{align}
has the solution
\[
f(w)=-\int_0^1\frac{1}{2\sqrt{t(1-t)}}\E\Big[Zg\big(\sqrt{t}w+\sqrt{1-t}Z\big) \Big]\d t.
\]
\end{lemma}

\begin{exercise}
Prove Lemma \ref{lem:stein-soln}.
\end{exercise}

\begin{lemma}\label{lem:stein-soln2}
The Stein solution to $g$ in Lemma \ref{lem:stein-soln} satisfies
\[
\norm{f}_\infty\leq 1,~~ \norm{f'}_{\infty}\leq \sqrt{\frac{2}{\pi}},~~{\rm and}~~ \norm{f''}_{\infty}\leq 2.
\]
\end{lemma}
\begin{proof}
{\bf Part I.} Let $h(Z)=g\big(\sqrt{t}w+\sqrt{1-t}Z\big)$. By Lemma \ref{lem:stein-identity}, 
\[
\E[Zh(Z)]=\E[h'(Z)],
\]
yielding
\[
\E\Big[Zg\big(\sqrt{t}w+\sqrt{1-t}Z\big) \Big] = \sqrt{1-t}\E\Big[g'\big(\sqrt{t}w+\sqrt{1-t}Z\big) \Big]. 
\]
Accordingly,
\[
f(w)=-\int_0^1 \frac{1}{2\sqrt{t}}\E\Big[g'\big(\sqrt{t}w+\sqrt{1-t}Z\big) \Big] \d t
\]
and thus
\[
\norm{f}_{\infty} \leq \int_0^1 \frac{1}{2\sqrt{t}}\d t = 1.
\]

{\bf Part II.} According to the form of $f$, we have
\[
f'(w)=-\int_0^1\frac{1}{2\sqrt{1-t}}\E\Big[Zg'\big(\sqrt{t}w+\sqrt{1-t}Z\big) \Big]\d t
\]
so that
\[
\norm{f'}_{\infty} \leq \E|Z|=\sqrt{\frac{2}{\pi}}.
\]

{\bf Part III.} Taking derivative to \eqref{eq:stein-equation} yields
\[
f''(w)-f(w)-wf'(w)=g'(w)
\]
so that
\[
f''(w)=g'(w)+f(w)+w\big[wf(w)+g(w)-\E g(Z)\big].
\]
Some algebra gives
\[
g(w)-\E g(Z)=\int_{-\infty}^wg'(z)\Phi(z)\d z - \int_w^{\infty}g'(z)\overline\Phi(z) \d z,
\]
where $\overline\Phi(z):=1-\Phi(z)$. Similarly, we can write
\[
f(w)=-\sqrt{2\pi}e^{w^2/2}\Big[\overline\Phi(w)\int_{-\infty}^w g'(z)\Phi(z)\d z + \Phi(w)\int_w^{\infty}g'(z)\overline\Phi(z)\d z \Big].
\]
Combining together reaches
\begin{align*}
f''(w)=&g'(w)+w[g(w)-\E[g(Z)]]+(1+w^2)f(w)\\
=&g'(w)+\Big[w-\sqrt{2\pi}(1+w^2)e^{w^2/2}\overline\Phi(w) \Big]\int_{-\infty}^wg'(z)\Phi(z)\d z+\\
&\Big[-w-\sqrt{2\pi}(1+w^2)e^{w^2/2}\Phi(w)\Big]\int_w^{\infty}g'(z)\overline\Phi(z)\d z,
\end{align*}
so that
\[
\norm{f''}_{\infty}\leq 2.
\]
This completes the proof of the whole lemma.
\end{proof}

\begin{proof}[Proof of Lemma \ref{lem:stein-1}] According to Equation \eqref{eq:stein-equation}, for all $g\in\cL(1)$,
\[
\E[g(W)-g(Z)] = \E[f'(W)-Wf(W)],
\]
so that, by Lemma \ref{lem:stein-soln2},
\begin{align*}
d_W(W,Z)&=\sup_{g\in\cL(1)}\Big|\E[g(W)-g(Z)]  \Big|\\
&\leq \sup_{f\in\cL'(1)}\Big|\E[f'(W)-Wf(W)] \Big|.
\end{align*}
This completes the proof.
\end{proof}

%% file: chapters/combin-process.tex
\begin{partbacktext}
\part{Combinatorial Processes}
\end{partbacktext}

%% file: chapters/ep.tex
\chapter{Empirical Process Theory}\label{chapter:ep}

\section{Glivenko, Cantelli, Doob, Donsker, Kolmogorov, and Skorokhod}\label{sec:4history}

The story begins with a series of questions on (different notions of) convergences of the empirical CDF to the CDF. 

Assume $X_1,\ldots,X_n$ to be $\reals$-valued random variables sampled {\it independently} from a law $\P$ whose CDF is written as $F$. The empirical CDF is then defined to be
\[
F_n(t):=\frac{1}{n}\sum_{i=1}^n\ind(X_i\leq t),
\]
corresponding to a (random) probability measure, which we call the {\it empirical measure} and denote as $\PP_n$.

Glivenko, Cantelli, and Kolmogorov independently proved, appearing to be published in 1933 in the same journal of the same issue, the following theorem. We now often call it the Glivenko-Cantelli theorem.
\begin{theorem}[Glivenko-Cantelli-Kolmogorov]\label{thm:gck} For any distribution function $F$, as long as $X_1,\ldots,X_n$ are independently sampled from $F$, it holds true that
\[
\sup_{t\in\reals}\Big|F_n(t)-F(t)\Big| \stackrel{a.s.}{\to} 0.
\]
\end{theorem}
\begin{proof}
By the standard quantile transformation trick (cf. \citet[Proposition 1.2]{dudley2014uniform}), it suffices to consider $F$ corresponding to the law of the Lebesgue measure over $[0,1]$. In this case, for any fixed $\epsilon>0$, take a sufficiently large integer $k$ such that $k^{-1}<\epsilon/2$. By the strong law of large numbers, we then have
\[
\sup_{j=0,1,\ldots,k}\Big|F_n\Big(\frac{j}{k}\Big)-\frac{j}{k}\Big| \stackrel{a.s.}{\to} 0.
\]
In other words, there exists a set $A\in\cF$ such that $\Pr(A)=1$ and for all $\omega\in A$, there exists a sufficiently large $n_0$ such that, for any $n>n_0$, 
 \[
\sup_{j=0,1,\ldots,k}\Big|F_n\Big(\frac{j}{k}\Big)-\frac{j}{k}\Big|<\epsilon/2.
 \]
 For each $t\in[0,1]$, choose $j_t$ such that $(j_t-1)/k\leq t\leq j_t/k$. Then, for all $\omega\in A$,
 \[
 \frac{j_t-1}{k}-\frac{\epsilon}{2}<F_n\Big(\frac{j_t-1}{k}\Big)\leq F_n(t)\leq F_n\Big(\frac{j_t}{k}\Big)\leq \frac{j_t}{k}+\frac{\epsilon}{2},
 \]
 so that, uniformly,
 \[
 \Big|F_n(t)-t \Big|\leq \epsilon.
 \]
 This implies uniform almost sure convergence for $F_n$ to $F$.
\end{proof}

%It is worth pointing out that, in deriving the above result, one core component is bounding the expectation of the supremum,
%\[
%\E\sup_{t\in\reals}\Big|F_n(t)-F(t)\Big|.
%\]
%We will be repeatedly faced with handling such maximal inequalities in this note.

The next question is, of course, how fast does the empirical CDF (uniformly) converge to the CDF? Fixing any $t\in\reals$, the standard CLT (see, e.g., Theorem \ref{thm:LFL-CLT}) informs us that
\[
\sqrt{n}(F_n(t)-F(t))=\frac{1}{n}\sum_{i=1}^n\Big(\ind(X_i\leq t)-\P(X_i\leq t)\Big) \Rightarrow N\Big(0, F(t)(1-F(t))\Big).
\]
Therefore, the best rate we can expect to achieve is $\sqrt{n}$. It turns out to be indeed the right rate, and that there is even more we could say about this rate of convergence: it is actually exponentially fast.

\begin{theorem}[Dvoretzky-Kiefer-Wolfowitz-Massart] For any distribution $F$, as long as $X_1,\ldots,X_n$ are independently sampled from $F$, it holds true that 
\[
\Pr\Big(\sqrt{n}\sup_{t\in\reals}\Big| F_n(t)-F(t) \Big| >u \Big) \leq 2\exp(-2u^2),~~\text{for any }u>0.
\] 
\end{theorem}
\begin{proof}
See \cite{massart1990tight}.
\end{proof}

Lastly, a crown question remains: does there exist a weak convergence type result for the {\it stochastic process} $\{\sqrt{n}(F_n(t)-F(t)); t\in \reals\}$ to its limit, if the latter exists? 

Unlike the Glivenko-Cantelli and Dvoretzky-Kiefer-Wolfowitz-Massart theorems, the last question involves first a subtle question: 
\begin{quote} \normalsize
what do we mean weak convergence of a sequence of stochastic processes --- here it is $\{\sqrt{n}(F_n(t)-F(t)); t\in\reals\}$ indexed by the sample size $n$ --- to another stochastic process? 
\end{quote}
Donsker in his famous 1952 paper, \cite{donsker1952justification}, gave a (not quite right) uniform convergence argument. Weak convergence of stochastic processes was clarified later \citep{dudley1966weak,billingsley1968convergence}. The exact form of Definition \ref{def:weak-convergence} is due to \cite{hoffmann1991stochastic}. \cite{dudley1978central} devised Theorem \ref{thm:weak-convergence-key} and the exact form we took is due to \cite{gine1986empirical}.

Applying Theorem \ref{thm:weak-convergence-key} to $\{\sqrt{n}(F_n(t)-F(t)); t\in\reals\}$ yields the celebrated Donsker's theorem, due to \cite{donsker1952justification} combined with \cite{skorokhod1956limit} and \cite{kolmogorov1956skorokhod}. 

\begin{theorem}[Donsker-Skorokhod-Kolmogorov]\label{thm:donsker} The stochastic processes  $\{\sqrt{n}(F_n(t)-F(t)); t\in\reals\}$ weakly converges to a (scaled) Brownian bridge process, whose arbitrary finite-dimensional distribution is Gaussian with mean zero and variance
\[
\cov(Z(s), Z(t))=F(s \wedge t)-F(s)F(t),~~~\text{ for any }s,t\in\reals.
\]
\end{theorem}

\section{Dudley's metric entropy bounds}\label{chapter:dudley}

%Starting from this section, we will drop the outer measure/expectation notation by working on a separable version of each random variable/process under study; see Section 2.3.3 in \cite{vaart2023empirical} for a detailed discussion on the existence and appropriateness of replacing a random variable/process by its separable version. 

Stochastic equicontinuity, in the form of \eqref{eq:sec}, is nontrivial to handle. Dudley's ingenuous approach to it constitutes the metric entropy method. 

\begin{definition}[Covering and packing numbers] For any pseudo-metric space $( T,d)$, the {\it covering number}, denoted by $\cN( T,d,\epsilon)$, and the {\it packing number}, denoted by $\cD( T, d, \epsilon)$, stand for the minimal number of radius-$\epsilon$ balls to cover $ T$ and the largest number of $\epsilon$-separated points in $ T$, respectively.
\end{definition}

\begin{exercise} Please show that, for any pseudo-metric space and any $\epsilon>0$, 
\[
\cN( T,d,\epsilon)\leq \cD( T,d,\epsilon) \leq \cN( T,d,\epsilon/2).
\]
\end{exercise}

\begin{definition}
The {\it metric entropy} of $( T,d)$ is defined as $\log \cN( T,d,\epsilon)$. 
\end{definition}

\begin{example}
For any set $\Theta\subset\reals^d$, it is immediate that
\[
\cN(\Theta,\norm{\cdot},\epsilon) \leq \frac{{\rm Vol}(\Theta+\frac{\epsilon}{2}B)}{{\rm Vol}(\frac{\epsilon}{2}B)},
\]
where $B$ stands for the unit-ball under the metric $\norm{\cdot}$. In particular, when $\norm{\cdot}$ is the Euclidean norm, it holds true that
\[
\cN(\Theta,\norm{\cdot},\epsilon) \leq K\Big(\frac{\diam\Theta}{\epsilon}\Big)^d,
\]
where $K$ only depends on $\Theta$ and $d$.
\end{example}

\begin{lemma}\label{lem:monotone-entropy} 
\begin{enumerate}[label=(\roman*)]
\item Both $\cN( T,d,\epsilon)$ and $\cD( T,d,\epsilon)$ are decreasing functions with regard to $\epsilon$. 
\item Consider two metrics $d_1,d_2$ over the same set $ T$ such that
$d_1(t_1,t_2)\leq d_2(t_1,t_2)$ holds true for all $t_1,t_2\in T$. We then have
\[
\cN( T,d_1,\epsilon)\leq \cN( T,d_2,\epsilon), ~~\text{for all }\epsilon>0.
\]
\end{enumerate}
\end{lemma}
\begin{exercise}
Please prove Lemma \ref{lem:monotone-entropy}.
\end{exercise}

\begin{theorem}[Dudley]\label{thm:dudley1} Recalling Definition \ref{def:young}, let $\psi$ be a Young's function and assume $\{X(t);t\in T\}$ to be a {\it separable} stochastic process such that 
\[
\norm{X(t)}_{\psi}<\infty~~~{\rm and}~~~\norm{X(s)-X(t)}_{\psi}\leq d(s,t),~~~\text{ for any }s,t\in T.
\]
Let 
\[
\diam(T)=\diam(T;d):=\sup\Big\{d(s,t);s,t\in T\Big\}
\]
be the diameter of $( T,d)$. Then, it holds true that
\[
\E\sup_{s,t\in T}\Big|X(s)-X(t)\Big| \leq 8\int_0^{\diam(T)} \psi^{-1}\Big(\cN( T,d,\epsilon) \Big)\d \epsilon.
\]
\end{theorem}
\begin{proof}
We denote $D=\diam(T)$. Without loss of generality, let's assume
\[
\int_0^{D}\psi^{-1}\Big(\cN( T,d,\epsilon) \Big)\d \epsilon < \infty
\]
since otherwise the bound trivially holds. Then, notice that 
\[
\int_0^{D}\psi^{-1}\Big(\cN( T,d,\epsilon) \Big)\d \epsilon \geq \psi^{-1}(1)\cdot D~~~{\rm and}~~~\psi^{-1}(1)> \psi^{-1}(0)=0,
\]
we obtain that $D$ is finite.

In addition, since by condition, $d(s,t)=0$ implies $\Pr(X(s)=X(t))=1$, in the following we assume that $d$ is a proper metric without loss of generality.

{\bf Step 1.} We first assume $ T$ is finite. Take $\ell_0$ to be the largest integer such that $2^{-\ell}>D$ and $\ell_1$ to be the smallest integer such that all open balls $B(t,2^{-\ell})$, of center $t$ and radius $2^{-\ell}$, contain only one point, which then is of course $t$. It is immediate that $\ell_1\geq \ell_0$. For each integer $\ell$ between $\ell_0$ and $\ell_1$, let $T_\ell\subset T$ be (one of) the smallest set of ball centers $s$ such that $\{B(s,2^{-\ell});s\in T_\ell\}$ covers $ T$. Since $2^{-\ell_0}>D$, we have $T_{\ell_0}$ contains only one point, which we denote as $t_0$. In addition, by the definition of $\cN( T,d,\epsilon)$, we have
\[
|T_\ell|=\cN( T,d,2^{-\ell}).
\]

Define 
\[
h_{\ell}: T_\ell \to T_{\ell-1}, ~~~\ell_0<\ell\leq\ell_1 
\]
to be maps such that $t\in B(h_{\ell}(t),2^{-\ell+1})$. Let $k_{\ell}=h_{\ell+1}\circ\cdots\circ h_{\ell_1}$, with $k_{\ell_1}$ being the identity map, identify a chain starting from $t$. We then have, for any $t\in T$,
\[
X(t)-X(t_0) = \sum_{\ell=\ell_0+1}^{\ell_1}\Big\{X(k_\ell(t))-X(k_{\ell-1}(t))\Big\}
\]
so that 
\[
\sup_{t\in T}\Big|X(t)-X(t_0) \Big| \leq \sum_{\ell=\ell_0+1}^{\ell_1}\sup_{t\in T}|X(k_\ell(t))-X(k_{\ell-1}(t))|,
\]
and thus, by triangle inequality, 
\[
\sup_{s,t\in T}\Big|X(t)-X(s) \Big|\leq 2\sum_{\ell=\ell_0+1}^{\ell_1}\sup_{t\in T}|X(k_\ell(t))-X(k_{\ell-1}(t))|.
\]

Now, note that by definition,
\[
\Big|\Big\{X(k_\ell(t))-X(k_{\ell-1}(t)); t\in T \Big\} \Big| \leq \cN( T,d,2^{-\ell})
\]
and
\[
\sup_{t\in T}\bnorm{X(k_\ell(t))-X(k_{\ell-1}(t))}_{\psi}\leq \sup_{t\in T}d\Big(k_\ell(t),k_{\ell-1}(t)\Big)\leq 2^{-\ell+1}.
\]
Invoking Lemma \ref{lem:orlicz} then yields
\[
\E\sup_{t\in T}\Big|X(k_\ell(t))-X(k_{\ell-1}(t))\Big|\leq 2\cdot 2^{-\ell+1}\cdot \psi^{-1}\Big(\cN( T,d,2^{-\ell}) \Big),
\]
and accordingly,
\begin{align*}
\E\sup_{s,t\in T}\Big|X(t)-X(s) \Big| &\leq 2\sum_{\ell=\ell_0+1}^{\ell_1}2^{-\ell+1}\psi^{-1}\Big(\cN( T,d,2^{-\ell}) \Big)\\
&\leq 4\sum_{\ell\geq \ell_0}2^{-\ell}\psi^{-1}\Big(\cN( T,d,2^{-\ell}) \Big)\\
&\leq 8\sum_{\ell>\ell_0}\int_{2^{-\ell-1}}^{2^{-\ell}}\psi^{-1}\Big(\cN( T,d,\epsilon) \Big)\d \epsilon\\
&\leq 8\int_0^{D}\psi^{-1}\Big(\cN( T,d,\epsilon) \Big)\d \epsilon.
\end{align*}
This completes the proof.

{\bf Step 2.} Now, for general $ T$, since $\{X(t);t\in T\}$ is assumed separable, there exists a countable subset $ T_0\subset  T$ such that, 
\[
\E\sup_{s,t\in T}|X(s)-X(t)|=\E\sup_{s,t\in T_0}|X(s)-X(t)|.
\]
Next, by countability of $ T_0$, there exists a sequence of finite subsets $ T_n\subset  T_0$ that increases to $ T_0$. Accordingly, employing monotone convergence theorem, 
\begin{align*}
\E\sup_{s,t\in T_0}|X(s)-&X(t)|=\E\lim_{n\to\infty}\sup_{s,t\in T_n}|X(s)-X(t)|\\
&=\lim_{n\to\infty}\E\sup_{s,t\in T_n}|X(s)-X(t)|\leq 8\int_0^{D} \psi^{-1}\Big(\cN( T,d,\epsilon) \Big)\d \epsilon.
\end{align*}
This completes the proof.
\end{proof}

\begin{theorem}\label{thm:dudley2}
Assume $\{X(t);t\in T\}$ to be a separable stochastic process such that 
\[
\norm{X(t)}_{\psi}<\infty~~~{\rm and}~~~\norm{X(s)-X(t)}_{\psi}\leq d(s,t),~~~\text{ for any }s,t\in T.
\]
Then, it holds true that, for any $\eta,\delta>0$,
\[
\E\sup_{d(s,t)\leq \delta}\Big|X(t)-X(s)\Big| \leq 16\int_0^\eta \psi^{-1}\Big(\cN( T,d,\epsilon)\Big)\d\epsilon +\delta\psi^{-1}(\cN^2( T,d,\eta)).
\]
\end{theorem}
\begin{proof}
{\bf Step 1.} First, assume $ T$ is finite. Fix $\eta$ and $\delta$. In the proof of Theorem \ref{thm:dudley1}, choose to stop the chain not at $\ell_0$ but at some $\ell_{\eta} \in [\ell_0,\ell_1]$ such that $2^{-\ell_\eta}\leq \eta$. By following the proof of Theorem \ref{thm:dudley1}, it then holds true that
\[
X(t)-X(k_\ell(t))=\sum_{\ell=\ell_{\eta}}^{\ell_1}\Big\{X(k_\ell(t))-X(k_{\ell-1}(t)) \Big\},
\]
so that
\[
\E\sup_{t\in T}\Big|X(t)-X(k_\ell(t)) \Big| \leq 4\int_0^{\eta}\psi^{-1}\Big(\cN( T,d,\epsilon) \Big)\d\epsilon.
\]

Next, let's chain the set 
\[
D_\delta:= \Big\{(s,t)\in T\times  T; d(s,t)\leq \delta\Big\}. 
\]
To this end, introduce the set
\begin{align*}
V=\Big\{(x,y)\in T_{\ell_\eta} \times T_{\ell_\eta}; \text{ there exists } (u,v)\in D_\delta \text{ such that }
 k_{\ell_\eta}(u)=x, k_{\ell_\eta}(v)=y \Big\}.
\end{align*}
Now, if $(x,y)\in V$, then let $(u_{x,y}, v_{x,y})$ be a {\it fixed} pair in $D_\delta$ such that $k_{\ell_\eta}(u_{x,y})=x$ and $k_{\ell_\eta}(v_{x,y})=y$. 

For any $(s,t)\in D_\delta$, set $x=k_{\ell_\eta}(s)$, $y=k_{\ell_\eta}(t)$. It is then immediate that $(x,y)\in V$ and we define $(u_{x,y},v_{x,y})$ as above. Triangle inequality then implies
\begin{align*}
|X(s)-X(t)|\leq& |X(s)-X(k_{\ell_\eta}(s))|+|X(k_{\ell_\eta}(s))-X(u_{x,y})|+|X(u_{x,y})-X(v_{x,y})|\\
&+|X(v_{x,y})-X(k_{\ell_\eta}(t))|+|X(k_{\ell_\eta}(t))-X_t|\\
\leq& 4\sup_{r\in T}\Big|X(r)-X(k_{\ell_\eta}(r))\Big|+\sup_{(x,y)\in V}\Big|X(u_{x,y})-X(v_{x,y})\Big|.
\end{align*}
It remains to bound the last term. Notice that each $(x,y)\in T_{\ell_\eta}\times T_{\ell_\eta}$ is bound to one $(u_{x,y},v_{x,y})$, we have (a)
\[
\Big|\Big\{(u_{x,y},v_{x,y});(x,y)\in V  \Big\} \Big| \leq |V|=\cN^2( T,d,\eta),
\]
and (b) 
\[
\bnorm{X(u_{x,y})-X(v_{x,y})}_\psi \leq d(u_{x,y},v_{x,y})\leq \delta.
\]
Combining these two facts and invoking Lemma \ref{lem:orlicz} yield
\[
\E\sup_{(x,y)\in V}\Big|X(u_{x,y})-X(v_{x,y}) \Big| \leq \delta \psi^{-1}(\cN^2( T,d,\eta)).
\]
This completes the proof of the finite case.

{\bf Step 2.} Using the same trick as in the proof of Step 2 of Theorem \ref{thm:dudley1}, we deduce, for any countable subset $ T_0$ of $ T$, we have 
\[
\E\sup_{s,t\in T_0: d(s,t)\leq \delta}\Big|X(t)-X(s)\Big| \leq 16\int_0^\eta \psi^{-1}\Big(\cN( T,d,\epsilon)\Big)\d\epsilon +\delta\psi^{-1}(\cN^2( T,d,\eta)).
\]
Employing the separability assumption then finishes the proof.
\end{proof}

Specializing to the Orlicz-$\psi_2$ norm, the following two corollaries gives bounds that are implications of the above theorems.

\begin{corollary}\label{cor:dudley3}
Assume $\{X(t);t\in T\}$ to be a separable stochastic process such that 
\[
\norm{X(t)}_{\psi_2}<\infty~~~{\rm and}~~~\norm{X(s)-X(t)}_{\psi_2}\leq d(s,t),~~~\text{ for any }s,t\in T.
\]
Then, it holds true that, for any $\delta>0$,
\[
\E\sup_{d(s,t)\leq \delta}\Big|X(t)-X(s)\Big| \leq 18\int_0^\delta \sqrt{\log2\cN( T,d,\epsilon)}\d\epsilon.
\]
\end{corollary}
\begin{proof}
Notice that $\psi_2^{-1}(x)=\sqrt{\log (x+1)}\leq \sqrt{\log 2x}$ for any $x\geq 1$. We have, by picking $\delta=\eta$,
\begin{align*}
&16\int_0^\eta \psi_2^{-1}\Big(\cN( T,d,\epsilon)\Big)\d\epsilon +\delta\psi_2^{-1}(\cN^2( T,d,\eta))\\
\leq&16\int_0^\delta \sqrt{\log 2\cN( T,d,\epsilon)}\d\epsilon +2\delta\sqrt{\log 2\cN( T,d,\delta)},
\end{align*}
where
\[
\delta\sqrt{\log 2\cN( T,d,\delta)}\leq \int_0^\delta \sqrt{\log 2\cN( T,d,\epsilon)}\d\epsilon.
\]
This completes the proof.
\end{proof}

\begin{corollary}\label{cor:dudley4}
Under the same conditions as Corollary \ref{cor:dudley3}, we have 
\[
\bnorm{\sup_{s,t\in T}\Big|X(s)-X(t)\Big|}_{\psi_2} \leq 8\int_0^{\diam(T)} \psi^{-1}\Big(\cN( T,d,\epsilon) \Big)\d \epsilon
\]
and for any $\delta>0$,
\[
\bnorm{\sup_{d(s,t)\leq \delta}\Big|X(t)-X(s)\Big|}_{\psi_2}\leq 18\int_0^\delta \sqrt{\log2\cN( T,d,\epsilon)}\d\epsilon.
\]
\end{corollary}
\begin{proof}
In the proof of all the above arguments, replace the arguments involving Lemma \ref{lem:orlicz} by Lemma \ref{lem:orlicz2}.
\end{proof}

\section{Empirical processes}\label{sec:4ep}

Let $X_1,X_2, \ldots, X_n$ be i.i.d. $\cX$-valued random values following the same law $\P$. Consider a {\it countable}\footnote{The assumption of countability is enforced, of course, for avoiding the measurability issue, which is an unfortunate technical obstacle most statisticians may not be really interested in. In Chapter \ref{chap:cpt}, we will see how the finite population paradigm can offer an alternative way to settling the measurability issue.} class of measurable functions, $\cF$,
\[
\cF= \Big\{f\in\cF; f: \cX\to\reals\Big\},
\]
for which a $\P$-square-integrable envelope function $F$ exists such that 
\begin{align}\label{eq:envelope}
F\geq 1, \int F^2\d\P<\infty, ~~{\rm and}~~|f(x)|\leq F(x) \text{ for all }x\in\cX \text{ and } f\in\cF. 
\end{align}
This implies that all $f$'s in $\cF$ are also $\P$-square-integrable.

Let 
\[
\PP_n=\frac{1}{n}\sum_{i=1}^n\delta_{X_i}, \text{ with }\delta_x \text{ denoting the Dirac measure at } x,
\]
represent the (random) empirical measure generated by $\{X_i;i\in[n]\}$. 

In the following, we adopt a commonly used notation that, for arbitrary signed measure $Q$ and $Q$-integrable function $g$, we write
\[
Qg = Q(g):= \int g\d Q.
\]
Accordingly, $(\PP_n-\P)f$ could be understood as 
\[
\PP_n f-\P f = \int f \d\PP_n -\int f \d\P=\frac{1}{n}\sum_{i=1}^n f(X_i)- \E_{X\sim \P} [f(X)].
\]

\begin{definition}
The stochastic process $\{f\mapsto \PP_nf;f\in\cF\}$, indexed by $f$, is called an {\it empirical process}.
\end{definition}

Next, let's define the largest deviation from the ``sample mean'' to the ``population mean'' as
\[
\bnorm{\PP_n-\P_n}_{\cF}:=\sup_{f\in\cF}\Big|(\PP_n-\P)f \Big|.
\]
Due to the countability condition on $\cF$, the above extreme value is always measurable. 

If the function class is, say,
\[
\cF=\Big\{f_t(x)=\ind(x \leq t); t \in \mathbb{Q} \Big\},
\]
we then obtain a Glivenko-Cantelli type statistic
\[
\bnorm{\PP_n-\P}_{\cF}=\sup_{t\in\mathbb{Q}}\Big|F_n(t)-F(t) \Big|.
\]
It is thus natural to call any function class $\cF$ that satisfies a uniform convergence like Theorem \ref{thm:gck} a $\P$-Glivenko-Cantelli class, with $\P$ highlighting the role of the data generating process. 

\begin{definition}
$\cF$ is said to be a {\it weakly} or {\it strongly $\P$-Glivenko-Cantelli class} if $\E\norm{\PP_n-\P}_{\cF} \to 0$ or if $\norm{\PP_n-\P}_{\cF} \to 0$ almost surely, as $n\to\infty$. 
\end{definition}

We can similarly define the operator $\GG_n(\cdot)$ as 
\[
\GG_nf=\GG_n(f):=\sqrt{n}(\PP_n-\P)f =\frac{1}{\sqrt{n}}\sum_{i=1}^n \Big(f(X_i)-\P f\Big).
\]
The random variable $\GG_nf$ gives a Donsker-type statistic; recall Theorem \ref{thm:donsker}. Indeed, by the finite-dimensional central limit theorem (e.g., Theorem \ref{thm:mlfclt}), 
\begin{align*}
&\frac{1}{\sqrt{n}}\sum_{i=1}^n\Big[ \{f_1(X_i)-\P f_1(X_i)\},\ldots, \{f_m(X_i)-\P f_m(X_i)\}  \Big] \\
\Rightarrow &(\GG_\P f_1,\ldots, \GG_\P f_m),~~\text{for any }f_1,\ldots,f_m\in\cF, m\in \mathbb{N},
\end{align*}
where $\GG_\P(\cF):=\{\GG_\P(f); f\in\cF\}$ is a centered Gaussian process with the same covariance structure as the process $\{\GG_n;f\in\cF\}$:
\[
\E \GG_\P(f)\GG_\P(g)=\P(f-\P f)(g-\P f).
\]
W may refer to $\GG_\P(\cF)$ as the {\it $\P$-bridge process indexed by $\cF$}; again, recall Theorem \ref{thm:donsker} and the Brownian bridge.

For weak convergence in $L^{\infty}(\cF)$ to make any sense, we may have to first require the corresponding Gaussian process $\GG_\P(\cF)$ to be well-defined and nice in a certain sense.

\begin{definition} We say that $\cF$ is $\P$-pre-Gaussian if the $\P$-bridge process $\GG_\P(\cF)$ admits a version (recall Definition \ref{def:2version}) whose sample paths are all bounded and uniformly continuous for its intrinsic $L^2$-distance 
\[
d_\P^2(f,g):= \P(f-g)^2-\{\P(f-g)\}^2,~~f,g\in\cF,
\]
which further produces a pseudo-metric space $(\cF, d_\P)$.
\end{definition}
\begin{definition}We say that the class $\cF$ satisfying
\begin{align}\label{eq:bounded}
\sup_{f\in\cF}\Big|f(x)-\P f\Big|<\infty,~~{\rm for~almost~all~}x\in\cX
\end{align}
is a {\it $\P$-Donsker class} if $\cF$ is $\P$-pre-Gaussian and $\GG_n(\cF):=\{\GG_n(f);f\in\cF\}$ weakly converges in $L^{\infty}(\cF)$ to the Gaussian process $\GG_\P(\cF)$ as $n\to\infty$.
\end{definition}

Theorem \ref{thm:weak-convergence-key} then immediately translates the above definition to the following theorem, which states the Donsker theorem for a general (countable) function class.

\begin{theorem}\label{thm:se2}
Assume that $\cF$ is countable and satisfies \eqref{eq:envelope} and \eqref{eq:bounded}. Then the following two conditions are equivalent:
\begin{enumerate}[label=(\roman*)]
\item $\cF$ is a $\P$-Donsker class.
\item There exists a totally bounded pseudo-metric space $(\cF,d)$ such that
\[
\lim_{\delta\to 0}\limsup_{n\to\infty}\P\Big\{ \sup_{d(f,g)\leq\delta}\Big|\GG_nf-\GG_ng\Big|\geq \epsilon \Big\}=0,
\]
for all $\epsilon>0$. 
\end{enumerate}
\end{theorem}

Now it is apparent that verifying either $\P$-Glivenko-Cantelli or $\P$-Donsker reduces to proving a maximal inequality. Dudley's metric entropy methods, of course, provided an answer to this call.

\section{Glivenko-Cantelli bounds}\label{chapter:gc}

This section discusses the weak and strong Glivenko-Cantell properties under an entropy condition and some further boundedness condition of $\cF$. 

\begin{theorem}[$\P$-Glivenko-Cantelli]\label{thm:gc} Assume $\cF$ is countable and admits an envelope function $F>0$ such that $\P F<\infty$. 
Suppose for any fixed $\epsilon>0$, we have 
\[
\P\Big(\lim_{n\to\infty}\frac{\log\cN(\cF,L^1(\PP_n),\epsilon\norm{F}_{L^1(\PP_n)})}{n}=0\Big)=1. 
\]
It then holds true that $\cF$ is both weakly and strongly $\P$-Glivenko-Cantelli.
\end{theorem}
\begin{proof}
{\bf Step 1.} In the first step, we employ the symmetrization trick to transfer the study of $\norm{\PP_n-\P_n}_{\cF}$ to that of a Rademacher process. 

\begin{lemma}[Symmetrization]\label{lem:symmetrization} For any countable and $\P$-integrable function class $\{g\in\cG; g:\cX\to\reals\}$, we have
\[
\E\sup_{g\in\cG}\Big|(\PP_n-\P)g \Big| \leq 2\E\sup_{g\in\cG}\Big|\frac{1}{n}\sum_{i=1}^n\epsilon_ig(X_i) \Big|.
\]
\end{lemma}
\begin{proof}
Let $Z_1,\ldots, Z_n$ be an independent copy of $X_1,\ldots,X_n$. We then have, for any $g\in\cG$,
\begin{align*}
 \E\sup_{g\in\cG}\Big|\PP_n g-\P g\Big| =&  \E\sup_{g\in\cG}\Big|\frac{1}{n}\sum_{i=1}^n g(X_i)-\frac{1}{n}\E \Big[g(Z_i)\Big]\Big|\\
 =&  \E\sup_{g\in\cG}\Big|\frac{1}{n}\sum_{i=1}^n \Big\{g(X_i)-\E \Big[g(Z_i) \mid X_1,\ldots,X_n\Big]\Big\}\Big|\\
 =& \E\sup_{g\in\cG}\Big|\E\Big(\frac{1}{n}\sum_{i=1}^n \{g(X_i)-g(Z_i)\} \mid X_1,\ldots,X_n\Big)\Big|\\
 \leq& \E\Big\{\E\Big[\sup_{g\in\cG}\Big|\frac{1}{n}\sum_{i=1}^n (g(X_i)-g(Z_i))\Big| \mid X_1,\ldots,X_n \Big]\Big\}\\
 =& \E\sup_{g\in\cG}\Big|\frac{1}{n}\sum_{i=1}^n(g(X_i)-g(Z_i))\Big|.
\end{align*}
%Here the main trick we use is the convexity of the $\sup(\cdot)$ function. 

We then employ the Rademacher sequence $\epsilon_1,\ldots,\epsilon_n$ to conclude 
\begin{align*}
\E\sup_{g\in\cG}\Big|\frac{1}{n}\sum_{i=1}^n(g(X_i)-g(Z_i))\Big| &= \E\sup_{g\in\cG}\left|\frac{1}{n}\sum_{i=1}^n\epsilon_i(g(X_i)-g(Z_i))\right|\\
&\leq 2\E\sup_{g\in\cG}\left|\frac{1}{n}\sum_{i=1}^n\epsilon_ig(X_i)\right|,
\end{align*}
where the last inequality is via the triangle inequality. 
\end{proof}

{\bf Step 2.} For any constant $M>0$, introduce a ``truncated'' function class
\[
\cF_M := \Big\{f\ind(F\leq M); f\in\cF \Big\}.
\] 
Employing the symmetrization trick and by definition of $\cF_M$, we have
\begin{align*}
&\E\bnorm{\PP_n-\P}_{\cF} \\
\leq& \E\sup_{f\in\cF}\Big|\frac{1}{n}\sum_{i=1}^n f(X_i)\ind(F(X_i)\leq M)-\P f\ind(F\leq M)\Big|+\\
&\quad \E\sup_{f\in\cF}\Big|\frac{1}{n}\sum_{i=1}^n f(X_i)\ind(F(X_i)> M)-\P f\ind(F>M)\Big|\\
\leq& \E\bnorm{\PP_n-\P}_{\cF_M}+2\P F\ind(F>M)\\
\leq&2\E\sup_{g\in\cF_M}\Big|\frac{1}{n}\sum_{i=1}^n\epsilon_ig(X_i) \Big|+2\P F\ind(F>M).
\end{align*}
Now, for any $g_1,g_2\in\cF_M$ such that $\norm{g_1-g_2}_{L^1(\PP_n)}=\PP_n|g_1-g_2|\leq \epsilon$, we have
\[
\Big|\frac{1}{n}\sum_{i=1}^n\epsilon_ig_1(X_i)-\frac{1}{n}\sum_{i=1}^n\epsilon_ig_2(X_i) \Big| \leq \frac{1}{n}\sum_{i=1}^n\Big|g_1(X_i)-g_2(X_i) \Big|\leq \epsilon.
\]
Thusly, letting $T_{\epsilon}$ denote the set of centers in an $\epsilon$-net in $(\cF_M,L^1(\PP_n))$, we obtain
\begin{align*}
\E\bnorm{\PP_n-\P}_{\cF} \leq 2\E\sup_{g\in T_{\epsilon}}\Big|\frac{1}{n}\sum_{i=1}^n\epsilon_ig(X_i) \Big|+2\epsilon + 2\P F\ind(F>M).
\end{align*}

{\bf Step 3.} Combining Hoeffding's inequality (Theorem \ref{thm:2hoeffding}) with the interplay between subgaussian distribution and tail probability inequalities (Lemma \ref{lem:subgaussian}) yields that, for any $g\in\cF_M$, conditioning on $X_1,\ldots,X_n$,
\[
\bnorm{\frac{1}{n}\sum_{i=1}^n\epsilon_ig(X_i)}_{\psi_2}\leq  \cdot \sqrt{\frac{6}{n^2}\sum_{i=1}^ng^2(X_i)} \leq M\sqrt{\frac{6}{n}}.
\]
Noting that $\psi_2^{-1}(x)=\sqrt{\log(x+1)}\leq \sqrt{\log 2x}$ for any $x\geq 1$, Lemma \ref{lem:orlicz} then yields, conditioning on $X_1,\ldots,X_n$,
\[
\E_{\epsilon}\sup_{g\in T_{\epsilon}}\Big|\frac{1}{n}\sum_{i=1}^n\epsilon_ig(X_i) \Big| \leq \sqrt{\log 2\cN(\cF_M,L^1(\PP_n),\epsilon)}\cdot M\sqrt{\frac{6}{n}}.
\]

{\bf Step 4.} Wrapping up all, we obtain, conditioning on $X_1,\ldots,X_n$,
\begin{align*}
\E_{\epsilon}\bnorm{\PP_n-\P}_{\cF} \leq 2M\Big\{\sqrt{\frac{6\log 2\cN(\cF_M,L^1(\PP_n),\epsilon\norm{F}_{L^1(\PP_n)})}{n}}\Big\}+2\epsilon\norm{F}_{L^1(\PP_n)} \\
+2 \P F\ind(F>M).
\end{align*}
 By assumption, $\E_{\epsilon}\norm{\PP_n-\P}_{\cF}$ is bounded by $M$. In addition, by assumption,
 \[
\P\Big(\lim_{n\to\infty}\frac{\log\cN(\cF,L^1(\PP_n),\epsilon\norm{F}_{L^1(\PP_n)})}{n}=0\Big)=1. 
\]
Invoking Lemma \ref{lem:monotone-entropy} and notice that, for any $f,g\in\cF$
\[
\P_n|f\ind(F\leq M)-g\ind(F\leq M)|\leq \P_n|f-g|,
\]
we have 
\[
\cN(\cF_M,L^1(\PP_n),\epsilon\norm{F}_{L^1(\PP_n)}) \leq \cN(\cF,L^1(\PP_n),\epsilon\norm{F}_{L^1(\PP_n)})
\]
so that $\log \cN(\cF_M,L^1(\PP_n),\epsilon\norm{F}_{L^1(\PP_n)})/n$ also goes to 0 almost surely. This implies that the first term on the righthand side converges to 0 almost surely. Reverse Fatou's Lemma then implies, for any fixed $M>0$ and $\epsilon>0$,
\[
\limsup_{n\to\infty}\E\bnorm{\PP_n-\P}_{\cF}\leq \E\Big[\limsup_{n\to\infty}\E_{\epsilon}\bnorm{\PP_n-\P}_{\cF}\Big] \leq 2\epsilon\norm{F}_{L^1(\P)} + 2\P F\ind(F>M).
\]
Now, since $\P F<\infty$, Markov inequality yields that $\ind(F>M)\to 0$ in probability as $M\to\infty$. Then, by dominated convergence theorem,
\[
\lim_{M\to\infty}\P F\ind(F>M)=0.
\]
Letting $\epsilon\to 0$ and $M\to \infty$ then finishes the proof of the first part. 

Almost sure convergence, on the other hand, is established by showing that $\norm{\PP_n-\P}_{\cF}$ is a reverse submartingale with regard to a particular filtration; see \citet[Lemma 2.4.5]{vaart1996empirical}.
\end{proof}

To close this section, let's add a remark that, assuming further that $\cF$ has a constant finite upper bound, the strong $\P$-Glivenko-Cantelli property can be proven using the following Talagrand's inequality \citep{talagrand1996new,bousquet2002concentration} combined with the first Borel-Cantelli lemma.

\begin{theorem}[Talagrand's inequality, Bousquet's version]\label{thm:talagrand-inequ} Assume 
\[
\sup_{f\in\cF}\bnorm{f-\P f}_{L^\infty} \leq M.
\]
It then holds true that, for any $t\geq 0$,
\[
\Pr\Big(\bnorm{\PP_n-\P}_{\cF}\geq \E\bnorm{\PP_n-\P}_{\cF} + \sqrt{2V_nt/n}+\frac{Mt}{3n} \Big)\leq e^{-t},
\]
where
\[
V_n:=2M\cdot \E\bnorm{\PP_n-\P}_{\cF}+\sup_{f\in\cF^2}\P(f-\P f)^2.
\]
\end{theorem}
\begin{proof}
Theorem 3.3.9 in \cite{gine2021mathematical}.
\end{proof}

\section{Donsker bounds}\label{chapter:donsker}

This section establishes bounds on 
\[
\sup_{f\in\cF}\GG_n(f) ~~~{\rm and} ~~~\sup_{d(f,g)<\delta}\GG_n(f-g)
\]
based on the uniform entropy condition:
\begin{align}\label{eq:uniform-entropy-condition}
\int_0^{2}\sup_Q\sqrt{\log2\cN(\cF,L^2(\Q),\epsilon\norm{F}_{L^2(\Q)})}\d\epsilon < \infty,
\end{align}
where the supremum is over all finitely discrete probability measures over $(\cX,\mathcal{A})$.

The following is the master theorem.
\begin{theorem}[Master theorem]\label{thm:master}
Assume that $0\in \cF$ is countable and satisfies \eqref{eq:envelope} and \eqref{eq:uniform-entropy-condition}. Condition\eqref{eq:bounded} then holds true and we further have
\[
\E\sup_{f\in\cF}\Big|\GG_n f \Big|\leq 8\sqrt{6}\norm{F}_{L^2(\P)}\int_0^2\sup_\Q \sqrt{\log2\cN(\cF,L^2(\Q),\epsilon\norm{F}_{L^2(\Q)})}\d\epsilon,
\]
where  the supremum is over all finitely discrete probability measures over $(\cX,\mathcal{A})$.
\end{theorem}
\begin{proof}
We employ the same symmetrization trick as in the proof of Theorem \ref{thm:gc} to transfer the study of $\GG_n(\cF)$ to a Rademacher process. Then, for any $f,g\in\cF$, conditioning on $X_1,\ldots,X_n$,
\[
\norm{\GG_n(f-g)}_{\psi_2}\leq  \cdot \sqrt{\frac{6}{n}\sum_{i=1}^n(f(X_i)-g(X_i))^2}=\sqrt{6}\norm{f-g}_{L^2(\PP_n)}.
\]
Noting that (a) $\psi_2^{-1}(x)=\sqrt{\log(x+1)}\leq \sqrt{\log(2x)}$ whenever $x\geq 1$ and (b) the stochastic processes $\GG_n(\cF)$ is trivially separable since $\cF$ is countable, Theorem \ref{thm:dudley1} implies
\begin{align*}
\E\Big[\sup_{f\in\cF}|\GG_nf| &\mid X_1,\ldots,X_n\Big] \leq 8\sqrt{6}\int_0^{\diam(\cF)}\sqrt{\log 2\cN(\cF,L^2(\PP_n),\epsilon)}\d\epsilon\\
%&=16\sqrt{6}\int_0^{D}\sqrt{\log 2\cN(\cF,L^2(\PP_n),\epsilon)}\d\epsilon\\
&\leq 8\sqrt{6}\int_0^{2\norm{F}_{L^2(\PP_n)}}\sqrt{\log 2\cN(\cF,L^2(\PP_n),\epsilon)}\d\epsilon\\
&=8\sqrt{6}\norm{F}_{L^2(\PP_n)}\int_0^2\sqrt{\log 2\cN(\cF,L^2(\PP_n),\epsilon\norm{F}_{L^2(\PP_n)})}\d\epsilon\\
&\leq 8\sqrt{6}\norm{F}_{L^2(\PP_n)}\int_0^2\sup_\Q\sqrt{\log 2\cN(\cF,L^2(\Q),\epsilon\norm{F}_{L^2(\Q)})}\d\epsilon,
\end{align*}
where in the last step the supremum is over all finitely discrete probability measures, which always include the empirical measure $\PP_n$. 

Lastly, by the law of total expectation,
\begin{align*}
\E\sup_{f\in\cF}|\GG_ng| &= \E\Big[\E\Big[\sup_{f\in\cF}|\GG_ng| \mid X_1,\ldots,X_n\Big] \Big]\\
&\leq 8\sqrt{6}\norm{F}_{L^2(\P)}\int_0^2 \sup_{\Q}\sqrt{\log2\cN(\cF,L^2(\Q),\epsilon\norm{F}_{L^2(\Q)})}\d\epsilon,
\end{align*}
where in the last step we used Jensen's inequality to derive $\E\norm{F}_{L^2(\PP_n)}\leq \norm{F}_{L^2(\P)}$. This completes the proof of the master theorem.
\end{proof}

A corollary of Theorem \ref{thm:master} gives a Donsker-type bound. 

\begin{corollary}\label{cor:donsker}
Assume that $\cF$ is countable and satisfies \eqref{eq:envelope} and \eqref{eq:uniform-entropy-condition}. We then have
\[
\lim_{\delta\to0}\limsup_{n\to\infty}\E\sup_{\norm{f-g}_{L^2(\P)}<\delta}\Big|\GG_n (f-g) \Big| = 0
\]
and $\cF$ is $\P$-Donsker.
\end{corollary}
\begin{proof}
{\bf Step 1.} Let's introduce two function classes
\[
\cF^{\rm diff}=\Big\{f-g: f,g\in\cF\}~~{\rm and}~~\cF_{\delta}:=\Big\{f-g: f,g\in\cF, \norm{f-g}_{L^2(\P)}< \delta\}.
\]
It then holds true that 
\[
\E\Big[\sup_{h\in\cF_{\delta}}|\GG_nh| \mid X_1,\ldots,X_n\Big] \leq 8\sqrt{6}\int_0^{D_\delta}\sqrt{\log 2\cN(\cF_\delta,L^2(\PP_n),\epsilon)}\d\epsilon,
\]
where $D_\delta$ is the diameter of $(\cF_{\delta},L^2(\PP_n))$.

\begin{lemma}\label{lem:diff-covering}
For any measure $\Q$, any function class $\cF$, and any $p\in[1,\infty]$, it holds true that
\[
\cN(\cF^{\rm diff},L^p(\Q),\epsilon)\leq \cN^2(\cF,L^p(\Q),\epsilon/2),
\]
\end{lemma}

{\bf Step 2.}  Using Lemma \ref{lem:diff-covering}, we can continue to write
\begin{align*}
\E\Big[\sup_{h\in\cF_{\delta}}\Big|\GG_nh\Big| \mid X_1,\ldots,X_n\Big] &\leq 8\sqrt{6}\int_0^{D_\delta}\sqrt{\log 2\cN^2(\cF,L^2(\PP_n),\epsilon/2)}\d\epsilon\\
&=16\sqrt{6} \int _0^{D_\delta/2}\sqrt{\log 2\cN^2(\cF,L^2(\PP_n),\epsilon)}\d\epsilon\\
&\leq 32\sqrt{3} \int _0^{D_\delta/2}\sup_\Q\sqrt{\log 2\cN(\cF,L^2(\Q),\epsilon)}\d\epsilon,
\end{align*}
so that
\[
\E\Big[\sup_{h\in\cF_{\delta}}|\GG_nh|\Big] \leq 32\sqrt{3}\cdot \E\Big[ \int _0^{D_\delta/2}\sup_\Q\sqrt{\log 2\cN(\cF,L^2(\Q),\epsilon)}\d\epsilon \Big].
\]

{\bf Step 3.} The aim of this step is to establish that $D_{\delta}\to 0$ in probability as $\delta\to0$, so that the above integral on the righthand side goes to 0.
\begin{lemma}\label{lem:diff-gc}
The function class $\cF_\delta^2:=\{h^2;h\in\cF_\delta\}$ is $\P$-Glivenko-Cantelli.
\end{lemma}
\begin{proof}[Proof of Lemma \ref{lem:diff-gc}] 
Let's check the conditions of Theorem \ref{thm:gc}. We have, for any $h^2=(f-g)^2\in\cF_\delta^2$, we have $h^2\leq (2F)^2$ and $\P(2F)^2<\infty$ since $F$ is $\P$-square-integrable. The first condition is thus checked.

Regarding the second condition, notice that for any $h_1^2,h_2^2\in\cF_{\delta}^2$,
\begin{align*}
\PP_n|h_1^2-h_2^2|=&\PP_n\Big\{|h_1-h_2|\cdot |h_1+h_2|\Big\}\leq \norm{h_1-h_2}_{L^2(\P_n)}\cdot 4\norm{F}_{L^2(\P_n)}.
\end{align*}
Invoking Lemma \ref{lem:monotone-entropy}, it then holds true that
\[
\cN(\cF_{\delta}^2,L^1(\PP_n),\epsilon\norm{4F^2}_{L^1(\PP_n)})\leq \cN(\cF_{\delta},4\norm{F}_{L^2(\P_n)}L^2(\PP_n),\epsilon\norm{4F^2}_{L^1(\PP_n)}).
\]
Noticing that, since $\norm{F^2}_{L^1(\PP_n)}=(\norm{F}_{L^2(\PP_n)})^2$, using Lemma \ref{lem:diff-covering},
 \begin{align*}
 \cN(\cF_{\delta},4\norm{F}_{L^2(\P_n)}L^2(\PP_n),\epsilon\norm{4F^2}_{L^1(\PP_n)})&= \cN(\cF_{\delta},L^2(\PP_n),\epsilon\norm{F}_{L^2(\PP_n)})\\
 &\leq  \cN^2(\cF,L^2(\PP_n),\epsilon\norm{F}_{L^2(\PP_n)}/2)\\
&\leq \sup_\Q\cN^2(\cF,L^2(\Q),\epsilon\norm{F}_{L^2(\Q)}/2).
\end{align*}
Now, by Condition \eqref{eq:uniform-entropy-condition}, for any $\epsilon>0$, there exists a constant $K=K(\epsilon)$ only depending on $\epsilon$ such that 
\[
\sup_\Q\log2\cN(\cF,L^2(\Q),\epsilon\norm{F}_{L^2(\Q)})< K(\epsilon).
\] 
Accordingly, for any fixed $\epsilon>0$, 
\[
\sup_\Q\log\cN^2(\cF,L^2(\Q),\epsilon\norm{F}_{L^2(\Q)}/2)
\]
is bounded, and thus, Theorem \ref{thm:gc} implies the result.
\end{proof}

Get back to the main proof. Now, by construction, $\sup_{h\in\cF_\delta}\P h^2<\delta^2$. Lemma \ref{lem:diff-gc} further implies
\[
\E\sup_{h\in\cF_\delta}(\PP_n-\P)h^2 \to 0
\]
so that
\begin{align*}
D_{\delta}^2=\sup_{h_1,h_2\in\cF_\delta}\PP_n(h_1&-h_2)^2\leq 4 \sup_{h\in\cF_\delta}\PP_nh^2\\
 &\leq 4\sup_{h\in\cF_\delta}|(\PP_n-\P)h^2|+4\sup_{h\in\cF_\delta}\P h^2 \leq 4\delta^2+o_{\P}(1).
\end{align*}
Accordingly, by dominated convergence theorem,
\[
\lim_{\delta\to0}\limsup_{n\to\infty}\E\sup_{h\in\cF_\delta}|\GG_nh|=0.
\]

{\bf Step 4.} In order to prove that $\cF$ is $\P$-Donsker, it remains to show $(\cF,L^2(\P))$ is totally bounded. For any $f,g\in\cF$,
\[
\P(f-g)^2\leq \PP_n(f-g)^2+|(\PP_n-\P)(f-g)^2|.
\]
The following two observations are true.

(a) By \eqref{eq:uniform-entropy-condition}, $\cN(\cF,L^2(\PP_n),\epsilon)$ is universally bounded. 

(b) In addition, notice that Lemma \ref{lem:diff-gc} actually showed that there exists a sequence of finitely discrete measures $\PP_n$ such that
\[
\E\norm{(\PP_n-\P)h^2}_{\cF^{\rm diff}} \to 0~~~{\rm as}~~n\to\infty.
\]
Fix an $\epsilon>0$. There then exists some $n$ such that $\E\norm{(\PP_n-\P)h^2}_{\cF^{\rm diff}}\leq \epsilon^2$. 

Wrapping both up, we have, for any fixed $\epsilon$, for all $n$ large enough, 
\[
\cN(\cF,L^2(\P),\epsilon)\leq \cN(\cF,L^2(\PP_n),\sqrt{2}\epsilon) \leq \sup_\Q\cN(\cF,L^2(\Q),\sqrt{2}\epsilon)<\infty,
\]
so that $(\cF,L^2(\P))$ is totally bounded. This completes the proof.
\end{proof}

\begin{exercise}
Please give a proof of Lemma \ref{lem:diff-covering}.
\end{exercise}

\section{VC arguments}\label{chap:vc}

This section aims to give a bound on 
\begin{align}\label{eq:EP-uniform-entropy}
\int_0^{2}\sup_\Q\sqrt{\log 2\cN(\cF,L_2(\Q),\epsilon\norm{F}_{L^2(\Q)})}\d \epsilon
\end{align}
using the VC argument.

\subsection{Basic properties}

Consider a class of sets $\cC:=\{C\in \cC, C\subset \cX\}$ and any sample $x_1^n=\{x_1,\ldots,x_n\}\subset \cX$ of size $n$. We define $\cC$'s growth function as follows. 
\begin{definition}
The growth function $\Pi_{\cC}(n)$ is defined as 
\[
\Pi_{\cC}(n):=\max_{x_1^n\subset \cX}|x_1^n \cap \cC|.
\]
\end{definition}

\begin{definition}[shattering]
$\cC$ is said to {\it shatter} a class $T\subset \cX$ if $|T\cap \cC|=2^{|T|}$.
\end{definition}

\begin{definition}[VC dimension]
The {\it VC dimension} (or called {\it VC index}) of $\cC$, written as $\nu(\cC)$, is the largest $n$ such that there exists a set $T\subset \cX$, $|T|=n$, and $\cC$ shatters it. 
\end{definition}

When the quantity $\nu(\cC)$ is finite, the class of sets $\cC$ is said to be a {\it VC-class}. 

\begin{example}\label{example:1}
Consider the class $\cC_{\rm left}:=\{(-\infty,a]; a\in\reals\}$. We have $\nu(\cC_{\rm left})=1$. On the other hand, it is easy to derive that $\Pi_{\cC_{\rm left}}(n)\leq n+1=(n+1)^{\nu(\cC_{\rm left})}$.
\end{example}

\begin{example}
Consider the class $\cC_{\rm two}:=\{(b,a]; a,b\in\reals\}$. We have $\nu(\cC_{\rm two})=2$. On the other hand, it is easy to derive that $\Pi_{\cC_{\rm left}}(n)\leq (n+1)^2=(n+1)^{\nu(\cC_{\rm two})}$.
\end{example}

The following result shows that, for any VC class, the cardinality of $x_1^n\cap \cC$ can grow at most polynomially in $n$. This is named the Sauer's Lemma. 

\begin{lemma}[Vapnik-Chervonenkis, Sauer, and Shelah]\label{prop:VC} Consider a set class $\cC$ with $\nu(\cC)<\infty$. Then, for any collection of points $x_1^n=(x_1,\ldots,x_n)$, we have
\[
\Big|x_1^n\cap \cC\Big|\leq \sum_{i=0}^{\nu(\cC)}\binom{n}{i} \leq \min\Big\{ (n+1)^{\nu(\cC)}, \Big(\frac{en}{\nu(\cC)} \Big)^{\nu(\cC)}\Big\}.
\]
\end{lemma}

\begin{proof} The first inequality could be established through the following more general inequality. The second is a simple algebra and is left to the readers. 
\begin{lemma}\label{lem:sauer} Let $A$ be a finite set and let $\cU$ be a class of subsets of $A$. Then
\[
|\cU|\leq \Big| \Big\{ B\subset A ~|~B\text{ is shattered by }\cU   \Big\} \Big|.
\]
\end{lemma}
To see how this lemma immediately proves Sauer's lemma, note that $B\subset A$ is shattered by $\cC$ meaning that $|B|\leq \nu(\cC)$. Consequently, if we let $A=x_1^n$ and set $\cU=\cC\cap A$, then Lemma \ref{lem:sauer} yields
\[
|x_1^n\cap \cC|=|\cC\cap A| \leq \Big| \Big\{ B\subset A ~|~B\leq \nu(\cC)   \Big\} \Big|\leq \sum_{i=0}^{\nu(\cC)}\binom{n}{i}.
\]
It remains to prove Lemma \ref{lem:sauer}. For a given $x\in A$, let's define an operator on sets $U\in\cU$ via
\begin{align*}
T_x(U)= \begin{cases}
U \setminus \{x\}~~~{\rm if~}x\in U~{\rm and}~U\setminus \{x\}\not\in\cU\\
U~~~{\rm otherwise.}
\end{cases}
\end{align*}
We let $T_x(\cU)$ be the new class of sets defined by applying $T_x$ to each member of $\cU$, namely, $T_x(\cU):=\Big\{T_x(U) ~|~U\in\cU\Big\}$.

(1) We first show that $T_x$ is a one-to-one map between $\cU$ and $T_x(\cU)$, and hence $|\cU|=|T_x(\cU)|$. This is equivalent to proving that, for any sets $U,U'\in\cU$ such that $T_x(U)=T_x(U')$, we must have $U=U'$ (the reverse is simple). This is by the following case-by-case investigation:
\begin{itemize}
\item Case 1: $x\not\in U$ and $x\not\in U'$. We then have $U=T_x(U)=T_x(U')=U'$.
\item Case 2: $x\not\in U$ and $x\in U'$. In this case, we have $U=T_x(U)=T_x(U')$, so  that $x\in U'$ but $x\not\in T_x(U')$. But this means that $T_x(U')=U'\setminus \{x\}\not\in \cU$, which contradicts the fact that $T_x(U')=U\in\cU$. By symmetry, the case $x\in U$ and $x\not\in U'$ is identical.
\item Case 3: $x\in U\cap U'$. If both $U\setminus \{x\}$ and $U'\setminus \{x\}$ belong to $\cU$, then $U=T_x(U)=T_x(U')=U'$. If neither $U\setminus \{x\}$ nor $U\setminus \{x\}$ belongs to $\cU$, then we also have $U\setminus \{x\}=U'\setminus \{x\}$, yielding $U=U'$. Lastly, if $U\setminus \{x\} \not\in \cU$ but $U'\setminus \{x\}\in\cU$, then $T_x(U)=U\setminus \{x\}\not\in\cU$ but $T_x(U')=U'\in\cU$, which is a contradiction.
\end{itemize}

(2) We secondly show that if $T_x(\cU)$ shatters a set $B$, then so does $\cU$. If $x\not\in B$, then both $\cU$ and $T_x(\cU)$ pick out the same set of subsets of $B$, and the claim must be true. Otherwise, if $x\in B$, since $T_x(\cU)$ shatters $B$, for any subset $B'\subset B\setminus \{x\}$, there is a subset $T\in T_x(\cU)$ such that $T\cap B=B'\cup \{x\}$. Since $T=T_x(\cU)$ for some subset $U\in\cU$ and $x\in T$, we conclude that both $U$ and $U\setminus \{x\}$ must belong to $\cU$, so that $\cU$ also shatters $B$.

(3) We now conclude the lemma. Define the weight function $\omega(\cU)=\sum_{U\in\cU}|U|$. Note that applying a transformation $T_x$ can only reduce this weight function: $\omega(T_x(\cU))\leq \omega(\cU)$. Consequently, by applying the transformations $\{T_x\}$ to $\cU$ repeatedly, we can obtain a new class of sets $\cU'$ such that $|\cU|=|\cU'|$ and the weight $\omega(\cU')$ is minimal. Then, for any $U\in\cU'$ and any $x\in U$, we have $U\setminus \{x\}\in \cU'$ (otherwise, we have $\omega(T_x(\cU'))<\omega(\cU')$, contradicting minimality). Therefore, the set class $\cU'$ shatters any one of its elements. Noting that $\cU$ shatters at least as many subsets as $\cU'$, and $|\cU|=|\cU'|$, the proof is complete.
\end{proof}

\subsection{VC stability}

The property of having finite VC-dimension is preserved under a number of basic operations, as summarized in the following (refer to, for example, Lemma 9.7 in \cite{kosorok2008introduction}, Proposition 3.6.7 in \cite{gine2021mathematical}, and Theorem 13.5 in \cite{devroye2013probabilistic}). %They are also known as stability results in David Pollard's sense. 

\begin{theorem}[Stability] \label{thm:s1} Let $\cC$ and $\cD$ be VC-classes on $\cX$ with growth functions $\Pi_{\cC}(n)$ and $\Pi_{\cD}(n)$ and VC dimensions $V_{\cC}$ and $V_{\cD}$. Let $\cE$ be VC-class on $\cW$ with growth function $\Pi_{\cE}(n)$ and VC dimension $V_{\cE}$. We then have
\begin{itemize}
\item[(1)] $\cC^C$ has VC-dimension $V_{\cC}$ and growth function $\Pi_{\cC}(n)$;
\item[(2)] $\cC\cap \cD=\{C\cap D; C\in\cC, D\in\cD\}$ has growth function $\leq \Pi_{\cC}(n)\Pi_{\cD}(n)$;
\item[(3)] $\cC\cup \cD=\{C\cup D; C\in\cC, D\in\cD\}$ has growth function $\leq \Pi_{\cC}(n)\Pi_{\cD}(n)$;
\item[(4)] $\cD\times \cE$ has growth function $\leq \Pi_{\cC}(n)\Pi_{\cD}(n)$;
\item[(5)] $\phi(\cC)$ has VC-dimension $V_{\cC}$ if $\phi$ is one-to-one;
\item[(6)] $\psi^{-1}(\cC)$ has VC-dimension $\leq V_{\cC}$.
\end{itemize}
\end{theorem}

\begin{remark}
When you have an upper bound on the growth function of a given class of sets, by the definition of VC dimension, you also obtain an upper bound on that class by noticing that $\nu(\cC)$ is the largest $n$ such that $2^{n} = \Pi_{\cC}(n)$, and for any $n\in\mathbb{N}$, $\Pi_{\cC}(n)\leq (n+1)^{\nu(\cC)}$.
\end{remark}

Theorem \ref{thm:s1} is a nice result, but we still need something to begin with. Regarding any given real-valued function $g:\cX\to \reals$, it defines a ``classification" function by the set 
\[
S_g:=\Big\{x\in\cX~|~g(x)\leq 0\Big\}. 
\]
In this way, we can associate the function class $\cG$ with the collection of subsets $\cS(\cG):=\{S_g; g\in\cG\}$. 

In case the function class $\cG$ is a vector space, the following result upper bounds the VC-dimension of the associated ``classification" class $\cS(\cG)$.

\begin{proposition}\label{prop:linear} Let $\cG$ be a vector space of functions $g:\reals^d\to\reals$ with dimension ${\rm dim}(\cG)<\infty$. Then the class $\cS(\cG)$ has VC-dimension at most ${\rm dim}(\cG)$.
\end{proposition}
\begin{proof}
By definition of VC-dimension, we need to show that no collection of $n={\rm dim}(\cG)+1$ points in $\reals^d$ can be shattered by $\cS(\cG)$. To this end, fix a collection $x_1^n$ of $n$ points in $\reals^d$, and consider the following sets:
\[
\Big\{(g(x_1),\ldots,g(x_n))^{\top}, g\in\cG\Big\}. 
\]
We then have, the range of the above sets is a linear subspace of $\reals^n$ with dimension at most ${\rm dim}(\cG)=n-1<n$. Therefore, there must exist a non-zero vector $a\in\reals^n$ such that $\langle a, (g(x_1),\ldots,g(x_n))^{\top}\rangle=0$ for all $g\in\cG$. We may assume, W.L.O.G., that at least one entry $a_i$ of $a$ is positive, and then write
\[
\sum_{\{i|a_i>0\}} a_ig(x_i)=\sum_{\{i;a_i<0\}}(-a_i)g(x_i)~~{\rm for~all~}g\in\cG.
\]

Now suppose that there exists some $g\in\cG$ such that the associate classification class $S_g=\{x\in\reals^d;g(x)\leq 0\}$ includes only the subset $\{x_i:a_i\leq 0\}$. For such a function $g$, the LHS of the above equation would be strictly positive, while the RHS would be non-positive, which is a contradiction. We thus proved that $\cS(\cG)$ cannot shatter $x_1^n$, and finish the proof.
\end{proof}

\begin{example}[Linear functions in $\reals^d$] For a pair $(a,b)\in\reals^d\times\reals$, consider the function class $f_{a,b}(x)=a^{\top}x+b$, and consider the family $\cL^d=\{f_{a,b}~|~(a,b)\in\reals^d\times\reals\}$. The associated classification is the collection of all half-spaces of the form $H_{a,b}:=\{ x\in\reals^d ~|~a^{\top}x+b\leq 0  \}$. Since the family $\cL^d$ forms a vector space of dimension $d+1$, we have $\cS(\cL^d)$ has VC-dimension at most $d+1$.
\end{example}

\begin{example}[Sphere in $\reals^d$] Consider the sphere $S_{a,b}:=\{x\in\reals^d; \norm{x-a}_2\leq b\}$ where $(a,b)\in\reals^d\times\reals^+$. Let $\mathbb{S}^d$ denote the collection of all such spheres. If we define the function
\[
f_{a,b}(x):=\norm{x}_2^2-2\sum_{j=1}^da_jx_j+\norm{a}_2^2-b^2
\]
Then we have $\cS_{a,b}=\{x\in\reals^d; f_{a,b}(x)\leq 0\}$, so that the sphere is a classification set of the function $f_{a,b}$. In order to leverage Proposition \ref{prop:linear}, we define a feature map $\phi:\reals^d\to\reals^{d+1}$ via
\[
\phi(x)=(x_1,\ldots,x_d,1), 
\]
and then consider the functions of the form
\[
g_c(x):=c^{\top}\phi(x)+\norm{x}_2^2,~~~{\rm where}~x\in\reals^{d+1}.
\]
The family of functions $\{g_c;c\in\reals^{d+1}\}$ is a vector space of dimension $d+2$, and it contains the functions $f_{a,b}$. We thus conclude $\nu(\mathbb{S}^d)\leq d+2$.
\end{example}

\begin{remark}
The VC-dimension should never be confused with the degree of freedom (or simply the number of parameters) in statistics. in fact, if you have a nonlinear classification function class, it is very possible that you will have a much higher VC-dimension than the number of parameters in your function. As an extreme case, the function class $\{\ind(\sin ax>0); a\in\reals\}$ can have infinite VC-dimension.
\end{remark}

\subsection{VC subgraph classes of functions}

%Definition 3.6.8 in \cite{gine2021mathematical}

\begin{definition}
The subgraph of a real function $f$ on $\cX$ is the set
\[
G_f := \Big\{ (x,t): x\in\cX, t\in\reals, t\leq f(x)   \Big\}.
\] 
A class of functions $\cF$ is VC subgraph of index (VC dimension) $\nu$ if the class of sets $\cC:=\{G_f; f\in\cF\}$ is VC of index $\nu$.
\end{definition}

\begin{exercise}\label{example:indicator}
Suppose that $\cC$ is a VC class of index $\nu(\cC)$. Show that the class of functions $\cF:=\{\ind_C; C\in\cC\}$ is VC subgraph of index $\nu(\cC)$.
\end{exercise}

%Lemma 2.6.15 in \cite{vaart1996empirical}

\begin{example}Any finite-dimensional vector space $\cF$ of measurable functions $f:\cX\to\reals$ is VC-subgraph of index $\leq {\rm dim}(\cF)+1$.
\end{example}
\begin{proof}
The proof resembles that of Proposition \ref{prop:linear}. Take any collection of $n={\rm dim}(\cF)+2$ points $(x_1,t_1),\ldots,(x_n,t_n)$ in $\cX\times\reals$. Since $\cF$ is a vector space, we have 
\[
\{(f(x_1)-t_1,\ldots,f(x_n)-t_n)^{\top}, f\in\cF\}
\]
are contained in a $({\dim \cF}+1)=(n-1)$-dimensional subspace of $\reals^n$. Hence, there exists a nondegenerate vector $a\ne 0$ such that
\[
\sum_{a_i>0}a_i\big(f(x_i)-t_i\big) = \sum_{a_i<0}(-a_i)(f(x_i)-t_i),~~{\rm for~every~}f\in\cF,
\]
where by default the sum over an empty set is set to be 0. WLOG, we pick out an $a$ such that there exists at least one positive entry. For this vector, the set $\{(x_i,t_i):a_i>0\}$ cannot be of the form $\{(x_i,t_i):t_i<f(x_i)\}$, since if then the LHS of the equation would be positive, and the RHS will be nonpositive.  This concludes that the $\cF$ is VC-subgraph of index $\leq {\dim \cF}+1$.
\end{proof}

%\begin{example}[Lemma 2.6.16 in VW1996 and Proposition 3.6.12 in GN2015] Let $f$ be a function of bounded $p$-variation, $p\geq 1$. Then the collection $\cF$ of translations and dilations:
%\[
%\cF:=\{f(tx-s):t>0,s\in\reals\}
%\]
%is of VC index at most 3.
%\end{example}

We are now ready to state the main theorem in this chapter, which is due to \cite{dudley1978central} and \cite{pollard1982central}. %Theorem 3.6.9 in GN2015

\begin{theorem}[Dudley-Pollard Universality Theorem]\label{thm:Dudley-Pollard} Let $\cF$ be a non-empty VC subgraph class of index $\nu$, and have an envelop $F\in L^p(\Omega, \cA, \Q)$ for some $1\leq p<\infty$. Set 
\[
m_{v,w}:=\max\Big\{ m\in\mathbb{N}: \log m\geq m^{1/\nu-1/w}  \Big\}
\]
for some $w>\nu$. We then have
\[
\cD(\cF,L^p(\Q), \epsilon \norm{F}_{p,\Q}) \leq m_{v,w} \vee \Big[ 2^{w/\nu}\Big( \frac{2^{p+1}}{\epsilon^p} \Big)^w  \Big].
\] 
\end{theorem}
\begin{proof}
The proof uses probabilistic method tracing back to Paul Erdos and many other mathematicians who worked on number theory via probabilistic construction techniques. We omit $\Q$ in the norm when no confusion is made.

Let $f_1,\ldots, f_m$ be a maximal collection of functions in $\cF$ satisfying 
\[
Q|f_i-f_j|^p >\epsilon^pQF^p,~~{\rm for}~i\ne j,
\]
so that $m=D(\cF,L^p(Q),\epsilon\norm{F}_p)$. For some $k$ to be specified later, let $\{(x_i,t_i);i\in[k]\}$ be i.i.d. random vectors with law
\[
{\rm Pr}\Big\{(x,t)\in A\times [a,b] \Big\} = \frac{\int_A \lambda[(-F(x))\vee a, F(x)\wedge b]F^{p-1}(x)dQ(x)}{2QF^p}
\]
for $A\subset \cX$, real numbers $a<b$, and Lebesgue measure $\lambda$. In other words, $x_i$ is chosen according to the law $P_F(A)=Q(\ind_AF^p)/QF^p$, and given $x_i$, $t_i$ is chosen uniformly on $[-F(x_i), F(x_i)]$.

The probability that at least two graphs have the same intersection with the sample $\{(x_i,t_i), i\in[k]\}$ is at most 
\begin{align*}
&{m \choose 2}\max_{i\ne j}{\rm Pr}(C_i~{\rm and}~C_j~\text{have the same intersection with the sample})\\
=& {m \choose 2}\max_{i\ne j}\prod_{r=1}^k {\rm Pr}\{ (x_r,t_r)\not\in C_i \Delta C_j  \}\\
=& {m \choose 2}\max_{i\ne j}\prod_{r=1}^k \Big[ 1- {\rm Pr}\Big\{ (x_r,t_r)\in C_i \Delta C_j\Big\}  \Big]\\
=& {m \choose 2}\max_{i\ne j}\prod_{r=1}^k \Big[ 1- {\rm Pr}\Big\{ (x_r,t_r): t_r~{\rm is~between}~f_i(s_r), f_j(s_r) \Big\}  \Big]\\
=& {m \choose 2} \max_{i\ne j}\Big[1-\frac{1}{\norm{F}_p^p} \int\frac{|f_i-f_j|}{2F}F^p dQ  \Big]^k\\
\leq&  {m \choose 2} \max_{i\ne j}\Big[1-\frac{1}{\norm{F}_p^p} \int\frac{|f_i-f_j|^p}{(2F)^p}F^p dQ  \Big]^k\\
\leq& {m \choose 2} \Big[ 1-\frac{\epsilon^p}{2^p} \Big]^k\\
\leq& {m \choose 2}\exp(-\epsilon^pk/2^p),
\end{align*}
where in the last equation we use $1-x\leq \exp(-x)$. 

Let $k$ be such that this probability is less than 1. Then there exists a set of $k$ elements such that graphs $C_i\in\cC,~1\leq i\leq m$, intersect different subsets of this set, which implies that  $\prod_{\cC}(k)\geq m$. On  the other hand, the smallest $k$ such that ${m \choose 2}\exp(-\epsilon^pk/2^p)<1$ satisfies $k\leq (2^{p+1}/\epsilon^p)\log m$. Then, by Sauer's Lemma, we have
\[
m \leq 2k^\nu \leq 2\Big( \frac{2^{p+1}}{\epsilon^p}\log m \Big)^\nu.
\]
Some algebra then gives the desired bound.
\end{proof}

\begin{example} Using the above corollary, it is immediate to prove that
\[
\E\sqrt{n}\norm{\PP_n-P}_{\cG} = O(1)
\]
by noticing that $\cG$ is a VC-subgraph of index 1 and $\int_0^2\sqrt{\log (A/\epsilon)}d\epsilon  <\infty$.
\end{example}

We close this section with the VC-subgraph stability result, which is left for the students to verify. %Lemma 2.6.18 in VW1996
 
\begin{lemma}[VC-subgraph stability] \label{lem:vc-subgraph}
Let $\cF$ and $\cG$ be VC-subgraph classes of functions on a set $\cX$ and $g:\cX\to\reals$, $\phi:\reals\to\reals$, and $\psi:\cZ\to \cX$ fixed functions. Then
\begin{itemize}
\item[(i)] $\cF\wedge \cG:=\{f\wedge g: f\in\cF, g\in\cG\}$ is VC-subgraph;
\item[(ii)] $\cF \vee \cG$ is VC-subgraph;
\item[(iii)] $\{\cF>0\}:=\{\{f>0\}:f\in\cF\}$ is VC;
\item[(iv)] $-\cF$ is VC-subgraph;
\item[(v)] $\cF+g:=\{f+g:f\in\cF\}$ is VC-subgraph;
\item[(vi)] $\cF\cdot g:=\{fg:f\in\cF\}$ is VC-subgraph;
\item[(vii)] $\cF\circ \psi:=\{f(\psi):f\in\cF\}$ is VC-subgraph;
\item[(viii)] $\phi\circ\cF$ is VC-subgraph for monotone $\phi$.
\end{itemize}
\end{lemma}

\begin{exercise}
Please prove Lemma \ref{lem:vc-subgraph}.
\end{exercise}

\subsection{VC-hull and VC-major}

This section briefly introduces the VC-hull and VC-major classes, without touching too much detail due to the time limit. VC-hull and VC-major classes generalize the VC-subgraph (sometimes just referred to as VC) classes of functions.

%Definition 3.6.13 in GN2015

\begin{definition}[Convex hull] Given a class of functions $\cF$, ${\rm co}(\cF)$ is defined as the convex hull of $\cF$, that is
\[
{\rm co}(\cF)= \Big\{ \sum_{f\in\cF}\lambda_ff: f\in \cF, \sum_{f}\lambda_f=1, \lambda_f>0, \lambda_f\ne 0~{\rm only~for~finitely~many~}f  \Big\},
\]
and $\bar{\rm co}(\cF)$ is defined as the pointwise sequential closure of ${\rm co}(\cF)$, that is, $f\in\bar{\rm co}(\cF)$ if there exist $f_n\in {\rm co}(\cF)$ such  that $f_n(x)\to f(x)$ for all $x\in\cX$ as $n\to\infty$. 
\end{definition}

%Definition 3.6.13 in GN2015

\begin{definition}[VC-hull]
If the class $\cF$ is VC-subgraph, then we say that $\bar{\rm co}(\cF)$ is a VC-hull class of functions. 
\end{definition}

\begin{example}\label{example:4.8} Let $\cF$ be the class of all monotone nondecreasing functions $f:\reals\to [0,1]$. Then $\cF\in\bar{\rm co}(\cG)$, where $\cG:=\{\ind_{(x,\infty)}, \ind_{[x,\infty)}: x\in\reals\}$.
\end{example}
\begin{proof}
For any $f:\reals \to [0,1]$, we could define  
\[
f_n=\frac{1}{n}\sum_{i=1}^{n-1}\ind_{\{f>i/n\}}=\sum_{j=0}^{n-1}\frac{j}{n}\ind_{\{j/n<f\leq(j+1)/n\}}.
\]
It is immediate that 
\[
\sup_{x\in\reals}|f_n(x)-f(x)|\leq 1/n.
\]
On the other hand, since $f$ is monotone nondecreasing (with possible jumps), we have the sets $\{f>i/n\}$ are all half lines, rendering that $\ind_{\{f>i/n\}}\in\cG$.
\end{proof}

\begin{definition}[VC-major] $\cF$ is a VC-major class if the collection of set $\{x: f(x)\geq t\}_{t\in\reals, f\in\cF}$ is a VC-class.
\end{definition}

%Lemma 2.6.13 in VW1996
\begin{lemma}\label{lem:4.8} A bounded VC-major class is a scalar multiple of a VC-hull class.
\end{lemma}
\begin{proof}
A given function $f:\cX\to [0,1]$ is the uniform limit of the sequence
\[
f_m=\sum_{i=1}^m\frac{1}{m}\ind(f>i/m).
\]
Thus, a given class of functions $f:\cX\to[0,1]$ is contained in the pointwise sequential closure of the convex hull of $\{\ind(f>t): f\in\cF, t\in\reals\}$, which is VC-subgraph using  Example \ref{example:indicator} and the definition of VC-major class. This then finishes the proof.
\end{proof}

The last result in this chapter, which we shall not prove, is the Universality Theorem on VC-hull, and hence also on bounded VC-major classes.

%Theorem 3.6.17 in GN2015

\begin{theorem}[Universality Theorem on VC-hull] Let $\Q$ be a probability measure on $(\cX,\sigma(\cX))$, and let $\cF$ be a collection of measurable functions with envelope $F\in L_2(\Q)$ such that
\[
N(\cF,L_2(\Q),\epsilon\norm{F}_{L_2(\Q)})\leq C\epsilon^{-w},~~{\rm for}~0<\epsilon\leq 1.
\]
Then there exists a constant $K$ depending only on $C$ and $w$ such that
\[
\log N(\bar{\rm co}(\cF),L_2(\Q),\epsilon\norm{F}_{L_2(\Q)}) \leq K\epsilon^{-2w/(w+1)},~~{\rm for}~0<\epsilon\leq 1.
\]
\end{theorem}
\begin{proof}
See Theorem 3.6.17 in \cite{gine2021mathematical}.
\end{proof}

\section{Bracketing argument}\label{chap:bracket}

The bracketing entropy argument gives a second bound of
\[
\int_0^{2}\sup_\Q\sqrt{\log 2\cN(\cF,L_2(\Q),\epsilon\norm{F}_{L^2(\Q)})}\d \epsilon.
\]

\begin{definition}[Bracketing brackets and numbers] Consider a normed function class $\cF$ equipped with a norm $\norm{\cdot}$. An $\epsilon$-bracket with upper and lower bounds $l, u$ (need not belong to $\cF$) contains all functions $f\in\cF$ such that
\[
l\leq f\leq u~~~{\rm and}~~~\norm{u-l}\leq\epsilon.
\]
The bracketing number, denoted by $\cN_{[]}(\cF,\norm{\cdot},\epsilon)$, stands for the minimum number of $\epsilon$-brackets to cover $\cF$.
\end{definition}

\begin{lemma}\label{lem:bracket-1}
For any normed function space satisfying that 
\[
|f| \leq |g| \text{ implies } \norm{f}\leq \norm{g},
\]
it holds true that
\[
\cN(\cF,\norm{\cdot},\epsilon)\leq \cN_{[]}(\cF,\norm{\cdot},2\epsilon)
\]
\end{lemma}

\begin{exercise}
Please give a proof of Lemma \ref{lem:bracket-1}.
\end{exercise}

%\subsection{Smooth functions}

For any real function $f:\cX(\subset\reals^d)\to\reals$ and integer vector $\bk=(k_1,\ldots,k_d)^\top$, let $D^{\bk}$ be the classic differential operator so that
\[
D^{\bk}f(\bx)=\frac{\partial^{k}}{\partial x_1\cdots\partial x_k}f(\bx),~~~\text{ with }\bx=(x_1,\ldots,x_d)^\top \text{ and }k=\sum_{j=1}^dk_j.
\]
Considering the smoothness parameter $\alpha>0$, introduce
\[
\norm{f}_{\alpha} = \max_{k\leq  \underline\alpha}\sup_{\bx\in\cX}|D^{\bm k}f| + \max_{k=\underline\alpha}\sup_{\bx,\by\in{\rm int}(\cX)}\frac{|D^{\bm k}f(\bx)-D^{\bm k}f(\by)|}{\norm{\bx-\by}^{\alpha-\underline\alpha}},
\]
where $\underline\alpha$ is defined to the greatest integer strictly smaller than $\alpha$, ${\rm int}(\cdot)$ represents the interior of the input set, and we take the convention that $0/0=0$. 

Lastly define 
\[
\cC_M^{\alpha}(\cX) := \Big\{f:\cX\to\reals: \norm{f}_{\alpha}\leq M  \Big\}.
\]

\begin{proposition}\label{prop:smooth} Assume $\cX$ to be a bounded convex subset of $\reals^d$ with a nonempty interior. It then holds true that
\[
\log \cN(\cC_1^{\alpha}(\cX),\norm{\cdot}_{\infty},\epsilon)\leq K \Big(\frac{1}{\epsilon} \Big)^{d/\alpha},
\]
and
\[
\sup_\Q\log\cN_{[]}(\cX,L_p(\Q),\epsilon)\leq K \Big(\frac{1}{\epsilon} \Big)^{d/\alpha}, ~~~\text{ for all }p\geq 1, \epsilon>0,
\]
where $K$ is a constant only depending on $\alpha, d, $ and $\cX$, and the supremum is taken over all probability measures on $\reals^d$.
\end{proposition}
\begin{proof}
Theorem 2.7.1 and Corollary 2.7.2 in \cite{vaart1996empirical}.
\end{proof}

%Lemma 4.3.9 in Talagrand

%\subsection{Other examples}

\begin{proposition} The class of functions
\[
\cF^{\rm BV} = \Big\{f:\reals\to[-1,1], f \text{ has bounded variation }  \Big\}
\]
satisfies
\[
\sup_Q\log\cN_{[]}(\cF^{\rm mon},L_p(Q),\epsilon)\leq K\Big(\frac{1}{\epsilon} \Big),~~~\text{ for all }p\geq 1, \epsilon>0,
\]
where $K$ only depends on $p$ and the variation bound and the supremum is over all probability measures on $\reals$.
\end{proposition}
\begin{proof}
Example 19.11 in \cite{van2000asymptotic}.
\end{proof}

\begin{theorem}[Parametric class]\label{thm:entropy-parametric}  Let $\cF:=\{f_{\theta},\theta\in\Theta\}$ be a parametrized class of functions indexed by a parameter set $\Theta$. Let $\cF$ and $\Theta$ be coupled with norms $\|\cdot\|_{\cF}$ and $\|\cdot\|_{\Theta}$, respectively.  
\begin{enumerate}[label=(\roman*)]
\item Assuming the existence of a universal constant $L>0$ such that 
\[
\|f_{\theta}(\cdot)-f_{\theta'}(\cdot)\|_{\cF}\leq L\|\theta-\theta'\|_{\Theta},
\]
we have
\[
\cN(\cF,\epsilon,\|\cdot\|_{\cF})\leq \cN(\Theta,\epsilon/L,\|\cdot\|_{\Theta}).
\]
\item Assume the existence of a function $F$ such that $\norm{F}_{\cF}<\infty$ and 
\[
|f_{\theta}(x)-f_{\theta'}(x)|\leq F(x)\norm{\theta-\theta'}_{\Theta},~~~\text{for all }x\in\cX.
\]
Then, 
\[
\cN_{[]}(\cF,\norm{\cdot}_{\cF},2\epsilon\norm{F}_{\cF})\leq \cN(\Theta, \norm{\cdot}_{\Theta},\epsilon).
\]
\end{enumerate}
\end{theorem}
\begin{proof}
The first assertion is by definition. For proving the second assertion, let $\theta_1,\ldots,\theta_m$ be the center of $m$ $\epsilon$-radius balls that cover $\Theta$. It is then true that 
\[
\Big[f_{\theta_i}-\epsilon F, f_{\theta_i}+\epsilon F\Big],~~i=1,\ldots,m
\]
form $m$ brackets that also cover $\cF$. These brackets have size
\[
\norm{f_{\theta_i}+\epsilon F-(f_{\theta_i}-\epsilon F)}_{\cF}=2\epsilon\norm{F}_{\cF},
\]
which proves the assertion.
\end{proof}

\section{Notes}

{\bf Chapter \ref{sec:4history}.} The contribution of Glivenko, Cantelli, and Kolmogorov to the development of uniform law of large numbers and central limit theorems concerning empirical CDF and CDF was mentioned in the main text. Donsker's original argument is flawed due to some measurability issues. Skorohod \citep{skorokhod1956limit} introduced a special metric on the {\it cadlag} space (functions that are right-contrinuous and left-hand limits) such that the latter is separable, and accordingly fixed the measurability problem. The proof of Theorem \ref{thm:gck} comes from \citet[Chapter 1]{dudley2014uniform}.

A notable result we did not mention in the text is Koml\'os–Major–Tusn\'ady (KMT) approximation \citep{komlos1975approximation}, which gave the sharp rate of convergence in Donsker's theorem (Theorem \ref{thm:donsker}); see, also, \cite{bretagnolle1989hungarian} for some explicit constants in the KMT bound and \citet[Section 1.4]{dudley2014uniform} for a proof of Bretagnolle and Massart's result.\\

{\bf Chapter \ref{chapter:dudley}.} Kolmogorov \citep{kolmogorov1955bounds} introduced the entropy concept, which was developed in more details in \cite{kolmogorov1959varepsilon}. Lorentz \citep{lorentz1966metric} coined the name ``metric entropy''. Dudley's metric entropy bounds, in the form of Theorems \ref{thm:dudley1} and \ref{thm:dudley2}, are due to Dudley \citep{dudley1967sizes,dudley1973sample}; the present versions and the proofs were adopted from \citet[Chapter 11]{ledoux2013probability}. Corollary \ref{cor:dudley4} is deduced from \citet[Theorem 2.2.4]{vaart1996empirical}.

In this chapter we did not address the work of Talagrand that gave sharp upper and lower bounds for stochastic processes, improving on Dudley's chaining; this is Talagrand's {\it generic chaining} \citep{talagrand1992simple,talagrand1996majorizing,talagrand2001majorizing,talagrand2014upper}. Chapter \ref{chap:cpt} will introduce a Bernstein-type generic chaining bound following Talagrand's argument.\\

{\bf Chapter \ref{sec:4ep}.} Empirical process theory has been studied in depth in the past half a century, and results were summarized in books such as, notably, \cite{vaart1996empirical}, \cite{kosorok2008introduction}, \cite{gine2021mathematical}, which this chapter closely followed. The ending notes in Chapters 1 and 2 of \cite{vaart1996empirical} as well as \citet[Section 3.8]{gine2021mathematical} sorted out the history of the development of the empirical process theory, which we refer the readers of interest to. \\

{\bf Chapter \ref{chapter:gc}.} The generalized Glivenko-Cantelli theorem was first discussed in \cite{blum1955convergence}. Vapnik and Chervonenkis \citep{vapnik1971uniform} made attempts to related uniform convergence to $L^1$-covering numbers. The present version is adapted from \citet[Chapter 2.4]{vaart1996empirical}. 

The symmetrization trick was due to \cite{kahane1985some} and \cite{hoffmann1974sums}. 

Talagrand's inequality, with constants unspecified, is due to \cite{talagrand1994sharper}, \cite{talagrand1995concentration}, and \cite{talagrand1996new}. The present, simplified, version of Talagrand's upper tail inequality is due to \cite{ledoux1997talagrand}; the lower tail part, which was not presented in this book, was established in \cite{samson2000concentration}. \\

{\bf Chapter \ref{chapter:donsker}.} The study of uniform central limit theorems is initiated by Dudley in \cite{dudley1978central}, followed by \cite{pollard1982central} and \cite{koltchinskii1981central}. The present uniform entropy version is due to Gin\'e and Zinn \citep{gine1984some}. The proofs were adapted from \citet[Chapter 2.5.1]{vaart1996empirical}, essentially from \cite{arcones1993limit}.\\

{\bf Chapter \ref{chap:vc}.}  All results in this chapter come from \cite{vaart1996empirical}, \cite{kosorok2008introduction}, and \cite{gine2021mathematical}. To name a few, Theorem \ref{thm:Dudley-Pollard} is Theorem 3.6.9 in \cite{gine2021mathematical}; Lemma \ref{lem:vc-subgraph} is Lemma 2.6.18 in \cite{vaart1996empirical}; Example \ref{example:4.8} is Example 3.6.14 in \cite{gine2021mathematical}; Lemma \ref{lem:4.8} is Lemma 2.6.13 in \cite{vaart1996empirical}. \\

{\bf Chapter \ref{chap:bracket}.} All results in this chapter come from \cite{vaart1996empirical} and \cite{van2000asymptotic}; see the references pointed in the text for details. 

%% file: chapters/cp.tex
\chapter{Permutation Process Theory}\label{chap:cpt}

Let $\cZ$ be a vector space. Consider the following finite population,
\[
\Big\{z_i=z_{N,i}\in\cZ, i\in[N]\Big\}, 
\]
which is assumed non-random but contains {\it not} necessarily distinct points. Here $N$ stands for the {\it population size}. The finite population distribution is denoted by $\P_N$ and is defined as
\[
\P_N=\frac{1}{N}\sum_{i=1}^N\delta_{z_i},~~\text{ with } \delta_z \text{ denoting the Dirac measure at } z.
\]
Let the sampling paradigm be uniform without replacement, yielding the following random measure that we call the {\it permutation measure}:
\[
\PP_{\pi,n} := \frac{1}{n}\sum_{\pi(i)\leq n}\delta_{z_i}=\frac{1}{n}\sum_{i=1}^N\delta_{z_i}\ind(\pi(i)\leq n),
\]
where $\pi$ stands for the uniform permutation over $[N]$. Here $n$ stands for the {\it sample size}.

Let $\cF=\{f:\cZ\to\reals\}$ be a class of functions. This chapter is focused on the following combinatorial process that is a natural counterpart of the empirical process:
\[
\Big\{\PP_{\pi,n}f:=\frac{1}{n}\sum_{\pi(i)\leq n}f(z_i); f\in\cF\Big\}.
\]
Each element in the above process has expectation
\[
\P_Nf=\frac{1}{N}\sum_{i=1}^Nf(z_i).
\]
We define
\[
\bnorm{\PP_{\pi,n}-\P_N}_{\cF}:=\sup_{f\in\cF}\Big|(\PP_{\pi,n}-\P_N)f\Big|~~{\rm and}~~\GG_{\pi,n}f=\sqrt{n}(\PP_{\pi,n}-\P_Nf).
\]

It is apparent that, in finite population, every random variable considered is trivially measurable so that, starting from now on, {\it we drop the countability condition for $\cF$}.

In addition, the asymptotic setting this chapter focuses on is $n=n_N\to\infty$ as $N\to\infty$.

\section{Rosen processes}\label{chap:rosen}

Before addressing more complex permutation processes, let's first establish an elementary result characterizing the limiting behavior of the finite-population sampling process. This result, originally due to Rosen \citep{rosen1964limit}, is presented here in the form given by Billingsley \citep[Theorem 24.1]{billingsley1968convergence}.

Consider here $\cZ\subset\reals$ and for each $t\in[0,1]$, define
\[
Z_N(t) := \sum_{\pi(i)\leq \lceil Nt\rceil}z_i=\sum_{i=1}^Nz_i\ind(\pi(i)\leq \lceil Nt\rceil),
\]
where $\lceil\cdot\rceil$ represents the ceiling of the input. The process 
\[
Z_N=\Big\{Z_N(t);t\in[0,1]\Big\} 
\]
then monitors the sampling process from the beginning ($t=0$) to the end when the pool is exhausted ($t=1$). 

The following theorem characterizes the stochastic behavior of $Z_N$.

\begin{theorem}[Rosen-Billingsley]\label{thm:rosen} Assume 
\[
\sum_{i=1}^Nz_i=0,~~ \sum_{i=1}^Nz_i^2=1,~~ {\rm and}~~ \max_{i\in[N]}|z_i|\to 0. 
\]
It holds true then that
\[
Z_N \Rightarrow \BB(0,1)~~~\text{ in } \ell^{\infty}([0,1]),
\]
where $\BB(0,1)$ stands for the standard Brownian bridge, i.e., the Gaussian process with mean  zero and variance 
\[
\cov(\BB(s),\BB(t))=s\wedge t-st,~~\text{ for any }s,t\in\reals.
\]
\end{theorem}

Before proving this theorem, let's first explore an ``empirical process'' type proof under a strong assumption that $\limsup_{N\to\infty}N\sum_{i=1}^Nz_i^4\leq K$ for some finite universal constant $K$. To employ Theorem \ref{thm:weak-convergence-key},  it is first straightforward to verify that $$(Z_N(t_1),\ldots,Z_N(t_m))^\top$$
weakly converges to the distribution of $(\BB(t_1),\ldots,\BB(t_m))^\top$ for any positive integer $m$ and any $t_1,\ldots,t_m\in[0,1]$. 

We next verify stochastic equicontinuity. Note that every function is measurable in a discrete space. Accordingly, we can replace the outer probability in Theorem \ref{thm:weak-convergence-key} by the regular probability. In addition, for any $\delta>0$ and $s<t$ such that $t-s<\delta$, we have 
\[
Z_N(t)-Z_N(s)=\sum_{i=1}^N z_i\ind(\lceil Ns\rceil<\pi(i)\leq \lceil Nt\rceil ).
\]

Applying Corollary \ref{cor:c-hoeffding-2} to the above combinatorial sum and noticing that
\begin{align*}
\var\{Z_N(t)-Z_N(s)\} \lesssim t-s ~~~{\rm  and}~~~\\
\sqrt{\sum_{i=1}^Nz_i^4}\cdot\sqrt{\sum_{j=1}^N\Big(\ind\Big(j\in (\lceil Ns\rceil,\lceil Nt\rceil]\Big)-\frac{\lceil Nt\rceil-\lceil Ns\rceil}{N}\Big)^4} \lesssim (t-s)^{1/2},
\end{align*}
we obtain, by Lemma \ref{lem:subgaussian}, that there exists a universal constant $K$ such that
\[
\norm{Z_N(t)-Z_N(s)}_{\psi_2} \leq K|t-s|^{1/4}~~~\text{ for all }t,s\in[0,1].
\]
Lastly, we have 
\begin{enumerate}[label=(\roman*)]
\item $([0,1],|\cdot|^{1/4})$ is a metric space;  
\item for any $\epsilon\in (0,1]$, $\cN([0,1],|\cdot|^{1/4},\epsilon) \leq 2/\epsilon^{4}$ and accordingly is totally bounded;
\item it holds true that
\begin{align*}
Z_N(t)-Z_N(s)&=\sum_{i=1}^N z_i\ind(\lceil Ns \rceil<\pi(i)\leq \lceil Nt \rceil)\\
 &\leq \sqrt{\lceil Nt \rceil-\lceil Ns \rceil} \leq \sqrt{2N}\cdot\sqrt{t-s},
\end{align*}
and $Z_N$ is thusly separable in $([0,1],|\cdot|^{1/4})$. 
\end{enumerate}
Dudley's entropy bound (Corollary \ref{cor:dudley3}) then yields
\[
\E\sup_{|s-t|^{1/4}\leq \delta}\Big|Z_N(t)-Z_N(s) \Big| \lesssim \int_0^\delta \sqrt{\log \Big(\frac{4}{\epsilon^{4}}\Big)}\d \epsilon,
\]
which converges to 0 as $\delta\to 0$ by the dominated convergence theorem. Employing Markov's inequality then proves stochastic equicontinuity. This finishes the proof.

\begin{proof}[Proof of Theorem \ref{thm:rosen}] The above argument does not fully exploit the partial sum structure in $Z_N$, and thus ends up with an unnecessarily strong uniformity requirement on $z_i$'s. This highlights the limit of Dudley's metric entropy approach.

Instead, to prove stochastic equicontinuity,  this time let's fix some constant $\delta>0$ such that, without loss of generality, $1/\delta$ is an integer. It then holds true that 
\begin{align*}
&\Pr\Big(\sup_{|s-t|\leq \delta}\Big|Z_N(t)-Z_N(s)\Big|> \epsilon\Big)\\
\leq &\Pr\Big(\max_{0\leq j\delta<1}\sup_{j\delta\leq t\leq (j+1)\delta}\Big|Z_N(t)-Z_N(j\delta)\Big|> \epsilon/2 \Big)+\\
&\Pr\Big(\max_{0< j\delta\leq 1}\sup_{(j-1)\delta\leq s\leq j\delta}\Big|Z_N(j\delta)-Z_N(s)\Big|> \epsilon/2 \Big).
\end{align*}
For the first term, notice that for any $\epsilon>0$,
\begin{align*}
&\Pr\Big(\max_{0\leq j\delta<1}\sup_{j\delta\leq t\leq (j+1)\delta}\Big|Z_N(t)-Z_N(j\delta)\Big|> \epsilon \Big)\\
\leq& \frac{1}{\delta} \Pr\Big(\sup_{j\delta\leq t\leq (j+1)\delta}\Big|Z_N(t)-Z_N(j\delta)\Big|> \epsilon \Big)= \frac{1}{\delta} \Pr\Big(\sup_{0\leq t\leq \delta}\Big|Z_N(t)\Big|> \epsilon \Big)\\
=& \frac{1}{\delta}\Pr\Big(\sup_{0\leq t\leq \delta}\Big|\sum_{\pi(i)\leq \lceil Nt \rceil}z_i\Big|> \epsilon \Big).
\end{align*} 

Next, we need a Levy-type inequality for partial sum processes of exchangeable entries. The following result is due to \cite{pruss1998maximal}.
\begin{lemma}[Pruss's inequality] Fix a constant $\gamma>1$. Assume $X_1,\ldots,X_{\lceil n\gamma\rceil}$ are exchangeable. Then, for any $k\in[n]$ and any $t>0$,
\[
\Pr\Big(\max_{\ell\in[k]}\Big|\sum_{i=1}^\ell X_i\Big|>t \Big) \leq c\cdot \Pr\Big(\Big|\sum_{i=1}^kX_i\Big|>\frac{t}{c}  \Big),
\]
where $c=c(\gamma)\in(0,\infty)$ is a constant only depending on $\gamma$.
\end{lemma}
Using this lemma and invoking Theorem \ref{thm:cb-ineq}, it is then obvious that there exists a constant $K>0$ such that
\begin{align*}
\Pr\Big(\sup_{0\leq t\leq \delta}\Big|\sum_{\pi(i)\leq \lceil Nt \rceil}z_i\Big|> \epsilon \Big) &\leq c\Pr\Big(\Big|\sum_{\pi(i)\leq \lceil N\delta \rceil}z_i\Big|> \epsilon/c \Big)\\
&\leq 2c\exp\Big(- \frac{K\epsilon^2}{\lceil N\delta\rceil/N+\epsilon\cdot\max_{i\in[N]}|z_i|} \Big).
\end{align*}
Accordingly,
\begin{align*}
&\Pr\Big(\max_{0\leq j\delta<1}\sup_{j\delta\leq t\leq (j+1)\delta}\Big|Z_N(t)-Z_N(j\delta)\Big|> \epsilon \Big)\\
 \leq&   \frac{1}{\delta} \cdot 2c\exp\Big(- \frac{K\epsilon^2}{\lceil N\delta\rceil/N+\epsilon\cdot\max_{i\in[N]}|z_i|} \Big),
 \end{align*}
 which will go to 0 as we first take $N\to\infty$ and then $\delta \downarrow 0$. The same holds true for the other half, and we thus complete the proof.
\end{proof}

%\begin{remark}
%A more careful analysis employing the martingale bounds as done in \citet[Theorem 24.1]{billingsley1968convergence} can replace Condition \eqref{eq:billingsley-thm-0} by assuming instead $\lim_{N\to\infty}\max_{i\in[N]}|z_i|=0$. However, such techniques are restricted to particular partial sum processes; in contrast, the above analysis bears the flavor to be carried over to more general permutation processes, which will be the focus of the subsequent sections.
%\end{remark}

\section{Maximal inequalities}\label{chap:cmi}

We first present the form of $\PP_{\pi,n}f$ as a combinatorial mean.

\begin{lemma}\label{lem:chaining-a-and-d}
The combinatorial mean 
\begin{align*}
&\PP_{\pi,n}f = \sum_{i=1}^Na_{i,\pi(i)}^f=\sum_{i=1}^N\frac{f(z_i)}{n}\ind(\pi(i)\leq n),\\
{\rm and}~~&(\PP_{\pi,n}-\P_N)f=\sum_{i=1}^N d_{i,j}^{f}=\sum_{i=1}^N\frac{f(z_i)-\P_N f}{n}\cdot \Big\{\ind(\pi(i)\leq n)-n/N\Big\}.
\end{align*}
\end{lemma}

The following is then a direct implication of Lemma \ref{lem:chaining-a-and-d} combined with Proposition \ref{prop:basic}.

\begin{lemma}\label{lem:combinatorial-basic-Donsker} For any $f,g\in\cF$, we have
\[
\E\Big[\GG_{\pi,n}f\Big]=0, ~~~\var\Big[\GG_{\pi,n}f\Big]=\frac{N-n}{N-1}\P_N(f-\P_Nf)^2,
\]
and
\[
\cov\Big[\GG_{\pi,n}f,\GG_{\pi,n}g\Big]=\frac{N-n}{N-1}(\P_N(fg)-(\P_Nf)(\P_Ng)).
\]
\end{lemma}

\begin{exercise}
Prove Lemma \ref{lem:chaining-a-and-d} and Lemma \ref{lem:combinatorial-basic-Donsker}.
\end{exercise}

We next present a critical bound, the Bobkov's inequality, that quantifies the subgaussian property of $\PP_{\pi,n}f$ and set a bound that is only related to the sampling ratio, $n/N$, and the variance under $\P_N$.

\begin{theorem}[Bobkov's inequality]\label{lem:comb-orlicz} It holds true that, for any $f\in\cF$,
\[
\bnorm{(\PP_{\pi,n}-\P_N)f}_{\psi_2}\leq \sqrt{\frac{12}{n}\Big(1+\frac{N}{n}\Big)}\bnorm{f-\P_N f}_{L^2(\P_N)},
\]
and for any $f,g\in\cF$, 
\[
\bnorm{(\PP_{\pi,n}-\P_N)(f-g)}_{\psi_2}\leq \sqrt{\frac{48}{n}\Big(1+\frac{N}{n}\Big)}\bnorm{f-g}_{L^2(\P_N)}.
\]
\end{theorem}
\begin{proof}
Using Lemma \ref{lem:chaining-a-and-d}, for any $f,g\in\cF$, we have
\[
d_{i,j}^f-d_{i,j}^g=\frac{1}{n}\Big(\ind(j\leq n)-\frac{n}{N}\Big)\Big[f(z_i)-g(z_i)-\P_N(f-g)\Big].
\]
Now we will invoke Theorem \ref{cor:c-hoeffding-2}. Notice that
\begin{align}\label{eq:m-est-1}
&\frac{1}{n^2}\sup_{s\in\cS_N}\sum_{i=1}^N\Big(f(z_i)-g(z_i)-\P_N(f-g)\Big)^2\Big(\ind(s(i)\leq n)-\frac{n}{N}\Big)^2\notag\\
\leq&\frac{1}{n}\cdot\frac{N}{n}\cdot \frac{1}{N}\sum_{i=1}^N\Big(f(z_i)-g(z_i)-\P_N(f-g)\Big)^2\notag\\
\leq& \frac{1}{n}\cdot\frac{4N}{n}\P_N(f-g)^2,
\end{align}
where in the last inequality we use the Jensen's inequality to derive
\[
(\P_N (f- g))^2\leq \P_N(f-g)^2.
\]
On the other hand,
\begin{align}\label{eq:m-est-2}
&\frac{1}{N-1}\sum_{i,j\in[N]}\Big|d_{i,j}^{f}-d_{i,j}^{g}\Big|^2 \notag\\
\leq~& \frac{2}{n^2(N-1)}\sum_{i,j\in[N]}\Big(\ind(j\leq n)-\frac{n}{N}\Big)^2[f(z_i)-g(z_i)]^2\notag\\
&~~~+\frac{2}{n^2(N-1)}\sum_{i,j\in[N]}\Big(\ind(j\leq n)-\frac{n}{N}\Big)^2[\P_N(f-g)]^2\notag\\
\leq~& \frac{1}{n}\Big\{\frac{2}{N}\sum_{i=1}^N[f(z_i)-g(z_i)]^2+2[\P_N(f-g)]^2\Big\}\notag\\
\leq~& \frac{4}{n}\cdot \P_N(f-g)^2,
\end{align}
where in the last inequality we use the fact that
\[
\sum_{j=1}^N\Big(\ind(j\leq n)-\frac{n}{N}\Big)^2=\sum_{j=1}^n\Big(1-\frac{n}{N}\Big)^2+\sum_{j=n+1}^N\Big(\frac{n}{N}\Big)^2=\frac{n(N-n)}{N}.
\]
Plugging \eqref{eq:m-est-1} and \eqref{eq:m-est-2} into Corollary \ref{cor:c-hoeffding-2} and applying Lemma \ref{lem:subgaussian} yields the second claim. The first claim is similar, and is left for the readers to verify.
\end{proof}

The term $N/n$ in the Orlicz $\psi_2$ norm bound is a little bit unpleasant if $n$ is asymptotically ignorant compared to $N$. We may, however, remove it by appealing to the Orlicz $\psi_1$ norm combined with Theorem \ref{thm:cb-ineq}. This is, however, at the cost of introducing an $L^{\infty}$ term. It results to a version of the Bernstein-Serfling inequality; see, e.g., \citet[Theorem 3.5]{bardenet2015concentration}.

\begin{theorem}[Bernstein-Serfling inequality]\label{lem:comb-orlicz-1} It holds true that, for any $f\in\cF$,
\[
\bnorm{(\PP_{\pi,n}-\P_N)f}_{\psi_1}\leq \frac{24\sqrt{2}}{n}\bnorm{f-\P_N f}_{L^\infty}+\sqrt{\frac{72}{n\log 2}}\bnorm{f-\P_N f}_{L^2(\P_N)},
\]
and for any $f,g\in\cF$, 
\[
\bnorm{(\PP_{\pi,n}-\P_N)(f-g)}_{\psi_1}\leq \frac{48\sqrt{2}}{n}\bnorm{f-g}_{L^\infty}+\sqrt{\frac{288}{n\log 2}}\bnorm{f-g}_{L^2(\P_N)}.
\]
\end{theorem}

\begin{exercise}
Please prove Theorem \ref{lem:comb-orlicz-1} using Theorem \ref{thm:cb-ineq} and then Lemma \ref{lem:2bernstein}.
\end{exercise}

Lastly, we highlight the following corollary that is a direct consequence of Hoeffding's convex ordering inequality (Theorem \ref{thm:hoeffding-amazing}). This result turns out to be quite useful in the subsequent sections.

\begin{corollary}[Hoeffding's convex ordering inequality, stochastic process version]\label{cor:5hcoi} For any $L$ fixed elements $w_1,\ldots,w_L\in\cW$, any function class $\cF=\{f:\cW\to\reals\}$, and any positive integer $\ell\in[L]$, we have
\[
\E\bnorm{\sum_{i=1}^\ell(\delta_{w_{\pi^{-1}(i)}}-\W_L)}_{\cF} \leq \E\bnorm{\sum_{i=1}^\ell(\delta_{\hat W_i}-\W_L)}_{\cF},
\]
where $\pi\in\cS_L$ is a uniform permutation, $\W_L:=L^{-1}\sum_{i=1}^L\delta_{w_i}$, and $\hat W_1,\hat W_2, \ldots,$ are i.i.d. drawn from $\W_L$.
\end{corollary}

\section{Permutation LLN and CLT}\label{sec:5plln}

\begin{theorem}[Permutation strong LLN]\label{thm:permutation-slln} Suppose that $f:\cZ\to\reals$ satisfies either
\[
\limsup_{N\to\infty}\P_N(f-\P_Nf)^2 <\infty ~~~{\rm and} ~~~ \liminf_{N\to\infty} (n/N)>0 
\]
or
\[
\limsup_{N\to\infty}\P_N(f-\P_Nf)^4 <\infty.
\]
It then holds true that 
\[
(\PP_{\pi,n}-\P_N)f\stackrel{a.s.}{\to}0.
\]
\end{theorem}
\begin{proof}
{\bf Case 1.} To prove the assertion under the first condition, Lemma \ref{lem:comb-orlicz} shows that there exists a universal constant $K>0$ such that, for all sufficiently large $n$, 
\[
\Pr\Big\{\Big|(\PP_{\pi,n}-\P_N)f\Big|>t\Big\}\leq 2\exp\Big(-\frac{nt^2}{K\sigma_N^2} \Big),~~\text{ for any }t>0.
\]
Additionally, by condition, 
\[
\sigma_N^2:=\P_N(f-\P_Nf)^2
\]
satisfies that $\limsup \sigma_N^2<\infty$. Thusly, we obtain, for any $t>0$,
\[
\sum_{n=1}^{\infty}\Pr\Big\{\Big|(\PP_{\pi,n}-\P_N)f\Big|>t\Big\}\leq 2\sum_{n=1}^{\infty}\exp\Big(-\frac{nt^2}{K\sigma_N^2} \Big)<\infty.
\]
Accordingly, invoking the first Borel-Cantelli lemma yields the conclusion. 

{\bf Case 2.}  To prove the assertion under the second condition, note that, by Theorem \ref{thm:hoeffding-amazing},  
\[
\E\Big|(\PP_{\pi,n}-\P_N)f\Big|^4 \leq \E\Big|(\hat\PP_{n}-\P_N)f\Big|^4,
\]
where $\hat\PP_n=n^{-1}\sum_{i=1}^n\delta_{\hat Z_i}$ and $\hat Z_i$'s are independently sampled from $\P_N$. The left-hand side can be further simplified to be
\begin{align*}
\E\Big|n(\hat\PP_{n}-\P_N)f\Big|^4&=n\P_N(f-\P_Nf)^4+3(n^2-n)\{\P_N(f-\P_Nf)^2\}^2\\
&\leq 3n^2\P_N(f-\P_Nf)^4.
\end{align*}
Markov's inequality then yields
\[
\Pr\Big\{\Big|(\PP_{\pi,n}-\P_N)f\Big|>t\Big\}\leq \frac{3\P_N(f-\P_Nf)^4}{n^2t^4}
\]
so that 
\[
\sum_{n=1}^{\infty}\Pr\Big\{\Big|(\PP_{\pi,n}-\P_N)f\Big|>t\Big\} \leq \sum_{n=1}^{\infty} \frac{3\P_N(f-\P_Nf)^4}{n^2t^4}<\infty.
\]
This completes the proof.
\end{proof}

%\begin{exercise}
%Please finish the proof of Theorem \ref{thm:permutation-slln}.
%\end{exercise}

\begin{theorem}[Permutation CLT]\label{thm:CLT-perm-sampling} Suppose that 
\[
n=n_N\to\infty,~~~\P_N\Rightarrow \P,~~~{\rm and}~~~\lim_{N\to\infty} \frac{n}{N}=\gamma\in[0,1).
\]
In addition, assume that the functions $f_1,\ldots,f_m:\cZ\to\reals$ satisfy that  
\[
\limsup_{N\to\infty}\P_N|f_i-\P_Nf_i|^3<\infty, ~~\P_Nf\to \P f,~~{\rm and}~~\P_N (f_if_j)\to \P (f_if_j) 
\]
for any $i,j\in[m]$. It then holds true that
\[
\GG_{\pi,n}\bm{f} \Rightarrow N(\bm{0}, \bSigma),
\]
where $\bm{f}:=(f_1,f_2,\ldots,f_m)^\top$ and $[\bSigma]_{ij}=(1-\gamma)\P(f_i-\P f_i)(f_j-\P f_j)$.
\end{theorem}
\begin{proof}
%Without loss of generality, it is assumed that $\bSigma$ is positive definite.
Let's employ the Cramer-Wold device and combine it with the combinatorial CLT (Theorem \ref{thm:CCLT-key}). In detail, for any $\bv\in\reals^m$ with $\norm{\bv}=1$, we have 
\[
[\GG_{\pi,n}\bm{f}]^\top\bv = \sum_{i=1}^N\frac{1}{\sqrt{n}}\sum_{j=1}^m v_j(f_j-\P_N f_j)\ind(\pi(i)\leq n).
\]
Lemma \ref{lem:combinatorial-basic-Donsker} shows that the variance of the above combinatorial sum is 
\begin{align*}
\var\Big([\GG_{\pi,n}\bm{f}]^\top\bv\Big) = (1-\gamma)\sum_{i=1}^m\sum_{j=1}^mv_iv_j\P_N(f_i-\P_Nf_i)(f_j-\P_N f_j)\\
\to (1-\gamma)\sum_{i=1}^m\sum_{j=1}^mv_iv_j\P(f_i-\P f_i)(f_j-\P f_j)~~~{\rm as}~N\to\infty.
\end{align*}
If $\sum_{i=1}^m\sum_{j=1}^mv_iv_j\P(f_i-\P f_i)(f_j-\P f_j)=0$, then $\var([\GG_{\pi,n}\bm{f}]^\top\bv)\to 0$ so that 
\[
[\GG_{\pi,n}\bm{f}]^\top\bv \Rightarrow 0 \stackrel{\rm a.s.}{=} \bv^\top N(\bm{0},\bSigma). 
\]
Accordingly, in the following we only have to focus on those $\bv$'s such that
\[
\sum_{i=1}^m\sum_{j=1}^mv_iv_j\P(f_i-\P f_i)(f_j-\P f_j)>0.
\]

For these $\bv$'s, we have 
\begin{align*}
&\frac{1}{Nn^{3/2}}\sum_{i,j\in[N]}\Big|\bv^\top\Big(\ind(j\leq n)-\frac{n}{N}\Big)\cdot\Big(\bm{f}(z_i)-\P_N\bm{f}(z_i)\Big)\Big|^{3}\\
\leq&\frac{1}{Nn^{3/2}}\sum_{i,j\in[N]}\Big|\ind(j\leq n)-\frac{n}{N}\Big|^{3}\cdot\|\bm{f}(z_i)-\P_N\bm{f}(z_i)\|^{3}\\
\leq& \frac{2}{Nn^{1/2}}\sum_{i=1}^N\Big(\sum_{j=1}^m|f_j-\P_Nf_j|^2\Big)^{3/2}= O(n^{-1/2}),
\end{align*}
where in the last equality we used (i) the fact that
\begin{align*}
\sum_{j\in[N]}\Big|\ind(j\leq n)-\frac{n}{N}\Big|^{3}=\sum_{j\leq n}\Big(1-\frac{n}{N}\Big)^3+\sum_{j>n}\frac{n^3}{N^3}\leq n+\frac{n^3}{N^2}\leq 2n,
\end{align*}
(ii) Jensen's inequality that
\[
\Big(\frac1m\sum_{j=1}^m|f_j-\P_Nf_j|^2\Big)^{3/2}\leq \frac1m\sum_{j=1}^m|f_j-\P_Nf_j|^3,
\]
and (iii) the condition that  
\[
\limsup_{N\to\infty}\P_N|f_j-\P_N f_j|^3 <\infty. 
\]
Accordingly, when applied to the combinatorial sum $[\GG_{\pi,n}\bm{f}]^\top\bv$, the righthand side of Theorem \ref{thm:CCLT-key} converges to 0, so that 
\[
[\GG_{\pi,n}\bm{f}]^\top\bv \Rightarrow N(0,\bv^\top\bSigma\bv).
\]
Invoking Corollary \ref{cor:cw-device} then  completes the proof.
\end{proof}

\section{Combinatorial Talagrand's inequality}\label{chap:cti}

We start with an introduction to Bobkov's entropy argument on uniform permutations. In the following, let 
\[
\cG_N:=\Big\{(i_1,i_2,\ldots,i_N)^\top; \{i_1,i_2,\ldots,i_N\}=[N]\Big\}
\]
contain all rearrangements of $[N]$.  Let 
\[
\bpi=\bpi_N:=(\pi^{-1}(1),\ldots,\pi^{-1}(N))^\top \in \cG_N
\]
be uniformly distributed over $\cG_N$.

\begin{definition}[Symmetric functions over $\cG_{N}$] A function $g:\cG_n\to\reals$ is said to be $(n,N)$-symmetric if $f$ is permutationally invariant with regard to the first $n$ and the second $N-n$ entries.
\end{definition}

\begin{definition}[Gradient magnitude of a function] For any $(n,N)$-symmetric function  $g:\cG_n\to\reals$, we define its gradient magnitude to be
\[
\Big|\nabla g(\bpi)\Big|^2 := \sum_{i\in \cI}\sum_{j\in\cJ}\Big( g(\bpi)-g(\bpi^{i,j})\Big)^2,
\]
where 
\[
\cI:=[n], ~~\cJ:=[N]/\cI,~~{\rm and}~~\pi^{i,j}_k:= \begin{cases}
\pi(k), \quad \text{if }k\ne i,j,\\
\pi(j), \quad \text{if } k=i,\\
\pi(i), \quad \text{if } k=j.
\end{cases}
\]
\end{definition}

The following result gives a log-Sobolev-type inequality.

\begin{lemma}[Bobkov's lemma]\label{lem:bobkov} Let $\bpi$ be uniformly distributed over $\cG_N$. For any $(n,N)$-symmetric function  $g:\cG_n\to\reals$, it then holds true that, 
\[
(N+2){\rm Ent}_{\bpi}(e^g)\leq \cE(e^g,g),
\]
where
\[
{\rm Ent}_{\bpi}(e^g) := \E\Big[e^{g(\bpi)} g(\bpi)\Big] - \E\Big[e^{g(\bpi)}\Big]\log \E\Big[e^{g(\bpi)}\Big]
\]
and
\[
 \cE(e^g,g) := \E\Big[\sum_{i\in\cI}\sum_{j\in\cJ}\Big(g(\bpi)-g(\bpi^{i,j})\Big)\Big(e^{g(\bpi)}-e^{g(\bpi^{i,j})}\Big)  \Big].
\]
\end{lemma}

A direct consequence of Bobkov's lemma is the following Hoeffding-type inequality for $(n,N)$-symmetric functions. 

\begin{corollary}[Bobkov's permutation inequality]\label{cor:bobkov} Let $\bpi$ be uniformly distributed over $\cG_N$ and $\Sigma^2$ be a positive finite constant. Assume $g$ to be an $(n,N)$-symmetric function such that 
\[
\P\Big(|\nabla g(\bpi)|^2\leq \Sigma^2 \Big) =1.
\]
It then holds true that, for any $t\geq 0$, 
\[
\Pr\Big\{\Big|g(\bpi)-\E g(\bpi) \Big|\geq t \Big\}\leq 2\exp\Big\{-\frac{(N+2)t^2}{4\Sigma^2} \Big\}.
\]
\end{corollary}

Specializing to such 
\[
g(\bpi)=(\PP_{\pi,n}-\P_N)f,
\]
Corollary \ref{cor:bobkov} gives a (up to constants) equivalent version of Theorem \ref{lem:comb-orlicz}; this explains why Theorem \ref{lem:comb-orlicz} is named after Sergey Bobkov. 
\begin{corollary} As taking $g(\bpi)=(\PP_{\pi,n}-\P_N)f$, we have
\[
\Big|\nabla g(\bpi)\Big|^2 \leq \Big(\frac{N}{n}\Big)^2 \bnorm{f-\P_Nf}_{L^2(\P_N)}^2
\]
and thusly
\[
\bnorm{(\PP_{\pi,n}-\P_N)f}_{\psi_2}\leq \sqrt{\frac{12N}{n^2}}\bnorm{f-\P_Nf}_{L^2(\P_N)}.
\]
\end{corollary}
\begin{proof}
By definition,
\begin{align*}
\Big|\nabla g(\bpi)\Big|^2 =  \frac{1}{n^2}\sum_{i\in\cI}\sum_{j\in\cI} (f(z_{\pi_i})-f(z_{\pi_j}))^2\leq \frac{1}{n^2}\sum_{1\leq i<j\leq N}(f(z_i)-f(z_j))^2\\
= \Big(\frac{N}{n}\Big)^2 \cdot \frac{1}{N}\sum_{i=1}^N\Big(f(z_i)-\frac{1}{N}\sum_{j=1}^Nf(z_j)\Big)^2 = \Big(\frac{N}{n}\Big)^2 \bnorm{f-\P_Nf}_{L^2(\P_N)}^2.
\end{align*}
The rest is straightforward.
\end{proof}

To establish a Talagrand's inequality in the form of Theorem \ref{thm:talagrand-inequ}, on the other hand, we need a refined version of Corollary \ref{cor:bobkov}. This is managed through a clever trick due to Ilya Tolstikhin \citep{tolstikhin2017concentration}.

\begin{corollary}[Bobkov's permutation inequality, Tolstikhin's version]\label{cor:bobkov-tol} Let $\bpi$ be uniformly distributed over $\cG_N$ and $\Sigma^2$ be a positive finite constant. Assume $g$ to be an $(n,N)$-symmetric function such that 
\begin{align*}
&\P\Big(|\nabla g(\bpi)|_+^2\leq \Sigma^2 \Big) =1\\
~~{\rm with}~~&\Big|\nabla g(\bpi)\Big|_+^2 := \sum_{i\in \cI}\sum_{j\in\cJ}\Big( g(\bpi)-g(\bpi^{i,j})\Big)^2\ind\Big(g(\bpi)\geq g(\bpi^{i,j})\Big).
\end{align*}
It then holds true that, for any $t\geq 0$, 
\[
\Pr\Big\{\Big|g(\bpi)-\E g(\bpi) \Big|\geq t \Big\}\leq 2\exp\Big\{-\frac{(N+2)t^2}{8\Sigma^2} \Big\}.
\]
\end{corollary}
\begin{proof}
Continuing from Bobkov's lemma (Lemma \ref{lem:bobkov}), we obtain
\begin{align*}
&\cE(e^g,g) \\
=& \frac{1}{N!}\sum_{\bsigma\in\cG_N}\Big[\sum_{i\in\cI}\sum_{j\in\cJ}\Big(g(\bsigma)-g(\bsigma^{i,j})\Big)\Big(e^{g(\bsigma)}-e^{g(\bsigma^{i,j})}\Big)  \Big]\\
=& \frac{2}{N!}\sum_{\bsigma\in\cG_N}\sum_{i\in\cI}\sum_{j\in\cJ}\Big(g(\bsigma)-g(\bsigma^{i,j})\Big)\Big(e^{g(\bsigma)}-e^{g(\bsigma^{i,j})}\Big)\ind\Big(g(\bsigma)\geq g(\bsigma^{i,j}) \Big).
\end{align*}
Notice that, for any $a,b\in\reals$, it holds true
\[
(a-b)(e^a-e^b)\leq \frac{e^a+e^b}{2}(a-b)^2.
\]
We thus obtain
\begin{align*}
&\frac{2}{N!}\sum_{\bsigma\in\cG_N}\sum_{i\in\cI}\sum_{j\in\cJ}\Big(g(\bsigma)-g(\bsigma^{i,j})\Big)\Big(e^{g(\bsigma)}-e^{g(\bsigma^{i,j})}\Big)\ind\Big(g(\bsigma)\geq g(\bsigma^{i,j}) \Big)\\
\leq& \frac{2}{N!}\sum_{\bsigma\in\cG_N}\sum_{i\in\cI}\sum_{j\in\cJ}\Big(g(\bsigma)-g(\bsigma^{i,j})\Big)^2\frac{e^{g(\bsigma)}+e^{g(\bsigma^{i,j})}}{2}\ind\Big(g(\bsigma)\geq g(\bsigma^{i,j}) \Big)\\
=& \frac{2}{N!}\sum_{\bsigma\in\cG_N}\sum_{i\in\cI}\sum_{j\in\cJ}\Big(g(\bsigma)-g(\bsigma^{i,j})\Big)^2\ind\Big(g(\bsigma)\geq g(\bsigma^{i,j}) \Big)\cdot e^{g(\bsigma)}\\
=&2\E\Big[\Big|\nabla g(\bpi)\Big|_+^2e^{g(\bpi)}  \Big].
\end{align*}
Accordingly, Lemma \ref{lem:bobkov} yields
\[
(N+2){\rm Ent}_{\bpi}(e^g) \leq 2\Sigma^2\E\Big[ e^{g(\pi)}\Big],
\]
so that for any $\lambda\in\reals$,
\[
(N+2){\rm Ent}_{\bpi}(e^{\lambda g}) \leq 2\lambda^2\Sigma^2\E\Big[ e^{\lambda g}\Big].
\]
Introducing $H(\lambda):=\E e^{\lambda g}$, standard entropy arguments then yield, for any $\lambda\in\reals$,
\[
\lambda H'(\lambda)-H(\lambda)\log H(\lambda) = {\rm Ent}_{\bpi}(e^{\lambda g}) \leq \frac{2\lambda^2\Sigma^2}{N+2} H(\lambda),
\]
so that, by rewriting $K(\lambda):=\lambda^{-1}\log H(\lambda)$, we have
\[
\frac{\d }{\d \lambda}K(\lambda)\leq \frac{2\Sigma^2}{N+2}.
\]
Using the initial condition that $K(0)=H'(0)/H(0)=\E g$, we obtain
\[
K(\lambda)=K(0)+\int_0^\lambda K'(t)\d t \leq \E g + \frac{2\Sigma^2\lambda}{N+2},
\]
so that 
\[
e^{-\lambda \E g}H(\lambda) = \E e^{\lambda(g-\E g)}\leq \exp\Big(\frac{2\Sigma^2\lambda^2}{N+2} \Big),
\]
which implies the Hoeffding-type inequality and thus completes the proof.
\end{proof}

We are now ready to introduce the main result of this section, a Talagrand-type inequality for combinatorial processes.

\begin{theorem}[Tolstikhin-Talagrand inequality]\label{thm:tol-tala} For any $t\geq 0$, it holds true that
\[
\Pr\Big\{\bnorm{\PP_{\pi,n}-\P_N}_{\cF}- \E\bnorm{\PP_{\pi,n}-\P_N}_{\cF}\geq t \Big\} \leq \exp\Big(-\frac{n^2t^2}{8N\Sigma_{\cF}^2}  \Big),
\]
where the constant $\Sigma_{\cF}^2$ is defined to be
\[
\Sigma_{\cF}^2 := \sup_{f\in\cF}\bnorm{f-\P_N f}_{L^2(\P_N)}^2.
\]
\end{theorem}
\begin{proof}
Conditioning on $\bpi$, let $\bar f=\bar f_{\bpi}$ be a function in the convex hull of $\cF$ such that
\[
Q(\bpi):=\bnorm{\PP_{\pi,n}-\P_N}_{\cF} = (\PP_{\pi,n}-\P_N)\bar f .
\]
It then holds true that 
\begin{align*}
\Big| \nabla Q_n  \Big|_+^2 &=  \sum_{i\in \cI}\sum_{j\in\cJ}\Big( Q(\bpi)-Q(\bpi^{i,j})\Big)^2\ind\Big(Q(\bpi)\geq Q(\bpi^{i,j})\Big)\\
&\leq \frac{1}{n^2}\sum_{i\in\cI}\sum_{j\in\cJ}\Big(\sum_{k=1}^n\bar f(z_{\pi_k})-\sum_{k=1}^n\bar f(z_{\pi_k^{i,j}})  \Big)^2\\
&=\frac{1}{n^2}\sum_{i\in\cI}\sum_{j\in\cJ}\Big(\bar f(z_{\pi_i})-\bar f(z_{\pi_j})  \Big)^2\\
&\leq  \frac{1}{n^2}\sum_{1\leq i<j\leq N}\Big(\bar f(z_{i})-\bar f(z_{j})  \Big)^2\\
&\leq \frac{N^2}{n^2}\Sigma_{\cF}^2,
\end{align*}
where in the first inequality we used the fact that, for any two functions $f,g:\cX\to\reals$ such that $\sup_{x\in\cX}f(x)=f(\bar x)$, we have 
\[
\Big(\sup_{x\in\cX}f(x)-\sup_{x\in\cX}g(x)\Big)^2\ind\Big(\sup_{x\in\cX}f(x)\geq \sup_{x\in\cX}g(x) \Big)\leq (f(\bar x)-g(\bar x))^2.
\]
Employing Corollary \ref{cor:bobkov-tol} then completes the proof.
\end{proof}

\section{Combinatorial Glivenko-Cantelli}\label{chap:cgc}

\begin{definition}[$(\P_N,\pi)$-Glivenko-Cantelli] A collection of functions $\{f, f\in\cF\}$ is said to be weakly $(\P_N,\pi)$-Glivenko-Cantelli if
\[
\E\sup_{f\in\cF}\Big|(\PP_{\pi,n}-\P_N)f\Big| \to 0,
\]
and strongly $(\P_N,\pi)$-Glivenko-Cantelli if
\[
\sup_{f\in\cF}\Big|(\PP_{\pi,n}-\P_N)f\Big| \stackrel{a.s.}{\to} 0.
\]
\end{definition}

%\begin{definition} A pair $(\P_N,\cF)$ is said to be $\P$-pre-Glivenko-Cantelli if $\cF$ admits an envelope $F$ and 
%\end{definition}

\begin{theorem}[Bracketing, version I]\label{thm:cgc-1} Let $\cF$ be an arbitrary class of functions such that (a) it admits an envelope function $F>0$ satisfying 
\[
\limsup_{N\to\infty}\cN_{[]}(\cF,L^1(\P_N),\epsilon)<\infty \text{ for any }\epsilon>0, 
\]
and either (a) $\lim_{N\to\infty}\P_NF^2/\sqrt{n}=0$, or (b) there exists some fixed constant $c>0$ such that $\limsup_{N\to\infty}\P_N F^{1+c}<\infty$. We then have $\cF$ is weakly $(\P_N,\pi)$-Glivenko-Cantelli.
\end{theorem}

\begin{proof}
Using Corollary \ref{cor:5hcoi}, we have 
\begin{align*}
\E\sup_{f\in\cF}\Big|(\PP_{\pi,n}-\P_N)f\Big| &= \E\bnorm{\frac1n\sum_{i=1}^n(\delta_{z_{\pi^{-1}(i)}}-\P_N)}_{\cF}\\
&\leq \E\bnorm{\frac1n\sum_{i=1}^n(\delta_{\hat Z_i}-\P_N)}_{\cF},
\end{align*}
where $\hat Z_i$'s are i.i.d. drawn from $\P_N$. 

Now fix an $\epsilon>0$ and let $[l_i,u_i]$ be the $\epsilon$-brackets such that their union covers $\cF$ and $\P_N(u_i-l_i)<\epsilon$. Let 
\[
\PP_n:=\sum_{i=1}^n\delta_{\hat Z_i}. 
\]
We have, for every $f\in\cF$, 
\[
(\PP_n-\P_N)f\leq \PP_n u_i-\P_N f=(\PP_n-\P_N)u_i+\P_N(u_i-f)\leq (\PP_n-\P_N)u_i+\epsilon.
\] 
Accordingly,
\[
\sup_{f\in\cF}(\PP_n-\P_N)f\leq \max_i(\PP_n-\P_N)u_i+\epsilon.
\]

{\bf Case 1.} Now, if condition (a) holds, then combining the condition that 
\[
\limsup_{N\to\infty}\cN_{[]}(\cF,L^1(\P_N),\epsilon)<\infty
\]
with Lemma \ref{lem:orlicz} (by choosing $\psi(x)=x^2$) yields
\begin{align*}
\E\sup_{f\in\cF}(\PP_n-\P_N)f&\leq \sqrt{\cN_{[]}(\cF,L^1(\P_N),\epsilon)}\cdot \max_{i}\norm{(\PP_n-\P_N)u_i}_{L^2}+\epsilon\\
&\leq  \sqrt{\cN_{[]}(\cF,L^1(\P_N),\epsilon)}\cdot\sqrt{\frac{\P_NF^2}{n}}+\epsilon,
\end{align*}
so that, under theorem conditions, 
\[
\limsup_{N\to\infty}\E\sup_{f\in\cF}(\PP_n-\P_N)f \leq \epsilon.
\]
Symmetrically, we have 
\[
\liminf_{N\to\infty}\E\inf_{f\in\cF}(\P_N-\PP_n)f\geq -\epsilon.
\]
Combining the above two finishes the proof under condition (a).

{\bf Case 2.} Next, if condition (b) holds, then using Lemma \ref{lem:orlicz} (by choosing $\psi(x)=x$) yields
\begin{align*}
\E\sup_{f\in\cF}(\PP_n-\P_N)f&\leq \cN_{[]}(\cF,L^1(\P_N),\epsilon)\cdot \max_{i}\E\Big|(\PP_n-\P_N)u_i\Big|+\epsilon.
\end{align*}
Picking a sufficiently large constant $M>0$, we have 
\begin{align*}
\E\Big|(\PP_n-\P_N)u_i\Big|&\leq \E\Big|(\PP_n-\P_N)u_i\ind(F\leq M)\Big| + \E\Big|(\PP_n-\P_N)u_i\ind(F>M)\Big|\\
&\leq \sqrt{\frac{M^2}{n}} + 2\P_N F\ind(F>M),
\end{align*}
where the second term can be controlled as
\begin{align*}
\P_N F\ind(F>M) \leq \P_N\Big[\frac{F^{1+c}}{F^{c}}\ind(F>M)  \Big] \leq \frac{\P_NF^{1+c}}{M^c}.
\end{align*}
This yields
\[
\limsup_{N\to\infty} \E\Big|(\PP_n-\P_N)u_i\Big| \leq 2\frac{\limsup\limits_{N\to\infty}\P_NF^{1+c}}{M^{c}},
\]
the righthand side of which converges to 0 as $M\to\infty$. Accordingly, 
\[
\limsup_{M\to\infty}\limsup_{N\to\infty}\E\sup_{f\in\cF}(\PP_n-\P_N)f \leq \epsilon.
\]
The other side is similar, and we thus finish the proof.
\end{proof}

\begin{theorem}[Bracketing, version II]\label{thm:comb-gc} Let $\cF$ be an arbitrary class of functions such that (a) it admits an envelope $F>0$ with $\limsup_{N\to\infty}\P_NF^2<\infty$, and (b) for any fixed $\epsilon>0$,
\[
\lim_{N\to\infty}\frac{N\log2\cN_{[]}(\cF,L^1(\P_N),\epsilon)}{n^2}=0.
\]
It then holds true that $\cF$ is weakly $(\P_N,\pi)$-Glivenko-Cantelli.
\end{theorem}
\begin{proof}
Fix an $\epsilon>0$ and let $[l_i,u_i]$ be the $\epsilon$-brackets such that their union covers $\cF$ and $\P_N(u_i-l_i)<\epsilon$. It then holds true that, for any $f\in\cF$,
\[
(\PP_{\pi,n}-\P_N)f \leq (\PP_{\pi,n}-\P_N)u_i+\P_N(u_i-f)\leq (\PP_{\pi,n}-\P_N)u_i+\epsilon
\]
so that, integrating Theorem \ref{lem:comb-orlicz} into Lemma \ref{lem:orlicz},
\begin{align*}
&\E\sup_{f\in\cF}(\PP_{\pi,n}-\P_N)f \leq \E\max_{i}(\PP_{\pi,n}-\P_N)u_i+\epsilon\\
 \leq& \sqrt{\log2\cN_{[]}(\cF,L^1(\P_N),\epsilon)}\cdot \sqrt{\frac{12}{n}\Big(1+\frac{N}{n}\Big)}\max_i\norm{u_i-\P_N u_i}_{L^2(\P_N)}+\epsilon\\
 \leq& \norm{2F}_{L^2(\P_N)}\cdot \sqrt{\frac{24N\cdot\log2\cN_{[]}(\cF,L^1(\P_N),\epsilon)}{n^2}}+\epsilon.
\end{align*}
Similarly, 
\begin{align*}
&\E\sup_{f\in\cF}(\P_N-\PP_{\pi,n})f \leq \E\max_{i}(\P_N-\PP_{\pi,n})l_i+\epsilon\\
 \leq& \sqrt{\log2\cN_{[]}(\cF,L^1(\P_N),\epsilon)}\cdot \sqrt{\frac{12}{n}\Big(1+\frac{N}{n}\Big)}\max_i\norm{l_i-\P_N l_i}_{L^2(\P_N)}+\epsilon\\
 \leq& \norm{2F}_{L^2(\P_N)}\cdot \sqrt{\frac{24N\cdot\log2\cN_{[]}(\cF,L^1(\P_N),\epsilon)}{n^2}}+\epsilon.
\end{align*}
Combining the above two inequalities and taking first $N\to\infty$ and then $\epsilon \to 0$ thus complete the proof.
\end{proof}

\begin{theorem}[Bracketing, version III] Let $\cF$ be an arbitrary class of functions satisfying (a) there exists some fixed constant $c>0$ such that $\limsup_{N\to\infty}\P_N F^{1+c}<\infty$, and (b) for any fixed $\epsilon>0$,
\[
\lim_{N\to\infty}\frac{\log2\cN_{[]}(\cF,L^1(\P_N),\epsilon)}{\sqrt{n}}=0.
\]
It then holds true that $\cF$ is weakly $(\P_N,\pi)$-Glivenko-Cantelli.

\end{theorem}
\begin{proof}
We use a truncation trick. Similar to the last proof, fix an $\epsilon>0$ and let $[l_i,u_i]$ be the $\epsilon$-brackets such that their union covers $\cF$ and $\P_N(u_i-l_i)<\epsilon$. Further fix an $M>0$. It then holds true that, for any $f\in\cF$,
\begin{align*}
(\PP_{\pi,n}-\P_N)f\ind(F\leq M) &\leq (\PP_{\pi,n}-\P_N)u_i\ind(F\leq M)+\P_N(u_i-f)\ind(F\leq M)\\
&\leq (\PP_{\pi,n}-\P_N)u_i\ind(F\leq M)+\epsilon.
\end{align*}
Accordingly, employing Theorem \ref{lem:comb-orlicz-1} and plugging into Lemma \ref{lem:orlicz}, we obtain
\begin{align*}
&\E\sup_{f\in\cF}(\PP_{\pi,n}-\P_N)f\ind(F\leq M) \leq \E\max_{i}(\PP_{\pi,n}-\P_N)u_i\ind(F\leq M)+\epsilon\\
 \leq& \log2\cN_{[]}(\cF,L^1(\P_N),\epsilon)\cdot \Big(\frac{96\sqrt{2}M}{n}+\sqrt{\frac{288}{n\log 2}}\cdot 2M\Big)+\epsilon,
\end{align*}
which shall decay to $\epsilon$ as $N$ goes to infinity.

On the other hand,
\[
\E\sup_{f\in\cF}(\PP_{\pi,n}-\P_N)f\ind(F> M)\leq 2\P_N F\ind(F>M).
\]
By the same argument as in the proof of Theorem \ref{thm:cgc-1} Case 2, $\P_N F\ind(F>M)\to 0$ as $M\to\infty$. 

Combining the above two arguments and applying the same strategy to the lower bound then complete the proof.
\end{proof}

\begin{theorem}\label{thm:comb-gc-3} Assume that $(\P_N,\cF)$ is $\P$-pre-Glivenko-Cantelli. Suppose further that (a)  $\liminf_{N\to\infty} (n/N)>0$, and (b) for any fixed $\epsilon>0$, we have
\begin{align*}
\lim_{N\to\infty}\frac{N\log\cN(\cF,L^1(\P_N),\epsilon\norm{F}_{L^1(\P_N)})}{n^2}=0.
\end{align*}
It then holds true that $\cF$ is weakly $(\P_N,\pi)$-Glivenko-Cantelli.
\end{theorem}
\begin{proof}
{\bf Step 1.} Similar to the proof of Theorem \ref{thm:gc}, for any constant $M>0$, let's introduce
\[
\cF_M:=\Big\{f\ind(F\leq M);f\in\cF\Big\}.
\]
It then holds true that
\begin{align*}
&\E\bnorm{\PP_{\pi,n}-\P_N}_{\cF} \\
\leq& \E\sup_{f\in\cF}\Big|\frac{1}{n}\sum_{\pi(i)\leq n} f(z_i)\ind(F(z_i)\leq M)-\P_Nf\ind(F\leq M)\Big|+\\
&\quad \E\sup_{f\in\cF}\Big|\frac{1}{n}\sum_{\pi(i)\leq n} f(z_i)\ind(F(z_i)> M)-\P_Nf\ind(F>M)\Big|\\
\leq& \E\bnorm{\PP_{\pi,n}-\P_N}_{\cF_M}+2\P_N F\ind(F>M).
\end{align*}

{\bf Step 2.} Now, for any $g_1,g_2\in\cF_M$ such that $\norm{g_1-g_2}_{L^1(\P_{N})}=\P_{N}|g_1-g_2|\leq \epsilon$, we have
\[
\norm{g_1-g_2}_{L^1(\PP_{\pi,n})}=\frac{1}{n}\sum_{i=1}^N\Big|g_1(X_i)-g_2(X_i)\Big|\ind(\pi(i)\leq n)\leq \frac{N}{n}\cdot \P_N|g_1-g_2|\leq \frac{N\epsilon}{n},
\]
and accordingly
\[
\norm{g_1-\P_Ng_1-(g_2-\P_Ng_2)}_{L^1(\PP_{\pi,n})} \leq \Big(1+\frac{N}{n}\Big)\epsilon,
\]
where in the last inequality we used the fact that $|\P_N(g_1-g_2)|\leq \P_N|g_1-g_2|$.

Thusly, letting $T_\epsilon$ denote the set of centers in the $\epsilon$-net of $(\cF_M,L^1(\P_N))$, we obtain
\[
\E\bnorm{\PP_{\pi,n}-\P}_{\cF}  \leq \E\sup_{g\in T_{\epsilon}}\Big|(\PP_{\pi,n}-\P_N)g\Big| + \Big(1+\frac{N}{n}\Big)\epsilon + 2\P_NF\ind(F>M).
\]

{\bf Step 3.} Integrating Theorem \ref{lem:comb-orlicz} into Lemma \ref{lem:orlicz}, we have
\begin{align*}
\E\sup_{g\in T_{\epsilon}}\Big|(\PP_{\pi,n}-\P_N)g \Big| \leq 2M\cdot \sqrt{\frac{12}{n}\Big(1+\frac{N}{n}\Big)\cdot\log 2\cN(\cF_M,L^1(\P_N),\epsilon)}.
\end{align*}
%On the other hand, if integrating Lemma \ref{lem:comb-orlicz-1} into Lemma \ref{lem:orlicz}, we have
%\begin{align*}
%\E\sup_{g\in T_{\epsilon}}\Big|(\PP_{\pi,n}-\P_N)g \Big| \leq \log2\cN(\cF_M,L^1(\P_N),\epsilon)\cdot \Big(\frac{96\sqrt{2}M}{n}+\sqrt{\frac{288}{n\log 2}}\cdot 2M\Big).
%\end{align*}

Employing the same arguments as in the proof of Theorem \ref{thm:gc} then completes the proof.
\end{proof}

Lastly, leveraging Theorem \ref{thm:tol-tala}, the above results could all yield strong Glivenko-Cantelli versions.

\begin{theorem}\label{thm:perm-strong-LLN} Using the notation in Theorem \ref{thm:tol-tala}. Suppose that $\P_N$ satisfies that
\[
\liminf_{N\to\infty}\frac{n^2}{N\Sigma_{\cF}^2\log n}>1.
\]
Then all the above theorems also yield a version of strong Glivenkp-Cantelli. 
\end{theorem}
\begin{proof}
Use Theorem \ref{thm:tol-tala} and the first Borel-Cantelli lemma.
\end{proof}

\section{Combinatorial Donsker}\label{chap:cd}

\begin{definition}[$(\P_N,\pi)$-Donsker] Remind that $\GG_{\pi,n}f=\sqrt{n}(\PP_{\pi,n}-\P_Nf)$. A collection of functions $\{f, f\in\cF\}$ is said to be $(\P_N,\pi)$-Donsker if
\[
\Big\{\GG_{\pi,n}f, f\in\cF\Big\} \Rightarrow \mathbb{G},
\]
where $\mathbb{G}$ is a tight Borel probability measure in $L^{\infty}(\cF)$.
\end{definition}

\begin{definition}[$\P$-pre-Donsker] A pair $(\P_N,\cF)$ is said to be $\P$-pre-Donsker if  there exists a probability measure $\P$ over a measurable space $(\bar \cZ,\cA)$ such that (a) $(\cF,L^2(\P))$ is totally bounded; and (b) any $f,g\in\cF$ are $\cA$-measurable and satisfy
\begin{align*}
\limsup_{N\to\infty}\P_N|f-\P_N f|^3<\infty,~
\P_N f \to \P f, ~{\rm and}~\P_N(fg)\to \P (fg).
\end{align*}
\end{definition}

\begin{lemma}[Separability]\label{lem:comb-sep} The stochastic process $\{\GG_{\pi,n}f;f\in\cF\}$ is separable in $(\cF,L^2(\P_N))$ and $(\cF,L^{\infty})$ if the latter pseudo-metric space itself is separable.
\end{lemma}
\begin{proof}
We have, for any $f,g\in\cF$,
\begin{align*}
\Big|\GG_{\pi,n}(f-g)\Big|&=\Big|\frac{1}{\sqrt{n}}\sum_{i=1}^N[f(z_i)-\P_Nf-(g(z_i)-\P_Ng)]\ind(\pi(i)\leq n) \Big|\\
&\leq \frac{1}{\sqrt{n}}\sum_{i=1}^N|f(z_i)-g(z_i)|+\sqrt{n}|\P_N(f-g)|\\
&\leq \frac{N+n}{\sqrt{n}}\norm{f-g}_{L^2(\P_N)}\\
&\leq \frac{N+n}{\sqrt{n}}\norm{f-g}_{L^\infty}.
\end{align*}
Accordingly, $\GG_{\pi,n}f$ is continuous in $(\cF,L^2(\P_N)$ and also $(\cF,L^{\infty})$ if the latter space is separable, and thus also separable.  
\end{proof}

\begin{lemma}[Stochastic equicontinuity]\label{lem:combinatorial-sec} Assume 
\[
\int_0^{\infty}\sqrt{\log 2\cN(\cF,L^2(\P_N),\epsilon)}\d \epsilon<\infty. 
\]
It then holds holds true that, for any $\delta>0$,
\begin{align*}
\E\sup_{\norm{f-g}_{L^2(\P_N)}\leq \delta}\Big|\GG_{\pi,n}(f-g)\Big| \leq 72\sqrt{\frac{6N}{n}}\int_0^\delta \sqrt{\log2\cN(\cF,L^2(\P_N),\epsilon)}\d\epsilon.
\end{align*}
\end{lemma}
\begin{proof}
By condition, for any $\epsilon>0$, we have $\log 2\cN(\cF,L_2(\P_N),\epsilon)<\infty$. Accordingly, the pseudo-metric space $(\cF,L^2(\P_N))$ is separable. 

In view of Theorem \ref{lem:comb-orlicz}, for applying Corollary \ref{cor:dudley3}, it remains to show that $\{\GG_{\pi,n}f\}$ is separable, which has been done in Lemma \ref{lem:comb-sep}, Invoking Corollary \ref{cor:dudley3} then completes the proof.
\end{proof}

\begin{theorem}[Donsker, version I]\label{thm:combinatorial-Donsker} Assume (a) $\lim_{N\to\infty} n/N=\gamma\in (0,1)$; (b) $(\P_N,\cF)$ is $\P$-pre-Donsker; and (c)
\[
\limsup_{N\to\infty}\int_0^{\diam(\cF)}\sqrt{\log 2\cN(\cF,L^2(\P_N),\epsilon)}\d \epsilon<\infty.
\]
The class $\cF$ is then $(\P_N,\pi)$-Donsker.
\end{theorem}
\begin{proof}
We use Theorem \ref{thm:weak-convergence-key} and Lemma \ref{lem:combinatorial-sec} to prove Theorem \ref{thm:combinatorial-Donsker}. 

{\bf Step 1.} Theorem \ref{thm:CLT-perm-sampling} verified that, for any finitely many functions $f_1,\ldots,f_m\in\cF$, denoting $
\bm{f}=(f_1,f_2,\ldots,f_m)^\top$, we have $\GG_{\pi,n}\bm{f}$ is asymptotically normal. 

{\bf Step 2.} We then verify stochastic equicontinuity in $(\cF,L^2(\P))$. To this end, fixing an $\epsilon>0$, let's first construct balls $B_1,\ldots,B_{\cN(\cF,L^2(\P_N),\epsilon)}$ to cover $\cF$ in $L^2(\P_N)$ metric, and then let $T_{\epsilon}$ denote the union of centers of radius-$\epsilon$ balls that cover balls $B_j$'s in $L^2(\P)$ norm. By condition, it then holds true that
\[
|T_\epsilon| \leq \cN(\cF,L^2(\P_N),\epsilon)\times \cN(\cF,L^2(\P),\epsilon)<\infty.
\]
We then have, for any $f\in\cF$, there exist $f_{\epsilon}\in T_{\epsilon}$ such that
$\norm{f-f_{\epsilon}}_{L^2(\P_N)}\leq 2\epsilon$ and $\norm{f-f_{\epsilon}}_{L^2(\P)}\leq \epsilon$. Furthermore, for any $f,g\in\cF$,
\begin{align*}
\norm{f-g}_{L^2(\P_N)}&\leq \norm{f-f_{\epsilon}}_{L^2(\P_N)} + \norm{f_{\epsilon}-g_{\epsilon}}_{L^2(\P_N)}+\norm{g-g_{\epsilon}}_{L^2(\P_N)}\\
&\leq 4\epsilon+\norm{f_{\epsilon}-g_{\epsilon}}_{L^2(\P_N)},
\end{align*}
and similarly,
\begin{align*}
\norm{f-g}_{L^2(\P)}&\geq \norm{f_{\epsilon}-g_{\epsilon}}_{L^2(\P)}-\norm{f-f_{\epsilon}}_{L^2(\P)} -\norm{g-g_{\epsilon}}_{L^2(\P)}\\
&\geq \norm{f_{\epsilon}-g_{\epsilon}}_{L^2(\P)}-2\epsilon.
\end{align*}
Accordingly,
\[
\sup_{f,g\in\cF}\Big\{\norm{f-g}_{L^2(\P_N)}-\norm{f-g}_{L^2(\P)}\Big\} \leq 6\epsilon+\sup_{f_\epsilon,g_\epsilon\in T_{\epsilon}}\Big\{\norm{f_{\epsilon}-g_{\epsilon}}_{L^2(\P_N)}- \norm{f_{\epsilon}-g_{\epsilon}}_{L^2(\P)}\Big\}.
\]
Now for any $f,g\in\cF$, we have
\[
\norm{f-g}^2_{L^2(\P_N)}=\P_N f^2+ \P_N g^2-2\P_N f \P_N g \to \P f^2+ \P g^2-2\P f \P_N g=\norm{f-g}^2_{L^2(\P)}
\]
so that
\[
\norm{f_{\epsilon}-g_{\epsilon}}_{L^2(\P_N)}- \norm{f_{\epsilon}-g_{\epsilon}}_{L^2(\P)} \to 0.
\]
Accordingly, 
\[
\lim_{N\to\infty}\sup_{f,g\in\cF}\Big\{\norm{f-g}_{L^2(\P_N)}-\norm{f-g}_{L^2(\P)}\Big\} = 0,
\]
and thus,
\begin{align*}
&\lim_{\delta\downarrow 0}\limsup_{N\to\infty}\E\sup_{\norm{f-g}_{L^2(\P)}\leq \delta}\Big|\GG_{\pi,n}(f-g)\Big|
\leq \lim_{\delta\downarrow 0}\limsup_{N\to\infty}\E\sup_{\norm{f-g}_{L^2(\P_N)}\leq 2\delta}\Big|\GG_{\pi,n}(f-g)\Big|.
\end{align*}
By condition that
\[
\limsup_{N\to\infty}\int_0^{\diam(\cF)}\sqrt{\log 2\cN(\cF,L_2(\P_N),\epsilon)}\d \epsilon<\infty
\]
and the dominated convergence theorem, the righthand side limit goes to 0, which then completes the proof.
\end{proof}

\begin{corollary}[Uniform entropy argument] Assume the first two conditions in Theorem \ref{thm:combinatorial-Donsker} hold. If in addition $\cF$ admits an envelope function $F>0$ such that $\limsup_{N\to\infty}\P_N F^2<\infty$. We then have, if $\cF$ further satisfies 
\[
\int_0^{2}\sup_Q\sqrt{\log 2\cN(\cF,L_2(Q),\epsilon\|F\|_{L_2(Q)})}\d \epsilon<\infty,
\]
where the supremum is taken over all finitely discrete probability measure, then $\cF$ is $(\P_N,\pi)$-Donsker.
\end{corollary}
\begin{proof}
We have
\begin{align*}
&\limsup_{N\to\infty}\int_0^{\diam(\cF)}\sqrt{\log 2\cN(\cF,L_2(\P_N),\epsilon)}\d \epsilon \\
\leq &\limsup_{N\to\infty}\int_0^{2\norm{F}_{L_2(\P_N)}}\sqrt{\log 2\cN(\cF,L_2(\P_N),\epsilon)}\d \epsilon \\
=& \limsup_{N\to\infty}\norm{2F}_{L_2(\P_N)}\int_0^{2}\sqrt{\log 2\cN(\cF,L_2(\P_N),\epsilon\norm{F}_{L_2(\P_N)})}\d \epsilon\\
\leq& \limsup_{N\to\infty} 2\norm{F}_{L_2(\P_N)}\cdot \int_0^{2}\sup_Q\sqrt{\log 2\cN(\cF,L_2(Q),\epsilon\|F\|_{L_2(Q)})}\d \epsilon\\
<&\infty,
\end{align*}
which completes the proof.
\end{proof}

\begin{theorem}[Donsker, version II]\label{thm:cdonsker} Assume that $\cF$ admits an envelope $F>0$ such that (a) $(\P_N,\cF)$ is $\P$-pre-Donsker; (b) $\lim_{N\to\infty} n/N=\gamma\in [0,1)$; (c) it holds true that 
\[
\int_0^{2}\sup_{\Q}\sqrt{\log \cN(\cF,L^2(\Q),\epsilon\norm{F}_{L^2(\Q)})}\d \epsilon<\infty,
\]
where the supremum is over all finitely discrete probability measures; and (d) $\limsup_{N\to\infty}\P_NF^2<\infty$. Then $\cF$ is $(\P_N,\pi)$-Donsker.
\end{theorem}
\begin{proof}
Using Theorem \ref{thm:hoeffding-amazing}, we obtain
\begin{align*}
&\lim_{\delta\downarrow 0}\limsup_{N\to\infty}\E\sup_{\norm{f-g}_{L^2(\P)}\leq \delta}\Big|\GG_{\pi,n}(f-g)\Big|\\
\leq& \lim_{\delta\downarrow 0}\limsup_{N\to\infty}\E\sup_{\norm{f-g}_{L^2(\P_N)}\leq 2\delta}\Big|\GG_{\pi,n}(f-g)\Big|\\
\leq&  \lim_{\delta\downarrow 0}\limsup_{N\to\infty}\E\sup_{\norm{f-g}_{L^2(\P_N)}\leq 2\delta}\Big|\GG_{n}(f-g)\Big|.
\end{align*}
Here $\GG_n:=\sqrt{n}(\PP_n-\P_N)$, with $\PP_n:=n^{-1}\sum_{i=1}^n\delta_{\hat Z_i}$ and $\hat Z_i$'s i.i.d. drawn from $\P_N$. Invoking Corollary \ref{cor:donsker} and noting that conditional on the data, the Rademacher sequence in the proof of Theorem \ref{thm:master} is separable in $L^2(\PP_n)$, we then complete the proof. 
\end{proof}

\begin{theorem}[Donsker, version III] Assume that $\cF$ satisfies (a) $(\P_N,\cF)$ is $\P$-pre-Donsker; (b) $\lim_{N\to\infty} n/N=\gamma\in [0,1)$; (c)
\[
\limsup_{N\to\infty}\int_0^{\diam(\cF)}\sqrt{\log \cN_{[]}(\cF,L^2(\P_N),\epsilon)}\d \epsilon<\infty;
\]
and (d) $\limsup_{N\to\infty}\P_NF^2<\infty$. Then $\cF$ is $(\P_N,\pi)$-Donsker.
\end{theorem}
\begin{proof}
Combining the argument in the proof of Theorem \ref{thm:cdonsker} with the bracketing uniform central limit theorem in Theorem 2.5.6 of \cite{vaart1996empirical}.
\end{proof}

\section{Combinatorial generic chaining}

This section aims to convey the results of Michel Talagrand that tightens the entropy bounds of Dudley. This is through the use of {\it generic chaining}. We further discuss how to use generic chaining to derive alternative bounds to those in Chapter \ref{chap:cd} with the assumption that $\liminf n/N>0$ being waived.

First of all, define
\[
N_0=1, ~~~{\rm and}~~N_n=2^{2^n} \text{ for }n\geq 1
\]
to represent the size of partitions of a general pseudo-metric space $( T,d)$.

\begin{definition}
For any pseudo-metric space $( T,d)$, any $t\in T$, and any $T_n\subset T$, define 
\[
d(t,T_n)=\inf_{s\in T_n}d(t,s)~~~{\rm and}~~~e_n( T,d)=\inf_{T_n}\sup_{t}d(t,T_n),
\]
where the infimum is taken over all $|T_n|\leq N_n$.
\end{definition}

\begin{definition}
For any pseudo-metric space $( T,d)$ and any $\alpha>0$, define
\[
\gamma_{\alpha}( T,d)=\inf\sup_{t\in T}\sum_{n\geq 0}2^{n/\alpha}\diam(A_n(t)),
\]
where the infimum is taken over all increasing sequences $\cA_n$ of partitions of $ T$ such that $|\cA_n|\leq N_n$, and $A_n(t)$ denotes the unique element in $\cA_n$ that contains $t$. 
\end{definition}

Talagrand gave the following bound that improves on Dudley's chaining bound (i.e., Theorem \ref{thm:dudley1}).
\begin{theorem}[Talagrand's generic chaining bound]\label{thm:generic-chaining} Suppose that $\{X(t);t\in T\}$ is a {\it centered} and {\it separable} stochastic process such that
\[
\norm{X(t)}_{\psi_2}<\infty~~~{\rm and}~~~\norm{X(s)-X(t)}_{\psi_2}\leq d(s,t),~~\text{ for any }s,t\in T.
\]
There then exists a universal constant $K>0$ such that 
\[
\E\sup_{t\in T}X_t\leq K\gamma_2( T,d).
\]
\end{theorem}
\begin{proof}
See the proof of Theorem 2.2.22 in \cite{talagrand2014upper}.
\end{proof}

One of the truly marvelous insights delivered by Talagrand's generic chaining bound is that the corresponding bound, unlike Dudley's, is actually {\it sharp} (up to some constants, of course). This statement can be formalized rigorously when $\{X(t);t\in T\}$ is further assumed to be a Gaussian process. 

\begin{theorem}[Majorizing measure theorem] Under the conditions of Theorem \ref{thm:generic-chaining}, if we further assume that  $\{X(t);t\in T\}$ is Gaussian process, then there exists a universal constant $K>0$ such that
\[
\frac{1}{K}\gamma_2( T,d)\leq \E\sup_{t\in T}X_t \leq K\gamma_2( T,d).
\]
\end{theorem}
\begin{proof}
See the proof of Theorem 2.4.1 in \cite{talagrand2014upper}.
\end{proof}

Our interest in this section, however, is in employing the generic chaining techniques for bounding the suprema of stochastic processes under Bernstein-type tail conditions. This was stated in the following theorem.

\begin{theorem}[Bernstein chaining]\label{thm:talagrand-chaining-bernstein} Consider $ T$ to be equipped with two metrics $d_1,d_2$ such that the stochastic process $\{X(t);t\in T\}$ is centered, separable (in both $( T,d_1)$ and $( T,d_2)$), and satisfies
\begin{align}\label{eq:bchaining-condition}
\P(|X(s)-X(t)|\geq u)\leq 2\exp\Big(-\frac{u^2/2}{d_2(s,t)^2+ud_1(s,t)}  \Big),~~\text{for all }\notag\\
s,t\in T \text{ and }u>0.
\end{align}
There then exists a universal constant $K>0$ such that
\[
\E\sup_{s,t\in T}|X(s)-X(t)|\leq K\Big\{\gamma_1( T,d_1)+\gamma_2( T,d_2)\Big\}.
\]
\end{theorem}
\begin{proof}
Theorem 2.2.23 in \cite{talagrand2014upper}.
\end{proof}

A direct implication of the above theorem is the following Dudley-typer metric entropy bounds with two tails being explicitly separated out.

\begin{corollary}\label{cor:talagrand} Consider the stochastic process $\{X(t);t\in T\}$ that satisfies the conditions of Theorem \ref{thm:talagrand-chaining-bernstein}. It then holds true that
\[
\E\sup_{s,t\in T}|X(s)-X(t)|\leq K\Big\{\int_0^{\diam(T;d_2)}\sqrt{\log 2\cN( T,d_2,\epsilon)}\d \epsilon + \int_0^{\diam(T;d_1)}\log 2\cN( T,d_1,\epsilon)\d \epsilon\Big\},
\]
where $K>0$ is a universal constant.
\end{corollary}
\begin{proof}
The proof is based on the following lemma that connects $\gamma_{\alpha}( T,d)$ to $e_n( T,d)$.
\begin{lemma}[Corollary 2.3.2 in \cite{talagrand2014upper}] For any $\alpha>0$ and any metric space $( T,d)$, it holds true that
\[
\gamma_{\alpha}( T,d)\leq K(\alpha)\sum_{n\geq 0}2^{n/\alpha}e_n( T,d).
\]
\end{lemma}
Using the above lemma, we could then separately bound $\gamma_1( T,d_1)$ and $\gamma_2( T,d_2)$ in the following manner.

{\bf Bounding $\gamma_2( T,d_2)$.} Noticing that 
\[
e_n( T,d_2)=\inf\Big\{\epsilon: \cN( T,d_2,\epsilon)\leq N_n  \Big\},
\]
we obtain
\[
\epsilon<e_n( T,d_2) ~~\text{ implies }~~\cN( T,d_2,\epsilon)\geq 1+N_n.
\]
Accordingly, we have
\begin{align*}
\int_{e_{n+1}( T,d_2)}^{e_n( T,d_2)}\sqrt{\log\cN( T,d_2,\epsilon)}\d\epsilon &\geq \sqrt{\log(1+N_n)}(e_n( T,d_2)-e_{n+1}( T,d_2))\\
&\geq \sqrt{\log 2}\cdot 2^{n/2}(e_n( T,d_2)-e_{n+1}( T,d_2)),
\end{align*}
where in the second inequality we used the fact that 
\[
\log(1+N_n)\geq 2^n\log 2.
\]
Summing over all $n\geq 0$, we obtain
\begin{align*}
\int_0^{e_0( T,d_2)}\sqrt{\log \cN( T,d_2,\epsilon)}\d\epsilon\geq& \sqrt{\log 2}\sum_{n\geq 0}2^{n/2}(e_n( T,d_2)-e_{n+1}( T,d_2))\\
\geq & \Big(1-\frac{1}{\sqrt{2}}\Big)\sqrt{\log 2}\sum_{n\geq 0}2^{n/2}e_n( T,d_2),
\end{align*}
which gives the desired bound that
\[
\gamma_2( T,d_2)\leq K\int_0^{\diam(T;d_2)}\sqrt{\log 2\cN( T,d_2,\epsilon)}\d\epsilon.
\]

{\bf Bounding $\gamma_1( T,d_1)$.} The derivation is similar. We have
\begin{align*}
\int_{e_{n+1}( T,d_1)}^{e_n( T,d_1)}\log\cN( T,d_1,\epsilon)\d\epsilon &\geq \log(1+N_n)(e_n( T,d_1)-e_{n+1}( T,d_1))\\
&\geq \log 2\cdot 2^{n}(e_n( T,d_1)-e_{n+1}( T,d_1)),
\end{align*}
implying
\begin{align*}
\int_0^{e_0( T,d_1)}\log \cN( T,d_1,\epsilon)\d\epsilon\geq& \log 2\sum_{n\geq 0}2^{n}(e_n( T,d_1)-e_{n+1}( T,d_1))\\
= & \frac{\log 2}{2}\sum_{n\geq 0}2^{n}e_n( T,d_1),
\end{align*}
so that
\[
\gamma_1( T,d_1)\leq K\int_0^{\diam(T;d_1)}\log 2\cN( T,d_1,\epsilon)\d\epsilon.
\]

Combining the $\gamma_2( T,d_2)$ and $\gamma_1( T,d_1)$ bounds then completes the proof. 
\end{proof}

\begin{corollary} Suppose $\cF$ admits an envelope $F\geq 1$ and a sequence of positive numbers $\epsilon_N\to \infty$ such that 
\begin{align*}
\limsup_{N\to\infty} P_NF^2<\infty,~~~\lim_{N\to\infty}\frac{\int_0^{2\epsilon_N\sqrt{n}}\log 2\cN(\cF,L^\infty,\epsilon)\d\epsilon}{\sqrt{n}}=0,\\
{\rm and}~~~\limsup_{N\to\infty}\int_0^2\sqrt{\log 2\cN(\cF,L^2(\P_N),\epsilon\norm{F}_{L^2(\P_N)})}\d\epsilon<\infty.
\end{align*}
We then have
\[
\lim_{\delta\to 0}\limsup_{N\to\infty}\E\sup_{\norm{f-g}_{L^2(\P_N)}\leq \delta}\Big|\GG_{\pi,n}(f-g) \Big|=0.
\]
If further the second condition in Theorem \ref{thm:combinatorial-Donsker} holds,  then $\cF$ is $(\P_N,\pi)$-Donsker.
\end{corollary}
\begin{proof}
{\bf Step 1.} Introduce a small fixed constant $\delta>0$ and a large fixed constant $M>0$. Similar to the proof of Theorem \ref{thm:comb-gc-3}, define
\[
\cF_M:=\Big\{f\ind(F\leq M);f\in\cF\Big\}~~~{\rm and}~~~\cF_M^{\rm diff}:=\Big\{f-g;f,g\in\cF_M,\norm{f-g}_{L^2(\P_N)}\leq \delta\Big\}.
\]
By applying a similar analysis as Lemma \ref{lem:comb-orlicz-1}, we could then find two universal constants $K_1,K_2>0$ such that 
\[
\Big\{\GG_{\pi,n}f; f\in \cF_M^{\rm diff} \Big\}
\]
satisfies \eqref{eq:bchaining-condition} with
\[
d_1(f,g)=\frac{K_1}{\sqrt{n}}\norm{f-g}_{\infty}~~~{\rm and }~~~d_2(f,g)=K_2\norm{f-g}_{L^2(\P_N)}.
\]
In addition, by a similar argument as made in Lemma \ref{lem:comb-sep}, the process is separable and also centered. Accordingly, all the conditions in Corollary \ref{cor:talagrand} are satisfied. 

{\bf Step 2.} Combining the above arguments with Lemma \ref{lem:diff-covering}, we obtain that there exists a universal constant $C>0$ such that
\begin{align*}
&\E\sup_{f\in\cF_M^{\rm diff}}|\GG_{\pi,n}f| \\
\lesssim &  \P_N|\GG_{\pi,n}f_0|+ \int_0^{\diam(\cF_M^{\rm diff};d_2)}\sqrt{\log 2\cN(\cF_M^{\rm diff},d_2,\epsilon)}\d\epsilon + \int_0^{\diam(\cF_M^{\rm diff};d_1)}\log 2\cN(\cF_M^{\rm diff},d_1,\epsilon)\d\epsilon\\
 \lesssim & \delta + \int_0^{\diam(\cF_M^{\rm diff};d_2)}\sqrt{\log 2\cN(\cF_M,d_2,\epsilon/2)}\d\epsilon +\int_0^{\diam(\cF_M^{\rm diff};d_1)}\log 2\cN(\cF_M,d_1,\epsilon/2)\d\epsilon\\
 \lesssim& \delta + \int_0^{C\delta}\sqrt{\log 2\cN(\cF,L^2(\P_N),\epsilon)}\d\epsilon +\frac{1}{\sqrt{n}}\int_0^{2M}\log 2\cN(\cF_M,L^\infty,\epsilon)\d\epsilon\\
 \lesssim & \delta + \norm{F}_{L^2(\P_N)}\int_0^{C\delta}\sqrt{\log 2\cN(\cF,L^2(\P_N),\epsilon\norm{F}_{L^2(\P_N)})}\d\epsilon +\\
 &~~~\frac{1}{\sqrt{n}}\int_0^{2M}\log 2\cN(\cF_M,L^\infty,\epsilon)\d\epsilon.
\end{align*}

{\bf Step 3.} It remains to handle the part of $F>M$. Let's then define
\[
\bar\cF_M:=\Big\{f\ind(F> M);f\in\cF\Big\}~~~{\rm and}~~~\bar\cF_M^{\rm diff}:=\Big\{f-g;f,g\in\bar\cF_M,\norm{f-g}_{L^2(\P_N)}\leq \delta\Big\}.
\]
For them, we have
\begin{align*}
\E\sup_{f\in\bar\cF_M^{\rm diff}}|\GG_{\pi,n}f| \lesssim \sqrt{n}\P_NF\ind(F>M)\leq \P_NF^2 \cdot \frac{\sqrt{n}}{M}.
\end{align*}

Combining what we obtained and using the theorem conditions, the proof is thus done.
\end{proof}

\section{Examples}\label{sec:5examples}

\subsection{Example: Two-sample permutation processes}

This section considers the two-sample problem in the finite-population setting. Let 
\[
\Big\{x_i=x_{i,N},i\in[m], y_j=y_{j,N}, j\in[n]\Big\}  
\]
be $N=m+n$ non-random $\cZ$-valued points with $m=m_N$, $n=n_N$, and $\gamma_N:=m/N$ converges to some constant $\gamma\in(0,1)$ as $N\to\infty$. Write 
\[
\Big\{z_k=x_k\ind(k\leq m)+y_{k-m}\ind(k>m);k\in[N]\Big\}
\]
to represent the whole set. Let 
\[
\P_N=\frac{1}{m}\sum_{i=1}^m\delta_{x_i},~~~\Q_N=\frac{1}{n}\sum_{j=1}^n\delta_{y_i},~~~{\rm and}~~~\H_N=\frac{1}{N}\sum_{k=1}^N\delta_{z_i}=\frac{m}{N}\P_N+\frac{n}{N}\Q_N
\]
represent the probability measures of $\{x_i;i\in[m]\}$, $\{y_j;j\in[n]\}$, and $\{z_k;k\in[N]\}$, respectively. 

Assume that there exist probability measures $\P$ and $\Q$ over the measurable space $(\bar\cZ,\cZ)$ such that
\[
\P_N\Rightarrow \P~~~{\rm and}~~~\Q_N\Rightarrow \Q.
\]
We then immediately have $\H_N\Rightarrow \gamma \P+(1-\gamma)\Q=:\H$. In statistics, $\H$ is called the {\it mixture probability measure} or {\it mixture distribution} that mixes $\P$ and $\Q$.

Now let's consider the permutation measures
\[
\PP_{\pi,m}:=\frac{1}{m}\sum_{\pi(i)\leq m}\delta_{z_i}~~~{\rm and }~~~\QQ_{\pi,n}:=\frac{1}{n}\sum_{\pi(i)> m}\delta_{z_i}.
\]

The following corollary is a direct consequence of Theorem \ref{thm:combinatorial-Donsker}.

\begin{corollary}[Two-sample permutation process]\label{cor:two-sample-perm}
Assume 
\begin{enumerate}[label=(\roman*)]
\item $\lim_{N\to\infty} \gamma_N=\gamma\in (0,1)$;
\item  $(\P_N,\cF)$ and $(\Q_N,\cF)$ are $\P$-pre-Donsker and $\Q$-pre-Donsker, respectively;
\item it holds true that
\[
\limsup_{N\to\infty}\int_0^{\diam(\cF)}\sqrt{\log 2\cN(\cF,L_2(\P_N),\epsilon)}\d \epsilon<\infty
\]
and
\[
\limsup_{N\to\infty}\int_0^{\diam(\cF)}\sqrt{\log 2\cN(\cF,L_2(\Q_N),\epsilon)}\d \epsilon<\infty.
\]
\end{enumerate}
We then have 
\[
\Big\{\sqrt{m}(\PP_{\pi,m}-\H_N)f;f\in\cF\Big\} \Rightarrow \sqrt{1-\gamma}\cdot \BB_{\H},
\]
where $\BB_{\H}$ is a mean-zero Brownian bridge satisfying that, for any $f,g\in\cF$,
\[
\cov(\BB_{\H}(f),\BB_{\H}(g))=\H(f-\H f)(g-\H g).
\]
\end{corollary} 

\begin{exercise}
Prove Corollary \ref{cor:two-sample-perm}.
\end{exercise}

Corollary \ref{cor:two-sample-perm}, combined with the continuous mapping theorem (Theorem \ref{thm:cmt}), yields that
\[
\sqrt{\frac{mn}{m+1}}\bnorm{\PP_{\pi,m}-\QQ_{\pi,n}}_{\cF} = \frac{1}{\sqrt{1-\gamma_N}}\cdot \sqrt{m}\bnorm{\PP_{\pi,m}-\H_N}_{\cF}
\]
weakly converges to $\norm{\BB_{\cH}}_{\cF}$. We thus recover the famous two-sample Kolmogorov-Smirnov theorem in the finite population paradigm; see, e.g., Chapter 3.7 of \cite{vaart1996empirical}. The finite-population paradigm this time, again, avoids any use of outer measure and gives an elementary proof.

\subsection{Example: Mean absolute deviation}

Let's study the mean absolute deviation estimator
\[
\hat M_n=\frac{1}{n}\sum_{\pi(i)\leq n}\Big|z_i-\PP_{\pi,n}z\Big|
\]
in the finite population paradigm with $n/N\to \gamma\in(0,1)$ as $N\to\infty$. Introduce 
\[
M_N:=\P_N|z-\P_Nz|.
\]

\begin{proposition} It holds true that
\[
\E|\hat M_n-M_N|\lesssim \frac{\norm{z-\P_Nz}_{L^2(\P_N)}}{\sqrt{n}}.
\]
\end{proposition}
\begin{proof}
By triangle inequality,
\begin{align*}
|\hat M_n-M_N| &\leq \Big|\PP_{\pi,n}\Big(|z-\PP_{\pi,n}z|-|z-\P_Nz|\Big)\Big| + \Big|(\PP_{\pi,n}-\P_N)|z-\P_Nz|\Big|\\
&\leq \Big|(\PP_{\pi,n}-\P_N)z \Big| + \Big|(\PP_{\pi,n}-\P_N)|z-\P_Nz|\Big|.
\end{align*}
Invoking Lemma \ref{lem:comb-orlicz}, we obtain that the first term can be upper bounded as
\[
\E\Big|(\PP_{\pi,n}-\P_N)z \Big| \lesssim \frac{1}{\sqrt{n}}\norm{z-\P_Nz}_{L^2(\P_N)}.
\]
The second term can be similarly bounded, so that the proof is done.
\end{proof}

Now let's examine the limiting distribution of $\hat M_n$. The asymptotic setting in our mind is
\[
\lim_{N\to\infty}\frac{n}{N}=\gamma\in(0,1),~~~\P_N\Rightarrow \P ~~{\rm as}~~ N\to\infty, 
\]
and $\P$ admits a Lebesgue density. Introduce the function class
\[
\cF:=\Big\{f_t(x):=\Big|x-t\Big|;t\in [\P z-\delta_0, \P z+\delta_0] \Big\}
\]
for some fixed constant $\delta_0>0$. We can then decompose
\begin{align*}
\sqrt{n}(\hat M_n-M_N)=&\sqrt{n}(\PP_{\pi,n}f_{\PP_{\pi,n}z}-\P_Nf_{\P_Nz})\\
=&\sqrt{n}(\PP_{\pi,n}-\P_N)f_{\P_Nz}+\sqrt{n}(\PP_{\pi,n}-\P_N)(f_{\PP_{\pi,n}z}-f_{\P_N z})+\\
&~~\sqrt{n}\P_N(f_{\PP_{\pi,n}z}-f_{\P_N z}).
\end{align*}
Combining Theorem \ref{thm:combinatorial-Donsker} with Proposition \ref{prop:smooth} and Proposition \ref{prop:portmanteau} implies that 
\[
\lim_{N\to\infty}\sqrt{n}(\PP_{\pi,n}-\P_N)(f_{\PP_{\pi,n}z}-f_{\P_n z})=0.
\]
For the third term, using the classic result on Riemann sum convergence to the integral for bounded variation functions (cf. \citet[Theorem 1]{babu1985edgeworth}), we can obtain
\[
\sqrt{n}\Big\{\P_N(f_{\PP_{\pi,n}z}-f_{\P_N z})-\P(f_{\PP_{\pi,n}z}-f_{\P_N z})\Big\}\lesssim \sqrt{n}/N\to 0.
\]
Accordingly, we only have to consider
\[
\P(f_{\PP_{\pi,n}z}-f_{\P_N z})=\Psi(\PP_{\pi,n}z)-\Psi(\P_N z),
\]
where $\Psi(t)=\P|Z-t|$. Assume $\P$ has a Lebesgue density, then $\Psi$ is differentiable with derivative $2F(t)-1$, with $F$ standing for the CDF of $\P$. Delta method then implies
\[
\Psi(\PP_{\pi,n}z)-\Psi(\P_N z)=(2F(\P_N z)-1)(\PP_{\pi,n}-\P_N)z+o_{\P_N}(1/\sqrt{n}).
\]
Accordingly,
\begin{align*}
\sqrt{n}(\hat M_n-M_N) &= \sqrt{n}(\PP_{\pi,n}-\P_N)(f_{\P_N,z}+(2F(\P_Nz)-1)z)\\
&\Rightarrow N(0,\sigma_{\P}^2),
\end{align*}
with
\[
\sigma_{\P}^2 = (1-\gamma)\var_{\P}\Big[\Big|z-\P z\Big|+\Big\{2F(\P z)-1\Big\}z\Big].
\]
We thus give a finite-population version of the famous Example 19.25 in \cite{van2000asymptotic}.

\section{Notes}

{\bf Chapter \ref{chap:rosen}.} Theorem \ref{thm:rosen} is due to Ros\'en \citep{rosen1964limit}; see, also, \cite{rosen1967central} and \cite{rosen1967centralb}. We adopt the present form from Billingsley \citep{billingsley1968convergence}, where Billingsley also extent results to more general exchangeable sequences. The proof of Theorem \ref{thm:rosen}, which is a combination of a Levy-type tail bound due to \cite{pruss1998maximal} (see, also, \cite{pozdnyakov2013systematic} and references therein) and Chattejee's combinatorial Bernstein inequality (Theorem \ref{thm:cb-ineq}), is new. Lastly, Corollary \ref{cor:5hcoi} is Proposition A.1.9 in \cite{vaart1996empirical}, which credited the general vector-valued version to \citet[Corollary A.2.e]{marshall1979inequalities}; Theorem \ref{thm:hoeffding-amazing} gave a different proof. \\

{\bf Chapter \ref{chap:cmi}.} Theorem \ref{lem:comb-orlicz} was first discovered in an implicit manner by \citet[Theorem 2.1]{bobkov2004concentration} and later explicitly presented in \citet[Theorem 7 and Lemma 2]{tolstikhin2017concentration}. The present proof, using Chatterjee's method, is new yet quite straightforward. The highlight is that the resulting concentration inequalities involve constants that only depend on the sampling ratio ($n/N$) and the $L^2(\P_N)$ norm. 

To contrast, tail probability bounds for $(\PP_{\pi,n}-\P_N)f$ that involve an $L^{\infty}$-term, as presented in  Theorem \ref{lem:comb-orlicz-1}, are numerous. Serfling \citep{serfling1974probability} may be the first that gave an exponential concentration inequality of the Hoeffding type. Using the same martingale approach, \cite{bardenet2015concentration} obtained a set of exponential inequalities. Additionally, Greene and Wellner \citep{greene2017exponential,greene2016finite}, employing Kemperman's majorizing hypergeometric argument \citep[Section 4]{kemperman1973moment}, derived a set of hypergeometric tail inequalities that can be directly translated to tail probability bounds for  $(\PP_{\pi,n}-\P_N)f$. Finally, Sourav Chatterjee introduced new moment and concentration inequalities via Stein's method, discussed in depth in Chapter \ref{sec:3cmi}. More recent developments along this track include \cite{albert2019concentration}, \cite{lei2021regression}, \cite{sambale2022concentration}, \cite{polaczyk2023concentration}, and \cite{barber2024hoeffding}.\\

{\bf Chapter \ref{sec:5plln}.} Results in this chapter can be partially found in some survey sampling reviews and books; see, e.g., \cite{praskova2009asymptotics}, \citet[Chapter 2]{cochran1977sampling} and also \citet[Appendix Section 4]{lehmann2006nonparametrics}. The versions we choose to present, however, are not directly adopted from any book. \\

{\bf Chapter \ref{chap:cti}.} All results in this chapter come from the two papers of \cite{bobkov2004concentration} and \cite{tolstikhin2017concentration}. In particular, Lemma \ref{lem:bobkov} and Corollary \ref{cor:bobkov} are due to \citet[Theorem 2.1]{bobkov2004concentration}, and the rest is due to \cite{tolstikhin2017concentration}. \\

{\bf Chapters \ref{chap:cgc}-\ref{sec:5examples}.} Most results in these four chapters are genuinely new, with connections to related works mentioned in the text. 

One thing we did not mention is bootstrap in the finite population, which may be developed in the future  as a new chapter. For now, we just point readers of interest here to the papers of Bickel and Freedman \citep{bickel1984asymptotic}, Rao and Wu \citep{rao1988resampling}, and Booth, Butler, and Hall \citep{booth1994bootstrap}.

Another rather interesting track of study concerns empirical processes for design-based inference from a super-population perspective. This corresponds to the setting outlined in Proposition \ref{prop:iid}, but of course focuses on more complex design. It has generated fruitful results in the recent years, including, e.g., \cite{saegusa2013weighted}, \cite{boistard2017functional}, and \cite{han2021complex} among others.  We hope to cover them in the future.

%% file: chapters/rmt.tex
\chapter{Combinatorial Random Matrix Theory}
\label{chapter:crmt}

A rather special stochastic process, where ``chaining is rarely the way to go'' \citep[Page 304]{talagrand2014upper}, is the spectral process involving terms like $\{\bv^\top\Xb\bv;\norm{\bv}_2=1\}$; here $\Xb$ is understood to be a {\it random matrix}.  This chapter concerns spectral processes in the combinatorial domain.

\section{Technical preparation}\label{sec:6technical}

We start with an introduction to the matrix notation we will use in the following sections.

\begin{definition}[Matrix notation] Consider arbitrary $d$-dimensional real {\it symmetric} matrices $\Ab$ and $\Bb$.
\begin{enumerate}[label=(\roman*)]
\item Write $\Ab^\top$ as the transpose of $\Ab$.
\item Write its trace as $\tr(\Ab)$ and define $\bar\tr(\Ab):=d^{-1}\tr(\Ab)$.
\item For any $f:\reals\to\reals$, define
\[
f(\Ab):=\sum_{i=1}^df(\lambda_i)\bu_i\bu_i^\top,
\]
where $\Ab=\sum_{i=1}^d\lambda_i\bu_i\bu_i^\top$ is the eigen-decomposition of $\Ab$.
\item Write $\lambda_{\min}(\Ab)$ and $\lambda_{\max}(\Ab)$ to represent the smallest and largest eigenvalues of $\Ab$. 
\item Write its matrix spectral norm as $\norm{\Ab}_{\rm op}:=\max\{|\lambda_{\min}(\Ab)|,|\lambda_{\max}(\Ab)|\}$.
\item Define ${\rm dom}(\Ab):=[\lambda_{\min}(\Ab),\lambda_{\max}(\Ab)]$.
\item Write $\Ab\succeq \bm{0}$ if $\Ab$ is positive semidefinite.
\item Write $\Bb \succeq \Ab$, or equivalently $\Ab\preceq \Bb$ if $\Bb-\Ab\succeq \bm{0}$.
\end{enumerate}
\end{definition}

Some matrix inequalities that we will use in the following are put below. 

\begin{lemma}[Basic trace inequalities] \label{lemma:trace} Consider arbitrary $d$-dimensional real symmetric matrices $\Ab$ and $\Bb$.
\begin{enumerate}[label=(\roman*)]
\item If $\Ab\preceq \Bb$, then $\tr(\Ab)\leq \tr(\Bb)$.
\item Assume $f:\reals\to\reals$ to be increasing and $\Ab,\Bb$ to be commuting. Then $\Ab \preceq \Bb$ yields $f(\Ab)\preceq f(\Bb)$.
\item Let $f,g:\reals\to\reals$ be two functions such that $f(x)\leq g(x)$ for all $x\in {\rm dom}(\Ab)$. It then holds true that $f(\Ab)\preceq g(\Ab)$.
\item Supposing $\Ab\preceq \Bb$ and $\Pb\succeq \bm{0}$, then $\tr(\Ab\Pb)\leq \tr(\Bb\Pb)$.
\end{enumerate}
\end{lemma}

\begin{exercise}
Please prove Lemma \ref{lemma:trace}.
\end{exercise}

The following three lemmas will be used in the subsequent developments.

\begin{lemma}[Lieb's inequality]\label{lem:lieb} For any random symmetric matrix $\Xb$ and fixed symmetric matrix $\Ab$, we have
\[
\E\tr\exp(\Ab+\Xb) \leq \tr\exp(\Ab+\log \E e^{\Xb}).
\]
\end{lemma}

\begin{lemma}[Mean value trace inequality]\label{lem:mvt} Let $I$ be an interval in $\reals$, $g:I\to\reals$ be an increasing function, and $h:I\to\reals$ be a function of a convex derivative $h'$. Then, for any symmetric matrices $\Ab,\Bb\in \reals^{d\times d}$  such that ${\rm dom}(\Ab)\subset I$ and ${\rm dom}(\Bb)\subset I$, it holds true that
\begin{align*}
&\tr\Big[(g(\Ab)-g(\Bb))\cdot (h(\Ab)-g(\Bb))\Big]\\
\leq& \frac12 \tr\Big[(g(\Ab)-g(\Bb))\cdot (\Ab-\Bb)\cdot (h'(\Ab)+h'(\Bb))\Big].
\end{align*}
\end{lemma}

\begin{lemma}[Young's entropy inequality, matrix version]\label{lem:young-matrix} Assume $\Ub\in\reals^{d\times d}$ to be a positive semidefinite random matrix such that $\E\bar\tr [\Ub]=1$ and $\Vb\in\reals^{d\times d}$ to be a symmetric random matrix whose spectral norm is almost surely bounded. We then have
\[ 
\E\bar\tr(\Vb\Ub)\leq \log\E\bar\tr e^{\Vb}+\E\bar\tr(\Ub\log \Ub).
\]
\end{lemma}

\section{Matrix moment inequalities under independence}\label{sec:6indep}

To begin with, we first introduce two inequalities that bound the spectral norm of the sum of independent random matrices; they are in parallel to the Hoeffding's and Bernstein's inequalities (cf. Theorems \ref{thm:2hoeffding} and \ref{thm:bernstein}) in the random scalar setting. 

\begin{theorem}[Matrix Hoeffding's inequality]\label{thm:matrix-hoeffding} Let $\Xb_1,\ldots,\Xb_n\in\reals^{d\times d}$ be $n$ {\it independent} $d$-dimensional {\it symmetric} random matrices satisfying 
\[
\E \Xb_i=\bm{0}~~~{\rm and}~~~\norm{\Xb_i}_{\op}\leq M.
\]
It then holds true that, for any $t\geq 0$, 
\[
\P\Big(\bnorm{\sum_{i=1}^n\Xb_i}_{\op}\geq t  \Big) \leq 2d\cdot \exp\Big(-\frac{t^2}{2nM^2}\Big).
\]
\end{theorem}

\begin{theorem}[Matrix Bernstein's inequality]\label{thm:mbi} Let $\Xb_1,\ldots,\Xb_n\in\reals^{d\times d}$ be $n$ independent $d$-dimensional {\it symmetric} random matrices satisfying 
\[
\E \Xb_i=\bm{0}~~~{\rm and}~~~\norm{\Xb_i}_{\op}\leq M.
\]
Further denote 
\[
\sigma^2=\bnorm{\sum_{i=1}^n\E[\Xb_i^2]}_{\op}.
\]
We then have, for any $t\geq 0$,
\[
\Pr\Big(\bnorm{\sum_{i=1}^n\Xb_i}_\op \geq t  \Big) \leq 2d\cdot \exp\Big(-\frac{t^2}{3\sigma^2+3Mt} \Big).
\]
\end{theorem}

We will be focused on the proof of the matrix Bernstein's inequality, with the Hoeffding counterpart left as an exercise to the readers of interest. We follow the proof of Tropp \citep{tropp2012user}, which is concisely summarized in Vershynin's book \citep{vershynin2018high}.

Let's start with a reduction step based on Lieb's inequality. Note that this step demands an independence structure.

\begin{lemma}\label{lem:mbi-1} Supposing $\Xb_i$'s are independent random symmetric matrices, it then holds true that, for any $\lambda\in\reals$, 
\[
\E\tr\exp\Big(\lambda\sum_{i=1}^n\Xb_i \Big) \leq \tr\exp\Big(\sum_{i=1}^n\log\E e^{\lambda\Xb_i} \Big).
\]
\end{lemma}
\begin{proof}
This can be proven by repeatedly applying Lemma \ref{lem:lieb} to the sum $\sum_{i=1}^n\Xb_i$, at step $m\in[n-1]$ conditional on $\bX_1,\ldots,\bX_{n-m}$.
\end{proof}

The next step is to bound the marginals. This is established based on the following lemma.

\begin{lemma}\label{lem:mbi-2} Suppose $\Xb$ is a random matrix satisfying
\[
\E \Xb=\bm{0}~~~{\rm and}~~~\norm{\Xb}_{\rm op}\leq M.
\]
It then holds true that, for any $|\lambda|<3/M$,
\[
\E\exp(\lambda\Xb) \preceq \exp\Big(\frac{\lambda^2/2}{1-|\lambda|M/3}\cdot \E\Xb^2   \Big).
\]
\end{lemma}
\begin{proof}
This is by Taylor expansion. In particular, it holds true that, 
\[
e^x \leq 1+x+\frac{1}{1-|x|/3}\frac{x^2}{2}, ~~~\text{for any }|x|< 3,
\]
which yields that
\[
e^{\lambda x} \leq 1+\lambda x + g(\lambda)x^2,~~~\text{for any }|x|\leq M~~{\rm and}~~|\lambda|<3/M,
\]
where we introduce
\[
g(\lambda):=\frac{\lambda^2/2}{1-|\lambda|M/3}.
\]
This argument, when applied to matrices, yields the desired bound.
\end{proof}

Now we are ready to finish the proof.
\begin{proof}[Proof of Theorem \ref{thm:mbi}] It suffices to bounding $\lambda_{\max}(\sum_{i=1}^n\Xb_i)$. To this end, we have, for any $\lambda\in\reals$
\begin{align*}
\E \exp\Big\{\lambda\cdot \lambda_{\max}\Big(\sum_{i=1}^n\Xb_i\Big)\Big\} \leq \E\tr\exp\Big(\lambda\sum_{i=1}^n\Xb_i \Big) \leq \tr\exp\Big(\sum_{i=1}^n\log\E e^{\lambda\Xb_i} \Big),
\end{align*}
where the second inequality is Lemma \ref{lem:mbi-1}. Now Lemma \ref{lem:mbi-2} implies
\[
\log\E\exp(\lambda\Xb_i) \preceq g(\lambda)\cdot \E\Xb_i^2,
\]
so that
\[
\tr\exp\Big(\sum_{i=1}^n\log\E e^{\lambda\Xb_i} \Big) \leq \tr\exp\Big(g(\lambda)\sum_{i=1}^n \E\Xb_i^2\Big) \leq d\exp\{g(\lambda)\sigma^2\}.
\]
The remaining is then standard. This proved the bound for the largest eigenvalue. A similar bound can be established for the smallest eigenvalue by applying the same technique to $-\Xb_i$'s.
\end{proof}

\section{Combinatorial matrix moment inequalities}\label{sec:cmmi}

Now let's move on to the combinatorial  counterpart of independent random matrix sums. To this end, let's introduce $\{\Ab_{j,k}\in\reals^{d\times d};j,k\in[N]\}$ to be a set of deterministic matrices. 

The following two are the combinatorial matrix Hoeffding and Bernstein inequalities. They are also the counterparts to the combinatorial scalar Hoeffding and Bernstein inequalities, Theorem \ref{thm:hb-ineq} and Theorem \ref{thm:cb-ineq}.

\begin{theorem}[Combinatorial matrix Hoeffding's inequality]\label{thm:cmhi} Assume 
\[
\sup_{i,j\in[N]}\bnorm{\Ab_{i,j}}_\op\leq M~~~{\rm and}~~~\sum_{i,j\in[N]}\Ab_{i,j}=\bm{0}. 
\]
It then holds true that, for any $t\geq 0$, 
\[
\Pr\Big(\bnorm{\sum_{i=1}^N\Ab_{i,\pi(i)}}_{\op}\geq t  \Big) \leq 2d\cdot \exp\Big(-\frac{t^2}{24NM^2}\Big).
\]
\end{theorem}

\begin{theorem}[Combinatorial matrix Bernstein's inequality]\label{thm:cmbi} Assume 
\[
\sup_{i,j\in[N]}\bnorm{\Ab_{i,j}}_\op\leq M~~~{\rm and}~~~\sum_{i,j\in[N]}\Ab_{i,j}=\bm{0}.
\]
Denote
\[
\sigma^2:=\frac{1}{N}\bnorm{\sum_{i,j\in[N]}\Ab_{ij}^2}_\op.
\]
It then holds true that, for any $t\geq 0$, 
\[
\Pr\Big(\bnorm{\sum_{i=1}^N\Ab_{i,\pi(i)}}_{\op}\geq t  \Big) \leq 2d\cdot \exp\Big(-\frac{t^2}{12\sigma^2+4\sqrt{2}Mt}\Big).
\]
\end{theorem}

\subsection{Chatterjee's method for matrices}

This chapter presents the revised Chatterjee's method, previously introduced in Chapter \ref{sec:chatterjee-method}, for handling matrices. This adaptation was carried out by Mackey, Jordan, Chen, Farrell, and Tropp \citep{mackey2014matrix}, and in the following we endeavor to give a summary of their ideas. 

Similar to Chapter \ref{sec:chatterjee-method}, let's assume $(X,X')$ to be an exchangeable pair of margins living in the space $\cX$. Introduce 
\[
f:\cX\to\reals^{d\times d}
\]
to be the objective matrix-valued function, which we coupled with a specifically chosen antisymmetric function, $F(X,X')=\alpha(f(X)-f(X'))$ for some $\alpha>0$, such that
\[
\E[F(X,X') \mid X]=f(X).
\]
Note that this implies $\E[f(X)]=\bm{0}$. Introduce further
\[
v(x):=\frac{\alpha}{2}\E\Big[\Big(f(X)-f(X')\Big)^2 \mid X=x  \Big].
\]
The following is the matrix analogue of Lemma \ref{lem:chatterjee-master}.

\begin{lemma}[Master lemma, matrix version]\label{lem:master-matrix} Assume that $(X,X')$ forms an exchangeable pair and $\norm{f(X)}_{\rm op}$ is almost surely bounded. Define
\[
M_{\lambda}:= \E\tr\Big[ \exp\Big\{\lambda f(X)\Big\}\Big],~~~\text{ for any }\lambda\in\reals.
\]
We then have
\begin{align*}
&\frac{\d}{\d \lambda}M_{\lambda} \leq \lambda\cdot \E\tr\Big[v(X)\exp\{\lambda f(X)\} \Big],~~~\text{ for any }\lambda\geq 0,\\
{\rm and}~~~&\frac{\d}{\d \lambda}M_{\lambda} \geq \lambda\cdot \E\tr\Big[v(X)\exp\{\lambda f(X)\} \Big],~~~\text{ for any }\lambda\leq 0.
\end{align*}
\end{lemma}
\begin{proof}
By the lemma condition, we can write
\[
\frac{\d}{\d \lambda}M_{\lambda} = \E\tr \Big[ \frac{\d}{\d \lambda}\exp\Big\{\lambda f(X)\Big\}\Big] = \E\tr\Big[f(X)e^{\lambda f(X)} \Big].
\]
Notice that \eqref{eq:2chatterjee-equ-1} still holds true in the matrix case. For any $\lambda>0$, we could then continue to write
\begin{align*}
\frac{\d}{\d \lambda}M_{\lambda} &= \frac{\alpha}{2}\E\tr\Big[\Big(f(X)-f(X')\Big)\cdot\Big(e^{\lambda f(X)}-e^{\lambda f(X')} \Big)\Big]\\
&\leq \frac{\alpha\lambda}{4} \E\tr\Big[ \Big(f(X)-f(X')\Big)^2\cdot (e^{\lambda f(X)}+e^{\lambda f(X')})\Big]\\
&= \frac{\alpha\lambda}{2} \E\tr\Big[ \Big(f(X)-f(X')\Big)^2\cdot e^{\lambda f(X)}\Big]\\
&=\lambda \E\tr\Big[v(X)\exp\{\lambda f(X)\} \Big],
\end{align*}
where the inequality is due to Lemma \ref{lem:mvt}. This completes the first half. The second half is symmetric.
\end{proof}

\begin{lemma}[Chatterjee's second lemma, matrix version]\label{lem:c2lemma-matrix} Assume the conditions in Lemma \ref{lem:master-matrix}. Suppose further that there exist finite positive constants $A,B$ such that
\[
\Pr\Big\{v(X)\preceq A f(X)+B\cdot\Ib_d \Big\}=1.
\]
It then holds true that
\[
\E\tr\Big[\exp(\lambda f(X)) \Big] \leq d\exp\Big\{\frac{B\lambda^2}{2(1-A\lambda)} \Big\},~~\text{ for all }\lambda \in [0,1/A),
\]
and for all $t\geq 0$,
\[
\Pr\Big(\lambda_{\max}(f(X))\geq t \Big)\leq d\exp\Big( -\frac{t^2}{2B+2At} \Big).
\]
\end{lemma}
\begin{proof}
The proof is essentially the same as that of Lemma \ref{lem:2chattejee-2} by invoking Lemma \ref{lemma:trace} and realizing that the only difference from Lemma \ref{lem:2chattejee-2} is the initial condition that, when $d\geq 1$, $M_0=\E\tr[\Ib_d]=d~~\text{so that } \log M_0=\log d$.
\end{proof}

\begin{lemma}[Chatterjee's thrid lemma, matrix version] Assume the conditions in Lemma \ref{lem:master-matrix} and introduce the function
\[
r(\psi):= \frac{1}{\psi}\log \E\bar\tr e^{\psi v(X)},~~\text{ for any }\psi>0.
\]
It holds true then that
\[
\log\E\bar\tr e^{\lambda f(X)} \leq \frac{\lambda^2r(\psi)}{2(1-\lambda^2/\psi)},~~\text{ for all }\psi>0 \text{ and } 0\leq \lambda<\sqrt{\psi},
\]
and thus, for any $t\geq 0$ and $\psi > 0$,
\[
\Pr\Big\{\lambda_{\max}(f(X)) \geq t  \Big\}\leq d\cdot \exp\Big\{ - \frac{t^2}{2r(\psi)+2t/\sqrt{\psi}} \Big\}.
\]
\end{lemma}
\begin{proof}
Let's introduce 
\[
m_{\lambda}:=\E\bar\tr e^{\lambda f(X)}=d^{-1}M_{\lambda}~~~{\rm and}~~W_{\lambda}:=\frac{1}{m_{\lambda}}\cdot e^{\lambda f(X)}.
\]
Similar to the proof of Lemma \ref{lem:2chattejee-3}, we have $\E \bar\tr (W_{\lambda})=1$ and $W_{\lambda}\succeq \bm{0}$. Lemma \ref{lem:master-matrix} then yields, for any $\lambda\geq 0$ and $\psi>0$,
\begin{align*}
&\frac{\d}{\d \lambda}m_{\lambda} \leq \lambda\cdot \E\bar\tr\Big[v(X)e^{\lambda f(X)} \Big] \leq \frac{\lambda m_{\lambda}}{\psi}\cdot \E\bar\tr\Big[\psi v(X)\cdot W_{\lambda} \Big]\\
\leq& \frac{\lambda m_{\lambda}}{\psi}\cdot \log\E\bar\tr e^{\psi v(X)}+\frac{\lambda m_{\lambda}}{\psi}\cdot\E\bar\tr(W_{\lambda}\log W_{\lambda})\\
=& \lambda m_{\lambda}r(\psi) + \frac{\lambda m_{\lambda}}{\psi}\cdot\E\bar\tr(W_{\lambda}\log W_{\lambda}).
\end{align*}
Now, 
\[
\log W_{\lambda} = \lambda f(X)-\log m_{\lambda}\cdot \Ib_d \preceq \lambda f(X),
\]
where in the inequality we used the fact that 
\[
\log m_{\lambda}= \log \E\bar\tr e^{\lambda f(X)}\geq \log\bar\tr e^{\lambda \E f(X)}=0.
\]
Therefore,
\begin{align*}
\frac{\d}{\d \lambda}m_{\lambda} \leq& \lambda m_{\lambda} r(\psi)+ \frac{\lambda m_{\lambda}}{\psi}\cdot \E\bar\tr\Big(\frac{\lambda}{m_{\lambda}}\cdot e^{\lambda f(X)}\cdot f(X) \Big)\\
=&\lambda m_{\lambda} r(\psi) + \frac{\lambda^2}{\psi}\cdot \frac{\d}{\d \lambda}m_{\lambda}.
\end{align*}
The rest is identical to Lemma \ref{lem:2chattejee-3}.
\end{proof}

\subsection{Proof of the combinatorial matrix Bernstein inequality}

Now let's apply the matrix version of the Chatterjee's method to give a proof of Theorem \ref{thm:cmbi}; the proof of Theorem \ref{thm:cmhi} is left as an exercise to the readers of interest.

Same as Chapter \ref{sec:cmi-chatterjee}, let's introduce a couple of $\pi$ to be
\[
\pi' :=\pi \circ (I,J) = \begin{cases}
\pi(i), \quad \text{if }i\ne I,J,\\
\pi(J), \quad \text{if } i=I,\\
\pi(I), \quad \text{if } i=J,
\end{cases}
\]
with $I,J$ uniformly and independently sampled from $[N]$. Let
\[
f(\pi)=\sum_{i=1}^N\Ab_{i,\pi(i)}~~~{\rm and}~~~F(\pi,\pi')=\frac{N}{2}\Big(\sum_{i=1}^N\Ab_{i,\pi(i)}-\sum_{i=1}^N\Ab_{i,\pi'(i)} \Big).
\]
Under the theorem condition, it is easy to verify that $\E f(\pi)=\bm{0}$. In addition, similar to the derivation in Chapter \ref{sec:cmi-chatterjee}, one could establish that
\begin{align*}
\E[F(\pi,\pi')\mid \pi]&=\frac{N}{2}\E\Big[\Ab_{I,\pi(I)}+\Ab_{J,\pi(J)}-\Ab_{I,\pi(J)}-\Ab_{J,\pi(I)}  \mid \pi\Big]\\
&=\sum_{i=1}^Nd_{i,\pi(i)}-\frac{1}{N}\sum_{i,j\in[N]}\Ab_{i,j}\\
&=f(\pi).
\end{align*}
Additionally, similar to Lemma \ref{lem:2chatterjee}, we have 
\begin{align*}
v(\pi)&=\frac{1}{4N}\sum_{i,j\in[N]}\Big(\Ab_{i,\pi(i)}+\Ab_{j,\pi(j)}-\Ab_{i,\pi(j)}-\Ab_{j,\pi(i)}\Big)^2\\
&\preceq \frac{1}{N}\sum_{i,j\in[N]}\Big(\Ab_{i,\pi(i)}^2+\Ab_{j,\pi(j)}^2+\Ab_{i,\pi(j)}^2+\Ab_{j,\pi(i)}^2\Big)\\
&=2\sum_{i=1}^N\Ab_{i,\pi(i)}^2+\frac{2}{N}\sum_{i,j=1}^N\Ab_{i,j}^2.
\end{align*}

\begin{lemma} Under the conditions of Theorem \ref{thm:cmbi}, we have
\[
r(\psi):= \frac{1}{\psi}\log \E\bar\tr e^{\psi v(X)} \leq 4\sigma^2 + \frac{8M^2\sigma^2\psi}{1-4M^2\psi},~~~\text{ for any }\psi \in \Big[0, \frac{1}{4M^2}\Big).
\]
\end{lemma}
\begin{proof}
Introduce
\[
\Wb := 2\sum_{i=1}^N\Ab_{i,\pi(i)}^2-\frac{2}{N}\sum_{i,j\in[N]}\Ab_{i,j}^2,
\]
whose expectation is easily verified to be $\bm{0}$. We then have
\[
v(\pi) \preceq \Wb+\frac{4}{N}\sum_{i,j=1}^N\Ab_{i,j}^2
\]
so that
\[
r(\psi)\leq 4\sigma^2 + \frac{1}{\psi}\log\E\bar\tr e^{\psi \Wb}.
\]
It remains to bound $\log\E\bar\tr e^{\psi \Wb}$. Again, leveraging the same idea as in Lemma \ref{lem:2chatterjee}, we obtain
\begin{align*}
v_{\Wb}(\pi)&=\frac{1}{N}\sum_{i,j\in[N]}\Big(\Ab_{i,\pi(i)}^2+\Ab_{j,\pi(j)}^2-\Ab_{i,\pi(j)}^2-\Ab_{j,\pi(i)}^2\Big)^2\\
&\preceq \frac{4}{N}\sum_{i,j\in[N]}\Big(\Ab_{i,\pi(i)}^4+\Ab_{j,\pi(j)}^4+\Ab_{i,\pi(j)}^4+\Ab_{j,\pi(i)}^4\Big)\\
&\preceq \frac{4M^2}{N}\sum_{i,j\in[N]}\Big(\Ab_{i,\pi(i)}^2+\Ab_{j,\pi(j)}^2+\Ab_{i,\pi(j)}^2+\Ab_{j,\pi(i)}^2\Big)\\
&= 4M^2\cdot \Big(\Wb+\frac{4}{N}\sum_{i,j=1}^N\Ab_{i,j}^2\Big)\\
&\preceq 4M^2 \Wb + 16M^2\sigma^2\cdot \Ib_d.
\end{align*}
Applying Lemma \ref{lem:c2lemma-matrix} then gives the desired bound for $\log\E\bar\tr e^{\psi \Wb}$. The rest is identical to the proof of Theorem \ref{thm:cb-ineq}.
\end{proof}

\section{Example: Combinatorial nonparametric regression}\label{sec:cn-reg}

In classical nonparametric regression problems, we have access to $n$ {\it independent realizations}, $(Y_1,\bX_1),\ldots, (Y_n,\bX_n)$, of a pair $(Y,\bX)\in \reals \times \reals^d$; here $Y$ is the response/outcome and $\bX$ represents a set of covariates/predictors. The goal of interest is to estimate the ``best predictor'', i.e., the conditional expectation of $Y$ given $\bX$,
\[
\psi(\bx):=\E[Y\mid \bX=\bx],
\]
based only on the observed values $\{(Y_i,\bX_i), i\in[n]\}$.

In a finite-population counterpart, it is assumed that there exists a finite population
\[
(y_1,\bx_1), (y_2,\bx_2), \ldots, \ldots,(y_N,\bx_N)\in\reals \times \reals^d,
\]
which are {\it non-random}. What we observed, on the other hand, is a randomly sampled size-$n$ subset of it. Using the permutation notation, the sample we observed is
\[
\Big\{(y_i,\bx_i), \pi(i)\leq n\Big\}.
\]

Since the population size is finite, the conditional expectation $\E_{\P_N}(Y \mid \bX=\bx)$ does not carry any statistical meaning. Instead, let's directly compare the nonparametric regression estimate with regard to the permutation measure $\PP_{\pi,n}$ to that with regard to the population law $\P_N$.

In detail, consider 
\[
\bp_K(\bx)=(p_{1K}(\bx),\ldots,p_{KK}(\bx))^\top
\]
to be a $K$-dimensional vector of basis functions 
\[
p_{jk}(\cdot): \reals^d\to\reals.
\]
The series estimator based on the random sample $\{(y_i,\bx_i), \pi(i)\leq n\}$ is defined to be the least square estimator (LSE) when projected to the linear space spanned by the bases: 
\begin{align*}
&\hat\psi_K(\bx)=\bp_K^\top(\bx)\hat\bbeta_K,\\
{\rm with}~~~&\hat\bbeta_K\in \argmin\limits_{\bb\in\reals^K}\frac{1}{n}\sum_{\pi(i)\leq n}\Big\{Y_i-\bp_K(\bx_i)^\top\bb\Big\}^2.
\end{align*}
The above LSE should approximate its population counterpart
\[
\psi_K(\bx)=\bp_K^\top(\bx)\bbeta_K,~~~{\rm with}~~\bbeta_K\in \argmin\limits_{\bb\in\reals^K}\frac{1}{N}\sum_{i=1}^N\Big\{Y_i-\bp_K^\top(\bx_i)\bb\Big\}^2.
\]

Let $\Ab^-$ be the generalized inverse of the matrix $\Ab$. It is clear then that 
\begin{align*}
&\hat\bbeta_K = \hat\Qb^{-}\cdot \frac{\Pb_{\pi}^\top\bY}{n},~~~{\rm with}~\bY:=(y_1,\ldots,y_N)^\top,\\
&\Pb_{\pi}:=\Big(\bp_K(\bx_1)\ind(\pi(1)\leq n),\ldots,\bp_K(\bx_N)\ind(\pi(N)\leq n)\Big)^\top\in\reals^{N\times k},\\
{\rm and}~~&\hat\Qb:=\Pb_{\pi}^\top\Pb_{\pi}/n=\frac{1}{n}\sum_{\pi(i)\leq N}\bp_K(\bx_i)\bp_K^\top(\bx_i).
\end{align*}
Symmetrically, we also have
\begin{align*}
&\bbeta_K=\Qb^{-}\cdot\frac{\Pb_N^\top\bY}{N}, \\
{\rm with}~~&\Pb_N:=(\bp_K(\bx_1),\ldots,\bp_K(\bx_N))^\top~{\rm and}~\Qb:=\Pb_N^\top\Pb_N/N=\P_N[\bp_K(\bx)\bp_K^\top(\bx)].
\end{align*}

%Next, to characterize the smoothness of the basis functions, let's introduce the multi-index derivatives. 
%\begin{definition}
%For any positive integer $k$, define $\Lambda_k$ to be the set of all $d$-dimensional vectors of nonnegative integers $\bt=(t_1,\ldots,t_d)^\top$ such that $|\bt|:=\sum_{i=1}^dt_i=k$. Define $\partial^{\bt}$ to be the corresponding partial derivative operator.
%\end{definition}

%Following the above definition, for each nonnegative integer $q$, let $\zeta_{q,K}:=\max_{\bt\in\Lambda_q}\sup_{\bx\in \cX}\norm{\partial^{\bt}\bp_K(\bx)}_2$.
Lastly, introduce
\[
\zeta_K:=\sup_{\bx\in\cX}\bnorm{\bp_K(\bx)}_2.
\]
This term satisfies $\zeta_K=O(K)$ for power series and $\zeta_K=O(K^{1/2})$ for fourier series, splines, compact supported wavelets, and piecewise polynomial regression. 

\begin{theorem}\label{thm:nonp-reg} Assume $\pi$ is a uniform random permutation, $\lambda_K:=\lambda_{\min}(\Qb)>0$, and $(1+\lambda_K^{-1}\zeta_{K})^2\log K/n\to 0$.
We then have
\[
\P_N(\hat\psi_K-\psi_K)^2=O_{\Pr}\Big(\lambda_K^{-1}(A_N^2\gamma_N+B_N^2) \Big),
\]
where $A_N:=(1+\lambda_K^{-1/2}\zeta_{K})\sqrt{\log K/n}$, $\gamma_N:=2(B_N^2+\norm{\Pb_N^\top\bY/N}_2^2)$, and
\[
B_N^2:= \frac{1}{n}\sum_{j=1}^K\P_N\Big(yp_{Kj}-\P_N\Big[yp_{Kj}\Big]\Big)^2.
\]
\end{theorem}
\begin{proof}
First, we have the following lemma that quantifies the deviation of $\hat\Qb$. 
\begin{lemma}\label{lem:nonp-reg-lem1} It holds true that $\E\bnorm{\Qb^{-1/2}\hat\Qb\Qb^{-1/2}-\Ib_K}_\op \lesssim A_N$.
\end{lemma}
\begin{proof}[Proof of Lemma \ref{lem:nonp-reg-lem1}]
Notice that
\[
\hat\Qb=\frac{1}{n}\sum_{i=1}^N\bp_K(\bx_i)\bp_K(\bx_i)^\top\ind(\pi(i)\leq n).
\]
It is immediate that
\begin{align*}
\Qb^{-1/2}\hat\Qb\Qb^{-1/2}-\Ib_K&=\frac{1}{n}\sum_{i=1}^N\Qb^{-1/2}\bp_K(\bx_i)\bp_K^\top(\bx_i)\Qb^{-1/2}\ind(\pi(i)\leq n)-\Ib_K,
\end{align*}
which can be written as some $\sum_{i=1}^N \Ab_{i,\pi(i)}$ with
\[
\Ab_{i,j}:=\frac{1}{n}\Qb^{-1/2}\bp_K(\bx_i)\bp_K^\top(\bx_i)\Qb^{-1/2}\ind(j\leq n)-\frac{1}{N}\Ib_K.
\]
It can then be easily checked that $\sum_{i,j\in[N]}\Ab_{i,j}=\bm{0}$.  Additionally, we have
\[
\bnorm{\Ab_{i,j}}_\op \leq \frac{1}{n}\Big(1+\lambda_K^{-1}\zeta_{K}^2\Big)
~~~{\rm and}~~~
\frac{1}{N}\bnorm{\sum_{i,j\in[N]}\Ab_{i,j}^2}_\op \leq \frac{1}{n}\Big(1+3\lambda_K^{-1}\zeta_{K}^2\Big).
\]
This is because, since $\lambda_K=\lambda_{\min}(\Qb)>0$, it holds true that 
\[
\norm{\Qb^{-1/2}}_\op=\lambda_K^{-1/2}<\infty,
\]
and thus
\[
\bnorm{\Qb^{-1/2}\bp_K(\bx_i)\bp_K(\bx_i)^\top\Qb^{-1/2}}_\op = \bnorm{\Qb^{-1/2}\bp_K(\bx_i)}^2\leq \lambda_K^{-1}\zeta_{K}^2.
\]
In addition, using the definition of $\Qb$, we obtain
\begin{align*}
&\bnorm{\P_N\Big[\Qb^{-1/2}\bp_K(\bx_i)\bp_K(\bx_i)^\top\Qb^{-1}\bp_K(\bx_i)\bp_K(\bx_i)^\top\Qb^{-1/2}  \Big]}_\op\\
\leq & \lambda_K^{-1}\zeta_{K}^2\bnorm{\P_N\Big[\Qb^{-1/2}\bp_K(\bx_i)\bp_K(\bx_i)^\top\Qb^{-1/2} \Big]}\\
= & \lambda_K^{-1}\zeta_{K}^2.
\end{align*}
Accordingly, invoking Theorem \ref{thm:cmbi} yields the claim.
\end{proof}

Next, getting back to the proof of Theorem \ref{thm:nonp-reg}, define
\[
A_n=\Big\{\lambda_{\min}(\Qb^{-1/2}\hat\Qb\Qb^{-1/2})>1/2\Big\}.
\]
Under the event $A_n$, it then holds true that $\hat\Qb$ is invertible. Accordingly,
\begin{align*}
&\ind_{A_n}\P_N(\hat\psi_K-\psi_K)^2 \\
=& \ind_{A_n}\P_N[\bp_K^\top\hat\bbeta_K-\bp_K^\top\bbeta_K]^2\\
=& \ind_{A_n}(\hat\bbeta_K-\bbeta_K)^\top\Qb(\hat\bbeta_K-\bbeta_K)\\
=&\ind_{A_n}\bnorm{\Qb^{1/2}(\hat\bbeta_K-\beta_K)}_2^2\\
\leq& 2\ind_{A_n}\bnorm{\Qb^{1/2}\Big[\hat\Qb^{-1}-\Qb^{-1}\Big]\frac{\Pb_{\pi}^\top\bY}{n}}_2^2+2\ind_{A_n}\bnorm{\Qb^{-1/2}\Big[\frac{\Pb_\pi^\top\bY}{n}-\frac{\Pb_N^\top\bY}{N}}_2^2\\
=&2\ind_{A_n}\bnorm{\Big[\Qb^{1/2}\hat\Qb^{-1}\Qb^{1/2}-\Ib_K\Big]\Qb^{-1/2}\frac{\Pb_{\pi}^\top\bY}{n}}_2^2+2\ind_{A_n}\bnorm{\Qb^{-1/2}\Big[\frac{\Pb_\pi^\top\bY}{n}-\frac{\Pb_N^\top\bY}{N}}_2^2\\
\leq&2\ind_{A_n}\lambda_K^{-1}\Big(\bnorm{\Qb^{1/2}\hat\Qb^{-1}\Qb^{1/2}-\Ib_K}_{\op}^2\bnorm{\frac{\Pb_{\pi}^\top\bY}{n}}_2^2+\bnorm{\frac{\Pb_\pi^\top\bY}{n}-\frac{\Pb_N^\top\bY}{N}}_2^2\Big).
\end{align*}

\begin{lemma}
For any symmetric real matrix $\Ab\in\reals^{K\times K}$, if $\norm{\Ab-\Ib_K}_\op\leq \lambda<1$, then $\norm{\Ab^{-1}-\Ib_K}_{\op} \leq \frac{\lambda}{1-\lambda}$.
\end{lemma}

Using the above lemma and noticing that
\[
\ind_{A_n}\Big(\Qb^{-1/2}\hat\Qb\Qb^{-1/2}\Big)^{-1}=\ind_{A_n}\Qb^{1/2}\hat\Qb^{-1}\Qb^{1/2},
\]
we obtain $\ind_{A_n}\norm{\Qb^{1/2}\hat\Qb^{-1}\Qb^{1/2}-\Ib_K}_{\op}=O_{\Pr}\Big\{(1+\lambda_K^{-1/2}\zeta_{K})\sqrt{\log K/n}\Big\}$.

\begin{lemma}
We have 
\[
\E\bnorm{\frac{\Pb_{\pi}^\top\bY}{n}}_2^2\leq 2(B_N^2+\norm{\Pb_N^\top\bY/N}_2^2) = \gamma_N
\]
and
\[
\E\bnorm{\frac{\Pb_\pi^\top\bY}{n}-\frac{\Pb_N^\top\bY}{N}}_2^2 \leq \frac{1}{n}\sum_{j=1}^K\P_N\Big(y_ip_{kj}-\P_N\Big[yp_{Kj}\Big]\Big)^2=B_N^2.
\]
\end{lemma}
\begin{proof}
We have
\begin{align*}
\E\bnorm{\frac{\Pb_\pi^\top\bY}{n}-\frac{\Pb_N^\top\bY}{N}}_2^2=&\E\bnorm{\frac{1}{n}\sum_{i=1}^N\Big(y_i\bp_K(\bx_i)-\P_N[y\bp_K]\Big)\ind(\pi(i)\leq n)}_2^2\\
=&\sum_{j=1}^K\E\Big|\frac{1}{n}\sum_{i=1}^N\Big(y_ip_{Kj}(\bx_i)-\P_N[yp_{Kj}]\Big)\ind(\pi(i)\leq n)  \Big|^2\\
\leq & \frac{1}{n}\sum_{j=1}^K\P_N\Big(y_ip_{kj}-\P_N\Big[yp_{Kj}\Big]\Big)^2.
\end{align*}
Lastly, employing the fact that
\[
\bnorm{\frac{\Pb_{\pi}^\top\bY}{n}}_2^2 \leq 2\Big(\bnorm{\frac{\Pb_\pi^\top\bY}{n}-\frac{\Pb_N^\top\bY}{N}}_2^2+\bnorm{\frac{\Pb_N^\top\bY}{N}}_2^2 \Big)
\]
completes the proof.
\end{proof}
Using the above lemma then yields the theorem statement.
\end{proof}

\section{Notes}

{\bf Chapter \ref{sec:6technical}.} Lemma \ref{lem:lieb} is due to Lieb \citep[Theorem 6]{lieb1973convex}; \citet[Chapter 8]{tropp2015introduction} gave a self-contained proof. Lemma \ref{lem:mvt} comes from \citet[Lemma 3.4]{mackey2014matrix}. Lemma \ref{lem:young-matrix} is adopted from \citet[Lemma A.3]{mackey2014matrix}, which credits it to \cite{carlen2010trace}. \\

{\bf Chapter \ref{sec:6indep}.} The results in this chapter are mostly rephrased from \citet[Chapter 5.4]{vershynin2018high}, with the exception of Theorem \ref{thm:matrix-hoeffding}, whose constants were sharpened to be optimal due to the work of Mackey, Jordan, Chen, Farrell, and Tropp \citep[Corollary 4.2]{mackey2014matrix}; cf. Theorem 1.3 in \cite{tropp2012user}.

Matrix Bernstein inequalities were first developed by Ahlswede and Winter \citep{ahlswede2002strong} and later independently improved upon by Oliveira \citep{oliveira2009concentration} (based on the Golden–Thompson inequality) and Tropp \citep{tropp2012user} (based on Lieb’s inequality). \\

{\bf Chapter \ref{sec:cmmi}.} All results in this chapter are adopted from \cite{mackey2014matrix}, who established a set of matrix concentration inequalities applicable to dependent settings. In a different context, Han and Li \citep{han2020moment} derived a similar matrix Bernstein inequality for matrix-valued time series, based on earlier work by Banna, Merlevède, and Youssef \citep{banna2016bernstein}.\\

{\bf Chapter \ref{sec:cn-reg}.} The standard references to the analysis of series estimators are \cite{newey1997convergence}, \cite{wasserman2006all}, and \cite{chen2007large}, whose theories are aimed at independently sampled data. Following this track, \cite{belloni2015some} was the first to introduce matrix deviation inequalities of Rudelson \citep{rudelson1999random} and Bernstein type (discussed in Chapter \ref{sec:6indep}) to the analysis of series estimators. This was followed by works such as \cite{chen2013optimal} and \cite{cattaneo2013optimal}, among many others. In particular, \citet[Section 4]{cattaneo2023rosenbaum} concerns another complex setting when the nonparametric regression covariates themselves are estimated from the same data.

Theorem \ref{thm:nonp-reg} represents the first attempt at analyzing nonparametric regressions in a finite population setting. The proof uses the combinatorial Bernstein inequality introduced in Chapter \ref{sec:cmmi} and is new, though we believe there is still substantial room to improve on the rate.

%% file: chapters/stat-app.tex
\begin{partbacktext}
\part{Statistical Applications}
\end{partbacktext}

%% file: chapters/m-est.tex
\chapter{Combinatorial M- and Z-estimators}\label{chap:m-est}

\section{General framework}\label{chap:7-general}

Many things about $M$-estimators are beyond any particular sampling paradigm, and it is better to present them first. To this end, let's define $\{(\Theta,d_N); N=1,2,\ldots\}$ to be a sequence of metric spaces and 
\[
M_N(\theta):\Theta\to\reals
\]
to be a {\it non-random objective function} that could possibly change with $N$. Let $n=n_N\in[N]$ be indexed by $N$ and is increasing to infinity along with $N$. Let
\[
\hat M_n(\theta):\Theta \to\reals
\]
be a {\it random} function that is usually designed to approximate $M_N(\cdot)$. Let 
\[
\theta_N \in \argmax_{\theta\in\Theta}M_N(\theta)~~~{\rm and}~~~\hat M_n(\hat\theta_n)=\sup_{\theta\in\Theta}\hat M_n(\theta)-o_{\Pr}(1)
\]
be the (approximate) maximizers of $M_N(\cdot)$ and $\hat M_n(\cdot)$, the latter of which, $\hat\theta_n$, is called an {\it M-estimator}. The goal of interest is to estimate and {\it infer} the unknown parameter $\theta_N$ using $\hat\theta_n$.

\subsection{M-estimators: consistency}

The following is the consistency theorem. 

\begin{theorem}[M-estimator, consistency]\label{thm:m-est-consist} Suppose that 
\begin{enumerate}[label=(\roman*)]
\item for any $\epsilon>0$, we have 
\[
\limsup_{N\to\infty} \sup_{d_N(\theta,\theta_N)>\epsilon}\Big\{M_N(\theta)-M_N(\theta_N)\Big\}<0;
\]
\item it holds true that $\E\sup_{\theta\in\Theta}|\hat M_n(\theta)-M_N(\theta)|\to 0$.
\end{enumerate}
We then have $d_N(\hat\theta_n,\theta_N)\stackrel{\Pr}{\to}0$ as $N\to\infty$.
\end{theorem}
\begin{proof}
By assumptions, 
\begin{align*}
&M_N(\hat\theta_n)-M_N(\theta_N) \\
=&M_N(\hat\theta_n) - \hat M_n(\hat\theta_n)+\hat M_n(\hat\theta_n)-\hat M_n(\theta_N)+\hat M_n(\theta_N)-M_N(\theta_N)\\
\geq & M_N(\hat\theta_n) - \hat M_n(\hat\theta_n) + \hat M_n(\theta_N)-M_N(\theta_N)-o_\Pr(1)\\
\geq &-2\sup_{\theta\in\Theta}\Big|\hat M_n(\theta)-M_N(\theta)\Big|-o_\Pr(1)\\
\geq & -o_\Pr(1).
\end{align*}
Now for every $\epsilon>0$, the first condition implies that there exists some $\eta=\eta_\epsilon>0$ such that for all sufficiently large $N$, 
\[
\sup_{\theta:d_N(\theta,\theta_N)>\epsilon}\Big\{M_N(\theta)-M_N(\theta_N)\Big\}<-\eta.
\]
Thus, the event $\{d_N(\hat\theta_n,\theta_N)>\epsilon\}$ is contained in the event $\{M_N(\hat\theta_n)-M_N(\theta_N)<-\eta\}$, the latter of which has probability tending to 0 in view of the above bounds. Accordingly, for any $\epsilon>0$, 
\[
\Pr\Big\{d_N(\hat\theta_n,\theta_N)>\epsilon\Big\} \leq \Pr\Big\{M_N(\hat\theta_n)-M_N(\theta_N)<-\eta\Big\} \to 0
\]
and the proof is thus complete.
\end{proof}

\subsection{M-estimators: rates of convergence}

In the following, consider an increasing sequence $r_n\to\infty$ as $N\to\infty$.

\begin{theorem}[M-estimator, rate of convergence]\label{thm:m-est-rate} Suppose that
\begin{enumerate}[label=(\roman*)]
\item there exists some fixed constant $\epsilon>0$ such that 
\[
\limsup_{N\to\infty}\sup_{d_N(\theta,\theta_N)\leq \epsilon}\frac{M_N(\theta)-M_N(\theta_N)}{d_N^2(\theta,\theta_N)} < 0;
\]
\item there exists some function $\phi:\reals^{>0}\to\reals^{>0}$ and a universal constant $K>0$ such that for all sufficiently small $\delta>0$ and all $N=1,2,\ldots$,
\[
\E\sup_{d_N(\theta,\theta_N)<\delta}\Big|(\hat M_n-M_N)(\theta)-(\hat M_n-M_N)(\theta_N) \Big| \leq \frac{K\phi(\delta)}{\sqrt{n}};
\]
\item the function $\phi$ satisfies that $\delta \mapsto \phi(\delta)/\delta^{\alpha}$ is decreasing for some fixed constant $\alpha<2$, and 
\[
r_n^2\phi\Big(\frac{1}{r_n}\Big)\leq \sqrt{n}~~~\text{for all }n=1,2,\ldots;
\]
\item the sequence of random values $\hat\theta_n$ satisfies 
\[
\hat M_n(\hat\theta_n)\geq \sup_{\theta\in\Theta}\hat M_n(\theta)-O_\Pr(r_n^{-2})~~~ {\rm and}~~~ d_N(\hat\theta_n,\theta_N)\stackrel{\Pr}{\to} 0.
\]
\end{enumerate}
It then holds true that $r_nd_N(\hat\theta_n,\theta_N)=O_{\Pr}(1)$.
\end{theorem}
\begin{proof}
For simplicity, let's assume that $\hat\theta_n$ maximizes $\hat M_n(\cdot)$. For any $N$, let's partition $\Theta$ to the cells
\[
S_{j,N}:=\Big\{\theta\in\Theta; r_nd_N(\theta,\theta_N)\in \Big(2^{j-1}, 2^j\Big]\Big\}, \text{ with } j=\ldots,-1,0,1,\ldots.
\]
Fix an integer $M$. If $r_nd_N(\hat\theta_n,\theta_N)>2^M$, then $\hat\theta_n$ is in one  $S_{j,N}$ with $j> M$ such that
\[
\sup_{\theta\in S_{j,N}} \hat M_n(\theta)\geq \hat M_n(\theta_N).
\]
In addition, for any $\eta>0$, we have, if $\hat\theta_N$ is in one cell $S_{j,N}$ with $2^j>\eta r_n$, then
\[
r_nd_N(\hat\theta_n,\theta_N)\geq 2^{j-1}>\eta r_n/2
\]
so that $2d_N(\hat\theta_n,\theta_N)>\eta$, and vise versa. Combining the above two facts, we obtain
\begin{align*}
&\Pr(r_nd_N(\hat\theta_n,\theta_N)>2^M)\\
=& \sum_{j> M,2^j\leq \eta r_n}\Pr\Big(\hat\theta_n\in S_{j,N}\Big)+ \sum_{2^j> \eta r_n}\Pr\Big(\hat\theta_n\in S_{j,N}\Big)\\
\leq& \sum_{j> M,2^j\leq \eta r_n}\Pr\Big(\sup_{\theta\in S_{j,N}} \hat M_n(\theta)\geq \hat M_n(\theta_N) \Big) + \Pr\Big(2d_N(\hat\theta_n,\theta_N)>\eta \Big).
\end{align*}
Now, choose $\eta<\epsilon$ so that for any $\kappa>0$, there exists some $N_\kappa$ such that for all $N\geq N_{\kappa}$,
\[
\sup_{d_N(\theta,\theta_N)\leq \eta}\frac{M_N(\theta)-M_N(\theta_N)}{d_N^2(\theta,\theta_N)} < \kappa.
\]
In addition, choose $\delta>\eta$ so that 
\[
\E\sup_{d_N(\theta,\theta_0)<\eta}\Big|(\hat M_n-M_N)(\theta)-(\hat M_n-M_N)(\theta_N) \Big| \leq \frac{K\phi(\delta)}{\sqrt{n}}.
\]
Then, for any $j>M$ and $2^j\leq \eta r_n$, and all sufficiently large $N$, 
\begin{align*}
&\sup_{\theta\in S_{j,N}} \hat M_n(\theta)- \hat M_n(\theta_N)\\
 \leq& \sup_{d_N(\theta,\theta_0)\leq \frac{2^j}{r_n}}\Big|(\hat M_n-M_N)(\theta)-(\hat M_n-M_N)(\theta_N) \Big| + \sup_{\theta\in S_{j,N}} M_N(\theta)-M_N(\theta_N)\\
\leq & \sup_{d_N(\theta,\theta_0)\leq \frac{2^j}{r_n}}\Big|(\hat M_n-M_N)(\theta)-(\hat M_n-M_N)(\theta_N) \Big| -\kappa \sup_{\theta\in S_{j,N}}d_N^2(\theta,\theta_N) \\
\leq & \sup_{d_N(\theta,\theta_0)\leq \frac{2^j}{r_n}}\Big|(\hat M_n-M_N)(\theta)-(\hat M_n-M_N)(\theta_N) \Big| -\kappa \frac{2^{2j-2}}{r_n^2}.
\end{align*}
Accordingly,
\begin{align*}
&\Pr(r_nd_N(\hat\theta_n,\theta_N)>2^M) \\
\leq& \sum_{j> M,2^j\leq \eta r_n}\Pr\Big(\sup_{\theta\in S_{j,N}} \hat M_n(\theta)\geq \hat M_n(\theta_N) \Big) + \Pr\Big(2d_N(\hat\theta_n,\theta_N)>\eta \Big)\\
\leq& \sum_{j> M,2^j\leq \eta r_n}\Pr\Big(\sup_{d_N(\theta,\theta_0)\leq \frac{2^j}{r_n}}\Big|(\hat M_n-M_N)(\theta)-(\hat M_n-M_N)(\theta_N) \Big|\geq \kappa \frac{2^{2j-2}}{r_n^2} \Big) + o_\Pr(1)\\
\leq&\frac{4K}{\kappa}\cdot \sum_{j>M} \frac{\phi(2^j/r_n)r_n^2}{ \sqrt{n}2^{2j}} + o_\Pr(1).
\end{align*}
Now, since $\delta \mapsto \phi(\delta)/\delta^{\alpha}$ is decreasing, we have, for any $c>1$, 
\[
\frac{\phi(c\delta)}{(c\delta)^{\alpha}}\leq \frac{\phi(\delta)}{\delta^\alpha}
\]
so that $\phi(c\delta)\leq c^\alpha \phi(\delta)$. Accordingly, setting $M$ to be positive, we have 
\[
\phi(2^j/r_n) \leq 2^{j\alpha}\phi(1/r_n)
\]
so that 
\[
\sum_{j>M} \frac{\phi(2^j/r_n)r_n^2}{\sqrt{n} 2^{2j}} \leq \sum_{j>M}2^{j\alpha-2j}\frac{r_n^2\phi(1/r_n)}{\sqrt{n}} \leq \sum_{j>M}2^{j\alpha-2j}<\infty
\]
since we set $\alpha<2$. Accordingly, letting $M=M_N\to\infty$, we obtain 
\[
\lim_{N\to \infty}\Pr\Big\{r_nd_N(\hat\theta_n,\theta_N)>2^{M_N}\Big\}\to 0.
\]
This proves the assertion. 
\end{proof}

\subsection{M-estimators: limiting distribution}

The last theorem in this section concerns weak convergence of $r_n(\hat\theta_n-\theta_N)$, assuming that $r_n(\hat\theta_n-\theta_N)=O_\Pr(1)$. 

\begin{theorem}[M-estimator, linearization]\label{thm:comb-m-est-clt} Suppose that
\begin{enumerate}[label=(\roman*)]
\item $\Theta$ is a subset of $\reals^d$ equipped with the Euclidean metric $\norm{\cdot}_2$ such that all $\btheta_N$ is bounded away from its boundary;
\item the mapping $M_N: \Theta \to \reals$ is twice continuously differentiable at $\btheta_N$ such that the Hessian matrix at $\btheta_N$, denoted as $\Vb_N$, satisfies
\[
\limsup_{N\to\infty}\lambda_{\max}(\Vb_N)<0;
\]
\item there exists a uniformly tight sequence of random vectors $\bZ_N$ such that, for any $K>0$,
\[
\sup_K\Big|r_n(\hat M_n-M_N)(\tilde\btheta_n)-r_n(\hat M_n-M_N)(\btheta_N) - (\tilde\btheta_n-\btheta_N)^\top \bZ_N \Big|= o_\Pr(r_n^{-1}),
\]
where the supremum is taken over all sequences $\tilde\btheta_n$ such that 
\[
\limsup_{N\to\infty} r_n\bnorm{\tilde\btheta_n-\btheta_N}_2\leq K;
\]
\item $\hat\theta_n$ satisfies that $r_n\norm{\hat\btheta_n-\btheta_N}_2=O_{\Pr}(1)$ and 
\[
\hat M_n(\hat\btheta_n)\geq \sup_{\btheta\in\Theta}\hat M_n(\btheta)-o_{\Pr}(r_n^{-2}).
\]
\end{enumerate}
It then holds true that
\[
r_n(\hat\btheta_n-\btheta_N)+\Vb_N^{-1}\bZ_n \stackrel{\Pr}{\to} 0.
\]
\end{theorem}
\begin{proof}
By Taylor expansion, we have, for every sequence $\bh_n=O_{\Pr}(r_n^{-1})$,
\[
M_N(\btheta_N+\bh_n)-M_N(\btheta_N)=\frac{1}{2}\bh_n^\top \Vb_N \bh_n + o_{\Pr}(r_n^{-2}).
\]
The asymptotic equicontinuity condition then ensures that 
\[
\hat M_n(\btheta_N+\bh_n)-\hat M_n(\btheta_N)=\frac{1}{2}\bh_n^\top \Vb_N \bh_n + r_n^{-1}\bh_n^\top\bZ_N+o_\Pr(r_n^{-2}).
\]
Substituting $\bh_n$ by $\hat\btheta_n-\btheta_N$ and $-r_n^{-1}V_N^{-1}\bZ_N$, which by conditions are both $O_\Pr(r_n^{-1})$, we have 
\begin{align*}
&\hat M_n(\hat\btheta_n)-\hat M_n(\btheta_N)\\
=&\frac{1}{2}(\hat\btheta_n-\btheta_N)^\top\Vb_N(\hat\btheta_n-\btheta_N)+r_n^{-1}(\hat\btheta_n-\btheta_N)^\top\bZ_N+o_\Pr(r_n^{-2})
\end{align*}
and
\[
\hat M_N(\hat\btheta_n-r_n^{-1}\Vb_N^{-1}\bZ_N)-\hat M_n(\btheta_N)=-\frac12r_n^{-2}\bZ_N^\top\Vb_N^{-1}\bZ_N+o_\Pr(r_n^{-2}).
\]
Subtracting the second from the first and recalling that $\hat M_n(\hat\btheta_n)\geq \sup_{\btheta\in\Theta}\hat M_n(\btheta)-o_{\Pr}(r_n^{-2})$, we obtain
\[
\frac12 (\hat\btheta_n-\btheta_N+r_n^{-1}\Vb_N^{-1}\bZ_N)^\top\Vb_N(\hat\btheta_n-\btheta_N+r_n^{-1}\Vb_N^{-1}\bZ_N) \geq -o_\Pr(r_n^{-2}).
\]
Using the fact that $\limsup \lambda_{\max}(\Vb_N)<0$, we arrive at the desired result that
\[
\bnorm{\hat\btheta_n-\btheta_N+r_n^{-1}\Vb_N^{-1}\bZ_N}_2^2=o_\Pr(r_n^{-2}),
\]
so that $r_n(\hat\btheta_n-\btheta_N)+\Vb_N^{-1}\bZ_N \stackrel{\Pr}{\to} 0$.
\end{proof}

\subsection{Z-estimators}

In practice there are settings where our task is not to {\it maximize} an objective function, but to solve an {\it estimating equation}. This leads to the Z-estimators. In this section,  consider a deterministic function
\[
\Psi_N: \Theta \to \reals^m
\]
coupled with a random function $\hat\Psi_n:\Theta\to\reals^m$ that is usually designed to approximate $\Psi_N$. Let 
\[
\Psi_N(\btheta_N)=0 ~~~{\rm and}~~~\hat\Psi_n(\hat\btheta_n)=0. 
\]

To approach the above $Z$-estimation problem, let us first consider a slightly more general setting, where a deterministic perturbation theorem is introduced. 

\begin{theorem}[Master theorem, Z-estimation]\label{thm:z-est-master} Consider $\btheta_0\in\reals^d$ to be an arbitrary point of $\Theta$ and $\psi(\cdot)$ and $\tilde\psi(\cdot):\Theta\to\reals^m$ to be two non-random continuous functions. Assume further that
\begin{enumerate}[label=(\roman*)]
\item there exist four constants $\epsilon, \kappa_1,\kappa_2,\kappa_3>0$ and a  matrix $\tilde\Db\in\reals^{d\times d}$ such that 
\[
\bnorm{\tilde\psi(\btheta_0+\bu)-\tilde\psi(\btheta_0)-\tilde\Db\bu}_2\leq \kappa_1+\kappa_2\norm{\bu}_2+\kappa_3\norm{\bu}_2^2
\]
holds true for all $\norm{\bu}_2\leq \epsilon$;
\item There exists an invertible matrix $\Db\in\reals^{d\times d}$ and three constants $\xi,\delta_1,\delta_2>0$ such that 
\[
\norm{\Db^{-1}}_{\rm op}\leq \xi, ~~\norm{\tilde\psi(\btheta_0)-\psi(\btheta_0)}_2\leq \delta_1,~~{\rm and}~~\norm{\tilde\Db-\Db}_{\rm op}\leq \delta_2.
\]
\item We have 
\[
2\xi(\delta_1+\kappa_1)\leq \epsilon, ~~\xi^2\kappa_3(\delta_1+\kappa_1)\leq 1/8, ~~{\rm and}~~\xi(\delta_2+\kappa_2)\leq 1/4.
\]
\end{enumerate}
The following two statements are then true.
\begin{enumerate}[label=(\roman*)]
\item Define $\rho:=2\xi(\delta_1+\kappa_1)$. There then exists some $\tilde\bu\in B(\bm{0},\rho)$ such that $\tilde\psi(\btheta_0+\tilde\bu)=\psi(\theta_0)$;
\item $\norm{\tilde\bu+\Db^{-1}(\tilde\psi(\btheta_0)-\psi(\btheta_0))}_2\leq \xi\kappa_1+2\xi^2(\delta_1+\kappa_1)(\delta_2+\kappa_2)+4\xi^3\kappa_3(\delta_1+\kappa_1)^2$.
\end{enumerate}
\end{theorem}
\begin{proof}
{\bf Claim 1.} Let us define the mapping 
\[
G(\cdot): \bu\mapsto \Db^{-1}(\tilde\psi(\btheta_0+\bu)-\psi(\btheta_0)).
\] 
We then have 
\begin{align*}
&G(\bu) = \Db^{-1}(\tilde\psi(\btheta_0)-\psi(\btheta_0))+\Db^{-1}(\tilde\psi(\btheta_0+\bu)-\tilde\psi(\btheta_0))\\
=&\Db^{-1}(\tilde\psi(\btheta_0)-\psi(\btheta_0))+\Db^{-1}\tilde\Db\bu+\Db^{-1}(\tilde\psi(\btheta_0+\bu)-\tilde\psi(\btheta_0)-\tilde\Db\bu)\\
=&\Db^{-1}(\tilde\psi(\btheta_0)-\psi(\btheta_0)) + \bu + \Db^{-1}(\tilde\Db-\Db)\bu+\Db^{-1}(\tilde\psi(\btheta_0+\bu)-\tilde\psi(\btheta_0)-\tilde\Db\bu).
\end{align*}
By theorem conditions, we have, for any $\norm{\bu}_2\leq \epsilon$
\begin{align}\label{eq:z-est-dumbgen1}
\bnorm{\Db^{-1}(\tilde\psi(\btheta_0)-\psi(\btheta_0))}_2\leq \xi\delta_1, ~~\bnorm{\Db^{-1}(\tilde\Db-\Db)\bu}_2\leq \xi\delta_2\norm{\bu}_2,\notag\\
~{\rm and}~\bnorm{\Db^{-1}(\tilde\psi(\btheta_0+\bu)-\tilde\psi(\btheta_0)-\tilde\Db\bu)}_2\leq \xi(\kappa_1+\kappa_2\norm{\bu}_2+\kappa_3\norm{\bu}_2^2).
\end{align}
Adding up, we reach that,  for any $\norm{\bu}\leq \rho$, 
\[
\norm{G(\bu)-\bu}_2\leq \xi(\delta_1+\kappa_1)+\xi(\delta_2+\kappa_2+\kappa_3\rho)\norm{\bu}_2\leq \rho.
\]
Brouwer's fixed point theorem then yields the existence of some $\tilde\bu\in B(\bm{0},\rho)$ such that $\tilde\bu-G(\tilde\bu)=\tilde\bu$, i.e., $G(\tilde\bu)=\bm{0}$. By the invertibility of $\Db$, we then obtain 
\[
\tilde\psi(\btheta_0+\tilde\bu)=\psi(\btheta_0),
\]
which proves the first assertion.

{\bf Claim 2.} Taking further, since $\norm{\tilde\bu}_2\leq \rho=2\xi(\delta_1+\kappa_1)$ and noticing 
\[
\Db^{-1}(\tilde\psi(\btheta_0)-\psi(\btheta_0)) + \tilde\bu + \Db^{-1}(\tilde\Db-\Db)\tilde\bu+\Db^{-1}(\tilde\psi(\btheta_0+\tilde\bu)-\tilde\psi(\btheta_0)-\tilde\Db\tilde\bu)=0,
\]
employing Equation \eqref{eq:z-est-dumbgen1} implies
\begin{align*}
\norm{\tilde\bu+\Db^{-1}(\tilde\psi(\btheta_0)-\psi(\btheta_0))}_2&\leq \xi\delta_2\norm{\tilde\bu}_2+\xi(\kappa_1+\kappa_2\norm{\tilde\bu}_2+\kappa_3\norm{\tilde\bu}_2^2)\\
&\leq \xi\kappa_1+2\xi^2(\delta_1+\kappa_1)(\delta_2+\kappa_2)+4\xi^3\kappa_3(\delta_1+\kappa_1)^2.
\end{align*}
This completes the proof.
\end{proof}

When taking 
\[
\psi=\Psi_N,~~\tilde\psi=\hat\Psi_n,~~\Db=\dot\Psi_N(\btheta_N),~~{\rm and}~~\tilde\Db=\dot{\hat\Psi}_n(\btheta_N),
\]
where the upper dot represents the function Jacobian, we recover the Z-estimation problem discussed earlier.

\begin{corollary}[Z-estimator]\label{cor:z-est} Suppose the existence of a fixed constant $\epsilon>0$ such that
\begin{enumerate}[label=(\roman*)]
\item $\Theta$ is a subset of $\reals^d$ equipped with the Euclidean metric $\norm{\cdot}_2$;
\item it holds true that
\[
\lim_{\epsilon\downarrow 0}\limsup_{N\to\infty}\sup_{\norm{\btheta-\btheta_N}_2\leq \epsilon}\frac{\norm{\Psi_N(\btheta)-\Psi_N(\btheta_N)-\dot \Psi_{N}(\btheta_N)(\btheta-\btheta_N)}_2}{\norm{\btheta-\btheta_N}_2}=0;
\]
\item we have
\[
\lim_{\epsilon\downarrow 0}\limsup_{N\to\infty}\E\sup_{\norm{\btheta-\btheta_N}_2\leq \epsilon}\frac{\norm{\hat\Psi_n(\btheta)-\hat\Psi_n(\btheta_N)-(\Psi_N(\btheta)-\Psi_N(\btheta_N))}_2}{\norm{\btheta-\btheta_N}_2} =0;
\]
\item we have
\[
\sqrt{n}\bnorm{\hat\Psi_n(\btheta_N)-\Psi_N(\btheta_N)}_2=O_\P(1)~~~{\rm and}~~~\sqrt{n}\bnorm{\dot{\hat\Psi_n}(\btheta_N)-\dot \Psi_N(\btheta_N)}_{\rm op} =O_\Pr(1);
\]
\item we further have that $\dot\Psi_N(\btheta_N)$ is  invertible such that
\[
\limsup_{N\to\infty}\bnorm{\Big[\dot\Psi_N(\btheta_N)\Big]^{-1}}_{\rm op}<\infty.
\]
\end{enumerate}
There then exists a root $\hat\btheta_n$ of $\hat\Psi_n$ such that 
\begin{enumerate}[label=(\roman*)]
\item $\sqrt{n}(\hat\btheta_n-\btheta_N)=O_{\Pr}(1)$; 
\item $(\hat\btheta_n-\btheta_N)+[\dot\Psi_N^{-1}(\btheta_N)]^{-1}(\hat\Psi_n-\Psi_N)(\btheta_N)=o_\Pr(n^{-1/2}).$
\end{enumerate}
\end{corollary}

\begin{exercise}
Prove Corollary \ref{cor:z-est} using Theorem \ref{thm:z-est-master}.
\end{exercise}

\section{Combinatorial M- and Z-estimators}\label{sec:com-m-z}

Let $\{m_{\btheta}: \cZ\to\reals;\btheta\in\Theta\subset\reals^d\}$ to be a collection of objective functions. In  a finite-population setting, the aim is to estimate
\[
\btheta_N \in  \argmax_{\btheta\in\Theta}\P_Nm_{\theta}
\]
using $\hat\btheta_n$ that maximizes $\PP_{\pi,n}m_{\btheta}$ over all possible $\btheta\in\Theta$.

The following three corollaries are specialized to the above setting.

\subsection{Combinatorial M-estimators: consistency}

\begin{corollary}[Combinatorial M-estimators, consistency] Suppose that
\begin{enumerate}[label=(\roman*)]
\item $m_{\btheta}$ is twice continuously differentiable in $\Theta$ and satisfies the existence of a fixed constant $\epsilon>0$ such that
\[
\limsup_{N\to\infty}\sup_{\norm{\btheta-\btheta_N}_2\leq \epsilon}\lambda_{\max}(P_N \ddot{m}_{\btheta})<0;
\]
\item the function class $\{m_{\btheta};\btheta\in\Theta\}$ is $(\P_N,\pi)$-Glivenko-Cantelli.
\end{enumerate}
There then exists a $\hat\btheta_n\in\argmax \PP_{\pi,n}m_{\btheta}$ such that $\norm{\hat\btheta_n-\btheta_N}_2=o_\Pr(1)$.
\end{corollary}
\begin{proof}
Let's verify the conditions of Theorem \ref{thm:m-est-consist} by choosing 
\[
d_N(\btheta_1,\btheta_2)=\norm{\btheta_1-\btheta_2}_2,~~~\text{ for any }\btheta_1,\btheta_2\in\Theta.
\]
The second condition holds by the definition of $(\P_N,\pi)$-Glivenko-Cantelli. For the first condition, by conditions of the theorem, we only have to consider those $\btheta$ that is sufficiently close to $\btheta_N$, for which there exists some $\tilde\btheta_N$ between $\btheta$ and $\btheta_N$ such that
\[
\P_Nm_\btheta-\P_N m_{\btheta_N} = \frac12 (\btheta-\btheta_N)^\top\P_{N}\ddot m_{\tilde\btheta_N}(\btheta-\btheta_N)
\]
and $\lambda_{\max}(\P_{N}\ddot m_{\tilde\btheta_N})\leq \text{(some constant) }\delta<0$ for all sufficiently large $N$. Accordingly, 
\[
\P_Nm_\btheta-\P_N m_{\btheta_N} \leq \frac12\delta \norm{\btheta-\btheta_N}_2^2<0
\]
for all sufficiently large $N$. This completes the proof.
\end{proof}

\subsection{Combinatorial M-estimators: rates of convergence}

\begin{corollary}[Combinatorial M-estimators, root-$n$-consistency]\label{cor:comb-m-est-rate} Suppose that $\limsup N/n <\infty$ and 
\begin{enumerate}[label=(\roman*)]
\item $\Theta$ is bounded subset of $\reals$;
\item $m_{\btheta}$ is twice continuously differentiable in $\Theta$ and satisfies the existence of a fixed constant $\epsilon>0$ such that
\[
 \limsup_{N\to\infty}\sup_{\norm{\btheta-\btheta_N}_2\leq\epsilon}\lambda_{\max}(P_N \ddot{m}_{\btheta})<0~~~{\rm and}~~~\limsup_{N\to\infty}\sup_{\norm{\btheta-\btheta_N}_2\leq\epsilon}\P_N (\dot m_{\btheta})^2<\infty;
\]
\item there exists a function $F\geq 1$ such that $\limsup_{N\to\infty} \norm{F}_{L^2(\P_N)}<\infty$ and the function class $\{m_{\btheta};\btheta\in\Theta\}$ satisfies 
\[
\Big|m_{\btheta}(z)-m_{\btheta}(\bz')\Big| \leq F(z)\norm{\btheta-\btheta'}_2,~~~\text{for all } z\in\cZ.
\]
\end{enumerate}
There then exists a $\hat\btheta_n\in\argmax \PP_{\pi,n}m_{\btheta}$ such that $\norm{\hat\btheta_n-\btheta_N}_2=O_\P(n^{-1/2})$.
\end{corollary}
\begin{proof}
Let's pick 
\[
d_N(\btheta,\btheta')=\norm{\btheta-\btheta'}_2.
\]
We first examine Condition (i) in Theorem \ref{thm:m-est-rate}. By Taylor expansion and using a similar argument as in the last proof, we have
\[
\P_N(m_{\btheta}-m_{\btheta_N})=\frac12 (\btheta-\btheta_N)^\top\P_N\ddot m_{\tilde\btheta} (\btheta-\btheta_N) \leq \delta \norm{\btheta-\btheta_N}_2^2
\]
for all sufficiently large $N$, for which
\[
\sup_{\norm{\btheta-\btheta_N}\leq \epsilon}\frac{\P_N(m_{\btheta}-m_{\btheta_N})}{\norm{\btheta-\btheta_N}_2^2}\leq \delta < 0.
\]
This verifies the first condition. Lastly, we need to establish a bound on the {\it continuity modulus} of $\GG_{\pi,n}(m_{\btheta}-m_{\btheta_N})$. To this end, we first have the existence of a fixed constant $K>0$ such that for all $\btheta$ sufficiently close to $\btheta_N$,
\begin{align*}
&\bnorm{m_{\btheta}-m_{\btheta_N}}_{L^2(\P_N)}^2=\P_N(m_{\btheta}-m_{\btheta_N})^2 = (\P_N\dot m_{\tilde\btheta}^2)\norm{\btheta-\btheta_N}_2^2 \leq K^2\norm{\btheta-\btheta_N}_2^2.
\end{align*}
Accordingly, invoking Lemma \ref{lem:combinatorial-sec},
\begin{align*}
&\E\sup_{\norm{\btheta-\btheta_N}<\delta}\Big|\GG_{\pi,n}(m_{\btheta}-m_{\btheta_N}) \Big| \leq \E\sup_{m\in\cM_{\delta}}\Big|\GG_{\pi,n} m \Big|\\
&~~~\lesssim \norm{F_\delta}_{L^2(\P_N)}\int_0^{2} \sqrt{\log2\cN(\cM_{\delta},L^2(\P_N),\epsilon\norm{F_\delta}_{L^2(\P_N)})}\d\epsilon,
\end{align*}
where $\cM_{\delta}:=\{m_{\btheta}-m_{\btheta_N};\norm{\btheta-\btheta_N}_2<\delta \}$ and $F_{\delta}=2\delta F$ is an envelope of $\cM_{\delta}$ by the theorem condition since, for any $\btheta,\btheta'\in B(\btheta_N,\delta)$,
\[
\Big|m_{\btheta}(z)-m_{\btheta_N}(z)-(m_{\btheta'}(z)-m_{\btheta_N}(z))\Big| \leq 2F(z)\norm{\btheta-\btheta_N}_2 \leq  2\delta F(z),~~~\text{for any }z\in\cZ.
\]
Next, Theorem \ref{thm:entropy-parametric} implies that 
\[
\cN_{[]}(\cM_{\delta},L^2(\P_N),2\epsilon\norm{F_{\delta}}_{L^2(\P_N)}) \leq \cN(\Theta, \norm{\cdot},\epsilon) \lesssim \Big(\frac{1}{\epsilon}\Big)^d.
\]
Thusly, we can continue to bound
\[
\E\sup_{\norm{\btheta-\btheta_N}<\delta}\Big|\GG_{\pi,n}(m_{\btheta}-m_{\btheta_N}) \Big| \lesssim \delta \norm{F}_{L^2(\P_N)}\int_0^{2}\sqrt{\log \Big(\frac{2}{\epsilon}\Big)}\d\epsilon \lesssim \delta.
\]
Then, as long as we can show $\cM$ is $(\P_{\pi,n})$-Glivenko-Cantelli, picking $r_n=1/\sqrt{n}$ and $\phi(\delta)=\delta$ completes the proof. Glivenk-Cantelli is proved by directly employing Theorem \ref{thm:comb-gc}.
\end{proof}

\subsection{Combinatorial M-estimators: CLT}

\begin{corollary}[Combinatorial M-estimators, CLT] Assume the conditions in Corollary \ref{cor:comb-m-est-rate} and further that all $\btheta_N$ is bounded away from the boundary of $\Theta$. In addition, assume the existence of a fixed constant $\epsilon>0$ such that, uniformly for all $\norm{\btheta-\btheta_N}_2\leq \epsilon$ so that $\btheta\to\btheta_N$, we have
\[
 \lim_{N\to\infty}\frac{\P_N[m_{\btheta}-m_{\btheta_N}-(\btheta-\btheta_N)^\top\dot m_{\btheta_N}]^2}{\norm{\btheta-\btheta_N}_2^2}=0.
\]
We then have the existence of a $\hat\btheta_n\in\argmax \PP_{\pi,n}m_{\btheta}$ such that
\[
\sqrt{n}(\hat\btheta_n-\btheta_N)+[\P_N\ddot m_{\btheta_N}]^{-1}\GG_{\pi,n}\dot m_{\btheta_N}=o_\P(1).
\]
%where $\Vb_N$ is the Hessian matrix of the mapping $\btheta\mapsto \P_Nm_{\btheta}$.
\end{corollary}
\begin{proof}
To apply Theorem \ref{thm:comb-m-est-clt}, it remains to verify Condition (iii) therein. It suffices to bound
\[
\sup_{\norm{\bh}\leq K}\Big|\GG_{\pi,n}\Big\{\sqrt{n}(m_{\btheta_N+\bh/\sqrt{n}}-m_{\btheta_N})-\bh^\top\dot m_{\btheta_N}\Big\}\Big|.
\]
Consider the function class 
\[
\cM_K:= \Big\{\sqrt{n}(m_{\btheta_N+\bh/\sqrt{n}}-m_{\btheta_N})-\bh^\top\dot m_{\btheta_N};\norm{\bh}_2\leq K  \Big\}.
\]
We have the following facts.
\begin{enumerate}[label=(\roman*)]
\item $\cM_K$ admits the envelope $F_K:=KF(z)+\sqrt{K}\norm{\dot m_{\btheta_N}}_2$ since for any $z\in\cZ$
\begin{align*}
&\Big|\sqrt{n}(m_{\btheta_N+\bh/\sqrt{n}}(z)-m_{\btheta_N}(z))-\bh^\top\dot m_{\btheta_N}(z)\Big|\\ 
\leq& \Big|\sqrt{n}(m_{\btheta_N+\bh/\sqrt{n}}(z)-m_{\btheta_N}(z))\Big| + \Big|\bh^\top\dot m_{\btheta_N}(z)\Big| \\
\leq& F(z)\norm{\bh}_2+\norm{\bh}_2\cdot \norm{\dot m_{\btheta_N}}_2\leq F_K.
\end{align*}
In addition, by assumption $\norm{F_K}_{L^2(\P_N)}\leq K\norm{F}_{L^2(\P_N)}+\sqrt{K}\P_N\norm{\dot m_{\btheta_N}}_2$ so that $\limsup_{N\to\infty} \norm{F_K}_{L^2(\P_N)}<\infty$.
\item For any $\bh,\bh'\in B(\bm{0},K)$, 
\begin{align*}
&\Big|\sqrt{n}(m_{\btheta_N+\bh/\sqrt{n}}-m_{\btheta_N})-\bh^\top\dot m_{\btheta_N}-\Big\{\sqrt{n}(m_{\btheta_N+\bh'/\sqrt{n}}-m_{\btheta_N})-\bh'^\top\dot m_{\btheta_N}\Big\}\Big|\\
=& \sqrt{n}\Big|m_{\btheta_N+\bh/\sqrt{n}}-m_{\btheta_N+\bh'/\sqrt{n}}\Big|\\
\leq & F \cdot \norm{\bh-\bh'}_2.
\end{align*}
\end{enumerate}
Combining the above two facts and employing Theorem \ref{thm:entropy-parametric}, we obtain $\GG_{\pi,n}\Big\{\sqrt{n}(m_{\btheta_N+\bh/\sqrt{n}}-m_{\btheta_N})-\bh^\top\dot m_{\btheta_N}\Big\}$ satisfies the stochastic equicontinuity condition.  Additionally, for any fixed $\bh\in B(\bm{0},K)$, by the theorem condition, 
\[
\P_N\Big[\sqrt{n}(m_{\btheta_N+\bh/\sqrt{n}}-m_{\btheta_N})-\bh^\top\dot m_{\btheta_N}\Big]^2\to 0.
\]
Thusly, Theorem \ref{thm:weak-convergence-key} implies that 
\[
\GG_{\pi,n}\Big\{\sqrt{n}(m_{\btheta_N+\bh/\sqrt{n}}-m_{\btheta_N})-\bh^\top\dot m_{\btheta_N}\Big\} \Rightarrow 0 ~~~\text{in }\ell^{\infty}(B(\bm{0},K)),
\]
so that, by Theorem \ref{thm:cmt}, 
\[
\sup_{\norm{\bh}_2\leq K}\Big|\GG_{\pi,n}\Big\{\sqrt{n}(m_{\btheta_N+\bh/\sqrt{n}}-m_{\btheta_N})-\bh^\top\dot m_{\btheta_N}\Big\}\Big| \stackrel{\P}{\to} 0.
\]
This completes the proof.
\end{proof}

\subsection{Combinatorial Z-estimators}

In permutation sampling, letting $\psi_{\btheta}(\cdot): \Theta\times\cZ\to\reals^m$ be a vector-valued function, we estimate the root of $\P_N\psi_{\btheta}$, 
\[
\btheta_N \in \Big\{\btheta\in\Theta; \P_N\psi_{\btheta}=0  \Big\},
\]
using the root of $\PP_{\pi,n}\psi_{\btheta}$, 
\[
\hat\btheta_n \in  \Big\{\btheta\in\Theta; \PP_{\pi,n}\psi_{\btheta}=0  \Big\}.
\]
The following is the main result, which is a direct consequence of Corollary \ref{cor:z-est}, and is left to the readers for verification. 

\begin{corollary}[Combinatorial Z-estimator]\label{cor:comb-z-est} Assume that $\limsup N/n <\infty$ and the existence of a fixed constant $\epsilon>0$ such that
\begin{enumerate}[label=(\roman*)]
\item $\Theta$ is a subset of $\reals^d$ equipped with the Euclidean metric $\norm{\cdot}_2$;
\item it holds true that
\[
\lim_{\epsilon\downarrow 0}\limsup_{N\to\infty}\sup_{\norm{\btheta-\btheta_N}_2\leq \epsilon}\frac{\norm{\P_N\psi_{\btheta}-\P_N\psi_{\btheta_N}-[\P_N\dot \psi_{\btheta_N}](\btheta-\btheta_N)}_2}{\norm{\btheta-\btheta_N}_2}=0;
\]
\item There exists a function $F>0$ such that $\limsup_{N\to\infty}\P_N F<\infty$ and 
\[
\bnorm{\dot\psi_{\btheta}(z)-\dot\psi_{\btheta_N}(z)}_{\rm op}\leq F(z)\norm{\btheta-\btheta_N}_2,~~\text{ for all }\norm{\btheta-\btheta_N}_2\leq \epsilon~~\text{and}~z\in\cZ;
\]
\item we have
\[
\limsup_{N\to\infty}\P\norm{\psi_{\btheta_N}-\P_N\psi_{\btheta_N}}_2^2<\infty ~\text{ and }~\limsup_{N\to\infty}\P\norm{\dot\psi_{\btheta_N}-\P_N\dot\psi_{\btheta_N}}_{\rm op}^2<\infty
\]
\item we further have that $\P_N\dot\psi(\btheta_N)$ is  invertible such that
\[
\limsup_{N\to\infty}\bnorm{\Big[\P_N\dot\psi(\btheta_N)\Big]^{-1}}_{\rm op}<\infty.
\]
\end{enumerate}
There then exists a root $\hat\btheta_n$ of $\hat\Psi_n$ such that 
\begin{enumerate}[label=(\roman*)]
\item $\sqrt{n}(\hat\btheta_n-\btheta_N)=O_{\Pr}(1)$; 
\item $(\hat\btheta_n-\btheta_N)+[\P_N\dot\psi(\btheta_N)]^{-1}(\PP_{\pi,n}-\P_N)\psi_{\btheta_N}=o_\Pr(n^{-1/2}).$
\end{enumerate}
\end{corollary}

\begin{exercise}
Prove Corollary \ref{cor:comb-z-est} using Corollary \ref{cor:z-est}.
\end{exercise}

\section{Notes}

{\bf Chapter \ref{chap:7-general}.} The standard reference to M- and Z-estimation theory is \citet[Chapter 5]{van2000asymptotic}, \citet[Part 3]{vaart1996empirical}, \citet[Chapters 13 and 14]{kosorok2008introduction}, and \cite{newey1994large}. Theorem \ref{thm:m-est-consist} is a rephrased version of Theorem 5.7 in \cite{van2000asymptotic}. Theorem \ref{thm:m-est-rate} is adapted from Theorem 3.2.5 in \cite{vaart1996empirical}. Theorem \ref{thm:comb-m-est-clt} comes from Theorem 3.2.16 in \cite{vaart1996empirical}. Here there exists some subtlety in handling a triangular array setting, which is, however, relatively straightforward. The author learnt Theorem \ref{thm:z-est-master} from Lutz D\"umbgen \citep{dumbgen1998tyler}, which devised this result implicitly in the analysis of Tyler's M-estimator. \\

{\bf Chapter \ref{sec:com-m-z}.} Results in this section is genuinely new, although the proof techniques are of course related to those presented in, e.g., \cite{vaart1996empirical}. It appears that there is some recent interest in such topics stemming from the econometric literature; cf. the works and references in \cite{abadie2014finite}, \cite{abadie2020sampling}, and \cite{xu2021potential}.

%% file: chapters/per-test.tex
\chapter{Permutation Tests}
\label{chapter:rbs}

This chapter is focused on {\it two-sample inference} using permutation methods. The general setup is as follows: suppose that we have two random data independently and uniformly sampled without replacement from two finite populations, and we aim to compare the difference, characterized by some functions, between these two finite populations using the sampled data..

\section{General setup}

Let's start with a description of the sampling paradigm. Suppose that we have two {\it finite populations} of size $M$ and $N$, respectively:
\[
x_1,x_2, \ldots,x_M\in\cZ,~~{\rm and}~~~ y_1, y_2, \ldots, y_N\in\cZ.
\]
Our interest is, again, in those data that are uniformly sampled without replacement from these two finite populations. 

Let $\sigma_1\in\cS_M, \sigma_2\in\cS_N$ be two independent uniform permutations. The observed data can then be described as
\[
x_{\sigma_1^{-1}(1)},x_{\sigma_1^{-1}(2)},\ldots,x_{\sigma_1^{-1}(m)}~~~{\rm and}~~y_{\sigma_2^{-1}(1)}, y_{\sigma_2^{-1}(2)},\ldots,y_{\sigma_2^{-1}(n)}.
\]
For any $k\in[m+n]$, denote 
\[
z_k=x_{\sigma_1^{-1}(k)}\ind(k\leq m)+y_{\sigma_2^{-1}(k-m)}\ind(k>m),
\]
so that the observed sample can also be written as 
\[
z_1,\ldots,z_m,z_{m+1},\ldots,z_{m+n}. 
\]
Write the corresponding distributions of the first, the second, and the whole populations and samples to be 
\begin{align*}
\P_M:=\frac{1}{M}\sum_{i=1}^M\delta_{x_i},~\Q_N:=\frac{1}{N}\sum_{i=1}^N\delta_{y_i}, ~{\rm and}~\H_{M+N}:=\frac{M}{M+N}\P_M+\frac{N}{M+N}\Q_N;\\
\PP_{\sigma_1,m}:=\frac{1}{m}\sum_{i=1}^m\delta_{x_{\sigma_1^{-1}(i)}},~ \QQ_{\sigma_2,n}:=\frac{1}{n}\sum_{i=1}^n\delta_{y_{\sigma_2^{-1}(i)}},~{\rm and}\\
~\HH_{m+n}:=\frac{m}{m+n}\PP_{\sigma_1,m}+\frac{n}{m+n}\QQ_{\sigma_2,n}.
\end{align*}
Lastly, introduce a second-layer uniform permutation $\pi\in\cS_{m+n}$. Define the  permutation distribution as
\[
\tilde \PP_{\pi,m}:=\frac{1}{m}\sum_{i=1}^{m}\delta_{z_{\pi^{-1}(i)}}~~~{\rm and}~~~\tilde \QQ_{\pi,n}:=\frac{1}{n}\sum_{i=m+1}^{n}\delta_{z_{\pi^{-1}(i)}},
\]
corresponding to the permuted first and second samples.

%Define $\theta(\cdot)$ to be an abstract functional that takes probability measures as input and outputs a real vector. The task is to infer the differences 
%\[
%\theta(\PP_{\sigma_1,m})-\theta(\QQ_{\sigma_2,n}) 
%\]
%using the distribution of $\theta(\tilde\PP_{\pi,m})-\theta(\tilde\QQ_{\pi,n})$.

\section{Two-sample testing in the sample domain}

Assume $\P_M\Rightarrow \P$ and $\Q_N\Rightarrow \Q$ for some well-defined probability measures $\P$ and $\Q$, and there exists three constants $\gamma\in(0,1)$ such that 
\[
\lim_{N\to\infty} \frac{m}{m+n}=\gamma. %,~~\lim_{N\to\infty} \frac{m}{M}=\gamma_1,~~{\rm and}~~\lim_{N\to\infty} \frac{n}{N}=\gamma_2.
\]
Here we assume $M=M_N, m=m_N, n=n_N$ all increase to infinity as $N \to \infty$.

%It is immediate then that 
%\[
%\H_{M+N}\Rightarrow \frac{\gamma\gamma_2}{\gamma\gamma_2+(1-\gamma)\gamma_1} \P+\frac{(1-\gamma)\gamma_1}{\gamma\gamma_2+(1-\gamma)\gamma_1}\Q=: \H.
%\]

The first result in this section is a direct implication of Theorem \ref{thm:combinatorial-Donsker}.

\begin{proposition}\label{prop:2sample} Assume 
\begin{enumerate}[label=(\roman*)]
\item $\lim_{N\to\infty} \frac{m}{m+n}=\gamma\in (0,1)$, $\lim_{N\to\infty} m/M=\gamma_1\in (0,1)$ and $\lim_{N\to\infty} n/N=\gamma_2\in (0,1)$;
\item there exist two probability measures $\P$ and $\Q$ over a measurable space $(\bar \cZ,\cA)$ such that $(\P_M,\cF)$ and $(\Q_N,\cF)$ are $\P$-pre-Donsker and $\Q$-pre-Donsker, respectively. 
\item it holds true that
\begin{align*}
\limsup_{N\to\infty}\int_0^{\diam(\cF;L^2(\P_M))}\sqrt{\log 2\cN(\cF,L_2(\P_M),\epsilon)}\d \epsilon<\infty~~{\rm and}\\
~~\limsup_{N\to\infty}\int_0^{\diam(\cF;L^2(\Q_N))}\sqrt{\log 2\cN(\cF,L_2(\Q_N),\epsilon)}\d \epsilon<\infty.
\end{align*}
\end{enumerate}
We then have 
\[
\Big\{\sqrt{m+n}(\PP_{\sigma_1,m}-\QQ_{\sigma_2,n}-(\P_M-\Q_N))f; f\in\cF \Big\} \Rightarrow \sqrt{\frac{1-\gamma_1}{\gamma}}\BB_{\P} + \sqrt{\frac{1-\gamma_2}{1-\gamma}}\BB_{\Q},
\]
where $\BB_{\P}$ and $\BB_{\Q}$ are two independent mean-zero Gaussian processes such that for any $f,g\in\cF$, 
\begin{align*}
\cov(\BB_{\P}(f),\BB_{\P}(g))=\P(f-\P f)(g-\P g) \\
{\rm and}~~~\cov(\BB_{\Q}(f),\BB_{\Q}(g))=\Q(f-\Q f)(g-\Q g).
\end{align*}
\end{proposition}

\begin{exercise}
Please prove Proposition \ref{prop:2sample}.
\end{exercise}

\section{Two-sample testing in the sample-permutation domain}

The following is the main theorem of this section.

\begin{theorem}\label{thm:2sample}
Assume that $\cF$ admits an envelope function $F>0$ such that (a) $(\P_M,\cF)$ and $(\Q_N,\cF)$ are $\P$-pre-Donsker and $\Q$-pre-Donsker, respectively; (b) $\limsup_{M\to\infty}\sup_{f\in\cF}\norm{f}_{L^8(\P_M)}<\infty$ and $\limsup_{N\to\infty}\sup_{f\in\cF}\norm{f}_{L^8(\Q_N)}<\infty$; (c) either
\[
\int_0^2\sup_{\tilde\Q}\sqrt{\log \cN(\cF,L^2(\tilde\Q),\epsilon\norm{F}_{L^2(\tilde\Q)})}\d \epsilon<\infty
\]
or
\begin{align*}
\limsup_{M\to\infty}\int_0^{\diam(\cF;L^2(\P_M))}\sqrt{\log \cN_{[]}(\cF,L^2(\P_M),\epsilon)}\d\epsilon<\infty\\
~~{\rm and}~~\limsup_{N\to\infty}\int_0^{\diam(\cF;L^2(\Q_N))}\sqrt{\log \cN_{[]}(\cF,L^2(\Q_N),\epsilon)}\d\epsilon<\infty.
\end{align*}
We then have 
\[
\Big\{\sqrt{m}(\tilde\PP_{\pi,m}-\HH_{m+n})f; f\in\cF \Big\} \Rightarrow \sqrt{1-\gamma}\BB_{\tilde{\H}},
\]
where $\BB_{\tilde{\H}}$ is a mean-zero Gaussian process such that for any $f,g\in\cF$,
\[
\cov(\BB_{\tilde{\H}}(f), \BB_{\tilde{\H}}(g))=\tilde\H(f-\tilde\H f)(g-\tilde\H g),
\]
and $\tilde\H := \gamma \P_M+(1-\gamma)\Q_N$.

If we additionally assume Proposition \ref{prop:2sample}(i) holds true, then the above weak convergence is true given almost every $\sigma_1,\sigma_2$.
\end{theorem}

\begin{corollary} (two-sample testing)
Assume the conditions in Theorem \ref{thm:2sample} hold. It then holds true that 
\[
\Big\{\sqrt{\frac{mn}{m+n}}\Big(\tilde\PP_{\pi,m}-\tilde\QQ_{\pi,n}\Big)f;f
\in\cF\Big\} \Rightarrow \BB_{\tilde \H},
\]
both unconditionally and conditionally given almost every $\sigma_1,\sigma_2$.
\end{corollary}

\section{Proofs}

The proof of Theorem \ref{thm:2sample} is lengthy, so we break it to several parts. First, we introduce some instrumental lemmas.

\begin{lemma}\label{lem:lecam-trick} Under the setup of Corollary \ref{cor:5hcoi},  we have, for any function class $\cF=\{f:\cW\to\reals\}$, 
\[
\E\bnorm{\sum_{i=1}^\ell(\delta_{\hat W_i}-\W_L)}_{\cF} \leq 4\E\bnorm{\sum_{i=1}^L\tilde N_i\delta_{w_i}}_{\cF},
\]
where $\tilde N_i=N_i-N_i'$ with $N_i, N_i'$ i.i.d. drawn from a Poisson distribution with mean $\ell/(2L)$, denoted as ${\rm Pois}(\ell/(2L))$. 
\end{lemma}
\begin{proof}
The studied summations all contain i.i.d. entries, so that standard empirical process techniques can be used. In particular, employing the symmetrization trick yields
\[
\E\bnorm{\sum_{i=1}^\ell(\delta_{\hat W_i}-\W_L)}_{\cF} \leq 2\E\bnorm{\sum_{i=1}^\ell \epsilon_i\delta_{\hat W_i}}_\cF,
\]
where $\epsilon_1,\epsilon_2,\ldots$ is the Rademacher sequence. Next, we argue that
\[
\E\bnorm{\sum_{i=1}^\ell \epsilon_i\delta_{\hat W_i}}_\cF \leq 2\E\bnorm{\sum_{i=1}^{N_0}\epsilon_i\delta_{\hat W_i}}_{\cF},
\]
where $N_0\sim {\rm Pois}(\ell)$. Introducing $Y_1,Y_2,\ldots$ to be i.i.d. ${\rm Pois}(1)$, we have
\begin{align*}
&\Big(1-\frac{1}{e}\Big)\E\bnorm{\sum_{i=1}^\ell \epsilon_i\delta_{\hat W_i}}_{\cF} = \E(Y_1\vee 1)\E\bnorm{\sum_{i=1}^\ell \epsilon_i\delta_{\hat W_i}}_{\cF}=\E\bnorm{\E_Y \sum_{i=1}^\ell(Y_i\vee 1)\epsilon_i\delta_{\hat W_i}}_{\cF}\\
\leq& \E\bnorm{\sum_{i=1}^\ell(Y_i\vee 1)\epsilon_i\delta_{\hat W_i}}_{\cF} \leq \E\bnorm{\sum_{i=1}^\ell \sum_{j=1}^{Y_i}\epsilon_{i,j}\delta_{\hat W_{i,j}}}_{\cF}, 
\end{align*}
where $\{\epsilon_{i,j}\}$'s are i.i.d. copies of $\epsilon_1$ and $\hat W_{i,j}$'s are independently drawn from $\W_L$, both independent of $Y_i$'s. Then, by the thinning property of Poisson processes, we have 
\[
\sum_{i=1}^\ell \sum_{j=1}^{Y_i}\epsilon_{i,j}\delta_{\hat W_{i,j}} \stackrel{d}{=} \sum_{i=1}^{N_0}\epsilon_i\delta_{\hat W_i},
\]
where it is noted that 
\[
N_0 \stackrel{d}{=} \sum_{i=1}^\ell Y_i.
\]
We thus obtain
\[
\E\bnorm{\sum_{i=1}^\ell(\delta_{\hat W_i}-\W_L)}_{\cF} \leq 4\E\bnorm{\sum_{i=1}^{N_0}\epsilon_i\delta_{\hat W_i}}_{\cF}.
\]
Lastly, we delete the randomness in $\epsilon_i$'s and $\hat W_i$'s. To this end, introduce
\[
N_k := \Big|\Big\{i\leq N_0: \hat W_i=w_k, \epsilon_i=1  \Big\}\Big|~~{\rm and}~~N_k' := \Big|\Big\{i\leq N_0: \hat W_i=w_k, \epsilon_i=-1  \Big\}\Big|.
\]
By construction, we then have
\[
\sum_{i=1}^{N_0}\epsilon_i\delta_{\hat W_i}=\sum_{i=1}^{L}(N_i-N_i')\delta_{w_i}.
\]
It remains to decipher the joint distribution of $N_i,N_i'$'s. Conditional on $N_0=m$, $(N_1,N_1,N_2,N_2',\ldots,N_L, N_L')$ are multinomially distributed with parameters $(m, 1/(2L),\ldots,1/(2L))$. Accordingly, unconditionally, we have that $(N_1,N_1,N_2,N_2',\ldots,N_L, N_L')$ are i.i.d. ${\rm Pois}(\ell/(2L))$. This completes the proof.
\end{proof}

\begin{lemma}[Multiplier inequality]\label{lem:multiplier}  Let $Z_1,\ldots,Z_n$ be a sequence of i.i.d. random values whose support is finite, and $\xi_1,\ldots,\xi_n$ be i.i.d. mean-zero random variables independent of $Z_1,\ldots,Z_n$. We then have 
\[
\E\bnorm{\frac{1}{\sqrt{n}}\sum_{i=1}^n\xi_i\delta_{Z_i}}_{\cF}\leq 2\sqrt{2}\norm{\xi_1}_{2,1}\cdot \max_{k\in[n]}\E\bnorm{\frac{1}{\sqrt{k}}\sum_{i=1}^k \epsilon_i\delta_{Z_i}}_{\cF},
\]
where $\epsilon_j$'s form a Rademacher sequence and, for any random variable $X$,
\[
\norm{X}_{2,1}:=\int_0^{\infty}\sqrt{\Pr(|X|>t)}\d t.
\]
\end{lemma}
\begin{proof}
{\bf Step 1.} First, assume that $\xi_i$'s are symmetrically distributed so that
\[
(\epsilon_1|\xi_1|,\ldots, \epsilon_n|\xi_n|)^\top\stackrel{d}{=} (\xi_1,\ldots,\xi_n)^\top.
\]
Accordingly, we have 
\[
\E\bnorm{\sum_{i=1}^n\xi_i\delta_{Z_i}}_{\cF}=\E\bnorm{\sum_{i=1}^n\epsilon_i|\xi_i|\delta_{Z_i}}_{\cF}=\E\bnorm{\sum_{i=1}^n\epsilon_i\tilde\xi_i\delta_{Z_i}}_{\cF},
\]
where $\tilde\xi_1\geq \tilde\xi_2\geq \cdots \geq \tilde\xi_n\geq \tilde\xi_{n+1}:=0$ are the reordered version of $|\xi_1|,\ldots,|\xi_n|$. Noting that
\[
\tilde\xi_i=\sum_{k=i}^n(\tilde\xi_k-\tilde\xi_{k+1}),
\]
we can continue writing
\begin{align*}
\E\bnorm{\sum_{i=1}^n\epsilon_i\tilde\xi_i\delta_{Z_i}}_{\cF} &= \E\bnorm{\sum_{i=1}^n\epsilon_i\delta_{Z_i}\sum_{k=i}^n(\tilde\xi_k-\tilde\xi_{k+1})}_{\cF}\\
&=\E\bnorm{\sum_{k=1}^n(\tilde\xi_k-\tilde\xi_{k+1})\sum_{i=1}^k\epsilon_i\delta_{Z_i}}_{\cF}\\
&\leq \E\sum_{k=1}^n\sqrt{k}(\tilde\xi_k-\tilde\xi_{k+1}) \max_{k\in[n]}\E\bnorm{\frac{1}{\sqrt{k}}\sum_{i=1}^k \epsilon_i\delta_{Z_i}}_{\cF}.
\end{align*}
Lastly, we bound
\[
\E\sum_{k=1}^n\sqrt{k}(\tilde\xi_k-\tilde\xi_{k+1}) =\E\sum_{k=1}^n\int_{\tilde\xi_{k+1}}^{\tilde\xi_k}\sqrt{k}\d t \leq \int_0^{\infty}\E\sqrt{|\{i:|\xi_i|\geq t\}|}\d t,
\]
where in the inequality we used the fact that
\[
k=\Big|\Big\{i:|\xi_i|>t \Big\}\Big|~~~{\rm for}~t\in(\tilde\xi_{k+1},\tilde\xi_k].
\]
Then,
\[
\int_0^{\infty}\E\sqrt{|\{i:|\xi_i|\geq t\}|}\d t \leq \int_0^{\infty}\sqrt{\E|\{i:|\xi_i|\geq t\}|}\d t=\int_0^{\infty}\sqrt{n\Pr(|\xi_i|\geq t)}\d t
\]
so that
\[
\E\bnorm{\frac{1}{\sqrt{n}}\sum_{i=1}^n\xi_i\delta_{Z_i}}_{\cF}\leq \norm{\xi_1}_{2,1}\cdot \max_{k\in[n]}\E\bnorm{\frac{1}{\sqrt{k}}\sum_{i=1}^k \epsilon_i\delta_{Z_i}}_{\cF}.
\]

{\bf Step 2.} For possibly asymmetric $\xi_i$'s, we employ the symmetrization trick and introduce $\eta_i$'s as an independent copy of $\xi_i$'s. Accordingly, since $\E\xi_i=0$,
\begin{align*}
\E\bnorm{\frac{1}{\sqrt{n}}\sum_{i=1}^n\xi_i\delta_{Z_i}}_{\cF}&=\E\bnorm{\frac{1}{\sqrt{n}}\sum_{i=1}^n(\xi_i-\E\xi_i)\delta_{Z_i}}_{\cF}\\
&=\E\bnorm{\frac{1}{\sqrt{n}}\sum_{i=1}^n(\xi_i-\E\eta_i)\delta_{Z_i}}_{\cF}\\
&\leq \E\bnorm{\frac{1}{\sqrt{n}}\sum_{i=1}^n(\xi_i-\eta_i)\delta_{Z_i}}_{\cF}.
\end{align*}
Now, applying the symmetric version of the multiplier inequality to $(\xi_i-\eta_i)$'s and using Exercise \ref{exe:21norm} (ahead) complete the proof.
\end{proof}

\begin{exercise}\label{exe:21norm} \begin{itemize}
\item[(1)] For any random variable $X$ and any $p>2$, 
\[
\frac12\norm{X}_2\leq \norm{X}_{2,1}\leq \frac{p}{p-2}\norm{X}_p.
\]
\item[(2)] For any random variables $X,Y$, 
\[
\norm{X+Y}_{2,1}^2\leq 4\norm{X}_{2,1}^2+4\norm{Y}_{2,1}^2.
\]
\end{itemize}
\end{exercise}

\begin{proof}[Proof of Theorem \ref{thm:2sample}] {\bf Step 1.} We first establish a finite-dimensional CLT. By Cramer-Wold device, it suffices to consider the marginal case. To this end, notice that, for any $f\in\cF$, using Corollary \ref{cor:two-sample-perm} and the permutation SLLN (Theorem \ref{thm:permutation-slln}), the random variable
\[
\sqrt{m}(\tilde\PP_{\pi,m}-\HH_{m+n})f \Rightarrow \sqrt{1-\gamma} \tilde\H(f-\tilde\H f)^2, \text{ conditional on almost all } \sigma_1,\sigma_2.
\]
This completes the first step.

{\bf Step 2.} Next, let's verify stochastic equi-continuity. Using first Corollary \ref{cor:5hcoi} followed by Lemma \ref{lem:lecam-trick},  we obtain
\[
\E\bnorm{\sqrt{m}(\tilde\PP_{\pi,m}-\HH_{m+n})}_{\cF_{\delta}}\leq 4\E\bnorm{\sum_{i=1}^{m+n}\tilde N_i\delta_{z_i}}_{\cF_{\delta}},
\]
where $\tilde N_i=N_i-N_i'$ and $N_i, N_i'$ are i.i.d. drawn from ${\rm Pois}(m/(m+n))$. Next, by triangle inequality and using Corollary \ref{cor:5hcoi} again, we obtain
\begin{align*}
\E\bnorm{\sum_{i=1}^{m+n}\tilde N_i\delta_{z_i}}_{\cF_{\delta}} &\leq \E\bnorm{\sum_{i=1}^m\tilde N_i\delta_{x_{\sigma_1^{-1}(i)}}}_{\cF_\delta}+  \E\bnorm{\sum_{i=1}^n\tilde N_{m+i}\delta_{y_{\sigma_2^{-1}(i)}}}_{\cF_\delta}\\
&\leq \E\bnorm{\sum_{i=1}^m\tilde N_i\delta_{\hat X_i}}_{\cF_\delta}+  \E\bnorm{\sum_{i=1}^n\tilde N_{m+i}\delta_{\hat Y_i}}_{\cF_\delta},
\end{align*}
where $\hat X_i$'s and $\hat Y_i$'s are independent and i.i.d. drawn from $\P_M$ and $\Q_N$, respectively. Lastly, employing Lemma \ref{lem:multiplier} yields
\begin{align*}
&\E\bnorm{\sum_{i=1}^m\tilde N_i\delta_{\hat X_i}}_{\cF_\delta}+  \E\bnorm{\sum_{i=1}^n\tilde N_{m+i}\delta_{\hat Y_i}}_{\cF_\delta}\\
 \leq& O(1)\cdot \Big(\max_{k\in[m]}\E\bnorm{\frac{1}{\sqrt{k}}\sum_{i=1}^k\epsilon_i\delta_{\hat X_i}}_{\cF_{\delta}}+\max_{k\in[n]}\E\bnorm{\frac{1}{\sqrt{k}}\sum_{i=1}^k\epsilon_i\delta_{\hat Y_i}}_{\cF_\delta} \Big).
\end{align*}
Applying the standard empirical processes Donsker's theorem (noticing that now everything is measurable) to the laws $\P_M$ and $\Q_N$ then completes the proof of the first part.

The conditional part is also true by appealing to Theorems \ref{thm:tol-tala} and \ref{thm:perm-strong-LLN}, and repeat the above calculation conditional on $\sigma_1$ and $\sigma_2$.
\end{proof}

\section{Notes}

This chapter is adapted mostly from Chapters 3.6 and 3.7 of \cite{vaart1996empirical}, and is currently still under construction. For example, %it lacks, obviously, a conditional version of Theorem \ref{thm:2sample}. In addition, 
the current focus lacks results on comparing, e.g., the stochastic behavior of some functionals of the samples. We hope to enrich this chapter in the following months, aligning this chapter more closely with results made in, e.g., \citet[Chapter 3.7]{vaart1996empirical}, \citet[Chapter 17]{romano2005testing}, and \cite{bickel1969distribution} and \cite{chung2013exact}.